\documentclass{amsart}
\usepackage{pgf}
\usepackage{amssymb}
\usepackage{dsfont}
\usepackage{txfonts}
\usepackage{graphicx}
\usepackage{amsmath}
\usepackage[colorlinks=true]{hyperref}
\usepackage{amsthm}
\usepackage{units}

\newtheorem{theo}{Theorem}[section]
\newtheorem{definition}[theo]{Definition}
\newenvironment{defi}{\begin{definition}\rm}{\end{definition}}
\newtheorem{remarque}[theo]{Remark}
\newenvironment{remark}{\begin{remarque}\rm}{\end{remarque}}
\newtheorem{exemple}[theo]{Example}
\newenvironment{ex}{\begin{exemple}\rm}{\end{exemple}}
\newtheorem{lemma}[theo]{Lemma}

\newtheorem{coro}[theo]{Corollary}
\newtheorem{nota}[theo]{Notation}
\newenvironment{notation}{\begin{nota}\rm}{\end{nota}}
\newtheorem{prf1}{\it {Idea of the proof}}

\newtheorem{prf}{\it{Proof}}

\newenvironment{demo}{\begin{prf}\rm}{\hfill$\Box$\end{prf}}

\newtheorem*{theorem*}{Theorem}

\def\varLim@#1#2{%
	\vtop{\m@th\ialign{##\cr
			\hfil$#1\operator@font Lim$\hfil\cr
			\noalign{\nointerlineskip\kern1.5\ex@}#2\cr
			\noalign{\nointerlineskip\kern-\ex@}\cr}}%
}
\def\varinjLim{%
	\mathop{\mathpalette\varLim@{\rightarrowfill@\textstyle}}\nmlimits@
}

\makeatletter
\def\moverlay{\mathpalette\mov@rlay}
\def\mov@rlay#1#2{\leavevmode\vtop{%
		\baselineskip\z@skip \lineskiplimit-\maxdimen
		\ialign{\hfil$\m@th#1##$\hfil\cr#2\crcr}}}
\newcommand{\charfusion}[3][\mathord]{
	#1{\ifx#1\mathop\vphantom{#2}\fi
		\mathpalette\mov@rlay{#2\cr#3}
	}
	\ifx#1\mathop\expandafter\displaylimits\fi}
\makeatother

\newcommand{\bigcupdot}{\charfusion[\mathop]{\bigcup}{\cdot}}

\def\N{\mathbb{ N}}
\def\Z{\mathbb{ Z}}

\DeclareMathOperator{\supp}{Supp}
\DeclareMathOperator{\ord}{ord}

\def\d{\mathbb{ \delta}}

\title{About the algebraic closure of formal power series in several variables.} % This is the full title of the paper
\author{Michel Hickel and Micka\"{e}l Matusinski}

\address{Michel Hickel and Mickaël Matusinski, Univ. Bordeaux, CNRS, Bordeaux INP, IMB, UMR 5251, F-33400 Talence, France}


\begin{document}

\begin{abstract}
	Let $K$ be a field of characteristic zero. We deal with the algebraic closure of the field of fractions of the ring of formal power series $K[[x_1,\ldots,x_r]]$, $r\geq 2$. More precisely, we view the latter as a subfield of an iterated Puiseux series field $\mathcal{K}_r$. On the one hand, given $y_0\in \mathcal{K}_r$ which is algebraic, we provide an  algorithm that reconstructs the space of all polynomials which annihilates $y_0$ up to a certain order (arbitrarily high). On the other hand, given a polynomial $P\in K[[x_1,\ldots,x_r]][y]$ with simple roots, we derive a closed form formula for the coefficients of a root $y_0$ in terms of the coefficients of $P$ and a fixed initial part of $y_0$.
\end{abstract}

\maketitle
\tableofcontents

%(respectivement convergentes $k{x_1,x_2}$)
\section{Introduction.}
Let $K$ be a field of characteristic zero and $\overline{K}$ its algebraic closure. Let $\underline{x}:=(x_1,\ldots,x_r)$ be an $r$-tuple of indeterminates where $r\in\Z$, $r\geq 2$. Let $K[\underline{x}]$ and $K[[\underline{x}]]$ denote respectively the domains of polynomials and of formal power series in $r$ variables with coefficients in $K$, and $K(\underline{x})$ and $K((\underline{x}))$ their fraction fields. Both fields embed naturally into $K((x_r))((x_{r-1}))\cdots((x_1))$, the latter being naturally endowed with the lexicographic valuation in the variables $(x_1,\ldots,x_r)$ (see Section \ref{section:preliminaries}).
By iteration of the classical Newton-Puiseux theorem (see e.g. \cite[Theorem 3.1]{walker_alg-curves} and \cite[p. 314, Proposition]{rib-vdd_ratio-funct-field}), one can derive a description of an algebraic closure of $K((x_r))((x_{r-1}))\cdots((x_1))$ in terms of iterated fractional Laurent series (see \cite[Theorem 3]{rayner_puiseux-multivar}\cite[p.151]{sathaye:newt-puiseux-exp_abh-moh-semigr}):
\begin{theorem*}
	The following field, where $L$ ranges over the finite extensions of $K$ in $\overline{K}$: $$\mathcal{L}_r:= \displaystyle\varinjlim_{p\in\mathbb{N}^*} \displaystyle\varinjlim_L L((x_r^{1/p}))((x_{r-1}^{1/p}))\cdots ((x_1^{1/p}))$$  is the algebraic closure of $K((x_r))((x_{r-1}))\cdots((x_1))$. 
\end{theorem*}
Within this framework, there are several results concerning those iterated fractional Laurent series which are solutions  of polynomial equations with coefficients either in $K(\underline{x})$ or $K((\underline{x}))$. More precisely, the authors provide necessary constraints on the supports of such a series (see \cite[Theorem 3.16]{mcdonald_puiseux-multivar}, \cite[ Th\'eor\`eme 2]{gonzalez-perez_singul-quasi-ord}, \cite[Theorem 13]{soto-vicente:polyhedral-cones}  \cite[Theorem 1]{aroca-ilardi:puiseux-multivar}, \cite[Theorem 1]{soto-vicente_puiseux-multivar}). More recently, Aroca, Decaup and Rond study more precisely the support of Laurent-Puiseux power series which are algebraic over $K[[\underline{x}]]$ (with certain results for $K$ of positive characteristic) \cite{aroca-rond:support-alg-series,aroca-decaup-rond:support-alg-laurent-series}. As asserted in \cite[2nd Theorem in p.56]{hickel-matu:puiseux-alg-multivar}, one can prove the following result (see the proof in Section \ref{section:preliminaries}), 
which could also be derived from the methods in  \cite[Theorem 1]{soto-vicente_puiseux-multivar} or \cite[Theorem 1]{aroca-ilardi:puiseux-multivar}:
\begin{theorem*} 
	The following field $\mathcal{K}_r$, where $L$ ranges over the finite extensions of $K$ in $\overline{K}$, is an algebraically closed extension of $K(\underline{x})$ and $K((\underline{x}))$ in $\mathcal{L}_r$:\\ 
	$$\mathcal{K}_r\ :=\ \displaystyle \varinjlim_{(p,\underline{q})\in \mathbb{N}^*\times\mathbb{N}^{r-1}} \displaystyle\varinjlim_L\ \ 
	L\left(\left(\,\left( \displaystyle\frac{x_1}{x_2^{q_1}}\right)^{1/p},\ldots,  \left( \displaystyle\frac{x_{r-1}}{x_r^{q_{r-1}}}\right)^{1/p} ,x_r^{1/p}\right)\right).$$
	%$ \ \ \ \ \mathcal{K}_r=\\ \displaystyle\varinjlim_{(p,\underline{q})\in \mathbb{N}^*\times\mathbb{N}^{r-1}}  \displaystyle\varinjlim_L
	%\displaystyle\varinjlim_{\underline{n}\in \mathbb{Z}^r}\ \  \left(\frac{x_1}{x_2^{q_1}}\right)^{n_1/p}\cdots \left(\frac{x_{r-1}}{x_r^{q_{r-1}}}\right)^{n_{r-1}/p} x_r^{n_r/p}\, L\left[\left[\left(\frac{x_1}{x_2^{q_1}}\right)^{1/p},\ldots, \left(\frac{x_{r-1}}{x_r^{q_{r-1}}}\right)^{1/p} ,x_r^{1/p}\right]\right]^*$\\
	
	%\noindent where $L\left[\left[\displaystyle\left(\frac{x_1}{x_2^{q_1}}\right)^{1/p},\ldots, \left(\displaystyle\frac{x_{r-1}}{x_r^{q_{r-1}}}\right)^{1/p} ,x_r^{1/p}\right]\right]^*$ denotes the group of invertible elements in\\ $L\left[\left[\left(\displaystyle\frac{x_1}{x_2^{q_1}}\right)^{1/p},\ldots, \left(\displaystyle\frac{x_{r-1}}{x_r^{q_{r-1}}}\right)^{1/p} ,x_r^{1/p}\right]\right]$.
\end{theorem*}

%In fact, $\mathcal{K}_r$ can be described as:
%$$ \displaystyle\varinjlim_{(p,\underline{q})\in \mathbb{N}^*\times\mathbb{N}^{r-1}}  \displaystyle\varinjlim_L
%\displaystyle\varinjlim_{\underline{n}\in \mathbb{Z}^r}\ \  \left(\frac{x_1}{x_2^{q_1}}\right)^{n_1/p}\cdots \left(\frac{x_{r-1}}{x_r^{q_{r-1}}}\right)^{n_{r-1}/p} x_r^{n_r/p}.L\left[\left[\left(\frac{x_1}{x_2^{q_1}}\right)^{1/p},\ldots, \left(\frac{x_{r-1}}{x_r^{q_{r-1}}}\right)^{1/p} ,x_r^{1/p}\right]\right],$$
%(see for instance Lemma \ref{lemme:Kr}).

Let $\tilde{y}_0\in \mathcal{K}_r$ and $\tilde{f},\tilde{g}\in L\left[\left[\left(\frac{x_1}{x_2^{q_1}}\right)^{1/p},\ldots, \left(\frac{x_{r-1}}{x_r^{q_{r-1}}}\right)^{1/p} ,x_r^{1/p}\right]\right]$ such that $\tilde{y}_0=\frac{\tilde{f}}{\tilde{g}}$. Let $\underline{\alpha}$ be the lexicographic valuation of $\tilde{g}$ (where it is understood that the valuation of $x_i^{1/p}$ is equal to $1/p$ times the valuation of $x_i$). Denote $\tilde{g}=a\underline{x}^{\underline{\alpha}}(1-\varepsilon)$ with $\varepsilon$ having positive valuation. We expand:
$$ \tilde{y}_0=\frac{\tilde{f}}{\tilde{g}}=\tilde{f}\,a^{-1}\underline{x}^{-\underline{\alpha}}\sum_{k\in\N}\varepsilon^k$$
as a generalized power series $\displaystyle\sum_{\underline{n}\in(\Z^r,\leq_{\textrm{lex}})} c_{\underline{n}/p}\underline{x}^{\underline{n}/p}$ (the latter is well defined by \cite[Theorem 3.4]{neumann:ord-div-rings}). We set: $$\mathrm{Supp} \left(\displaystyle\sum_{\underline{n}\in(\Z^r,\leq_{\textrm{lex}})} c_{\underline{n}/p}\underline{x}^{\underline{n}/p}\right):=\left\{\frac{1}{p}\underline{n}\in\left(\frac{1}{p}\Z^r,\leq_{\textrm{lex}}\right)\ |\ c_{\underline{n}/p}\neq 0\right\}.$$

Let us call the elements of $\mathcal{K}_r$ \emph{rational polyhedral Puiseux series}  (since one can observe that the support with respect to the variables $x_i$'s of such a series is included in the translation of some rational convex polyhedral cone). We are interested in those rational polyhedral Puiseux series that are algebraic over $K((\underline{x}))$, say the rational polyhedral Puiseux series which verify a polynomial equation $\tilde{P}(\underline{x},y)=0$ with coefficients which are themselves formal power series in $\underline{x}$:  $\tilde{P}(\underline{x},y)\in K[[\underline{x}]][y]\setminus\{0\}$. Let us call such a series \emph{algebroid}. If such a series $\tilde{y}_0$ admits a vanishing polynomial of degree at most $d$ in $y$, we will say that $\tilde{y}_0$ is \emph{algebroid of degree bounded by $d$}.

More precisely, we extend our previous work on algebraic (over $K(\underline{x})$) Puiseux series in several variables \cite{hickel-matu:puiseux-alg-multivar}, by dealing with the following analogous questions:\vspace{0.2cm}

\noindent$\bullet$ \textbf{Reconstruction of pseudo-vanishing polynomials for a given algebroid rational polyhedral Puiseux series.}\vspace{0.2cm}

In this part, for simplicity reasons, we will assume that $K$ is algebraically closed. For $\tilde{Q}(\underline{x},y)\in K[[\underline{x}]][y]$  a nonzero polynomial,  the \emph{$(\underline{x})$-adic order} of $\tilde{Q}$ is the maximum of the integers $k$ such that $\tilde{Q}(\underline{x},y)\in (\underline{x})^kK[[\underline{x}]][y]$ where $(\underline{x})$ denotes the ideal of $K[[\underline{x}]]$ generated by $x_1,\ldots,x_r$.

We consider $\tilde{y}_0=\frac{\tilde{f}}{\tilde{g}}$ with $\tilde{f},\tilde{g}\in K\left[\left[\left(\frac{x_1}{x_2^{q_1}}\right)^{1/p},\ldots, \left(\frac{x_{r-1}}{x_r^{q_{r-1}}}\right)^{1/p} ,x_r^{1/p}\right]\right]$ algebroid of degree bounded by $d$. For an arbitrarily large valuation $l\in \N$, we provide an algorithm which computes polynomials $\tilde{Q}(\underline{x},y)\in K[[\underline{x}]][y]$ such that the expansion of $\tilde{Q}(\underline{x},\tilde{y}_0)\in \mathcal{K}_r$ as a rational polyhedral Puiseux series has  valuation greater than $l$.

More precisely, let us denote $\zeta_i:=\left(\frac{x_i}{x_{i+1}^{q_i}}\right)^{1/p}$ for $i=1,\ldots,r-1$, and $\zeta_r:=x_r^{1/p}$. We suppose that for any $k\in\N$, one can compute all the coefficients of $\underline{\zeta}^{\underline{n}}$ with $n_1+\cdots+n_r\leq k$ in $\tilde{f}$ and $\tilde{g}$. Moreover, we assume that the lexicographic valuations with respect to $\underline{\zeta}$ of $\tilde{f}$ and $\tilde{g}$ are given.

\begin{theo}\label{theo:reconstr}
	Let $d\in\N^*$ and $\tilde{\nu}_0\in\N$. Let $\tilde{y}_0\in \mathcal{K}_r$ be algebroid of degree bounded by $d$.  We assume that there is a vanishing polynomial $\tilde{P}$  of degree bounded by $d$ and of $(\underline{x})$-adic order bounded by $\tilde{\nu}_0$. We consider formal power series $\tilde{f},\tilde{g}\in K\left[\left[\left(\frac{x_1}{x_2^{q_1}}\right)^{1/p},\ldots, \left(\frac{x_{r-1}}{x_r^{q_{r-1}}}\right)^{1/p} ,x_r^{1/p}\right]\right]$ such that $\tilde{y}_0=\frac{\tilde{f}}{\tilde{g}}$. Let $\underline{\beta}=(\beta_1,\ldots,\beta_r)$ be the lexicographic valuation of $\tilde{f}\tilde{g}$ with respect to the variables ${\zeta}_i:=\left(\frac{x_i}{x_{i+1}^{q_i}}\right)^{1/p}$, ${\zeta}_r:=x_r^{1/p}$, and $q_i':=q_i+\beta_{i+1}+1$ for $i=1,\ldots,r-1$.
	We set:
	$$\begin{array}{lccl}
		\tilde{L}:&\Z^r&\rightarrow &\Z\\
		&(n_1,\ldots,n_r)&\mapsto &n_r+q'_{r-1}n_{r-1}+q'_{r-1}q'_{r-2}n_{r-2}+\cdots+q'_{r-1}q'_{r-2}\cdots q'_1n_1.
	\end{array}$$
	%\textcolor{red}{NON pas de $p$, deja pris en compte ci-dessous quand on considere $(1/p)\underline{n}$ puis $L(\underline{n})$}
	The algorithm described in Section \ref{sect:reconstr-algebroid} provides for any $\nu\in\N$ a parametric description of the space of all the polynomials $\tilde{Q}_\nu(\underline{x},y)\in K[[\underline{x}]][y]$ with $\deg_y\tilde{Q}_\nu\leq d$ and of $(\underline{x})$-adic order bounded by $\tilde{\nu}_0$ such that, for any $\frac{1}{p}\underline{n}=\frac{1}{p}(n_1,\ldots,n_r)\in \mathrm{Supp}\,\tilde{Q}_\nu(\underline{x},\tilde{y}_0) $, one has:
	$$ \tilde{L}(\underline{n})\geq \nu.$$
\end{theo}

Note that the condition $ \tilde{L}(\underline{n})\geq \nu$ for $\frac{1}{p}\underline{n}\in \mathrm{Supp}\,\tilde{Q}_\nu(\underline{x},\tilde{y}_0) $ implies that  \emph{infinitely }many coefficients of $\tilde{Q}_\nu(\underline{x},\tilde{y}_0) $ vanish since $\underline{n}\in \Z^r$. With more information on $\tilde{y}_0$, we can use other linear forms $\tilde{L}$, see Theorem \ref{theo:reconstr-optimise}.\\

\noindent$\bullet$ \textbf{Description of the  coefficients of an algebroid rational polyhedral Puiseux series in terms of the coefficients of a vanishing polynomial.} \vspace{0.2cm}

Now, let a polynomial $\tilde{P}(\underline{x},y)\in K[[\underline{x}]][y]$ with only simple roots and  a root $\tilde{y}_0\in \mathcal{K}_r$ be given.  Up to a change of coordinates (see Section \ref{section:preliminaries}), we reduce to the case of a polynomial $ {P}(\underline{u},y)\in K[[\underline{u}]][y]$ whose support has constraints (see Lemma \ref{lemma:support-equ}), and a simple root $y_0\in L[[\underline{u}]]$ (where $[L:K]<\infty$). In Theorem \ref{theo:FS} and Corollary \ref{coro:FS}, we provide a closed form formula for the coefficients of $y_0$ in terms of the coefficients of $P$ and the coefficients of a fixed initial part of $y_0$. This is obtained as a consequence of a generalization of the multivariate Flajolet-Soria formula for Henselian equations (\cite{flajolet-soria:coeff-alg-series, sokal:implicit-function}), see Theorem \ref{theo:formule-FS}.\vspace{0,2cm}

Our article is organized as follows. In Section \ref{section:preliminaries}, we prove a monomialization lemma (Lemma \ref{lemme:monomialisation}) which is a key to reduce to the case of formal power series annihilating a polynomial whose support has constraints (Lemma \ref{lemma:support-equ}). This is done by a change of variable (\ref{equ:eclt1}) corresponding to the lexicographic valuation. Moreover, we distinguish two sets $\underline{s}$ and $\underline{t}$ of variables and we show that our series $y_0$ can be expanded as $y_0=\sum_{\underline{n}}c_{\underline{n}}(\underline{s}) \underline{t}^{\underline{n}}$ where the $c_{\underline{n}}(\underline{s})\in K[[\underline{s}]]$ are algebraic power series (see Lemma \ref{lemma:algebraicity}) of bounded degree (see Lemma \ref{lemma:algebraicity}). Section \ref{sect:depth} is devoted to the proof of the nested depth lemma (Theorem \ref{propo:nested}). It is used in the subsequent sections to ensure the finiteness of the computations. We use elementary properties on B\'{e}zout's identity and the resultant of two polynomials. In Section \ref{section:Wilc}, we show how to reconstruct all the polynomials of given bounded degrees which vanish at given several algebraic power series. This is based on Section \ref{sect:depth} and our previous work on algebraic multivariate power series \cite{hickel-matu:puiseux-alg}. In Section \ref{sect:reconstr-algebroid}, we prove our first main result, Theorem \ref{theo:reconstr} and its variant Theorem \ref{theo:reconstr-optimise}. Sections \ref{section:hensel} and \ref{section:flajo-soria} are devoted to our second question. In Section \ref{section:hensel}, we study what we call strongly reduced Henselian equations (see Definition \ref{defi:equ-hensel-red}) and prove a generalisation of the multivariate Flajolet-Soria formula (see Theorem \ref{theo:formule-FS}). In Section \ref{section:flajo-soria}, we prove how to reduce to the case of a strongly reduced Henselian equation (see Theorem \ref{theo:FS}) and, in the case of an equation with only simple roots, we derive a closed form formula for the coefficients of a solution $y_0$ in terms of the coefficients of the equation and of a bounded initial part of $y_0$ (see Corollary \ref{coro:FS}).

\section{Preliminaries}\label{section:preliminaries} 

Let us denote $\mathbb{N}:=\mathbb{Z}_{\geq 0}$ and $\mathbb{N}^*:=\mathbb{N}\setminus\{0\}=\mathbb{Z}_{>0}$. 
%We will consider, unless specified, that the groups $\mathbb{Z}^r$ and $\mathbb{Q}^r$ are ordered lexicographically. 
For any set $\mathcal{E}$, we  denote by $|\mathcal{E}|$ its cardinal. We systematically write the vectors using underlined letters, e.g. $\underline{x}:=(x_1,\ldots,x_r)$, $\underline{n}:=(n_1,\ldots,n_r)$, and in particular $\underline{0}:=(0,\ldots,0)$. Moreover, $\underline{x}^{\underline{n}}:=x_1^{n_1}\cdots x_r^{n_r}$. The floor function will be denoted by $\lfloor q \rfloor$ for $q\in\mathbb{Q}$.

For a polynomial $P(y)=\sum_{i=0}^d a_iy^i$ with coefficients $a_i$ in a domain and $a_d\neq 0$, we consider that its discriminant $\Delta_P$ is equal to the resultant of $P$ and $\displaystyle\frac{\partial P}{\partial y}$ (instead of the more usual convention $\Delta_P=\displaystyle\frac{(-1)^{d(d-1)/2}}{a_d}\mathrm{Res}\left(P,\displaystyle\frac{\partial P}{\partial y}\right)$).

\begin{notation}\label{nota:FS}
	For any sequence of nonnegative  integers $\underline{m}=\displaystyle\left(m_{\underline{i},j}\right)_{\underline{i},j}$ \emph{with finite support} and any sequence of scalars $\underline{a}=\displaystyle\left(a_{\underline{i},j}\right)_{\underline{i},j}$ indexed by  $\underline{i}\in\mathbb{Z}^r$ and $j\in\mathbb{N}$, we set:
	\begin{itemize}
		\item $\underline{m}!:=\displaystyle\prod_{\underline{i},j}m_{\underline{i},j}!$;
		\item $\underline{a}^{\underline{m}}:=\displaystyle\prod_{\underline{i},j}{a_{\underline{i},j}}^{m_{\underline{i},j}}$;
		\item  $\ |\underline{m}|:=\displaystyle\sum_{\underline{i},j}m_{\underline{i},j}$, $\ \ ||\underline{m}||:= \displaystyle\sum_{\underline{i},j}m_{\underline{i},j}\, j\in \N\ $ and $\ g(\underline{m}) := \displaystyle\sum_{\underline{i},j}m_{\underline{i},j}\, \underline{i}\in \Z^r$.
	\end{itemize}
	In the case where $\underline{k}=(k_0,\ldots,k_l)$, we set %$\underline{k}!:=\displaystyle\prod_{i=0}^lk_i!$, $|\underline{k}|:=\displaystyle\sum_{i=0}^lk_i$ and 
	$  \|\underline{k}\| :=\displaystyle\sum_{j=0}^l{k_j}\,j$. In the case where $\underline{k}=(k_{\underline{i}})_{\underline{i}\in \Delta}$ where $\Delta$ is a finite subset of $\mathbb{Z}^r$, we set $g(\underline{k}):=\displaystyle\sum_{\underline{i}\in \Delta}{k_i}\,\underline{i}$.
\end{notation}

We will consider the following orders on tuples in $\mathbb{Z}^r$:
\begin{description}
	\item[The lexicographic order] $\underline{n} \leq_{\textrm{lex}} \underline{m}$ $:\Leftrightarrow$ $n_1<m_1$ or $(n_1=m_1\ \textrm{ and } n_2<m_2)$ or $\cdots$ or $(n_1=m_1,\ n_2=m_2,\ \ldots \ \textrm{ and } n_r<m_r)$. 
	\item[The graded lexicographic order] $\underline{n} \leq_{\textrm{grlex}} \underline{m}$ $:\Leftrightarrow$ $|\underline{n }|<|\underline{m}|$ or $(|\underline{n }|=|\underline{m}|\ \textrm{ and } \underline{n} \leq_{\textrm{lex}} \underline{m})$. 
	\item[The product (partial) order] $\underline{n} \leq \underline{m}$  $:\Leftrightarrow$ $n_1\leq m_1 \ \textrm{ and } n_2\leq m_2\ \cdots \ \textrm{ and } n_r\leq m_r$.
\end{description}
Note that we will apply also the lexicographic order on $\mathbb{Q}^r$. Similarly, one has the \textbf{anti-lexicographic order} denoted by $\leq_{\textrm{alex}}$.\vspace{0.1cm}\\
\noindent Considering the restriction of $\leq_{\textrm{grlex}}$ to $\N^r$ (for which $\N^r$ has order type $\omega$), we denote by $S(\underline{k})$ (respectively  $A(\underline{k})$ for $\underline{k}\neq 0$), the \textbf{successor element} (respectively  the \textbf{predecessor element}) of $\underline{k}$ in $(\mathbb{N}^r,\leq_{\textrm{grlex}})$.\\

Given a variable $x$ and a field $K$, we call \textbf{Laurent series} in $x$ with coefficients in $ K$ any formal series $\sum_{n\geq n^0}c_nx^n$ for some $n^0\in \Z$ and $c_n\in K$ for any $n$. They consist in a field, which is identified with the fraction field $K((x))$ of $K[[x]]$.
To view the fields $K(\underline{x})$ and $K((\underline{x}))$ as embedded into $K((x_r))((x_{r-1}))\cdots((x_1))$ means that the rational fractions or formal meromorphic fractions can be represented as iterated formal Laurent series, i.e. Laurent series in $x_1$ whose coefficients are Laurent series in $x_2$, whose coefficients... etc. This corresponds to the following approach. As in \cite{rayner_puiseux-multivar, sathaye:newt-puiseux-exp_abh-moh-semigr}, we identify $K((x_r))((x_{r-1}))\cdots((x_1))$ with the field of generalized power series (in the sense of \cite{hahn:nichtarchim}, see also \cite{rib:series-fields-alg-closed})  with coefficients in $K$ and exponents in $\mathbb{Z}^r$ ordered lexicographically, usually denoted by $K\left(\left(X^{\mathbb{Z}^r}\right)\right)^{\mathrm{lex}}$. By definition, such a generalized series is a formal expression $s=\displaystyle\sum_{\underline{n}\in \mathbb{Z}^r}c_{\underline{n}}X^{\underline{n}}$ (say a map $\mathbb{Z}^r\rightarrow K$) whose support $\supp(s):=\{\underline{n}\in \mathbb{Z}^r\ |\ c_{\underline{n}}\neq 0\}$ is well-ordered. The field $K\left(\left(X^{\mathbb{Z}^r}\right)\right)^{\mathrm{lex}}$ comes naturally equipped with the following valuation of rank $r$:
$$\begin{array}{lccl}v_{\underline{x}}:&K\left(\left(X^{\mathbb{Z}^r}\right)\right)^{\mathrm{lex}}&\rightarrow&(\mathbb{Z}^r\cup\{\infty\},\leq_{\textrm{lex}})\\
	&s\neq 0 &\mapsto& \min(\supp(s))\\
	&0&\mapsto & \infty
\end{array}$$
The identification of $K\left(\left(X^{\mathbb{Z}^r}\right)\right)$ and $K((x_r))((x_{r-1}))\cdots((x_1))$ reduces to the identification $$X^{(1,0,\ldots,0)}=x_1\ ,\ \ X^{(0,1,\ldots,0)}=x_2 \ ,\ \ \ldots\ ,\  \  X^{(0,\ldots,0,1)}=x_r.$$
By abuse of terminology, we call  $K\left(\left(X^{\mathbb{Z}^r}\right)\right)^{\mathrm{lex}}$ or $K((x_r))((x_{r-1}))\cdots((x_1))$ the field of \textbf{(iterated) multivariate Laurent series}.
Note also that this corresponds to the fact that the power series in the rings $K[\underline{x}]$ and $K[[\underline{x}]]$ are viewed as expanded along $(\mathbb{Z}^r,\leq_{\textrm{lex}})$. \\
Similarly, the field $\mathcal{L}_r$ is a union of fields of generalized series $L\left(\left(X^{(\mathbb{Z}^r)/p}\right)\right)^{\mathrm{lex}}$ and comes naturally equipped with the valuation of rank $r$:
$$\begin{array}{lccl}v_{\underline{x}}:&\mathcal{L}_r&\rightarrow&(\mathbb{Q}^r\cup\{\infty\},\leq_{\textrm{lex}})\\
	&s\neq 0&\mapsto& \min(\supp(s))\\
	&0&\mapsto & \infty.
\end{array}$$

We will need another representation of the elements in $K(\underline{x})$ and $K\left(\left(\underline{x}\right)\right)$, via the embedding of these fields into the field $K\left(\left(X^{\mathbb{Z}^r}\right)\right)^{\mathrm{grlex}}$ with valuation:
$$\begin{array}{lccl}w_{\underline{x}}:&K\left(\left(X^{\mathbb{Z}^r}\right)\right)^{\mathrm{grlex}}&\rightarrow&(\mathbb{Z}^r\cup\{\infty\},\leq_{\textrm{grlex}})\\
	&s\neq 0&\mapsto& \min(\supp(s))\\
	&0&\mapsto & \infty.
\end{array}$$
and the same identification:
$$X^{(1,0,\ldots,0)}=x_1\ ,\ \ X^{(0,1,\ldots,0)}=x_2 \ ,\ \ \ldots\ ,\  \  X^{(0,\ldots,0,1)}=x_r.$$

For a polynomial $P(y)=\displaystyle\sum_{j=0}^da_jy^j\in K\left(\left(X^{\mathbb{Z}^r}\right)\right)^{\mathrm{grlex}}[y]$, we denote:
$$ w_{\underline{x}}(P(y)):=\min_{j=0,\ldots,d}\{w_{\underline{x}}(a_j)\}.$$

We will also use the following notations to keep track of the variables used to write the monomials. Given a ring $R$, we denote by $ R\left(\left(x_1^{\mathbb{Z}},\ldots,x_r^\mathbb{Z}\right)\right)^{\textrm{lex}}$ and $ R\left(\left(x_1^{\mathbb{Z}},\ldots,x_r^\mathbb{Z}\right)\right)^{\textrm{grlex}}$ the corresponding rings of generalized series $\displaystyle\sum_{\underline{n}\in \mathbb{Z}^r}c_{\underline{n}} \underline{x}^{\underline{n}}$ with coefficients $c_{\underline{n}}$ in $R$. Accordingly, let us write $ R\left(\left(x_1^{\mathbb{Z}},\ldots,x_r^\mathbb{Z}\right)\right)^{\textrm{lex}}_{\textrm{Mod}}$ and $ R\left(\left(x_1^{\mathbb{Z}},\ldots,x_r^\mathbb{Z}\right)\right)^{\textrm{grlex}}_{\textrm{Mod}}$ the subrings of series whose actual exponents are all bounded by below by some constant for the product order. Note that these subrings are both isomorphic to the ring $\displaystyle\bigcup_{\underline{n}\in\mathbb{Z}^r}\underline{x}^{\underline{n}}R[[\underline{x}]] $.
Let us write also $R\left(\left(x_1^{\mathbb{Z}},\ldots,x_r^\mathbb{Z}\right)\right)^{\textrm{lex}}_{\geq_{\textrm{lex}}\underline{0}}$ and $R\left(\left(x_1^{\mathbb{Z}},\ldots,x_r^\mathbb{Z}\right)\right)^{\textrm{grlex}}_{\geq_{\textrm{grlex}}\underline{0}}$ the subrings of series $s$ with $v_{\underline{x}}(s)\geq_{\textrm{lex}}\underline{0}$, respectively $w_{\underline{x}}(s)\geq_{\textrm{grlex}}\underline{0}$. \\

\begin{lemma}[Monomialization Lemma]\label{lemme:monomialisation}
	Let $f$ be non zero in $K[[\xi_1,\ldots,\xi_r]]$. There exists $\rho_1,\ldots,\rho_{r-1}\in \mathbb{N}$ such that, if we set
	\begin{equation}\label{equ:transfo}
		\left\{\begin{array}{lcl}
			\eta_1&:=&\displaystyle\frac{\xi_1}{{\xi_2}^{\rho_1}}\\
			&\vdots&\\
			\eta_{r-1}&:=&\displaystyle\frac{\xi_{r-1}}{{\xi_{r}}^{\rho_{r-1}}}\\
			\eta_r&:=&\xi_r
		\end{array}\right.
	\end{equation} 
	then $f(\xi_1,\ldots,\xi_r)=\underline{\eta}^{\underline{\alpha}}g(\eta_1,\ldots,\eta_r)$ where $\underline{\alpha}\in\mathbb{N}^r$ and $g$ is an invertible element of $K[[\eta_1,\ldots,\eta_r]]$. Moreover, for all $i=1,\ldots,r-1$, $\rho_i\leq 1+\beta_{i+1}$ where $\underline{\beta}:=v_{\underline{\xi}}(f)$.
\end{lemma}
\begin{demo}
	
	Let us write $f=\underline{\xi}^{\underline{\beta}}\,h$ where $\underline{\beta}=v_{\underline{\xi}}(f)$ and $h\in K\left(\left(\xi_1^{\mathbb{Z}},\ldots,\xi_r^\mathbb{Z}\right)\right)^{\textrm{lex}}_{\geq_{\textrm{lex}}\underline{0},\,\textrm{Mod}}$ with  $v_{\underline{\xi}}(h)=\underline{0}$. Note that $h$ can be written as $h=h_0+h_1$ where $h_0\in K\left(\left(\xi_2^{\mathbb{Z}},\ldots,\xi_r^\mathbb{Z}\right)\right)^{\textrm{lex}}_{\geq_{\textrm{lex}}\underline{0},\,\textrm{Mod}}$ with $v_{\underline{\xi}}(h_0)=\underline{0}$, and $h_1\in \xi_1K[[\xi_1]]\left(\left(\xi_2^{\mathbb{Z}},\ldots,\xi_r^\mathbb{Z}\right)\right)^{\textrm{lex}}_{\textrm{Mod}}$. If $h_1\in K[[\xi_1]]\left(\left(\xi_2^{\mathbb{Z}},\ldots,\xi_r^\mathbb{Z}\right)\right)^{\textrm{lex}}_{\geq_{\textrm{lex}}\underline{0},\,\textrm{Mod}}$, then we set $\rho_1=0$. Otherwise, let $\rho_1$ be the smallest positive integer such that:
	$$ \rho_1\geq \sup\{1\ ;\ (1-m_2)/m_1,\ \underline{m}\in\textrm{supp}\ h_1\}.$$
	Note that, since $m_1\geq 1$ and $m_2\geq -\beta_2$, we have that $\rho_1\leq 1+\beta_2$. We also remark that the supremum is achieved for $0\geq m_2\geq -\beta_2$ and $1+\beta_2 \geq m_1\geq 1$.
	Let $\eta_1:=\xi_1/{\xi_2}^{\rho_1}$. For every monomial in $h_1$, one has $\xi_1^{m_1}\xi_2^{m_2}\ldots \xi_r^{m_r}=\eta_1^{m_1}\xi_2^{m_2+\rho_1m_1}\ldots \xi_r^{m_r}$. Hence, $m_2+\rho_1m_1\geq 1$ by definition of $\rho_1$. So $(m_2+\rho_1m_1,\ldots,m_r)>_{\textrm{lex}}\underline{0}$, meaning that $h_1\in K[[\eta_1]]\left(\left(\xi_2^{\mathbb{Z}},\ldots,\xi_r^\mathbb{Z}\right)\right)^{\textrm{lex}}_{\geq_{\textrm{lex}}\underline{0},\, \textrm{Mod}}$ and that  $v(h_1)>_{\textrm{lex}}\underline{0}$ where here $v$ is the lexicographic valuation with respect to the variables $(\eta_1,\xi_2,\ldots,\xi_r)$. So $h\in K[[\eta_1]]\left(\left(\xi_2^{\mathbb{Z}},\ldots,\xi_r^\mathbb{Z}\right)\right)^{\textrm{lex}}_{\geq_{\textrm{lex}}\underline{0},\, \textrm{Mod}}$ and $v(h)=\underline{0}$. Note that the exponents $m_3,\ldots, m_r$ remain unchanged in the support of $h$. \\
	Suppose now that we have obtained $h\in  K[[\eta_1,\ldots,\eta_p]]\left(\left(\xi_{p+1}^{\mathbb{Z}},\ldots,\xi_r^\mathbb{Z}\right)\right)^{\textrm{lex}}_{\geq_{\textrm{lex}} \underline{0},\,\textrm{Mod}}$ and that $v(h)=\underline{0}$ where $v$ is now the lexicographic valuation with respect to the variables\\
	$(\eta_1,\ldots,\eta_p,\xi_{p+1},\ldots,\xi_r)$. The induction step is similar to the initial one. As before, let us write $h=h_0^{(p+1)}+h_1^{(p+1)}$ where  $h_0^{(p+1)}\in K[[\eta_1,\ldots,\eta_p]]\left(\left(\xi_{p+2}^{\mathbb{Z}},\ldots,\xi_r^\mathbb{Z}\right)\right)^{\textrm{lex}}_{\geq_{\textrm{lex}} \underline{0},\,\textrm{Mod}}$ with $v(h_0^{(p+1)})=\underline{0}$, and $$h_1^{(p+1)}\in \xi_{p+1}K[[\eta_1,\ldots,\eta_p,\xi_{p+1}]]\left(\left(\xi_{p+2}^{\mathbb{Z}},\ldots,\xi_r^\mathbb{Z}\right)\right)^{\textrm{lex}}_{\textrm{Mod}}.$$ If 	
	$$h_1^{(p+1)}\in K[[\eta_1,\ldots,\eta_p,\xi_{p+1}]]\left(\left(\xi_{p+2}^{\mathbb{Z}},\ldots,\xi_r^\mathbb{Z}\right)\right)^{\textrm{lex}}_{\geq_{\textrm{lex}}\underline{0},\,\textrm{Mod}},$$ then we set $\rho_{p+1}=0$.	Otherwise, let $\rho_{p+1}$ be the smallest positive integer such that:
	$$ \rho_{p+1}\geq \sup\left\{1\ ;\ (1-m_{p+2})/m_{p+1},\ \underline{m}\in\textrm{supp}\ h_1^{(p+1)}\right\}.$$
	Note that, since $m_{p+1}\geq 1$ and $m_{p+2}\geq -\beta_{p+2}$ (since these exponents $m_{p+2}$ remained unchanged until this step), we have that $\rho_{p+1}\leq 1+\beta_{p+2}$.	 If we set $\eta_{p+1}:=\xi_{p+1}/\xi_{p+2}^{\rho_{p+1}}$, then $h\in K[[\eta_1,\ldots,\eta_{p+1}]]\left(\left(\xi_{p+2}^{\mathbb{Z}},\ldots,\xi_r^\mathbb{Z}\right)\right)^{\textrm{lex}}_{\geq_{\textrm{lex}}\underline{0},\, \textrm{Mod}}$ and $v(h)=\underline{0}$ (where $v$ is now the lexicographic valuation with respect to the variables $(\eta_1,\ldots,\eta_{p+1},\xi_{p+2},\ldots,\xi_r)$).\\
	By iteration of this process, we obtain that $h \in K[[\eta_1,\ldots,\eta_{r-1}]]\left(\left(\xi_r^\mathbb{Z}\right)\right)^{\textrm{lex}}_{\geq_{\textrm{lex}}\underline{0},\, \textrm{Mod}}$  and $v(h)=\underline{0}$ (where $v$ is now the lexicographic valuation with respect to the variables $(\eta_1,\ldots,\eta_{r-1}, \xi_r)$), which means that $h\in K[[\eta_1,\ldots,\eta_{r-1},\xi_r]]$ with $h$ invertible. Since $\underline{\xi}^{\underline{\beta}}=\underline{\eta}^{\underline{\alpha}}$ for some $\underline{\alpha}\in\N^r$, the lemma follows.
\end{demo}

\begin{remark}\label{rem:monom}
	\begin{enumerate}
		\item[(i)]  Let $\tilde{y}_0:=\displaystyle\frac{\tilde{f}}{\tilde{g}}\in \mathcal{K}_r$.  There exist $(p,\underline{q})\in \mathbb{N}^*\times\mathbb{N}^{r-1}$ and $L$ with $[L:K]<+\infty$ such that $\tilde{y}_0\in L\left(\left(\left(\displaystyle\frac{x_1}{x_2^{q_1}}\right)^{1/p},\ldots, \left(\displaystyle\frac{x_{r-1}}{x_r^{q_{r-1}}}\right)^{1/p} ,x_r^{1/p}\right)\right)$. We note that we can rewrite 
		$\tilde{y}_0$ as a monomial (with integer exponents) times an invertible power series in other variables $\left(\left( \displaystyle\frac{x_1}{x_2^{q_1'}}\right)^{1/p},\ldots, \left(\displaystyle\frac{x_{r-1}}{x_r^{q_{r-1}'}}\right)^{1/p} ,x_r^{1/p}\right)$.
		
		Indeed, let us denote $\underline{\xi}=(\xi_1,\ldots,\xi_r):=\left(\left(\displaystyle\frac{x_1}{x_2^{q_1}}\right)^{1/p},\ldots, \left(\displaystyle\frac{x_{r-1}}{x_r^{q_{r-1}}}\right)^{1/p} ,x_r^{1/p}\right)$. So $\tilde{y}_0=\displaystyle\frac{\tilde{f}}{\tilde{g}}$ for some $\tilde{f},\tilde{g}\in L[[\underline{\xi}]]$. 
		By the preceding lemma, we can monomialize the product $\tilde{f}.\tilde{g}$, so $\tilde{f}$ and $\tilde{g}$ simultaneously, by a suitable transformation (\ref{equ:transfo}). Note that this transformation maps $L\left[\left[\left(\displaystyle\frac{x_1}{x_2^{q_1}}\right)^{1/p},\ldots, \left(\displaystyle\frac{x_{r-1}}{x_r^{q_{r-1}}}\right)^{1/p} ,x_r^{1/p}\right]\right]$ into some $L\left[\left[\left(\displaystyle\frac{x_1}{x_2^{q_1'}}\right)^{1/p},\ldots, \left(\displaystyle\frac{x_{r-1}}{x_r^{q'_{r-1}}}\right)^{1/p} ,x_r^{1/p}\right]\right]$. Indeed, a monomial in $\underline{\xi}$ is transformed into a monomial in $\underline{\eta}$, and one has that:
		\begin{center}
			$ {\eta_1}^{i_1/p}\cdots {\eta_{r-1}}^{i_{r-1}/p} {\eta_r}^{i_r/p}= \left(\displaystyle\frac{x_1/{x_2}^{q_1}}{\left(x_2/{x_3}^{q_2}\right)^{\rho_1}}\right)^{i_1/p}\cdots \left(\displaystyle\frac{x_{r-1}/{x_{r}}^{q_{r-1}}}{{x_r}^{\rho_{r-1}}}\right)^{i_{r-1}/p} {x_r}^{i_r/p} =	\ \ \ \ \ \ \ \ \ 	\ \ \ \ \ \ \ \ \  $\\	
			$	\ \ \ \ \ \ \ \ \ 	\left(\displaystyle\frac{x_1}{{x_2}^{q_1+\rho_1}}\right)^{i_1/p}\cdots \left(\displaystyle\frac{x_{r-1}}{{x_r}^{q_{r-1}+\rho_{r-1}}}\right)^{i_{r-1}/p} x_r^{i_r/p}{\left({x_3}^{q_2\rho_1}\right)}^{i_1/p}\left({x_4}^{q_3\rho_2}\right)^{i_2/p}\cdots \left({x_r}^{q_{r-1}\rho_{r-2}}\right)^{i_{r-2}/p}$
		\end{center}
		and we write ${\left({x_3}^{q_2\rho_1}\right)}^{i_1/p}= \left(\displaystyle\frac{x_3}{{x_4}^{q_3+\rho_3}}\right)^{q_2\rho_1i_1/p}{x_4}^{(q_3+\rho_3)q_2\rho_1i_1/p}$ and so on. Thus we obtain a monomial in the variables $\left(\left(\displaystyle\frac{x_1}{{x_2}^{q_1+\rho_1}}\right)^{1/p},\ldots, \left(\displaystyle\frac{x_{r-1}}{{x_r}^{q_{r-1}+\rho_{r-1}}}\right)^{1/p}, x_r^{1/p}\right)$.
		\item[(ii)] 	Let $f\in  K[[\underline{\xi}]]$, $\rho_1,\ldots,\rho_{r-1}\in \mathbb{N}$, and  $\underline{\eta}$  be as in the Monomialization Lemma \ref{lemme:monomialisation}. Let $\underline{\beta}=v_{\underline{\xi}}(f)$. If we replace  $\rho_1,\ldots,\rho_{r-1}$ by $\rho_1',\ldots,\rho_{r-1}'$ with $\rho_i'\geq \rho_i$ for all $i$, and we proceed to the corresponding change of variables $\underline{\eta}'$ as in (\ref{equ:transfo}), then we still have $f(\xi)=(\underline{\eta}')^{\underline{\alpha}}g'(\underline{\eta}')$ for some invertible $g'\in K[[\underline{\eta}']]$. So Lemma \ref{lemme:monomialisation} holds true if we take $1+\beta_{i+1}$ instead of $\rho_i$ whenever $\rho_i>0$.
		
		%Let such that \ref{equ:transfo}
	\end{enumerate}
\end{remark}

\begin{theo}\label{theo:K_r-alg-closed}
	$\mathcal{K}_r$ is an algebraically closed extension of $K((\underline{x}))$.
\end{theo}
\begin{proof} This is a consequence of Abhyankar-Jung Theorem \cite{abh:val-cent-local-domain}, see \cite[Theorem 1.3 and Propo 5.1]{parusinski-rond:abhyankar-jung}, and our Monomialization Lemma \ref{lemme:monomialisation}. Let $$P(y)=\sum_{i=0}^da_iy^i\in L\left[\left[\left(\frac{x_1}{x_2^{q_1}}\right)^{1/p},\ldots, \left(\frac{x_{r-1}}{x_r^{q_{r-1}}}\right)^{1/p} ,x_r^{1/p}\right]\right][y]$$
	where $[L:K]<+\infty$, $p\in\N^*$, $q_i\in\N$ for $i=1,..r-1$ and $a_d\neq 0$. We want to show that $P$ has a root in $\mathcal{K}_r$. Up to multiplication by ${a_d}^{d-1}$ and change of variable $z=a_dy$, we may assume that $P$ is monic.  Let us denote $\underline{\xi}=(\xi_1,\ldots,\xi_r):=\left(\left(\displaystyle\frac{x_1}{x_2^{q_1}}\right)^{1/p},\ldots, \left(\displaystyle\frac{x_{r-1}}{x_r^{q_{r-1}}}\right)^{1/p} ,x_r^{1/p}\right)$ and $P(y)=P(\underline{\xi},y)$. 
	Up to replacing $L$ by a finite algebraic extension of it, we may also suppose that 
	$$ P(\underline{0},y)=(y-c_1)^{\alpha_1}\cdots (y-c_m)^{\alpha_m}$$
	with $c_i\in L$. By Hensel's Lemma [CITE Raynaud Propo 5 4) and Lafon Alg locale, chap 12, theo 12.5 p.166], there exist polynomials $P_1(\underline{\xi},y),\ldots,P_m(\underline{\xi},y)$ such that $P_i(0,y)=(y-c_i)^{\alpha_i}$ ($i=1,..,m$) and $P=P_1\cdots P_m$. It is enough to show that $P_1$ has a root in  $\mathcal{K}_r$. By a change of variable $y=z-c_1$, we are lead to the case of a polynomial 
	$$P(\underline{\xi},y)=y^d+\sum_{i=0}^{d-1}a_i(\underline{\xi})y^i$$	with $a_i(0)=0$, $i=0,..,d-1$. By our Monomialization Lemma \ref{lemme:monomialisation} and Remark \ref{rem:monom}(i), we may assume that the discriminant of $P$ is monomialized. Hence, Abhyankar-Jung Theorem applies. Note that this last step may require to replace $L$ by a finite algebraic extension. 
\end{proof}

Let  $\tilde{y}_0\in \mathcal{K}_r$  be a non zero rational polyhedral Puiseux series.  Let us show that the existence of a nonzero polynomial $\tilde{P}(\underline{x},y)$ cancelling $\tilde{y}_0$ is equivalent to the one of a polynomial $P(\underline{u},y)$ cancelling $y_0\in L[[\underline{u}]]$, but with constraints on the support of $P$. 
%=\displaystyle\sum_{j=0}^d\displaystyle\sum_{\underline{i}\in\N^r} a_{\underline{i},j}\underline{u}^{\underline{i}}y^j

Indeed, by our Monomialization Lemma \ref{lemme:monomialisation} and Remark \ref{rem:monom}(i), there are $(p,\underline{q})\in \mathbb{N}^*\times\mathbb{N}^{r-1}$ such that, if we set:
\begin{equation}\label{equ:eclt1} \left(u_1,\ldots,u_{r-1},u_r\right):=\left(\left(\frac{x_1}{x_2^{q_1}}\right)^{1/p},\ldots, \left(\frac{x_{r-1}}{x_r^{q_{r-1}}}\right)^{1/p} ,x_r^{1/p}\right),\end{equation}
then we can rewrite $\tilde{y}_0 =\displaystyle\sum_{\underline{n}\geq \tilde{\underline{n}}^0} \tilde{c}_{\underline{n}}\underline{u}^{\underline{n}},\ \tilde{c}_{\tilde{\underline{n}}^0}\neq 0$. Let us denote $c_{\underline{n}}:=\tilde{c}_{\underline{n}+\tilde{\underline{n}}^0}$, and: 
\begin{equation}\label{equ:y0-y0tilde}
	\tilde{y}_0=\underline{u}^{\tilde{\underline{n}}^0} \displaystyle\sum_{\underline{n}\geq \underline{0}} c_{\underline{n}}\underline{u}^{\underline{n}}=\underline{u}^{\tilde{\underline{n}}^0} y_0\ \ \textrm{with}\  c_{\underline{0}}\neq 0.
\end{equation}
Hence, $y_0$ is a formal power series in $\underline{u}$ with coefficient in a finite algebraic extension $L$ of $K$. %\emph{For simplicity, we may assume from now that $K=L$ }
By the change of variable (\ref{equ:eclt1}), we have:
$$ x_k=u_k^pu_{k+1}^{pq_{k}}u_{k+2}^{pq_{k}q_{k+1}}\cdots u_{r}^{pq_{k}q_{k+1}\cdots q_{r-1}},\ \ \ \ \ \ \ \  k=1,\ldots,r$$
The rational polyhedral Puiseux series $\tilde{y}_0$ is a root of a polynomial $$\tilde{P}(\underline{x},y)=\displaystyle\sum_{j=0}^d\displaystyle\sum_{\underline{i}\in\N^r} \tilde{a}_{\underline{i},j}\underline{x}^{\underline{i}}y^j\,\in K\left[\left[\underline{x}\right]\right][y]$$ of degree $d$ in $y$ if and only if the power series $y_0=\displaystyle\sum_{\underline{n}\in \N^r} c_{\underline{n}}\underline{u}^{\underline{n}} \in L[[\underline{u}]]$ is a root of  $$\underline{u}^{\tilde{\underline{m}}^0}\tilde{P}\left(  u_1^pu_{2}^{pq_{1}}\cdots u_{r}^{pq_{1}q_{2}\cdots q_{r-1}}\    ,\ \ldots\ ,\ u_r^{p} ,\ \underline{u}^{\tilde{\underline{n}}^0}y\right),$$
the latter being a polynomial $P(\underline{u},y)$ in $K[[\underline{u}]][y]$ for $\tilde{\underline{m}}^0$ such that 
\begin{equation}\label{equ:m}
	\tilde{m}^0_k=\max\left\{0\, ;\, -\tilde{n}_k^0d\right\},\ k=1,\ldots,r%-1,\ \ \ \textrm{ and }\ \ \  m_r=\max\left\{0\, ;\, \left(1-n_r^0\right)d\right\}
	.\end{equation}
Note that the transformation is uniquely defined by $p,\underline{q},d$ and $\underline{\tilde{n}}^0$. \\

In the following lemma, we clarify the constraints on the support of the polynomial $P$.

\begin{lemma}\label{lemma:support-equ}
	With the notations of (\ref{equ:eclt1}), we set $\underline{u}=\left(\underline{ t}_0,\underline{ s}_1,\underline{ t}_1,\ldots,\underline{ s}_\sigma,\underline{ t}_\sigma\right)$ where $\underline{ t}_0$ might be empty, such that  $u_i\in \underline{ s}_k$ if and only if $q_i\neq 0$ (and, so $u_i\in \underline{ t}_k$ if and only if $q_i=0$). Moreover, we write $\underline{ s}:=\left(\underline{ s}_1,\ldots,\underline{ s}_\sigma\right)$ and  $\underline{ t}:=\left(\underline{ t}_0,\underline{ t}_1,\ldots,\underline{ t}_\sigma\right)$.
	%If we set $\underline{s}$ the tuple of variables among the $u_i$'s such that $q_i\neq 0$, and $\underline{t}$ the tuple of those such that $q_i=0$, then 
	Hence, a polynomial $\tilde{P}(\underline{x},y) \in K\left[\left[\underline{x}\right]\right][y]$ is changed by the transformation induced by (\ref{equ:eclt1}) and  (\ref{equ:m}) into a polynomial:
	$$ P(\underline{s},\underline{t},y)=\displaystyle\sum_{\underline{l}\geq \underline{0}}\displaystyle\sum_{j=0}^dP_{\underline{l},j}(\underline{s})y^j\,\underline{t}^{\underline{l}}\in K[\underline{s},y][[\underline{t}]]$$
	with %$P_{\underline{l}}(\underline{s},y)= \displaystyle\sum_{\underline{i},j} a_{\underline{i},j}\underline{s}^{\underline{k}}y^j$ satisfies 
	for any $i$ such that $u_i\in \underline{s}_k$, 
	\begin{equation}\label{equ:degree}
		\deg_{{u_i}}(P_{\underline{l},j}(\underline{s}))-(\tilde{m}^0_{i}+j\tilde{n}_{i}^0) \leq \displaystyle\frac{\deg_{u_{i+1}} (P_{\underline{l},j}(\underline{s})\, \underline{t}^{\underline{l}})-(\tilde{m}^0_{i+1}+j\tilde{n}_{i+1}^0)  }{q_i},\ \ j=0,..,d.
	\end{equation}
	% and for any monomial ${\underline{u}}^{\underline{ \alpha}}y^j$ in the support of $P$, 
	% \begin{equation}\label{equ:congruence}
		% \alpha_i-(\tilde{m}^0_{i}+j\tilde{n}_{i}^0)\equiv 0\, (p) .
		% \end{equation}
	Conversely, any polynomial $$ P(\underline{s},\underline{t},y)=\displaystyle\sum_{\underline{l}\geq \underline{0}}\displaystyle\sum_{j=0}^dP_{\underline{l},j}(\underline{s})y^j\,\underline{t}^{\underline{l}}\in K[\underline{s},y][[\underline{t}]]$$
	%satisfying  (\ref{equ:degree}) and (\ref{equ:congruence}) 
	comes from a unique polynomial $\tilde{P}(\underline{x},y) \in K\left[\left[\underline{x}\right]\right][y]$ by the transformation induced by (\ref{equ:eclt1}) and (\ref{equ:m}) if and only if each monomial ${\underline{u}}^{\underline{ \alpha}}y^j$ in the support of $P$ satisfies the following conditions:
	\begin{itemize}
		\item[(i)] $\underline{ \alpha}\geq \tilde{\underline{m}}^0+j\tilde{\underline{n}}^0;$
		\item[(ii)] $\forall i=1,\ldots,r,\ \ \alpha_i-(\tilde{m}^0_{i}+j\tilde{n}_{i}^0)\equiv 0\, (p) ;$
		\item[(iii)] For any $u_i\in \underline{s}_k$, $\ \alpha_i-(\tilde{m}^0_{i}+j\tilde{n}_{i}^0)\leq \displaystyle\frac{\alpha_{i+1}-(\tilde{m}^0_{i+1}+j\tilde{n}_{i+1}^0)}{q_i}.$
	\end{itemize}
	
	%\\ A DETAILLER?? $q_1\deg_{u_1}\leq \deg_{u_2}$ and $q_1q_2\deg_{u_1}+ q_2\deg_{u_2}\leq\deg_{u_3}$ etc
	%The support of $P$ is such that: $$\mathrm{supp} (P)\subseteq $$
\end{lemma}

\begin{demo} Let us collect the variables $x_i$ according to the distinction between $t_j$ and $s_k$ among the variables $u_l$. We set $\underline{x}_k$ for the sub-tuple of variables $x_i$ corresponding to $\underline{t}_k$, and $\underline{\xi}_k$ for $\underline{s}_k$ respectively. %Hence, $$\underline{\xi}_k=
	% \left(\left(\frac{x_{i_1}}{x_{i_1+1}^{q_{i_1}}}\right)^{1/p},\left(\frac{x_{i_1+1}}{x_{i_1+2}^{q_{i_1+1}}}\right)^{1/p},\ldots,  \left(\frac{x_{i_2-1}}{x_{i_2}^{q_{i_2-1}}}\right)^{1/p}\right).$$
	Let us consider a general monomial:
	\begin{equation}\label{equ:var-paquets}
		\underline{x}^{\underline{ n}}y^j = \underline{x}_0^{\underline{ n}_0}\,\underline{\xi}_1^{\underline{ m}_1}\,\underline{x}_1^{\underline{ n}_1}\cdots\underline{\xi}_\sigma^{\underline{ m}_\sigma}\,\underline{x}_\sigma^{\underline{ n}_\sigma}y^j. 
	\end{equation}
	where $\underline{n}=(\underline{ n}_0, \underline{m }_1,\underline{ n}_1,\ldots,\underline{ m}_\sigma,\underline{ n}_\sigma)$. For $k=1,\ldots,\sigma$, we denote $\underline{ \xi}_k=(x_{i_k},\ldots,x_{j_k-1})$ and $\underline{ x}_k=(x_{j_k},\ldots,x_{i_{k+1}-1})$, and accordingly $\underline{ m}_k=(n_{i_k},\ldots,n_{j_k-1})$ and $\underline{ n}_k=(n_{j_k},\ldots,n_{i_{k+1}-1})$ with $i_{\sigma+1}:=r+1$. For $k=0$ when $\underline{ t}_0$ is not empty, we denote $\underline{ x}_0=\underline{ t}_0=(x_{j_0},\ldots,x_{i_{1}-1})$ and $\underline{ n}_0=(n_{j_0},\ldots,n_{i_{1}-1})$ with $j_0:=1$.
	
	By the change of variable  (\ref{equ:eclt1}), for each $k=1,\ldots,\sigma$, we obtain that:
	$$ {\underline{\xi}_k}^{\underline{ m}_k}\,{\underline{x}_k}^{\underline{ n}_k}= \left(\left(\frac{x_{i_k}}{x_{i_k+1}^{q_{i_k}}}\right)^{1/p}\right)^{pn_{i_k}}
	\left( \left(\frac{x_{i_k+1}}{x_{i_k+2}^{q_{i_k+1}}}\right)^{1/p}\right)^{p(n_{i_k+1}+q_{i_k}n_{i_k})}\cdots  \hspace{4cm}$$
	$$\left(\left(\frac{x_{j_k-1}}{x_{j_k}^{q_{j_k-1}}}\right)^{1/p}\right)^{p(n_{j_k-1}+q_{j_k-2}n_{j_k-2} +q_{j_k-2}q_{j_k-3}n_{j_k-3}+\cdots+ q_{j_k-2}q_{j_k-3}\cdots q_{i_k}n_{i_k})} $$
	$$  \hspace{2cm}
	\times {\left( {x_{j_k}}^{1/p}\right)}^{p(n_{j_k}+q_{j_k-1}n_{j_k-1}+q_{j_k-1}q_{j_k-2}n_{j_k-2}+\cdots+ q_{j_k-1}q_{j_k-2}\cdots q_{i_k}n_{i_k})}$$
	$$  \hspace{3cm} \times{\left({ x_{j_k+1}}^{1/p}\right)}^{pn_{j_k+1}}\cdots {\left({x_{i_{k+1}-1}}^{1/p}\right)}^{pn_{i_{k+1}-1}}$$
	$$ = {u_{i_k}}^{pn_{i_k}} {u_{i_k+1}}^{p(n_{i_k+1}+q_{i_k}n_{i_k})}\cdots {u_{j_k-1}}^{p(n_{j_k-1}+q_{j_k-2}n_{j_k-2}+q_{j_k-2}q_{j_k-3}n_{j_k-3}+\cdots+ q_{j_k-2}q_{j_k-3}\cdots q_{i_k}n_{i_k})} $$
	$$
	{ u_{j_k}}^{p(n_{j_k}+q_{j_k-1}n_{j_k-1}+q_{j_k-1}q_{j_k-2}n_{j_k-2}+\cdots+ q_{j_k-1}q_{j_k-2}\cdots q_{i_k}n_{i_k})}{ u_{j_k+1}}^{pn_{j_k+1}}\cdots {u_{i_{k+1}-1}}^{pn_{i_{k+1}-1}}$$
	\begin{equation}\label{equ:formula}
		\begin{array}{l}
			= {s_{i_k}}^{pn_{i_k}} {s_{i_k+1}}^{p(n_{i_k+1}+q_{i_k}n_{i_k})}\cdots {s_{j_k-1}}^{p(n_{j_k-1}+q_{j_k-2}n_{j_k-2}+q_{j_k-2}q_{j_k-3}n_{j_k-3}+\cdots+ q_{j_k-2}q_{j_k-3}\cdots q_{i_k}n_{i_k})} \\
			\hspace{1cm}{ t_{j_k}}^{p(n_{j_k}+q_{j_k-1}n_{j_k-1}+q_{j_k-1}q_{j_k-2}n_{j_k-2}+\cdots+ q_{j_k-1}q_{j_k-2}\cdots q_{i_k}n_{i_k})}{ t_{j_k+1}}^{pn_{j_k+1}}\cdots {t_{i_{k+1}-1}}^{pn_{i_{k+1}-1}}.
		\end{array}
	\end{equation}
	
	Moreover, $y^j$ is transformed into 
	\begin{equation}\label{equ:translation}
		\underline{u}^{\tilde{\underline{ m}}^0+j\tilde{\underline{n}}^0}y^j.
	\end{equation}
	For $u_i\in \underline{ s}_k$, we denote by $c_i$ its exponent in Formula (\ref{equ:formula}). If $i<j_k-1$, then $u_{i+1}\in \underline{ s}_k$ and its exponent is $c_{i+1}=p(n_{i+1}+q_{i}n_i+\cdots +q_iq_{i-1}\cdots q_{i_k}n_{i_k}) =pn_{i+1}+q_ic_i.$ The total exponent of $u_i$ in the transform of $\underline{x}^{\underline{ n}}y^j$ is $c_i+\tilde{m}^0_i+j\tilde{n}_i^0$. So, 
	\begin{center}
		$ \deg_{u_{i+1}} (P_{\underline{l},j}(\underline{s})\,y^j \underline{t}^{\underline{l}})-(\tilde{m}^0_{i+1}+j\tilde{n}_{i+1}^0)  
		= \deg_{u_{i+1}} (P_{\underline{l},j}(\underline{s}))-(\tilde{m}^0_{i+1}+j\tilde{n}_{i+1}^0)
		\geq {q_i} \left(\deg_{u_{i}} (P_{\underline{l},j}(\underline{s}))-(\tilde{m}^0_{i}+j\tilde{n}_{i}^0)\right).$
		% \deg_{{u_i}}P_{\underline{l}}(\underline{s},y) .$
	\end{center} If $i=j_k-1$, then  $u_{i+1}=t_{j_k}\in \underline{ t}_k$. Likewise, its exponent in (\ref{equ:formula}) is $pn_{j_k}+q_{j_k-1}c_{j_k-1}$. We obtain that \begin{center}
		$ \deg_{u_{i+1}} (P_{\underline{l}}(\underline{s})\,y^j \underline{t}^{\underline{l}})-(\tilde{m}^0_{j_k}+j\tilde{n}_{j_k}^0)  =\deg_{t_{j_k}}\underline{t}^{\underline{l}}-(\tilde{m}^0_{j_k}+j\tilde{n}_{j_k}^0)  \geq {q_{j_k-1}} \left(\deg_{{u_{j_k-1}}}P_{\underline{l}}(\underline{s},y)-(\tilde{m}^0_{j_k-1}+j\tilde{n}_{j_k-1}^0)  \right).$
	\end{center}
	
	Conversely, we consider a monomial ${\underline{s}_k}^{\underline{ \lambda}}\, {\underline{t}_k}^{\underline{ \mu}}$. It is of the form (\ref{equ:formula}), that is, it comes from a monomial ${ \underline{\xi}_k}^{\underline{ m}_k}\,{\underline{x}_k}^{\underline{ n}_k}$, if and only if 
	$\deg_{{u_i}} {\underline{s}_k}^{\underline{ \lambda}} \leq \displaystyle\frac{\deg_{u_{i+1}} {\underline{s}_k}^{\underline{ \lambda}}\, {\underline{t}_k}^{\underline{ \mu}}}{q_i}$
	and $\lambda_i\equiv \mu_j\equiv 0\, (p)$, which are equivalent to the conditions (ii) and (iii). Taking into account the transformation (\ref{equ:translation}), this gives the converse part of the lemma.
	
	% \deg_{u_{i+1}} {\underline{s}_k}^{\underline{ \alpha}}\, {\underline{t}_k}^{\underline{ \beta}} 
	
	% $$ {\underline{s}_k}^{\underline{ \alpha}}\, {\underline{t}_k}^{\underline{ \beta}}=$$
\end{demo}

\begin{remark}
	Note that, if $\underline{x}^{\underline{n}}y^j\neq \underline{x}^{\underline{n}'}y^{j'}$, the transformation applied to these monomials gives $\underline{u}^{\underline{\alpha}}y^j\neq \underline{u}^{\underline{\alpha}'}y^{j'}$.
	%\begin{enumerate}
	%	\item Note that, if $\underline{x}^{\underline{n}}y^j\neq \underline{x}^{\underline{n}'}y^{j'}$, the transformation applied to these monomials gives $\underline{u}^{\underline{\alpha}}y^j\neq \underline{u}^{\underline{\alpha}'}y^{j'}$.
	%	\item 	?? With the notations of (\ref{equ:formula}), let $|||\underline{l}|||=\sum_{k=1}^\sigma l_{j_k}$. This controls the degree in $\underline{s}$ of the polynomials that one gets as coefficients of $\underline{t}^{\underline{l}}$ for any  $\underline{l}$ with same $|||\underline{l}|||$. Hence this controls the depth of the reconstruction for these polynomials. 
	%\end{enumerate}
\end{remark}

For the rest of this section, and also for Sections \ref{sect:depth}, \ref{section:Wilc} and \ref{sect:reconstr-algebroid}, we assume that the field $K$ is algebraically closed, hence $K=L=\overline{K}$.

\begin{remark}
	If for all $i$, $q_i=0$, namely if $u_i={x_i}^{1/p}$,  then any  $\tilde{y}_0=\displaystyle\frac{f}{g}$ with $f,g\in K[[\underline{u}]]$ 
	% =\displaystyle\sum_{\underline{n}\geq \tilde{\underline{n}}^0} \tilde{c}_{\underline{n}}\underline{x}^{\underline{n}/p}$
	is algebroid. Indeed, let $\theta_p$ denote a primitive $p$th root of unity. We set:
	\begin{align*}
		\tilde{P}(\underline{u},y)&:=&\prod_{i=1,\ldots,r}\prod_{k_i=0,\ldots,p-1}g\left({\theta_p}^{k_1}u_1,\ldots,{\theta_p}^{k_r}u_r\right) \left(y-\tilde{y}_0\left({\theta_p}^{k_1}u_1,\ldots,{\theta_p}^{k_r}u_r\right)\right)\\
		&=&\prod_{i=1,\ldots,r}\prod_{k_i=0,\ldots,p-1}\left[g\left({\theta_p}^{k_1}u_1,\ldots,{\theta_p}^{k_r}u_r\right) y-f\left({\theta_p}^{k_1}u_1,\ldots,{\theta_p}^{k_r}u_r\right)\right]. 
	\end{align*}
	Note that $\tilde{P}(\underline{u},\tilde{y}_0)=0$. Moreover, since $	\tilde{P}(u_1,\ldots,\theta_pu_i,\ldots,u_r,y)=\tilde{P}(\underline{u},y)$ for any $i=1,\ldots,r$, we conclude that $\tilde{P}\in K[[\underline{x}]][y]$. \\
	
	Consequently, from now on, we consider the case where $q_i\neq 0$ for at least one $i\in\{1,\ldots,r\}$.
\end{remark}

Let us denote by $\tau$ the number of variables in $\underline{s}$, and so $r-\tau$ is the number of variables in $\underline{t}$.  We consider $y_0=\displaystyle\sum_{\underline{m}\in\mathbb{N}^\tau,\,\underline{n}\in \mathbb{N}^{r-\tau}} c_{\underline{m} ,\underline{n}}\underline{s}^{\underline{m}}\underline{t}^{\underline{n}} =\displaystyle\sum_{\underline{n}\in \mathbb{N}^{r-\tau}} c_{\underline{n}}(\underline{s})\, \underline{t}^{\underline{n}}$  such that $c_{\underline{0},\underline{0}}\neq 0$ which satisfies an equation
$ P(\underline{s},\underline{t},y)=0$ where $ P$ agrees conditions (i), (ii) and (iii) of Lemma \ref{lemma:support-equ}. 

\begin{lemma}\label{lemma:algebraicity}
	The series $c_{\underline{n}}(\underline{s})\in K[[\underline{s}]]$,  $\underline{n}\in \mathbb{N}^{r-\tau}$, are all algebraic over $K(\underline{s})$, and lie in a finite extension of  $K(\underline{s})$.
\end{lemma}
\begin{demo} We consider $y_0 =\displaystyle\sum_{\underline{n}\in \mathbb{N}^{r-\tau}} c_{\underline{n}}(\underline{s})\, \underline{t}^{\underline{n}}$ root of a non-trivial polynomial $$ P(\underline{s},\underline{t},y)=\displaystyle\sum_{\underline{l}\in \mathbb{N}^{r-\tau}} P_{\underline{l}}(\underline{s},y)\,\underline{t}^{\underline{l}}\in K[\underline{s},y][[\underline{t}]]$$
	which satisfies conditions (i), (ii) and (iii).	We proceed by induction on $\mathbb{N}^{r-\tau}$ ordered by $\leq_{\rm grlex}$.  Given some $\underline{ n}\in \mathbb{N}^{r-\tau}$, we set \begin{equation}\label{equ:z_n-y_n}
		y_0=\tilde{z}_{\underline{n}}+c_{\underline{n}}\underline{t}^{\underline{n}}+y_{\underline{n}}
	\end{equation} with
	$\tilde{z}_{\underline{n}}=\sum_{\underline{\beta}<_{\textrm{grlex}} \underline{n} } c_{\underline{\beta}}\underline{t}^{\underline{\beta}}$, 	$y_{\underline{n}}=\sum_{\underline{\beta}>_{\textrm{grlex}} \underline{n} } c_{\underline{\beta}}\underline{t}^{\underline{\beta}}$,
	(and $z_{\underline{0}}:=0$ which corresponds to the initial step of the induction). We assume that the coefficients $c_{\underline{\beta}}$ of $\tilde{z}_{\underline{n}}$ belong to a finite extension $L_{\underline{n}}$ of $K(\underline{s})$. We set 
	\begin{equation}\label{equ:translation-Q}
		Q_{\underline{n}}(\underline{t},y):=P(\underline{s},\underline{t},\tilde{z}_{\underline{n}}+y)\in L_{\underline{n}}[y][[\underline{t}]]
	\end{equation}
	and we denote it by:
	$$ Q_{\underline{n}}(\underline{t},y)=\displaystyle\sum_{\underline{l}\geq \underline{0}}Q_{\underline{n},\underline{l}}(y)\,\underline{t}^{\underline{l}}.$$
	We claim that 
	\begin{equation}\label{equ:wtPQ}
		w_{\underline{t}}(P)=w_{\underline{t}}(Q_{\underline{n}}).
	\end{equation} This is clear if $\underline{n}=\underline{0}$.
	For $\underline{n}>_{\rm grlex}\underline{0}$, let $\underline{l}_0:=w_{\underline{t}}(P)$. We have
	\begin{center}
		$Q_{\underline{n}}(\underline{t},y)=P_{\underline{l}_0}(\underline{s}, \tilde{z}_{\underline{n}}+y)\underline{t}^{\underline{l}_0}+\cdots =\left( \displaystyle\sum_{j=0}^d \frac{1}{j!}\frac{\partial^j P_{\underline{l}_0}}{\partial y^j}(\underline{s},y){\tilde{z}_{\underline{n}}}^j \right)\underline{t}^{\underline{l}_0}+\cdots $
	\end{center}
	Let $d_{\underline{l}_0} :=\deg_y P_{\underline{l}_0}$: the coefficient of $y^{d_{\underline{l}_0}}$ in the previous parenthesis is not zero for $j=0$ but zero for $j\geq 1$. Namely, it is the coefficient of $P_{\underline{l}_0}(\underline{s},y)$, which is of the form  $a(\underline{s})y^{d_{\underline{l}_0}}\underline{t}^{\underline{l}_0}$ and therefore cannot overlap with other terms. 
	
	By Taylor's formula, we have that:
	$$Q_{\underline{n}}(\underline{t},C\underline{t}^{\underline{n}}+y)=\displaystyle\sum_{\underline{l}\geq_{\rm grlex} \underline{l}_0}\displaystyle\sum_{j=0}^d \frac{1}{j!} \frac{\partial^j Q_{\underline{n},\underline{l}} }{\partial y^j}(0)\, \left(C\underline{t}^{\underline{n}}+y\right)^j \,\underline{t}^{\underline{l}}.$$ 
	
	Recall that $y_{\underline{n}}\in  K[[\underline{s}]][[\underline{t}]]$ with  $w_{\underline{t}}(y_{\underline{n}})>_{\textrm{grlex}} \underline{n}$. Then $Q_{\underline{n}}(\underline{t},C\underline{t}^{\underline{n}}+y_{\underline{n}})\neq 0$ as a polynomial in $C$ (otherwise $P$ would have more than $d$ roots). Necessarily, $w_{\underline{t}}\left( Q_{\underline{n}}(\underline{t},C\underline{t}^{\underline{n}}+y_{\underline{n}})\right)$ is of the form $\underline{\omega}=\underline{l}_1+j_1\underline{ n}$. Indeed, let us consider $\underline{\omega}:=\min_{\underline{l},j}\left\{\underline{l}+j\underline{ n}\ |\, \frac{\partial^j Q_{\underline{n},\underline{l}} }{\partial y^j}(0)\neq 0\right\}$, and among the $(\underline{l},j)$'s which achieve this minimum, consider the term with the biggest $j$. This term cannot be cancelled.  The correspondent coefficient of $\underline{t}^{\underline{\omega}}$ in $Q_{\underline{n}}(\underline{t},C\underline{t}^{\underline{n}}+y_{\underline{n}})$ is a nonzero polynomial in $C$ of the form:
	\begin{equation}\label{equ:taylorQ_n}
		\sum_{\underline{l}_k+j_k\underline{ n}=\underline{\omega}} \frac{1}{j_k!} \frac{\partial^{j_k} Q_{\underline{n},\underline{l}_k} }{\partial y^{j_k}}(0)\,  {C }^{j_k}. 
	\end{equation}
	Since $y_0$ is a root of $P$, this polynomial needs to vanish for $C=c_{\underline{n}}$, which proves by the induction hypothesis that $ c_{\underline{n}}$ is itself algebraic over $K(\underline{s})$.
	
	Without loss of generality, we may assume that $y_0$ is a simple root of $P$, hence, $ \displaystyle\frac{\partial P}{\partial y}(\underline{s},\underline{t},y_0)$ $\neq 0$. With the same notations as above, we consider $\underline{n}_0:= w_{\underline{t}}\left(\displaystyle\frac{\partial P}{\partial y}(\underline{s},\underline{t},y_0)\right) 	\in \mathbb{N}^{r-\tau}$. For any $\underline{n}>_{\textrm{grlex}} \underline{n}_0$,  $ \displaystyle\frac{\partial Q_{\underline{n}}}{\partial y}(\underline{t},0)=\displaystyle\frac{\partial P}{\partial y}(\underline{s},\underline{t},\tilde{z}_{\underline{ n}})$ and $$w_{\underline{t}}\left(\displaystyle\frac{\partial Q_{\underline{n}}}{\partial y}(\underline{t},0)-\displaystyle\frac{\partial P}{\partial y}(\underline{s},\underline{t},y_0)\right)=w_{\underline{t}}\left(\displaystyle\frac{\partial P}{\partial y}(\underline{s},\underline{t},\tilde{z}_{\underline{ n}})-\displaystyle\frac{\partial P}{\partial y}(\underline{s},\underline{t},y_0)\right)\geq_{\textrm{grlex}} \underline{n}>_{\textrm{grlex}} \underline{n}_0. $$ 
	So $w_{\underline{t}}\left(\displaystyle\frac{\partial Q_{\underline{n}}}{\partial y}(\underline{t},0)\right)=\underline{n}_0$.

	By Taylor's formula:
	\begin{equation}\label{equ:yn}
		Q_{\underline{n}}(\underline{t},C\underline{t}^{\underline{n}}+y_{\underline{n}})=\displaystyle\sum_{j=0}^d \frac{1}{j!} \frac{\partial^j Q_{\underline{n}} }{\partial {y_{\underline{n}}}^j}(\underline{t},0)\, \left(C\underline{t}^{\underline{n}}+y\right)^j.\end{equation} 
	We have:
	\[ w_{\underline{t}}\left(\frac{\partial Q_{\underline{n}} }{\partial y}(\underline{t},0) \left(C\underline{t}^{\underline{n}}+y_{\underline{n}}\right)\right)= \underline{n}+\underline{n}_0,\]
	and for any  $j\geq 2$:
	\[ w_{\underline{t}}\left(\frac{\partial^j Q_{\underline{n}} }{\partial y^j}(\underline{t},0) \left(C\underline{t}^{\underline{n}}+y_{\underline{n}}\right)^j\right)\geq_{\mathrm{grlex}} 2\underline{n}>\underline{n}+\underline{n}_0. \]
	%	$$ w_{\underline{t}}\left(\frac{\partial^j Q_{\underline{n}} }{\partial y^j}(\underline{t},0) \right)+(j-1)\underline{ n}>_{\textrm{grlex}} w_{\underline{t}}\left(\frac{\partial Q_{\underline{n}} }{\partial y}(\underline{t},0) \right)=\underline{n}_0.$$
	%	So, we have that:
	%	$$w_{\underline{t}}\left(Q_{\underline{n}}(\underline{t},C\underline{t}^{\underline{n}}+y_{\underline{n}}) \right)=w_{\underline{t}}\left( Q_{\underline{n}}(\underline{t},0)+ \displaystyle\frac{\partial Q_{\underline{n}} }{\partial y}(\underline{t},0)C\underline{t}^{\underline{n}} \right)=\min\left\{w_{\underline{t}}\left( Q_{\underline{n}}(\underline{t},0)\right), \underline{n}_0+\underline{n}\right\}.$$
	We deduce by (\ref{equ:yn}) that $w_{\underline{t}}(Q_{\underline{n}}(\underline{t},0))  \geq_{\textrm{grlex}} \underline{n}+\underline{n}_0$ since, otherwise, $Q_{\underline{n}}(\underline{t},C\underline{t}^{\underline{n}}+y_{\underline{n}})$ could not vanish at $C=c_{\underline{n}}$. 
	Let us prove by induction on $\underline{n}\in\mathbb{N}^{r-\tau}$ ordered by $\leq_{\rm grlex}$,  $\underline{n}\geq_{\rm grlex}{\underline{n}_0}$, that the coefficients $c_{\underline{l}}$ of ${\underline{t}}^{\underline{l}}$ in $\tilde{z}_{\underline{n}}$ all belong to $L_{\underline{ n}_0}=K\left(\underline{s},c_{\underline{0}},\ldots,c_{\underline{n}_0}\right)$. The initial case is clear. Assume that the property holds for less than some given $\underline{n}$. Let us denote $\displaystyle\frac{\partial Q_{\underline{n}} }{\partial y}(\underline{t},0)=a_{\underline{n}_0}\underline{t}^{\underline{n}_0}+R(\underline{t})$ with $w_{\underline{t}}(R(\underline{t})) >_{\textrm{grlex}} \underline{n}_0$, $a_{\underline{n}_0}\neq 0$, and $ Q_{\underline{n}}(\underline{t},0)=b_{\underline{n}+\underline{n}_0}\underline{t}^{\underline{n}+\underline{n}_0}+S(\underline{t})$ with $w_{\underline{t}}(S(\underline{t})) >_{\textrm{grlex}} \underline{n}+\underline{n}_0$. By (\ref{equ:translation-Q}) and the induction hypothesis, $a_{\underline{n}_0}$ and  $b_{\underline{n}+\underline{n}_0}$ belong to $L_{\underline{n}_0}$. Looking at the coefficient of $\underline{t}^{\underline{n}+\underline{n}_0}$ in (\ref{equ:yn}) evaluated at $C=c_{\underline{n}}$, we get:
	\begin{equation}\label{equ:hensel-rec}
		a_{\underline{n}_0}c_{\underline{n}} +b_{\underline{n}+\underline{n}_0}=0.
	\end{equation}
	Hence we obtain that  $c_{\underline{n}}\in L_{\underline{n}_0}=K\left(\underline{s},c_{\underline{0}},\ldots,c_{\underline{n}_0}\right)$ for all $\underline{n}>_{\rm grlex}\underline{n}_0$.
\end{demo}

Let us recall that $A(\underline{n})$ denotes the predecessor element of $\underline{n}$ in $(\mathbb{N}^r,\leq_{\textrm{grlex}})$. The following lemma will be used in Section\ref{sect:reconstr-algebroid} in order to apply the results of Section \ref{section:Wilc}.

\begin{lemma}\label{lemme:control-deg-supp}
	Let $d$, $\tilde{m}^0$, $\tilde{n}^0$, $\underline{q}$, $p$ and $P$ be as above (see (\ref{equ:eclt1}) and (\ref{equ:m})). As in the proof of the previous lemma, we set $\underline{l}_0:=w_{\underline{t}}(P)$. We resume the notations of Lemma \ref{lemma:support-equ}. For $k=1,\ldots,\sigma$, with $\underline{s}_k=(u_{i_k},\ldots,u_{j_k-1})$, we denote 
	$$ e_{\underline{s}_k}:=\displaystyle\frac{1}{q_{i_k}q_{i_k+1}\cdots q_{j_k-1}}+\displaystyle\frac{1}{q_{i_k+1}\cdots q_{j_k-1}}+\cdots + \displaystyle\frac{1}{ q_{j_k-1}},$$
	and $\underline{\tilde{n}}^{0,\underline{s}_k}$ (respectively $\underline{\tilde{m}}^{0,\underline{s}_k}$), the multi-index obtained from $\tilde{n}^0$ (respectively $\tilde{m}^0$), by restriction to the components corresponding to the variables in $\underline{s}_k$.  %(i.e. $\tilde{n}^{0,\underline{s}_k}=(\tilde{n}_{i_k}^{0,\underline{s}_k},\ldots,\tilde{n}_{j_k-1}^{0,\underline{s}_k})$). 
	Likewise, we set $\underline{\tilde{n}}^{0,\underline{t}_k}$ and $\underline{\tilde{m}}^{0,\underline{t}_k}$ corresponding to the variables in $\underline{t}_k$ for $k=0,\ldots, \sigma$. Let $\underline{n}\in\N^{r-\tau} $, then there exists  $T_{\underline{n}}\in K[\underline{s},( C_{\underline{\beta}})_{\beta\leq_{ \mathrm{grlex}}\underline{n}}]\setminus\{0\}$ such that $T_{\underline{n}}(\underline{s},c_{\underline{0}},\ldots,c_{A(\underline{n})},c_{\underline{n}})=0$, $T_{\underline{n}}(\underline{s},c_{\underline{0}},\ldots,c_{A(\underline{n})},C_{\underline{n}})\not\equiv 0$ with 
	\begin{center}
		$\deg_{C_{\underline{ \beta}}}T_{\underline{n}}\leq d$,\\
		$\deg_{ \underline{ s}}T_{\underline{n}}\leq \left( |\underline{l}_{0}|+d\, |\underline{n}| \right)a+b$,
	\end{center}
	where 
	\begin{center}
		$a:=\displaystyle\sum_{k=1}^\sigma e_{\underline{s}_k} $,\\
		$b:=\varepsilon \left(\displaystyle\sum_{k=1}^\sigma |\underline{\tilde{n}}^{0,\underline{s}_k}|-\displaystyle\sum_{k=1}^\sigma {\tilde{n}}^{0,\underline{t}_k}_{j_k}  e_{\underline{s}_k}\right)+\displaystyle\sum_{k=1}^\sigma |\underline{\tilde{m}}^{0,\underline{s}_k}|-\displaystyle\sum_{k=1}^\sigma \tilde{m}^{0,\underline{t}_k}_{j_k}  e_{\underline{s}_k}$,
	\end{center}
	with $\tilde{n}^{0,\underline{t}_k}_{j_k}$ (respectively $\tilde{m}^{0,\underline{t}_k}_{j_k}$) the first component of $\underline{\tilde{n}}^{0,\underline{t}_k}$ (respectively $\underline{\tilde{m}}^{0,\underline{t}_k}$), and 
	\begin{center}
		$\varepsilon:=\left\{\begin{array}{ll}
			0&\textrm{ if } \displaystyle\sum_{k=1}^\sigma |\underline{\tilde{n}}^{0,\underline{s}_k}|-\displaystyle\sum_{k=1}^\sigma \tilde{n}^{0,\underline{t}_k}_{j_k}  e_{\underline{s}_k}\leq 0,\\
			d&\textrm{ if } \displaystyle\sum_{k=1}^\sigma |\underline{\tilde{n}}^{0,\underline{s}_k}|-\displaystyle\sum_{k=1}^\sigma \tilde{n}^{0,\underline{t}_k}_{j_k}  e_{\underline{s}_k}> 0.
		\end{array} \right.$
	\end{center}
\end{lemma}
\begin{proof}
	Resuming the notations and computations of the previous lemma (see (\ref{equ:z_n-y_n}) to (\ref{equ:taylorQ_n})), $c_{\underline{n}}$ is a root of a nonzero polynomial in $C$ of the form:
	$$ \sum_{\underline{l}_k+p_k\underline{ n}=\underline{\omega}} \frac{1}{p_k!} \frac{\partial^{p_k} Q_{\underline{ n},\underline{l}_k} }{\partial y^{p_k}}(0)\,  {C }^{p_k} $$
	where $\underline{\omega}:=w_{\underline{t}}\left( Q_{\underline{n}}(\underline{t},C\underline{t}^{\underline{n}}+y_{\underline{n}})\right)=\underline{l}_1+p_1\underline{ n}\leq_{\mathrm{grlex}} \underline{l}_1+d\,\underline{ n}$. Let us denote by $T_{\underline{n}}$ the polynomial obtained from the preceding expression by substituting $C_{\underline{n}}$ to $C$ and $C_{\underline{\beta}}$ to $c_{\underline{\beta}}$ for $\underline{\beta}<_{\mathrm{grlex}}\underline{n}$. More precisely, if we set 
	\[ \begin{aligned}
		H_{\underline{n}}(\underline{s},\underline{t},(C_{\underline{\beta}})_{\underline{\beta}\leq_{ \mathrm{grlex}}\underline{n}},y)= P\left(\underline{s},\underline{t},\sum_{\underline{\beta}\leq_{ \mathrm{grlex}}\underline{n}} C_{\underline{\beta}} \underline{t}^{\underline{\beta}}+y \right)\\
		=\sum_{\underline{l}\in \mathbb{N}^{r-\tau}} H_{\underline{n},\underline{l}}(\underline{s},(C_{\underline{\beta}})_{\underline{\beta}\leq_{ \mathrm{grlex}}\underline{n}},y)\underline{t}^{\underline{l}}
	\end{aligned} 
	\]
	then $T_{\underline{n}}(\underline{s},(C_{\underline{\beta}})_{\underline{\beta}\leq_{ \mathrm{grlex}}\underline{n}}):=H_{\underline{n},\underline{\omega}}(\underline{s},(C_{\underline{\beta}})_{\underline{\beta}\leq_{ \mathrm{grlex}}\underline{n}},0)$.
	
	Since $w_{ \underline{t}}(Q_{\underline{ n}})=w_{ \underline{t}}(P)$ by (\ref{equ:wtPQ}), we observe  that  $\underline{l}_0=\min_{\leq_{\mathrm{grlex}}} \left\{\underline{l}\ |\ \exists p,\ \displaystyle\frac{\partial^{p} Q_{\underline{ n},\underline{l}} }{\partial y^{p}}(0)\neq 0 \right\}$. Let $p_0 = \min \left\{p\ |\  \displaystyle\frac{\partial^{p} Q_{\underline{ n},\underline{l}_0} }{\partial y^{p}}(0)\neq 0 \right\}$. Then the coefficient of $C^{p_0} \underline{t}^{\underline{l}_0 +p_0\underline{n}}$ in the expansion of $Q_{\underline{ n}}(\underline{t},C \underline{t}^{\underline{n} }+y_{ \underline{n}})$ is not zero. Since we have that:
	$$Q_{\underline{ n}}(\underline{t},C\underline{t}^{\underline{n}}+y_{ \underline{n}})=\displaystyle\sum_{\underline{l}\geq \underline{0}}\displaystyle\sum_{j=0}^d \frac{1}{j!} \frac{\partial^j Q_{\underline{ n},\underline{l}} }{\partial y^j}(0)\, \left(C\underline{t}^{\underline{n}}+y_{ \underline{n}}\right)^j \,\underline{t}^{\underline{l}},$$ 
	the term $\displaystyle\frac{1}{p_0!} \frac{\partial^{p_0} Q_{\underline{ n},\underline{l}_0} }{\partial y^{p_0}}(0)\,  C ^{ p_0} \,\underline{t}^{\underline{l}_0+p_0\underline{n}}$ cannot overlap with other terms since the latter will necessarily be of the form $\displaystyle\frac{1}{(p-p_0)!p_0!} \frac{\partial^{p} Q_{\underline{ n},\underline{l}} }{\partial y^{p}}(0)\,  C ^{ p_0} \,\underline{t}^{\underline{l}+p_0\underline{n}}y_{ \underline{n}}^{p-p_0}$ with $\underline{l}\geq_{\mathrm{grlex}}\underline{l}_0$, $p\geq p_0$ and $w_{\underline{t}}(y_{\underline{n}})>_{\mathrm{grlex}} \underline{n}$. (see (\ref{equ:z_n-y_n})). So, $\underline{ \omega}\leq_{\mathrm{grlex}} \underline{l}_0+p_0\underline{n}\leq_{\mathrm{grlex}} \underline{l}_0+d\underline{n}$.\\

	Let us detail the expression of the connection between $P$ and $Q_{\underline{ n}}$.	We denote  $ P(\underline{s},\underline{t},y)=\displaystyle\sum_{\underline{l}\in \mathbb{N}^{r-\tau}}
	\left(\displaystyle\sum_{\underline{k}\in\N^\tau}\displaystyle\sum_{j=0}^d  a_{\underline{k},\underline{l},j}\underline{s}^{\underline{k}}y^j\right) \underline{t}^{\underline{l}}$, and we get:
	\begin{center}
		$ Q_{\underline{ n}}(\underline{s},\underline{t},y)	=P(\underline{s},\underline{t},\tilde{z}_{\underline{n}}+y )$\\
		$	=\displaystyle\sum_{\underline{l}\in \mathbb{N}^{r-\tau}}
		\left(\displaystyle\sum_{\underline{k}\in\N^\tau}\displaystyle\sum_{j=0}^d  a_{\underline{k},\underline{l},j}\underline{s}^{\underline{k}}\left(\sum_{\underline{\beta}<_{\textrm{grlex}} \underline{n} } c_{\underline{\beta}}\underline{t}^{\underline{\beta}}+y\right)^j\right) \underline{t}^{\underline{l}}$\\
		$	=\displaystyle\sum_{\underline{l}\in \mathbb{N}^{r-\tau}}	\left(\displaystyle\sum_{\underline{k}\in\N^\tau}\displaystyle\sum_{j=0}^d  a_{\underline{k},\underline{l},j}\underline{s}^{\underline{k}}\left(\sum_{|\underline{j}|=j} \frac{j!}{\underline{j}!} \left(\prod_{\underline{\beta}<_{\textrm{grlex}} \underline{n} }{ c_{\underline{\beta}}}^{j_{\underline{\beta}}}\right) y^{j_{\underline{n}}} \underline{t}^{g(\underline{j})-j_{\underline{n}}\underline{n}} \right)\right) \underline{t}^{\underline{l}}$\\
		$	=\displaystyle\sum_{\underline{l}\in \mathbb{N}^{r-\tau}}	\displaystyle\sum_{\underline{k}\in\N^\tau}\displaystyle\sum_{j=0}^d\sum_{|\underline{j}|=j}   a_{\underline{k},\underline{l},j}\underline{s}^{\underline{k}}  \frac{j!}{\underline{j}!} \left(\prod_{\underline{\beta}<_{\textrm{grlex}} \underline{n} }{ c_{\underline{\beta}}}^{j_{\underline{\beta}}}\right) y^{j_{\underline{n}}}  \underline{t}^{\underline{l}+g(\underline{j})-j_{\underline{n}}\underline{n}}$\\
	\end{center}
	where $\underline{j}=(j_{\underline{0}},\ldots,j_{\underline{n}})$ and $g(\underline{j})$ is as in Notation \ref{nota:FS}. Next, we evaluate $y$ at $C \underline{t}^{\underline{n}}+y_{ \underline{n}}$ and we consider the $(\underline{l},\underline{j})$'s such that $\underline{l}+g(\underline{j})=\underline{ \omega}$ for which the coefficient of $\underline{t}^{\underline{ \omega} }$ is the non-trivial polynomial of which $c_{\underline{n}}$ is a root. Then, the multi-indices $\underline{l}$ involved are such that $\underline{l}\leq_{\mathrm{grlex}}\underline{l}_0+d\underline{n}$. Consider such a monomial $\underline{s}^{\underline{ k} }\underline{t}^{\underline{l} }y^j$ written as $\underline{u}^{\underline{ \alpha} }y^j$ as in (\ref{equ:eclt1}). Recall that the elements of the support of $P$ satisfy Condition (iii) of Lemma \ref{lemma:support-equ}: for any $k=1,\ldots,\sigma$, for any $u_i\in \underline{s}_k$, $ \alpha_i-(\tilde{m}^0_{i}+j\tilde{n}_{i}^0)\leq \displaystyle\frac{\alpha_{i+1}-(\tilde{m}^0_{i+1}+j\tilde{n}_{i+1}^0)}{q_i}.$ For $\underline{s}_k=(u_{i_k},\ldots,u_{j_k-1})$ and $\underline{t}_k=(u_{j_k},\ldots,u_{i_{k+1}-1})$, we claim that for any $i=i_k,\ldots,j_k-1$,
	\begin{equation}\label{equ:control-supp}
		\alpha_i\leq \displaystyle\frac{\alpha_{j_k}}{q_iq_{i+1}\cdots q_{j_k-1}}+j \left( \tilde{n}^0_i-\displaystyle\frac{\tilde{n}^0_{j_k}}{q_iq_{i+1}\cdots q_{j_k-1}} \right)+\tilde{m}^0_i-\displaystyle\frac{\tilde{m}^0_{j_k}}{q_iq_{i+1}\cdots q_{j_k-1}}. 
	\end{equation}
	The case $i=j_k-1$ is given by Condition (iii). Suppose that the formula holds until $i+1$, i.e. 
	$$\alpha_{i+1}\leq \displaystyle\frac{\alpha_{j_k}}{q_{i+1}\cdots q_{j_k-1}}+j \left( \tilde{n}^0_{i+1}-\displaystyle\frac{\tilde{n}^0_{j_k}}{q_{i+1}\cdots q_{j_k-1}} \right)+\tilde{m}^0_{i+1}-\displaystyle\frac{\tilde{m}^0_{j_k}}{q_{i+1}\cdots q_{j_k-1}}.$$
	Since, by Condition (iii), we have $ \alpha_i\leq \displaystyle\frac{\alpha_{i+1}}{q_i}+j\left(\tilde{n}_{i}^0-\displaystyle\frac{\tilde{n}_{i+1}^0}{q_i}\right)+\tilde{m}_{i}^0-\displaystyle\frac{\tilde{m}_{i+1}^0}{q_i},$ we obtain the formula for $\alpha_i$ as expected.
	
	Now, we consider the sum for $i=i_k,\ldots,j_k-1$ of these inequalities (\ref{equ:control-supp}):
	$$\displaystyle\sum_{i=i_k}^{j_k-1}\alpha_i\leq \alpha_{j_k}e_{\underline{s}_k}+j\left(|\underline{\tilde{n}}^{0,\underline{s}_k}|-\tilde{n}^{0 }_{j_k}e_{\underline{s}_k}\right)+|\underline{\tilde{m}}^{0,\underline{s}_k}|-\tilde{m}^{0 }_{j_k}e_{\underline{s}_k}.$$
	Note that $\tilde{n}^{0 }_{j_k}=\tilde{n}^{0,\underline{t}_k}_{j_k}$ and  $\tilde{m}^{0 }_{j_k}=\tilde{m}^{0,\underline{t}_k}_{j_k}$. Moreover, $\alpha_{j_k}$ is equal to some $l_{\gamma}$ component of $\underline{l}$, so $\alpha_{j_k}\leq |\underline{l}_0|+d|\underline{n}|$. So,\begin{equation}\label{equ:alpha}
		\displaystyle\sum_{i=i_k}^{j_k-1}\alpha_i\leq \left(|\underline{l}_0|+d|\underline{n}|\right)e_{\underline{s}_k}+j\left(|\underline{\tilde{n}}^{0,\underline{s}_k}|-\tilde{n}^{0,\underline{t}_k}_{j_k}e_{\underline{s}_k}\right)+|\underline{\tilde{m}}^{0,\underline{s}_k}|-\tilde{m}^{0,\underline{t}_k}_{j_k}e_{\underline{s}_k}.
	\end{equation}
	Taking the sum for $k=1,\ldots,\sigma$, we obtain:
	$$|\underline{k}|\leq \left(|\underline{l}_0|+d|\underline{n}|\right)\displaystyle\sum_{i=1}^{\sigma}e_{\underline{s}_k}+j\left(\displaystyle\sum_{i=1}^{\sigma}|\underline{\tilde{n}}^{0,\underline{s}_k}|-\displaystyle\sum_{i=1}^{\sigma}\tilde{n}^{0,\underline{t}_k}_{j_k}e_{\underline{s}_k}\right)+\displaystyle\sum_{i=1}^{\sigma}|\underline{\tilde{m}}^{0,\underline{s}_k}|-\displaystyle\sum_{i=1}^{\sigma}\tilde{m}^{0,\underline{t}_k}_{j_k}e_{\underline{s}_k}.$$
	Since $0\leq j\leq d$, we finally obtain:
	$$|\underline{k}|\leq \left(|\underline{l}_0|+d|\underline{n}|\right)\displaystyle\sum_{i=1}^{\sigma}e_{\underline{s}_k}+\varepsilon\left(\displaystyle\sum_{i=1}^{\sigma}|\underline{\tilde{n}}^{0,\underline{s}_k}|-\displaystyle\sum_{i=1}^{\sigma}\tilde{n}^{0,\underline{t}_k}_{j_k}e_{\underline{s}_k}\right)+\displaystyle\sum_{i=1}^{\sigma}|\underline{\tilde{m}}^{0,\underline{s}_k}|-\displaystyle\sum_{i=1}^{\sigma}\tilde{m}^{0,\underline{t}_k}_{j_k}e_{\underline{s}_k}.$$
\end{proof}

\begin{remark}\label{rem:control-deg}
	From the previous proof, we observe that, for any monomial $\underline{s}^{\underline{ k} }\underline{t}^{\underline{l} }y^j$ in the support of a polynomial $P$ which satisfies the conditions of Lemma \ref{lemma:support-equ}, one has that:
	\begin{equation}\label{equ:control-deg}
		| \underline{ k}|\leq a |\underline{ l}|+b,
	\end{equation}
	where $a$ and $b$ are as in Lemma \ref{lemme:control-deg-supp}. To see this, use $\alpha_{j_k}\leq |\underline{l}|$ in place of $\alpha_{j_k}\leq |\underline{l}_0|+d|\underline{n}|$ in (\ref{equ:alpha}).
\end{remark}

\begin{ex}	For $r=2$, let $p,q\in\N^*$ and $\underline{\tilde{n}}^0=(\tilde{n}^0_1,\tilde{n}^0_2)\in \Z^2$. 
	\begin{enumerate}
		\item Let us consider:  
		\[\tilde{y}_0=\displaystyle\left(\frac{x_1}{x_2^q}\right)^{\tilde{n}^0_1/p}x_2^{\tilde{n}^0_2/p} \displaystyle\sum_{i,j=0}^{p-1}\left(\frac{1}{1-x_2}\frac{x_2^{q}}{x_2^{q}-x_1}\right) \left(\frac{x_1}{x_2^q}\right)^{i/p} x_2^{j/p}\in \mathcal{K}_2.\]
		The series $\tilde{y}_0$ is algebroid, even algebraic, since it is a finite sum and product of algebraic series. Hence, $\left(u_1,u_2\right)=\left( \left(\displaystyle\frac{x_1}{x_2^{q_1}}\right)^{1/p}, {x_2}^{1/p}\right)=(s,t)$. Moreover, it has a full support: \[\left\{ \frac{1}{p}\underline{\tilde{n}}^0+\left(\frac{k}{p},\, \frac{l-qk}{p} \right)\ |\ (k,l)\in\N^2 \right\}.\vspace{0.2cm}\]
		\begin{center}
			\includegraphics[scale=0.65]{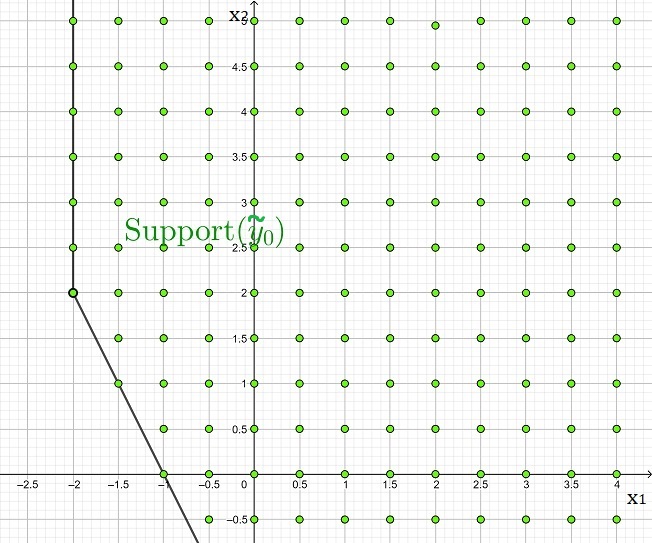} 
		\end{center}
		\item  Let us consider 
		\[\tilde{y}_0=\displaystyle\left(\frac{x_1}{x_2^q}\right)^{\tilde{n}^0_1/p}x_2^{\tilde{n}^0_2/p} \left(\frac{1}{1-{x_2}^{1/p}}\right) \exp\left(\left(\frac{x_1}{x_2^q}\right)^{1/p}\right) \in \mathcal{K}_2.\]
		The series $\tilde{y}_0$ is transcendental over $K[[x_1,x_2]]$. Indeed, with the same notations as above, $\tilde{y}_0=s^{\tilde{n}^0_1/p}t^{\tilde{n}^0_2/p}\displaystyle\frac{1}{1-t}\exp(s)$ is algebroid if and only if $\exp(s)$ is algebraic by Lemma \ref{lemma:algebraicity}. This is clearly not the case. Moreover, $\tilde{y}_0$  has the same support as above. 
	\end{enumerate} 
\end{ex}

\begin{remark}
	In \cite[Question 7.2]{kuhlmann-krapp-serra:generalised-LRR}, the authors ask whether $K((\underline{x}))$ is a Rayner field. The above example with $p=1$ provides us with two series having same support, the first belonging to $K((\underline{x}))$, and the second not. Following the argument after \cite[Question 7.2]{kuhlmann-krapp-serra:generalised-LRR}, this shows that $K((\underline{x}))$ is not a Rayner field.
\end{remark}

\section{A nested depth lemma.}\label{sect:depth}

\begin{lemma}\label{lemme:bezout-coeff}
	Let $d_{\underline{x}},\, d,\, \d_{\underline{x}},\, \d\in \mathbb{N}^*$. 
	Given two  polynomials $P\in K\left[\underline{x},y\right]\setminus\{0 \}$, $\deg_{\underline{x}}P\leq d_{\underline{x}},\ \deg_yP\leq d$, and $Q\in K\left[\underline{x},y\right]\setminus\{0 \},\ \deg_{\underline{x}}Q\leq \d_{\underline{x}},\ \deg_yQ\leq \d$, we denote by  $R\in  K\left[\underline{x}\right] $  their resultant. It satisfies $\deg_{\underline{x}}R\leq d\d_{\underline{x}}+\d d_{\underline{x}}$. Moreover, in the B\'ezout identity:
	$$AP+BQ=R,$$ one can choose the polynomials $S,\, T \in K\left[\underline{x},y\right]$	which satisfy:
	$$\left\{ \begin{array}{ll}
		\deg_{\underline{x}}A\leq d_{\underline{x}}(\d-1)+\d_{\underline{x}}d& \deg_yA\leq \d-1\\
		\deg_{\underline{x}}B\leq d_{\underline{x}}\d+\d_{\underline{x}}(d-1)& \deg_yB\leq d-1
	\end{array}\right.$$
	
	%one has that $$\mathrm{ord}_{\underline{x}}Q(\underline{x},y_0)\leq d\d_{\underline{x}}+ \dd_{\underline{x}}.$$
\end{lemma}
\begin{demo} 
	We consider the following linear map: 
	$$\begin{array}{lccl}
		\varphi:&K(\underline{x})[y]_{\d}\times K(\underline{x})[y]_{d}&\rightarrow &K(\underline{x})[y]_{d+\d}\\ 
		&(A,B)&\mapsto& AP+BQ,
	\end{array}	$$
	where $K(\underline{x})[y]_{n}$ denotes the $K(\underline{x})$-vector space of polynomials of degree less than $n$ in $y$.  
	The matrix $M$ of $\varphi$ in the standard basis $\{(y^i,0)\}\cup\{(0,y^j)\}$ and $\{y^k\}$ is the Sylvester matrix of $P$ and $Q$.	
	%Recall that the B\'ezout identity translates into a linear map of the vector space $(K(\underline{x}))^{d+\d}$, whose expression $(A,B)\mapsto AP+BQ$ in the standard basis $\{y^j\}$ is given by the Sylvester matrix of $P$ and $Q$. 
	The polynomial $R\in   K\left[\underline{x}\right] $ is its determinant. So, $\deg_{\underline{x}}R\leq d\d_{\underline{x}}+\d d_{\underline{x}}$. Let $M'$ be the matrix of cofactors of $M$. From the relation ${M}.\, {^tM'}=R\, \mathrm{Id}_{d+\d}$, one deduces the B\'ezout identity $AP+BQ=R$, the coefficients of $A$ and $B$ being minors of $M$ of  maximal order minus 1.
\end{demo}

\begin{lemma}\label{lemme:minor-rank}
	Let $\mathfrak{A}$ be a domain and $\mathfrak{K}$ its field of fractions. Given $n\in\N$, $n\geq 2$, we consider an  $n\times n$ matrix $M=(m_{i,j})$ with coefficients in $\mathfrak{A}$. We suppose that $M$ (as a matrix with coefficients in $\mathfrak{K}$) has rank $n-p$ for some $1\leq p<n$. Then there exists a vector $V\in \mathfrak{A}^n\setminus\{0\}$ whose nonzero coefficients are equal, up to sign $\pm$, to minors of order $n-p$ of $M$ and such that $M.V=0$. 
\end{lemma}
\begin{proof}
	Without loss of generality, we can suppose that the minor of order $n-p$, say $\Delta$, given by the first $n-p$ rows and columns is not zero. Denote $V:=(\Delta_1,\ldots, \Delta_n)$.  For $k>n-p+1$, set $\Delta_k:=0$. For $k=n-p+1$, set $\Delta_k:=(-1)^{n-p+1}\Delta\neq 0$. For $ k< n-p+1$, we set $\Delta_k$ equal to $(-1)^k$ times the minor of $M$ given by the first $n-p$ rows, and all but the $k$'th first $n-p+1$ columns. Denote  $M.V:=(c_1,\ldots,c_n)$. We claim that $M.V=0$. Indeed, $ c_1= \displaystyle\sum_{j=1}^{n-p+1} m_{1,j}\Delta_j$ which is the determinant of the $(n-p+1)\times(n-p+1)$-matrix $(\delta_{i,j})$ with $\delta_{i,j}=m_{i,j}$ for $1\leq i\leq n-p$ and $1\leq j\leq n-p+1$, and   $\delta_{n-p+1,j}=m_{1,j}$ for $1\leq j\leq n-p+1$. This determinant vanishes since it has two identical rows. Similarly, we have that $c_2=\cdots=c_{n-p}=0$. 
	
	Now, $c_{n-p+1}=\displaystyle\sum_{j=1}^{n-p+1} m_{n-p+1,j}\Delta_j$, which is equal to a minor of order $n-p+1$ of $M$. It vanishes since $M$ has rank $n-p$. Similarly, $c_{n-p+2}=\ldots=c_n=0$.
\end{proof}

\begin{lemma}\label{lemme:sylvester-rank}
	Let $\mathfrak{A}$ be a domain and $\mathfrak{K}$ its field of fractions.
	%Let $\mathfrak{A}$ and $\mathfrak{K}$ be as above and $\overline{\mathfrak{K}}$ an algebraic closure of $\mathfrak{K}$.
	Let $P_1,P_2\in \mathfrak{A}[y]\setminus\{0\}$ of positive degrees $d_1\geq d_2$ respectively. %, and $z_0\in \overline{\mathfrak{K}}$.
	The Sylvester matrix of $P_1$ and $P_2$ has rank at least $d_1$. 
	
	Moreover, 
	% such that $P(z_0)=0$ and $Q(z_0)\neq 0$.
	it  has rank $d_1$ if and only if   $aP_1=BP_2$  for some $a\in \mathfrak{A}$ and $B\in \mathfrak{A}[y]\setminus\{0\}$. 
	
	In this case, one can take $a={q_{d_2}}^{d_1-d_2 + 1}$ (where $q_{d_2}$ is the coefficient of $y^{d_2}$ in $P_2$) and the coefficients of such a polynomial $B$ can be computed as homogeneous polynomial formulas in the coefficients of $P_1$ and $P_2$ of degree $d_1-d_2+1$, each monomial consisting of $d_1-d_2$ coefficients of $P_2$ times  1 coefficient of $P_1$.
\end{lemma}
\begin{proof}
	As in the proof of Lemma \ref{lemme:bezout-coeff}, we denote by $M_{P_1,P_2}$  the Sylvester matrix of $P_1$ and $P_2$. By definition, its $d_1$ columns corresponding to the coefficients of $y^lP_2$, $l=0,\ldots,d_1-1$, being upper triangular are linearly independent (and the same holds for the $d_2$ columns corresponding to the coefficients of $y^kP_1$).  Hence,  $M_{P_1,P_2}$ has rank at least $\max\{d_1,d_2\}=d_1$. 
	
	Moreover, an equality $aP_1=BP_2$ translates exactly into a linear relation between the column corresponding to $P_1$ and the  columns corresponding to  $y^lP_2$ for $l=0,\ldots,d_1-d_2$. In this case, the linear relation repeats mutatis mutandi between   the column corresponding to $y^k P_1$ and the  columns corresponding to  $y^lP_2$ for $l=k,\ldots,d_1-d_2+k$, corresponding to an equality $ay^kP_1=y^kBP_2$.
	
	Let us consider the submatrix $N_{P_1,P_2}$ of $M_{P_1,P_2}$ consisting of the  column corresponding to $P_1$ and the columns corresponding  to  $y^lP_2$ for $l=0,\ldots,d_1-d_2$. It has rank $d_1-d_2+1$. By the previous lemma, there exists a nonzero vector in the kernel of  $N_{P_1,P_2}$, given by minors of order  $d_1-d_2+1$. More precisely, we are in the case of a Cramer system encoding an equality $BP_2 = aP_1$, with in particular  $a={q_{d_2}}^{d_1-d_2+1}$ corresponding to the determinant of the matrix of the linear map $B\mapsto BP_2$. By Cramer's rules, the coefficients of $B$ are  computed as determinants which indeed give homogeneous polynomial formulas with monomials consisting of $d_1-d_2$ coefficients of $P_2$ and 1 coefficient of $P_1$.  
\end{proof}

\begin{lemma}\label{lemme:ordreQ-alg}
	Let $d_{\underline{x}},\, d,\, \d_{\underline{x}},\, \d\in \mathbb{N}^*$ and $P,\, Q\in K\left[\underline{x},y\right]\setminus\{0 \}$, $\deg_{\underline{x}}P\leq d_{\underline{x}},\ \deg_yP\leq d,\, \deg_{\underline{x}}Q\leq \d_{\underline{x}},\ \deg_yQ\leq \d$. For any series   $c_0%=\displaystyle\sum_{\underline{n}\in\N^r} c_{\underline{n}}\underline{x}^{\underline{n}}
	\in K\left[\left[\underline{x}\right]\right]$ 
	such that  $P(\underline{x},c_0)=0$ and $Q(\underline{x},c_0)\neq 0$, one has that $$\mathrm{ord}_{\underline{x}}Q(\underline{x},c_0)\leq \d_{\underline{x}}d+ d_{\underline{x}}\d.$$
\end{lemma} 
\begin{demo}
	%\displaystyle\nicefrac{K[\underline{x},y]}{\mathfrak{I}_0}
	Let $ c_0 $ be a series as in the statement of Lemma \ref{lemme:ordreQ-alg}. We consider the prime ideal $\mathfrak{I}_0:=\left\{R(\underline{x},y)\in K\left[\underline{x},y\right]\ |\ R(\underline{x},c_0)=0\right\}$.
	Since $\mathfrak{I}_0\neq (0)$, $$\dim \left(K[\underline{x},y]/\mathfrak{I}_0\right)=\mathrm{trdeg}_K\mathrm{Frac}\left(K[\underline{x},y]/\mathfrak{I}_0\right)\leq r.$$ But, in $\mathrm{Frac}\left(K[\underline{x},y]/\mathfrak{I}_0\right)$, the elements $\overline{x_1},\ldots,\overline{x_r}$ are algebraically independant (if not, we would have $T(\overline{x_1},\ldots,\overline{x_r})=\overline{0}$ for some non trivial $T\in K[\underline{X}]$, i.e. $T(x_1,\ldots,x_r)\in \mathfrak{I}_0$, a contradiction). Thus, $\mathfrak{I}_0$ is a height one prime ideal of the factorial ring $K\left[\underline{x},y\right]$. It is generated by an irreducible polynomial $P_0(\underline{x},y)\in K\left[\underline{x},y\right]$. We set $d_{x,0}:=\deg_{\underline{x}} P_0$ and $d_{y,0}:=\deg_y P_0$. Note also that, by factoriality of $K\left[\underline{x},y\right]$, $P_0$ is also irreducible as an element of $K\left(\underline{x}\right)[y]$.\\
	Let $P$ be as in  the statement of Lemma \ref{lemme:ordreQ-alg}. One has that $P=SP_0$ for some $S\in K\left[\underline{x},y\right]$. Hence $d_{x,0}\leq d_{\underline{x}}$ and $d_{y,0}\leq d$. Let $Q\in K\left[\underline{x},y\right]$ be such that $Q(\underline{x},c_0)\neq 0$ with $\deg_{\underline{x}} Q\leq \d_{\underline{x}}$, $\deg_yQ\leq \d$. So $P_0$ and $Q$ are coprime in $K\left(\underline{x}\right)[y]$. Their resultant $R(\underline{x})$ is nonzero. One has the following B\'ezout relation in $K\left[\underline{x}\right][y]$:
	$$A(\underline{x},y)P_0(\underline{x},y)+B(\underline{x},y)Q(\underline{x},y)=R(\underline{x}).$$
	We evaluate at $y=c_0$:
	$$0+B(\underline{x},c_0)Q(\underline{x},c_0)=R(\underline{x}).$$
	%So $\mathrm{ord}_{\underline{x}} Q(\underline{x},c_0)\leq \deg_{\underline{x}} r(\underline{x})$. But, %the resultant is a  determinant consisting of at most $d_{y,0}$ coefficients of degree in $\underline{x}$ at most $\d_{\underline{x}}$ and at most $\d$ coefficients of degree in  $\underline{x}$ at most $d_{x,0}$.
	%of order at most $d_{y,0}+\d\leq d+\d$ whose entries are  polynomials in $K\left[\underline{x}\right]$ of degree at most $\max\{d_x,d_{0,x}\}= d_x$. So, $ \deg_{\underline{x}} r(\underline{x})\leq 2\,d_xd_y$. Hence, one has that: $\mathrm{ord}_{\underline{x}} Q(\underline{x},c_0)\leq  2\,d_xd_y$. So,
	But, by Lemma \ref{lemme:bezout-coeff},  $ \deg_{\underline{x}} R \leq d_{y,0}\d_{\underline{x}}+ \d d_{x,0}\leq  d\d_{\underline{x}}+ \d d_{\underline{x}} $. Hence, one has that: $$\mathrm{ord}_{\underline{x}} Q(\underline{x},c_0)\leq \mathrm{ord}_{\underline{x}}R \leq  \deg_{\underline{x}} R\leq   d\d_{\underline{x}}+ \d d_{\underline{x}}.$$
\end{demo}

%\subsection{Nested depth lemma}
\begin{theo}\label{propo:nested}
	Let  $i,\,d_{\underline{x}},\, d,\, \d_{\underline{x}},\, \d \in \mathbb{N}$, $d\geq 2$, $\d\geq 1$. There exists $\omega(i,d_{\underline{x}}, d, \d_{\underline{x}}, \d )\in\N$ minimal such that:
	
	{ for any $j=0,\ldots,i$, given   $c_j=\displaystyle\sum_{\underline{n}\in\N^r} c_{j,\underline{n}}\underline{x}^{\underline{n}}\in K\left[\left[\underline{x}\right]\right]$  power series   
		satisfying some equations ${P}_j(\underline{x},c_0,\ldots,c_j)=0$ % and  ${P}_{1}(\underline{x},c_0,c_1)=0$
		where 	$P_{j}\in K\left[\underline{x},z_0,z_1,\ldots,z_j \right]\setminus\{0 \}$, $\deg_{\underline{x}}P_{j}\leq d_{\underline{x} },$ 	 %$ P_{1}\in K\left[\underline{x},Y,Z\right]\setminus\{0 \}$, $\deg_{\underline{x}}{P}_{1}\leq d_{x,1},\ \deg_y{P}_{1}\leq d_{y,1}$ and
		$\deg_{z_k}{P}_{j}\leq d$ for $k=0,\ldots,j$, and $P_j (\underline{x},c_0,\ldots,c_{j-1},z_j)\not\equiv 0$,  and given  $Q_i\in K\left[\underline{x},z_0,z_1,\ldots,z_i \right]\setminus\{0 \}$, $\deg_{\underline{x}}Q_i\leq \d_{\underline{x}},\ 
		\deg_{z_j}Q_i\leq \d$ for $j=0,\ldots,i$ a polynomial
		such that $Q_i(\underline{x},c_0,c_1,\ldots,c_i)\neq 0$,  one has that $$\ord_{\underline{x}}Q_i(\underline{x},c_0,c_1,\ldots,c_i)\ \leq\   \omega(i,d_{\underline{x}}, d, \d_{\underline{x}}, \d ).$$}
	
	Moreover, for $\d\geq 3$:	   
	\begin{equation}\label{equ:estim0}
		\begin{array}{c} \omega(i,d_{\underline{x}}, d, \d_{\underline{x}}, \d )\leq (2.3^{d^{i-1}+\cdots+d^2+d+1} -2^i3^{d^{i-1}+\cdots+d^2+d-(i-1)}) d^{d^{i-1}+\cdots+d^2+d+1} d_{\underline{x}}\d^{d^i}+\\ 2^i.3^{d^{i-1}+\cdots+d^2+d-(i-1)}  d^{d^{i-1}+\cdots+d^2+d+2}  \d_{\underline{x}}  \d^{d^i-1} . \end{array}
	\end{equation} 
	So, for $d\geq 3$:
	\begin{equation}\label{equ:estim1}
		\omega(i,d_{\underline{x}}, d, d_{\underline{x}}, d )\leq 2.3^{d^{i-1}+\cdots+d^2+d+1}    
		d_{\underline{x}}  d^{d^{i}+\cdots+d^2+d+1} .  
	\end{equation}  
	Finally, for any $\varepsilon>0$, there is $\d_\varepsilon$ such that, for $\d\geq \d_\varepsilon$:
	\begin{equation}\label{equ:estim2}
		\begin{array}{c}	\omega(i,d_{\underline{x}}, d, \d_{\underline{x}}, \d )	\leq \hspace{8cm}\\
			\left(	2.(2+\varepsilon)^{d^{i-1}+\cdots+d^2+d+1} - 
			(1+\varepsilon)^i.(2+\varepsilon)^{d^{i-1}+\cdots+d^2+d-(i-1)}
			\right)d^{d^{i-1}+\cdots+d^2+d+1} d_{\underline{x}}\d^{d^i}	+\\
			(1+\varepsilon)^i.(2+\varepsilon)^{d^{i-1}+\cdots+d^2+d-(i-1)}  d^{d^{i-1}+\cdots+d^2+d+2}  \d_{\underline{x}}  \d^{d^i-1} ,    \end{array}
	\end{equation}	 and for $d\geq \d_\varepsilon$: 
	\begin{equation}\label{equ:estim3}
		\omega(i,d_{\underline{x}}, d, d_{\underline{x}}, d )	\leq 
		2.(2+\varepsilon)^{d^{i-1}+\cdots+d^2+d+1} 
		d^{d^i+d^{i-1}+\cdots+d^2+d+1} d_{\underline{x}}.	
	\end{equation}
\end{theo} 
\begin{demo}
	We proceed by induction on $i\in\N$, the case $i=0$ being Lemma \ref{lemme:ordreQ-alg} where we set $d^{i-1}+\cdots+d^2+d+1:=0 $, $d^{i-1}+\cdots+d^2+d+2:=d^{i-1}+\cdots+d^2+d+1+1=1$ and  $d^{i-1}+\cdots+d^2+d-(i-1):=0$ and where we get: 
	$$\mathrm{ord}_{\underline{x}}Q_0(\underline{x},c_0)\leq \d_{\underline{x}}d+ d_{\underline{x}}\d.$$ 
	Suppose that the property holds until some rank $i-1\geq 0$, and consider polynomials $P_i$ and $Q_i$ as in the statement of the theorem. 	Let $R_1$ be the resultant of $P_{i}$ and $Q_i$ with respect to $z_i$, and the following B\'ezout identity according to Lemma \ref{lemme:bezout-coeff} (where $\underline{x}$ there stands for $\underline{x}$ or $z_j$, $j=0,..,i-1$, here):
	$$A_1P_i+B_1Q_i=R_1.$$
	There are two cases. If $R_1(\underline{x},c_0,\ldots,c_{i-1})\neq 0$,  since $R_1\in K\left[\underline{x},z_0,\ldots,z_{i-1}\right]$ with $\deg_{\underline{x}}{R_1}\leq d_{\underline{x}}\d +\d_{\underline{x}}d,\ \deg_{z_j}{R_1}\leq 2d \d$ for $j=1,\ldots,i-1$, we deduce  from the induction hypothesis that $\mathrm{ord}_{\underline{x}} R_1(\underline{x},c_0,\ldots,c_{i-1})\leq \omega(i-1,d_{\underline{x}},d,d_{\underline{x}}\d +\d_{\underline{x}}d, 2d \d )$.  % where $\omega_{i-1^,1}$ is the bound obtained at the rank $i-1$ with the parameters  $((d_{x,j},d_{y,j},d_{z_1,j},\ldots,d_{z_{j},j})_{j=0,\ldots,i-1})$ for the $P_j$'s, and $(d_{x,i}\d_{z_i,i}+\d_{x,i}d_{z_i,i},\,d_{y,i}\d_{z_i,i}+\d_{y,i}d_{z_i,i},\,d_{z_1,i}\d_{z_i,i}+\d_{z_1,i}d_{z_i,i},\ldots)$ instead of $(\d_{x,i-1},\,\d_{y,i-1},\,\d_{z_1,i-1},\ldots)$ for $Q_{i-1}$.
	So, by the B\'ezout identity: $$\mathrm{ord}_{\underline{x}} Q_i(\underline{x},c_0,\ldots,c_{i})\leq \mathrm{ord}_{\underline{x}}R_1(\underline{x},c_0,\ldots,c_{i-1}) \leq \omega (i-1,d_{\underline{x}},d,d_{\underline{x}}\d +\d_{\underline{x}}d, 2d \d ).$$
	If   $R_1(\underline{x},c_0,\ldots,c_{i-1})=0$, then $B_1(\underline{x},c_0,\ldots,c_{i-1},c_i)=0$. There are several sub-cases.
	
	\begin{lemma}\label{lemma:bezout-ind}
		If   $R_1(\underline{x},c_0,\ldots,c_{i-1})=0$, then there exist $A,B\in K\left[\underline{x},z_0,\ldots,z_{i}\right] $ such that $B(\underline{x},c_0,\ldots,c_{i-1},c_i)=0$, $B(\underline{x},c_0,\ldots,c_{i-1},z_i)\not\equiv 0$ and
		$$ A(\underline{x},c_0,\ldots,c_{i-1},z_i)P_i(\underline{x},c_0,\ldots,c_{i-1},z_i)+B(\underline{x},c_0,\ldots,c_{i-1},z_i) Q_i(\underline{x},c_0,\ldots,c_{i-1},z_i)=0$$
		with  $\deg_{\underline{x}}B\leq d_{\underline{x}}\d +\d_{\underline{x}}(d-1),\ \deg_{z_j}B\leq (2d-1) \d$ for $j=1,\ldots,i-1$, and  $\deg_{z_i} B\leq d-1$.
	\end{lemma}
	\begin{proof}
		If $B_1(\underline{x},c_0,\ldots,c_{i-1},z_i)\not\equiv 0$, we take $A=A_1$ and $B=B_1$, noticing by Lemma  \ref{lemme:bezout-coeff} that $\deg_{\underline{x}}B_1\leq d_{\underline{x}}\d +\d_{\underline{x}}(d-1),\ \deg_{z_j}B_1\leq (2d-1) \d$ for $j=1,\ldots,i-1$, and  $\deg_{z_i} B_1\leq d-1$.
		
		If $B_1(\underline{x},c_0,\ldots,c_{i-1},z_i)\equiv 0$, necessarily $A_1(\underline{x},c_0,\ldots,c_{i-1},z_i)\equiv 0$. 
		
		Let us denote $\tilde{P}_i:=P_i(\underline{x},c_0,\ldots,c_{i-1},z_i)$ and $\tilde{Q}_i:=Q_i(\underline{x},c_0,\ldots,c_{i-1},z_i)$, hence $\tilde{P}_i,\tilde{Q}_i\in K[\underline{x},c_0,\ldots,c_{i-1}][z_i]$,  with degrees $\tilde{d}$ and $\tilde{\d}$ in $z_i$ respectively. Note that $\tilde{d}\geq 1$ and $\tilde{\d}\geq 1$ (if not, $R_1(\underline{x},c_0,\ldots,c_{i-1})\neq 0$). Let $M_{\tilde{P}_i,\tilde{Q}_i}$ be the Sylvester matrix of $\tilde{P}_i$ and $\tilde{Q}_i$, and $\tilde{d}+\tilde{\d}-p$ its rank. Hence, $p\geq 1$. Suppose that $p=1$. Let us denote by $M'_{\tilde{P}_i,\tilde{Q}_i}$ the matrix of cofactors of $M_{\tilde{P}_i,\tilde{Q}_i}$, and by  $^tM'_{\tilde{P}_i,\tilde{Q}_i}$ its transpose. At least one of the columns of $^tM'_{\tilde{P}_i,\tilde{Q}_i}$ is not zero. Since we have that $M_{\tilde{P}_i,\tilde{Q}_i}.^tM'_{\tilde{P}_i,\tilde{Q}_i}=0$, this column determines a non-trivial relation
		$$\tilde{A}\tilde{P}_i+\tilde{B}\tilde{Q}_i=0$$
		where the coefficients of $\tilde{A},\tilde{B}$ are given by the coefficients of this column. Moreover, $\tilde{B}(\underline{x},c_0,\ldots,c_{i-1},c_i)=0$ since $\tilde{P}_i(\underline{x},c_0,\ldots,c_{i-1},c_i)=0$ and  $\tilde{Q}_i(\underline{x},c_0,\ldots,c_{i-1},c_i)\neq 0$, and $\tilde{B}(\underline{x},c_0,\ldots,c_{i-1},z_i)\not\equiv 0$ (if not, we would have $\tilde{A}(\underline{x},c_0,\ldots,c_{i-1},z_i)\equiv 0$ since \\  $\tilde{P}_i(\underline{x},c_0,\ldots,c_{i-1},z_i)\not\equiv 0$). The coefficients of $\tilde{B}$ are homogeneous polynomial formulas in $\tilde{\d}$ coefficients of $\tilde{P}_i$ and $\tilde{d}-1$ coefficients of $\tilde{Q}_i$. Lifting these formulas to $K[\underline{x},z_0,\ldots,z_{i-1},z_i]$ by replacing the $c_j$'s by the $z_j$'s, we obtain $A$ and $B$ with  $\deg_{\underline{x}}B\leq d_{\underline{x}}\tilde{\d} +\d_{\underline{x}}(\tilde{d}-1),\ \deg_{z_j}B\leq d \tilde{\d} +\d (\tilde{d}-1)$ for $j=1,\ldots,i-1$, and  $\deg_{z_i} B\leq \tilde{d}-1$. We conclude since $\tilde{\d}\leq \d$ and  $ \tilde{d}\leq d$. 
		
		Suppose that  $p\geq 2$. The $\tilde{\d}$ columns corresponding to the coefficients of the ${z_i}^k \tilde{P}_i$'s, $k=0,..,\tilde{\d}-1$, are linearly independent (since they form an upper triangular system). We complete them with $\tilde{d}-p$ columns corresponding to the coefficients of the ${z_i}^k \tilde{Q}_i$ to a maximal linearly independent family. There is a non-zero minor, say $\Delta$, of maximal order $\tilde{\d}+\tilde{d}-p$ of this family. Proceeding as in Lemma \ref{lemme:minor-rank}, there is a non-zero vector $V$ in the kernel of $M_{\tilde{P}_i,\tilde{Q}_i}$ whose coefficients are minors of order   $\tilde{\d}+\tilde{d}-p$. More precisely, except for  $\Delta$, the other minors are  obtained by replacing a column of $\Delta$ by the corresponding part of another column of $M_{\tilde{P}_i,\tilde{Q}_i}$. Hence, they consist of either  $\tilde{d}-p+1$ columns with coefficients of $\tilde{Q}_i$ and $\tilde{\d}-1$ columns with coefficients of $\tilde{P}_i$, or  $\tilde{d}-p$ columns with coefficients of $\tilde{Q}_i$ and $\tilde{\d}$ columns with coefficients of $\tilde{P}_i$. We translate the relation $M_{\tilde{P}_i,\tilde{Q}_i}.V=0$ to a non-trivial relation 
		$$\tilde{A}\tilde{P}_i+\tilde{B}\tilde{Q}_i=0$$
		where the coefficients of $\tilde{A},\tilde{B}$ are given by the coefficients de $V$. Moreover,\\  $\tilde{B}(\underline{x},c_0,\ldots,c_{i-1},c_i)=0$ since  $\tilde{P}_i(\underline{x},c_0,\ldots,c_{i-1},c_i)=0$ and  $\tilde{Q}_i(\underline{x},c_0,\ldots,c_{i-1},c_i)\neq 0$, and $\tilde{B}(\underline{x},c_0,\ldots,c_{i-1},z_i)\not\equiv 0$ (if not, we would have $\tilde{A}(\underline{x},c_0,\ldots,c_{i-1},z_i)\equiv 0$ since \\  $\tilde{P}_i(\underline{x},c_0,\ldots,c_{i-1},z_i)\not\equiv 0$). The coefficients of $\tilde{B}$ are homogeneous polynomial formulas in at most $\tilde{\d}$ coefficients of $\tilde{P}_i$ and $\tilde{d}-p+1$ coefficients of $\tilde{Q}_i$. Lifting these formulas to $K[\underline{x},z_0,\ldots,z_{i-1},z_i]$ by replacing the $c_j$'s by the $z_j$'s, since $p\geq 2$, we obtain $A$ and $B$ with  $\deg_{\underline{x}}B\leq d_{\underline{x}}\tilde{\d} +\d_{\underline{x}}(\tilde{d}-1),\ \deg_{z_j}B\leq d \tilde{\d} +\d (\tilde{d}-1)$ for $j=1,\ldots,i-1$, and  $\deg_{z_i} B\leq \tilde{d}-1$. We conclude since $\tilde{\d}\leq \d$ and  $ \tilde{d}\leq d$. 
	\end{proof}

	% Then we consider the biggest power ${\tilde{Q}_i}^l$ which divides $\tilde{P}_i$ and we set $\hat{P}_i:=\tilde{P}_i/{\tilde{Q}_i}^l$: we are reduced to the case of two polynomials $\hat{P}_i$ and $\tilde{Q}_i$ where neither of them divides the other.
	
	We denote by $B_1$ the polynomial  $B$ of the previous lemma. In any case, we are in position to replace $P$ by $B_1$, with  $\deg_{\underline{x}}B_1\leq d_{\underline{x}}\d +\d_{\underline{x}}(d-1),\ \deg_{z_j}B_1\leq (2d-1) \d$ for $j=1,\ldots,i-1$, and  $\deg_{z_i} B_1\leq d-1$. We obtain another B\'ezout identity:
	$$A_2B_1+B_2Q_i=R_2$$
	with $R_2$ the resultant of $B_1$ and $Q_i$ with respect to $z_i$, $$\deg_{\underline{x}}{R_2}\leq  (d_{\underline{x}}\d +\d_{\underline{x}}(d -1) )\d + \d_{\underline{x}}(d -1) =  d_{\underline{x}}{\d}^2+\d_{\underline{x}} ((d-1)\d +(d -2)+1),$$ likewise,  for $j=1,\ldots,i-1$, $$\deg_{z_j}{R_2}\leq d \d^2+\d  ((d -1)\d +(d-2)+1). $$
	Moreover, $$\begin{array}{lcl}
		\deg_{\underline{x}}{B_2}&\leq&  (\deg_{\underline{x}}{B_1})\d +\d_{\underline{x}}(\deg_{z_i}{B_1} -1)\\
		&\leq&  (d_{\underline{x}}\d +\d_{\underline{x}}(d-1))\d +\d_{\underline{x}}(d-1 -1)=d_{\underline{x}}\d^2+\d_{\underline{x}}(\d(d-1)+d-2),
	\end{array}$$ 
	and likewise,  for $j=1,\ldots,i-1$, 
	$$\begin{array}{lcl}
		\deg_{z_j}{B_2}&\leq &(\deg_{z_j}B_1)\d + (\deg_{z_i}B_1-1)\d\\
		&\leq& (2d-1)\d^2+(d-2)\d =d\d^2+\d (\d(d-1)+d-2),
	\end{array} $$
	and 
	$$ \deg_{z_i}B_2\leq \deg_{z_i} B_1-1\leq d-2.$$
	
	If $R_2(\underline{x},c_0,\ldots,c_{i-1})\neq 0$, we proceed as before Lemma \ref{lemma:bezout-ind}, and  we obtain:
	\begin{center}
		$ \mathrm{ord}_{\underline{x}} Q_i(\underline{x},c_0,\ldots,c_i)\leq  \mathrm{ord}_{\underline{x}} R_2(\underline{x},c_0,\ldots,c_{i-1})\leq \omega\left(i-1,d_x,\,d,\,d_{\underline{x}}{\d}^2+\d_{\underline{x}} ((d-1)\d +(d -2)+1),\, d \d^2+\d  ((d -1)\d +(d-2)+1)\right).
		$
	\end{center}
	Note that this new bound for  $\mathrm{ord}_{\underline{x}} Q_i(\underline{x},c_0,\ldots,c_{i-1},c_i)$ has  increased with respect to the previous one, since $d\leq (d-1)(\d+1)=(d-1)\d +(d -2)+1$ for any $d\geq 2$, $\d\geq 1$. At worst, one can have repeatedly the second case with successive  B\'ezout identities:
	$$A_kB_{k-1}+B_kQ_i=R_k$$
	with $R_k(\underline{x},c_0,\ldots,c_{i-1})=0$ where for $j=0,\ldots,i-1$,  
	$$\left\{\begin{array}{lcl}
		\deg_{\underline{x}}{R_k}&\leq& d_{\underline{x}}\d^k+\d_{\underline{x}}\left(\d^{k-1}(d-1)+{\d}^{k-2}(d-2)+\cdots+{\d}(d-(k-1))+(d-k)+1\right)\\
		\deg_{z_j}{R_k}&\leq&
		d{\d}^k+\d\left({\d}^{k-1}(d-1)+{\d}^{k-2}(d-2)+\cdots+{\d}(d-(k-1))+(d-k)+1\right),
	\end{array}\right.$$ and with  
	$$\left\{\begin{array}{lcl}
		\deg_{\underline{x}}{B_k}&\leq&  d_{\underline{x}}{\d }^k+\d_{\underline{x}}\left({\d}^{k-1}(d-1)+{\d}^{k-2}(d-2)+\cdots+{\d}(d-k+1)+(d-k)\right)\\
		\deg_{z_j}{B_k}&\leq&
		d{\d}^k+\d\left({\d}^{k-1}(d-1)+{\d}^{k-2}(d-2)+\cdots+{\d}(d-k+1)+(d-k)\right)\\
		\deg_{z_i}B_k&\leq& d-k.
	\end{array}\right.$$
	The greatest bound is obtained for $k=d-1$, for which $B_{d-1}$ has $\deg_{z_i}B_{d-1}= 1$. In this case, $B_{d-1}$ has $c_i$ as unique root and $Q_i(\underline{x},c_0,\ldots,c_{i-1},c_i)\neq 0$, so $R_d(\underline{x},c_0,\ldots,c_{i-1})\neq 0$. We set for $n,m\in\N^*$:
	$$\begin{array}{lcl}
		\phi(n,m)&:=&(n-1){m}^{n-1}+(n-2){m}^{n-2}+\cdots+{m} +1\\
		&=&\left((n-1){m}^{n-2}+(n-2){m}^{n-3}+\cdots+2{m} +1\right)m+1\\
		&=& \displaystyle\frac{(n-1)m^{n+1}-nm^n+m^2-m+1}{(m-1)^2} \textrm{ for }m\neq 1
	\end{array}
	$$
	%NB: ON PEUT AVOIR $c_1\in K[x,c_0]$ MAIS AVEC DE LA MULTIPLICITE, CE QUI CONDUIT A UNE EQUATION AVEC DU DEGRE EN $c_2$\\
	We have for $j=0,\ldots,i-1$:
	$$\left\{\begin{array}{lcl}
		\deg_{\underline{x}}R_{d}&\leq& d_{\underline{x}}\d^{d}+\d_{\underline{x}}\phi(d ,\d )  \\
		%	\deg_y{R_{d_{z_i,i}}}&\leq& d_{y,i}\d_{z_i,i}^{d_{z_i,i}}+\d_{y,i}\phi(d_{z_i,i},\d_{z_i,i})\\
		\deg_{z_j}R_{d}&\leq& d\d^{d}+\d\phi(d,\d),
	\end{array}\right.$$
	By the induction hypothesis, $ \mathrm{ord}_{\underline{x}} R_{d}(\underline{x},c_0,\ldots,c_{i-1})$ is bounded by \\ $\omega\left(i-1,d_{\underline{x}},d, d_{\underline{x}}\d^{d}+\d_{\underline{x}}\phi(d ,\d ), d\d^{d}+\d\phi(d,\d) \right)$.
	We get the corresponding expected bound:
	$$ \mathrm{ord}_{\underline{x}} Q_i(\underline{x},c_0,\ldots,c_{i-1},c_i)\leq \omega \left(i-1,d_{\underline{x}},d, d_{\underline{x}}\d^{d}+\d_{\underline{x}}\phi(d ,\d ), d\d^{d}+\d\phi(d,\d) \right),$$
	which proves the existence of $\omega(i,d_{\underline{x}},d, \d_{\underline{x}},\d)$ with  
	\begin{equation}\label{equ:omega_i}
		\omega \left(i,d_{\underline{x}},d,\d_{\underline{x}},\d \right)\leq \omega\left(i-1,d_{\underline{x}},d, d_{\underline{x}}\d^{d}+\d_{\underline{x}}\phi(d ,\d ), d\d^{d}+\d\phi(d,\d) \right).
	\end{equation}
	To bound $\omega(i,d_{\underline{x}},d, \d_{\underline{x}},\d)$, we need to find estimates for $\phi$. \\
	First step: for $n,m\geq 2$,
	$$\phi(n,m)\leq (n-1)m^n.$$
	Indeed, $\phi(n,m)=\displaystyle\frac{(n-1)m^{n+1}-nm^n+m^2-m+1}{(m-1)^2}$. For $n\geq 2$, $-nm^n+m^2-m+1\leq 0$, so $\phi(n,m)\leq \displaystyle\frac{(n-1)m^{n+1}}{(m-1)^2}$ and $\displaystyle\frac{(n-1)m^{n+1}}{(m-1)^2}\leq (n-1) m^n\Leftrightarrow \displaystyle\frac{ m}{(m-1)^2}\leq 1 \Leftrightarrow m^2- 3m+1\geq 0$ with $\Delta=5$ et $m=(3+\sqrt{5})/2< 3$. This holds for $m\geq 3$.	For $m=2$, we compute:
	$$\phi(n,2)=(n-1)2^{n+1}-n2^n+3\leq (n-1)2^n\Leftrightarrow 3\leq 2^n$$
	This holds for $n\geq 2$. On the other hand, this does not hold for $m=1$ and $n\geq 3$.\\
	Second step: for $n\geq 3$, $m\geq 2$,
	\begin{equation}\label{equ:phi}
		\phi(n,m)\leq (2n-3)m^{n-1}
	\end{equation}
	Indeed, from the first step:  
	\[ \begin{array}{lcl}
		\phi(n,m):=(n-1){m}^{n-1}+(n-2){m}^{n-2}+\cdots+{m} +1&=&(n-1){m}^{n-1}+\phi(n-1,m)\\
		&\leq& (n-1)m^{n-1}+(n-2)m^{n-1}\\
		&\leq & (2n-3)m^{n-1}
	\end{array}	\]

	Let $\varepsilon>0$. For $n\geq 2$, since $-nm^n+m^2-m+1\leq 0 $, the inequality 
	\begin{equation}\label{equ:phi-eps}
		\phi(n,m)\leq (1+\varepsilon)(n-1)m^{n-1}
	\end{equation}  is implied by $$\displaystyle\frac{(n-1)m^{n+1}}{(m-1)^2}\leq (1+\varepsilon)(n-1)m^{n-1}\Leftrightarrow \displaystyle\frac{m^{2}}{(m-1)^2}\leq 1+\varepsilon. $$
	This holds for $m$ large enough, say for $m\geq m_\varepsilon$, since $ \displaystyle\frac{m^{2}}{(m-1)^2}$ decreases to 1.
	
	Now, let us prove the estimates for $\omega(i,\ldots)$ by induction on $i$. 	For $i=0$, $\omega(0,\ldots)\leq d\d_{\underline{x}}+ \d d_{\underline{x}}$ by Lemma \ref{lemme:ordreQ-alg}. Suppose that the estimates (\ref{equ:estim0}), (\ref{equ:estim1}), (\ref{equ:estim2}) and  (\ref{equ:estim3}) hold until some $i\geq 0$. By (\ref{equ:omega_i}):
	\begin{center}
		$\omega \left(i+1,d_{\underline{x}},d,\d_{\underline{x}},\d \right)\leq \omega \left(i,d_{\underline{x}},d, d_{\underline{x}}\d^{d}+\d_{\underline{x}}\phi(d ,\d ), d\d^{d}+\d\phi(d,\d) \right)$\\
		$\leq \omega \left(i,d_{\underline{x}},d, d_{\underline{x}}\d^{d}+\d_{\underline{x}}
		(2d-3)\d^{d-1}, d\d^{d}+\d(2d-3)\d^{d-1} \right)$\\
		$\leq \omega\left(i,d_{\underline{x}},d, d_{\underline{x}}\d^{d}+\d_{\underline{x}}
		2d\d^{d-1}, d\d^{d}+\d2d\d^{d-1} \right)$\\
		$\leq \omega \left(i,d_{\underline{x}},d, d_{\underline{x}}\d^{d}+\d_{\underline{x}}
		2d\d^{d-1}, 3d\d^{d} \right)$\\
		$\leq (2.3^{d^{i-1}+\cdots+d^2+d+1} -2^i3^{d^{i-1}+\cdots+d^2+d-(i-1)}) d^{d^{i-1}+\cdots+d^2+d+1} d_{x}(3d\d^{d})^{d^i}+ 2^i.3^{d^{i-1}+\cdots+d^2+d-(i-1)}  d^{d^{i-1}+\cdots+d^2+d+2}  (d_{\underline{x}}\d^{d}+\d_{\underline{x}}
		2d\d^{d-1})  (3d\d^{d})^{d^i-1} $\\
		$\leq (2.3^{d^i+d^{i-1}+\cdots+d^2+d+1} -2^i3^{d^i+d^{i-1}+\cdots+d^2+d-(i-1)}) d^{d^i+d^{i-1}+\cdots+d^2+d+1} d_{x}\d^{d^{i+1}}+ 
		2^i.3^{d^i+d^{i-1}+\cdots+d^2+d-(i-1)-1}  d^{d^i+d^{i-1}+\cdots+d^2+d+1}  
		d_{\underline{x}} \d^{d^{i+1}}+
		2^{i+1}.3^{d^i+d^{i-1}+\cdots+d^2+d-i}  d^{d^i+d^{i-1}+\cdots+d^2+d+2}  \d_{\underline{x}}   \d^{d^{i+1}-1}$\\
		$\leq (2.3^{d^i+d^{i-1}+\cdots+d^2+d+1} -2^i3^{d^i+d^{i-1}+\cdots+d^2+d-(i-1)}) d^{d^i+d^{i-1}+\cdots+d^2+d+1} d_{x}\d^{d^{i+1}}+ 
		\displaystyle\frac{1}{3}2^i3^{d^i+d^{i-1}+\cdots+d^2+d-(i-1)} d^{d^i+d^{i-1}+\cdots+d^2+d+1}  
		d_{\underline{x}} \d^{d^{i+1}}+
		2^{i+1}.3^{d^i+d^{i-1}+\cdots+d^2+d-i}  d^{d^i+d^{i-1}+\cdots+d^2+d+2}  \d_{\underline{x}}   \d^{d^{i+1}-1}$\\
		$\leq (2.3^{d^i+d^{i-1}+\cdots+d^2+d+1} -\displaystyle\frac{2}{3}2^i3^{d^i+d^{i-1}+\cdots+d^2+d-(i-1)}) d^{d^i+d^{i-1}+\cdots+d^2+d+1} d_{x}\d^{d^{i+1}}+
		2^{i+1}.3^{d^i+d^{i-1}+\cdots+d^2+d-i}  d^{d^i+d^{i-1}+\cdots+d^2+d+2}  \d_{\underline{x}}  \d^{d^{i+1}-1}$\\
		$\leq (2.3^{d^i+d^{i-1}+\cdots+d^2+d+1} -2^{i+1}3^{d^i+d^{i-1}+\cdots+d^2+d-i}) d^{d^i+d^{i-1}+\cdots+d^2+d+1} d_{x}\d^{d^{i+1}}+
		2^{i+1}.3^{d^i+d^{i-1}+\cdots+d^2+d-i}  d^{d^i+d^{i-1}+\cdots+d^2+d+2}  \d_{\underline{x}}  \d^{d^{i+1}-1}.$
	\end{center}
	This proves   (\ref{equ:estim0}), and also (\ref{equ:estim1}) by letting $\d\leq d$ and  $\d_{\underline{x}}\leq d_{\underline{x}}$.
	
	Similarly, given $\varepsilon>0$, we use (\ref{equ:omega_i}) and (\ref{equ:phi-eps}) with $\d\geq \d_\varepsilon$ and, since $d-1<d$, we get:
	\begin{center}
		$\omega\left(i+1,d_{\underline{x}},d,\d_{\underline{x}},\d \right)\leq \omega\left(i,d_{\underline{x}},d, d_{\underline{x}}\d^{d}+\d_{\underline{x}}
		(1+\varepsilon)d\d^{d-1}, (2+\varepsilon)d\d^{d} \right)$\\
		$\leq (2.(2+\varepsilon)^{d^{i-1}+\cdots+d^2+d+1} -(1+\varepsilon)^i(2+\varepsilon)^{d^{i-1}+\cdots+d^2+d-(i-1)}) d^{d^{i-1}+\cdots+d^2+d+1} d_{x}((2+\varepsilon)d\d^{d})^{d^i}+ (1+\varepsilon)^i.(2+\varepsilon)^{d^{i-1}+\cdots+d^2+d-(i-1)}  d^{d^{i-1}+\cdots+d^2+d+2}  (d_{\underline{x}}\d^{d}+\d_{\underline{x}}	(1+\varepsilon)d\d^{d-1})  ((2+\varepsilon)d\d^{d})^{d^i-1} $\\
		$\leq (2.(2+\varepsilon)^{d^i+d^{i-1}+\cdots+d^2+d+1} -(1+\varepsilon)^i(2+\varepsilon)^{d^i+d^{i-1}+\cdots+d^2+d-(i-1)}) d^{d^i+d^{i-1}+\cdots+d^2+d+1} d_{x}\d^{d^{i+1}}+$\\ 
		$(1+\varepsilon)^i(2+\varepsilon)^{d^i+d^{i-1}+\cdots+d^2+d-(i-1)-1}  d^{d^i+d^{i-1}+\cdots+d^2+d+1}  
		d_{\underline{x}} \d^{d^{i+1}}+$\\
		$(1+\varepsilon)^{i+1}(2+\varepsilon)^{d^i+d^{i-1}+\cdots+d^2+d-i}  d^{d^i+d^{i-1}+\cdots+d^2+d+2}  \d_{\underline{x}}   \d^{d^{i+1}-1}$\\
		$\leq (2.(2+\varepsilon)^{d^i+d^{i-1}+\cdots+d^2+d+1} -(1+\varepsilon)^i(2+\varepsilon)^{d^i+d^{i-1}+\cdots+d^2+d-(i-1)}) d^{d^i+d^{i-1}+\cdots+d^2+d+1} d_{x}\d^{d^{i+1}}+ 
		\displaystyle\frac{1}{(2+\varepsilon)}(1+\varepsilon)^i(2+\varepsilon)^{d^i+d^{i-1}+\cdots+d^2+d-(i-1)} d^{d^i+d^{i-1}+\cdots+d^2+d+1}  d_{\underline{x}} \d^{d^{i+1}}+$\\
		$(1+\varepsilon)^{i+1}(2+\varepsilon)^{d^i+d^{i-1}+\cdots+d^2+d-i}  d^{d^i+d^{i-1}+\cdots+d^2+d+2}  \d_{\underline{x}}   \d^{d^{i+1}-1}$\\
		$\leq (2.(2+\varepsilon)^{d^i+d^{i-1}+\cdots+d^2+d+1} -\displaystyle\frac{(1+\varepsilon)}{(2+\varepsilon)}(1+\varepsilon)^i(2+\varepsilon)^{d^i+d^{i-1}+\cdots+d^2+d-(i-1)}) d^{d^i+d^{i-1}+\cdots+d^2+d+1} d_{x}\d^{d^{i+1}}+
		(1+\varepsilon)^{i+1}.(2+\varepsilon)^{d^i+d^{i-1}+\cdots+d^2+d-i}  d^{d^i+d^{i-1}+\cdots+d^2+d+2}  \d_{\underline{x}}  \d^{d^{i+1}-1}$\\
		$\leq (2.(2+\varepsilon)^{d^i+d^{i-1}+\cdots+d^2+d+1} -(1+\varepsilon)^{i+1}(2+\varepsilon)^{d^i+d^{i-1}+\cdots+d^2+d-i}) d^{d^i+d^{i-1}+\cdots+d^2+d+1} d_{x}\d^{d^{i+1}}+$\\
		$(1+\varepsilon)^{i+1}(2+\varepsilon)^{d^i+d^{i-1}+\cdots+d^2+d-i}  d^{d^i+d^{i-1}+\cdots+d^2+d+2}  \d_{\underline{x}}  \d^{d^{i+1}-1}.$
	\end{center}
	This proves   (\ref{equ:estim2}), and also (\ref{equ:estim3}) by letting $\d\leq d$ and  $\d_{\underline{x}}\leq d_{\underline{x}}$.
\end{demo}

\section{Total reconstruction of vanishing polynomials for several algebraic series.}\label{section:Wilc}

In the present section, we provide several improvements of \cite{hickel-matu:puiseux-alg-multivar}.

\subsection{Total reconstruction in the algebraic case.}\label{sect:reconstr-alg}

\begin{defi}\label{defi:alg-relative}
	\begin{itemize}
		\item 	Let $\mathcal{F}'$ and $\mathcal{G}'$ be two strictly increasing finite sequences of pairs $(\underline{k},j)\in\left(\mathbb{N}^\tau\times\mathbb{N}\right)_{\textrm{alex}*}$ ordered anti-lexicographically: 
		$$(\underline{k}_1,j_1) \leq_{\textrm{alex}*} (\underline{k}_2,j_2)\Leftrightarrow  j_1 < j_2\textrm{ or } (j_1 = j_2\ \textrm{and}\ \underline{k}_1 \leq_{\textrm{grlex}} \underline{k}_2).$$
		We suppose additionally that  $ (\underline{k}_1,j_1) \geq_{\textrm{alex}*} \left(\underline{0},1\right)>_{\textrm{alex}*}(\underline{k}_2,j_2)$ for any $ (\underline{k}_1,j_1)\in \mathcal{F}'$ and $(\underline{k}_2,j_2)\in \mathcal{G}'$ (thus the elements of $\mathcal{G}'$ are  ordered pairs of the form $(\underline{k}_2,0)$, and those of  $\mathcal{F}'$ are of the form  $(\underline{k}_1,j_1),\ j_1\geq 1$). \\
		We denote $d_{y'}':=\max\{j,\ (\underline{k},j)\in\mathcal{F}'\}$ and  $d_{\underline{ s}}':=\max\{|\underline{k}|,\ (\underline{k},j)\in\mathcal{F}'\cup\mathcal{G}'\}$.		
		\item 	 We say that a series $y_0'=\displaystyle\sum_{\underline{m}\in \mathbb{N}^\tau} c_{\underline{m}}\underline{s}^{\underline{m}}\in K[[\underline{s}]]$ is \textbf{algebraic relatively to $(\mathcal{F}',\mathcal{G}')$} if there exists a polynomial $P(\underline{s},y')=\displaystyle\sum_{(\underline{k},j)\in\mathcal{F}'\cup\mathcal{G}'} a_{\underline{k},j}\underline{s}^{\underline{k}}{y'}^j\in K[\underline{s},y']\setminus\{0\}$ such that $P(\underline{s},y_0')=0$. 
		\item Let $d_{y'}', d_{\underline{ s}}'\in \N$, $d_{y'}'\geq 1$. We say that a series $y_0' \in K[[\underline{s}]]$ is \textbf{algebraic of degrees bounded by $d_{y'}'$ and $d_{\underline{ s}}'$} if it is algebraic  relatively to $(\mathcal{F}',\mathcal{G}')$ where $\mathcal{F}'$ and $\mathcal{G}'$ are the complete sequences of indices $(\underline{k},j)\in\left(\mathbb{N}^\tau\times\mathbb{N}\right)_{\textrm{alex}*}$ with $j\leq d_{y'}'$ and  $|\underline{k}|\leq d_{\underline{ s}}'$. 
	\end{itemize}
\end{defi}

Let us consider a series $Y_0'=\displaystyle\sum_{\underline{m}\in \mathbb{N}^\tau} C_{\underline{m}}\underline{s}^{\underline{m}}\in K[(C_{\underline{m}})_{\underline{m}\in\mathbb{N}^\tau}][[\underline{s}]]$ where $\underline{s}$ and the $C_{\underline{m}}$'s are variables. We denote the multinomial expansion of the $j$th power ${Y_0'}^j$  of $Y_0'$ by:
$${Y_0'}^j=\displaystyle\sum_{\underline{m}\in\mathbb{N}^\tau} C_{\underline{m}}^{(j)}\underline{s}^{\underline{m}}.$$
where $C_{\underline{m}}^{(j)}\in  K[(C_{\underline{m}})_{\underline{m}\in\mathbb{N}^\tau}]$.  
For instance, one has that   $C_{\underline{0}}^{(j)}={C_{\underline{0}}}^j$. For $j=0$, we set ${Y_0'}^0:=1$. More generally, for any $\underline{m}$ and any $j\leq |\underline{m}|$, $C_{\underline{m}}^{(j)}$ is a homogeneous polynomial of degree $j$ in the $C_{\underline{k}}$'s for $\underline{k}\in\mathbb{N}^\tau$, $\underline{k}\leq \underline{m}$, with coefficients in $\mathbb{N}^*$. 

Now suppose we are given a series $y_0'=\displaystyle\sum_{\underline{m}\in\mathbb{N}^\tau} c_{\underline{m}}\underline{s}^{\underline{m}}\in K[[\underline{s}]]\setminus\{0\}$. For any $j\in\mathbb{N}$, we denote the multinomial expansion of ${y_0'}^j$ by:
$${y_0'}^j=\displaystyle\sum_{\underline{m}\in\mathbb{N}^\tau} c_{\underline{m}}^{(j)}\underline{s}^{\underline{m}}.$$
So, $c_{\underline{m}}^{(j)}=C_{\underline{m}}^{(j)}(c_{\underline{0}},\ldots,c_{\underline{m}})$.

\begin{defi}\label{defi:mat_Wilc}
	%Supposons donnés deux entiers $d_x\in\mathbb{N}_{>0}$ et $d_y\in\mathbb{N}_{>0}$. 
	Let $y_0'=\displaystyle\sum_{\underline{m}\in\mathbb{N}^\tau} c_{\underline{m}}\underline{s}^{\underline{m}}\in K[[\underline{s}]]\setminus\{0\}$.
	\begin{enumerate}
		\item 
		Given a pair $(\underline{k},j)\in \mathbb{N}^\tau\times\mathbb{N}$, we call \textbf{Wilczynski vector} $V_{\underline{k},j}$ (associated to $y_0'$) the infinite vector with components $\gamma_{\underline{m}}^{\underline{k},j}$ with $\underline{m}\in\mathbb{N}^\tau$ ordered with $\leq_{\textrm{grlex}}$:\\
		- if $j\geq 1$:
		$$V_{\underline{k},j}:=\left(\gamma_{\underline{m}}^{\underline{k},j}\right)_{\underline{m}\in\mathbb{N}^\tau} \textrm{ with } \gamma_{\underline{m}}^{\underline{k},j}=\left\{\begin{array}{ll}
			=c_{\underline{m}-\underline{k}}^{(j)}& \textrm{ if }\underline{m}\geq\underline{k}\\
			=0 & \textrm{ otherwise }
		\end{array}\right.$$
		- otherwise: 1  in the $\underline{k}$th position and 0 for the other coefficients,
		$$V_{\underline{k},0}:=(0,\ldots,1,0,0,\ldots,0,\ldots).$$
		So $\gamma_{\underline{m}}^{\underline{k},j}$ is the coefficient of $s^{\underline{m}}$ in the expansion of $s^{\underline{k}}{y_0'}^j$.
		\item Let $\mathcal{F}'$ and $\mathcal{G}'$ be two sequences as in Definition \ref{defi:alg-relative}. We associate to $\mathcal{F}'$, $\mathcal{G}'$ and $y_0'$ the \textbf{ (infinite)  Wilczynski matrix } whose columns are the corresponding vectors $V_{\underline{k},j}$:
		$$M_{\mathcal{F}',\mathcal{G}'}:=(V_{\underline{k},j})_{(\underline{k},j)\in\mathcal{F}'\cup \mathcal{G}'}\, ,$$
		$\mathcal{F}'\cup\mathcal{G}'$ being ordered by $ \leq_{\textrm{alex}*}$ as in  Definition \ref{defi:alg-relative}.\\ 
		We also define the \textbf{reduced Wilczynski matrix}, $M_{\mathcal{F}',\mathcal{G}'}^{red}$: it is the matrix obtained from $M_{\mathcal{F}',\mathcal{G}'}$ by removing the columns indexed in $\mathcal{G}'$,  and also removing the corresponding rows (suppress the $\underline{k}$th row for any $(\underline{k},0)\in\mathcal{G}'$). This amounts exactly to remove the rows containing the coefficient 1 for some Wilczynski vector indexed in $\mathcal{G}'$. For $(\underline{i},j)\in\mathcal{F}'$, we also denote by $V_{\underline{i},j}^{red}$ the corresponding vectors obtained from $V_{\underline{i},j}$ by suppressing the $\underline{k}$th row for any $(\underline{k},0)\in\mathcal{G}'$ and we call them \textbf{reduced Wilczynski vectors}.
	\end{enumerate}
\end{defi}
The following result is \cite[Lemma 3.2]{hickel-matu:puiseux-alg-multivar}:

\begin{lemma}[generalized Wilczynski]\label{lemme:wilcz}
	The series $y_0'$ is algebraic relatively to $(\mathcal{F}',\mathcal{G}')$ if and only if all the minors of order $|\mathcal{F}'\cup\mathcal{G}'|$ of the  Wilczynski matrix $M_{\mathcal{F}',\mathcal{G}'}$ vanish, or also if and only if all the minors of order  $|\mathcal{F}'|$ of the reduced Wilczynski matrix $M_{\mathcal{F}',\mathcal{G}'}^{red}$ vanish.
\end{lemma}

Let us give an outline of the reconstruction process of \cite{hickel-matu:puiseux-alg-multivar}. Let $\mathcal{F}'$ and $\mathcal{G}'$ be two sequences as in Definition \ref{defi:alg-relative} and $y_0'=\displaystyle\sum_{\underline{m}\in\mathbb{N}^\tau} c_{\underline{m}}\underline{s}^{\underline{m}}\in K[[\underline{s}]]\setminus\{0\}$ be algebraic relatively to $(\mathcal{F}',\mathcal{G}')$. Our purpose is to describe the $K$-vector space whose non-zero elements are the polynomials  $P(\underline{s},y')=\displaystyle\sum_{(\underline{k},j)\in\mathcal{F}'\cup\mathcal{G}'} a_{\underline{k},j}\underline{s}^{\underline{k}}{y'}^j\in K[\underline{s},y']\setminus\{0\}$ such that $P(\underline{s},y_0')=0$.  The components of the infinite vector computed as $M_{\mathcal{F}',\mathcal{G}'}\cdot (a_{\underline{k},j})_{(\underline{k},j)\in\mathcal{F}'\cup\mathcal{G}'}$ are exactly the coefficients of the expansion of $P(\underline{s},y_0')$ in $K[[\underline{s}]]$. 
Let us now remark that, in the infinite vector $M_{\mathcal{F}',\mathcal{G}'}\cdot (a_{\underline{k},j})_{(\underline{k},j)\in\mathcal{F}'\cup\mathcal{G}'}$, if we remove the components indexed by  $\underline{k}$ for $(\underline{k},0)\in\mathcal{G}'$, then we get exactly the infinite vector $M_{\mathcal{F}',\mathcal{G}'}^{red}\cdot (a_{\underline{k},j})_{(\underline{k},j)\in\mathcal{F}'}$. The vanishing of the latter means precisely that the rank of  $M_{\mathcal{F}',\mathcal{G}'}^{red}$ is less than $|\mathcal{F}|$. 
Conversely, if the columns of  $M_{\mathcal{F}',\mathcal{G}'}^{red}$ are dependent for certain $\mathcal{F}'$ and $\mathcal{G}'$, we denote by $(a_{\underline{k},j})_{(\underline{k},j)\in\mathcal{F}'}$ a corresponding sequence of coefficients of a nontrivial vanishing linear combination of the column vectors. Then it suffices to note that the remaining  coefficients $a_{\underline{k},0}$ for  $(\underline{k},0)\in\mathcal{G}'$ are  uniquely determined  as follows: 
\begin{equation}\label{equ:terme-cst0}
	a_{\underline{k},0}=-\displaystyle\sum_{(\underline{i},j)\in\mathcal{F}',\, \underline{i}\leq \underline{k}} a_{\underline{i},j}c_{\underline{k}-\underline{i}}^{(j)}\,.
\end{equation} 

We consider a maximal  family $\mathcal{F}''\subsetneq\mathcal{F}'$ such that the corresponding reduced Wilczynski vectors are $K$-linearly independent. Proceeding as in Lemma 3.7 in  \cite{hickel-matu:puiseux-alg-multivar}, $\mathcal{F}''$ is such a family if and only if, in the reduced Wilczynski matrix $M_{\mathcal{F}',\mathcal{G}'}^{red}$, there is a nonzero minor $\det(A)$ where $A$ has columns indexed in $\mathcal{F}''$ and lowest row with index $\underline{m}$ such that $|\underline{m}|\leq 2d_{\underline{s}}'d_{y'}'$ and $\mathcal{F}''$ is maximal with this property. Moreover, among such $A$'s,  we take one that has its lowest row having an index minimal for $\leq_{\mathrm{grlex}}$, and we denote the latter index by $\hat{\underline{ p}}$.

For any $(\underline{k}_0,j_0)\in\mathcal{F}'\setminus\mathcal{F}''$, the   family of reduced Wilczynski vectors $(V_{\underline{k},j}^{red})$ with $(\underline{k},j)\in \mathcal{F}''\cup\{(\underline{k}_0,j_0)\}$ is $K$-linearly dependent. There is a unique relation:
\begin{equation}\label{equ:wilcz}
	V_{\underline{k}_0,j_0}^{red} =\sum_{(\underline{k},j)\in \mathcal{F}''}\lambda_{\underline{k},j}^{\underline{k}_0,j_0} V_{\underline{k},j}^{red}\ \  \textrm{ with }\ \  \lambda_{\underline{k},j}^{\underline{k}_0,j_0}\in K.
\end{equation}

We consider the restriction of  $M_{\mathcal{F}',\mathcal{G}'}^{red}$ to the rows of $A$. For these rows, by Cramer's rule, we reconstruct the linear combination (\ref{equ:wilcz}).
%between any column of index  $(\underline{k},j)\in\mathcal{F}'\setminus\mathcal{F}''$ and the columns in $\mathcal{F}''$.
The coefficients $\lambda_{\underline{k},j}^{\underline{k}_0,j_0}$ of such a linear combination are quotients of homogeneous polynomials with integer coefficients  in terms of the entries of these restricted matrix, hence quotients of polynomials in the corresponding $c_{\underline{m}}$'s, $|\underline{m}|\leq 2d_{\underline{s}}'d_{y'}'$. \\

Let $P(\underline{s},y')=\displaystyle\sum_{(\underline{k},j)\in\mathcal{F}'\cup\mathcal{G}'} a_{\underline{k},j}\underline{s}^{\underline{k}}{y'}^j\in K[\underline{s},y']\setminus\{0\}$. One has $P(\underline{s},y_0')=0$ if and only if (\ref {equ:terme-cst0}) holds as well as:
\begin{center}
	$\displaystyle	\sum_{(\underline{k},j)\in \mathcal{F}''}a_{\underline{k},j} V_{\underline{k},j}^{red}+\sum_{(\underline{k}_0,j_0)\in \mathcal{F}'\setminus\mathcal{F}''}a_{\underline{k}_0,j_0} V_{\underline{k}_0,j_0}^{red}=\underline{0} 
	$\\
	$\Leftrightarrow  	\displaystyle\sum_{(\underline{k},j)\in \mathcal{F}''}a_{\underline{k},j} V_{\underline{k},j}^{red}+\sum_{(\underline{k}_0,j_0)\in \mathcal{F}'\setminus\mathcal{F}''}a_{\underline{k}_0,j_0} \left(\sum_{(\underline{k},j)\in \mathcal{F}''}\lambda_{\underline{k},j}^{\underline{k}_0,j_0} V_{\underline{k},j}^{red}\right)=\underline{0} $\\
	$\Leftrightarrow  	\displaystyle\sum_{(\underline{k},j)\in \mathcal{F}''}\left( a_{\underline{k},j} +\sum_{(\underline{k}_0,j_0)\in \mathcal{F}'\setminus\mathcal{F}''}a_{\underline{k}_0,j_0}  \lambda_{\underline{k},j}^{\underline{k}_0,j_0}  \right)V_{\underline{k},j}^{red}=\underline{0} $\\
	$\Leftrightarrow  \forall 	(\underline{k},j)\in \mathcal{F}'',\ \  a_{\underline{k},j} =-\displaystyle\sum_{(\underline{k}_0,j_0)\in \mathcal{F}'\setminus\mathcal{F}''}a_{\underline{k}_0,j_0}  \lambda_{\underline{k},j}^{\underline{k}_0,j_0},$\\
\end{center}

\begin{lemma}\label{lemma:total-reconstr}
	Let $\mathcal{F}',\mathcal{G}',d_{\underline{s}}',d_{y'}', y_0',\mathcal{F}''$ be as above. Then, the $K$-vector space of polynomials   $P(\underline{s},y')=\displaystyle\sum_{(\underline{k},j)\in\mathcal{F}'\cup\mathcal{G}'} a_{\underline{k},j}\underline{s}^{\underline{k}}{y'}^j\in K[\underline{s},y']$ such that $P(\underline{s},y_0')=0$ is the set of polynomials such that 
	\begin{equation}\label{equ:coeff-param}
		\forall 	(\underline{k},j)\in \mathcal{F}'',\ \  a_{\underline{k},j} =-\displaystyle\sum_{(\underline{k}_0,j_0)\in \mathcal{F}'\setminus\mathcal{F}''}a_{\underline{k}_0,j_0}  \lambda_{\underline{k},j}^{\underline{k}_0,j_0},
	\end{equation} 	
	and 
	\begin{equation}\label{equ:terme-cst1}
		\forall (\underline{k},0)\in\mathcal{G}',\ \ a_{\underline{k},0}=-\displaystyle\sum_{(\underline{i},j)\in\mathcal{F}',\, \underline{i}\leq \underline{k}} a_{\underline{i},j}c_{\underline{k}-\underline{i}}^{(j)}\, ,
	\end{equation} 
	where the $\lambda_{\underline{k},j}^{\underline{k}_0,j_0}$'s are computed as in (\ref{equ:wilcz}) as quotients of polynomials with integer coefficients in the $c_{\underline{m}}$'s for $|\underline{m}|\leq 2d_{\underline{s}}'d_{y'}'$.
\end{lemma}

\begin{remark}
	Note that the set of polynomials   $P(\underline{s},y')\in K[\underline{s},y']$ with support in $\mathcal{F}'\cup\mathcal{G}'$ such that $P(\underline{s},y_0')=0$ is a $K$-vector space of dimension $|\mathcal{F}'|-|\mathcal{F}''|\geq 1$.
\end{remark}

\subsection{Total algebraic reconstruction in the non-homogeneous case.}
Let $\mathcal{F}',\mathcal{G}', d_{y'}',d_{\underline{ s}}'$ be as in Definition \ref{defi:alg-relative}.

\subsubsection{First case.}\label{sect:non-homog-1stcase}
Let  $y_0'=\displaystyle\sum_{\underline{m}\in\mathbb{N}^\tau} c_{\underline{m}}\underline{s}^{\underline{m}}\in K[[\underline{s}]]$ be \emph{algebraic relatively to $(\mathcal{F}',\mathcal{G}')$}.  %	We denote $d_{y'}':=\max\{j,\ (k,j)\in\mathcal{F}'\}$ and  $d_{\underline{ s}}':=\max\{|\underline{k}|,\ (\underline{k},j)\in\mathcal{F}'\cup\mathcal{G}'\}$.
Let  $i,\,d_{\underline{s}},\,d' \in \mathbb{N}$, $d'\geq 3$, $d_{\underline{ s}}'\leq d_{\underline{ s}}$ and $d_{y'}'\leq d'$.  For any $j=0,\ldots,i$, we consider power series   $y_j'=\displaystyle\sum_{\underline{m}\in\N^\tau} c_{j,\underline{m}}\underline{s}^{\underline{m}}\in K\left[\left[\underline{s}\right]\right]$  
which satisfy some equations ${P}_j(\underline{s},y'_0,\ldots,y'_j)=0$ % and  ${P}_{1}(\underline{x},c_0,c_1)=0$
where 	$P_{j}\in K\left[\underline{s},z_0,z_1,\ldots,z_j \right]\setminus\{0 \}$, $P_{j}(\underline{s},y_0',\ldots,y_{j-1}',z_j)\not\equiv 0$, $\deg_{\underline{s}}P_{j}\leq d_{\underline{s} },$ 	
$\deg_{z_k}{P}_{j}\leq d'$ for $k=0,\ldots,j$.    In particular, $c_{\underline{m}}=c_{0,\underline{m}} $ for any $\underline{m}$. 
Let $z'=R(\underline{s},y'_0,\ldots,y'_i)\in K[[\underline{s}]]\setminus\{0\}$, where $R\in K\left[\underline{s},z_0,z_1,\ldots,z_i \right]\setminus\{0 \}$ with $\deg_{\underline{s}}R\leq d_{\underline{s} },$ 	
$\deg_{z_k}R\leq d'$ for $k=0,\ldots,i$.

We want to determine when  there is a polynomial  $P(\underline{s},y')=\displaystyle\sum_{(\underline{k},j)\in\mathcal{F}'\cup\mathcal{G}'} a_{\underline{k},j}\underline{s}^{\underline{k}}{y'}^j\in K[\underline{s},y']\setminus\{0\}$ such that $P(\underline{s},y_0')=z'$ and, subsequently, to reconstruct all such possible $P$'s. 

Let $V$ be the infinite vector with components the coefficients of $z'$, and $V^{red}$ the corresponding reduced vector as in Definition \ref{defi:mat_Wilc}. For $\mathcal{F}''$ as in the previous section, we have $P(\underline{s},y_0')=z'$ if and only if: 
\begin{equation}\label{equ:wilcz-2nd-membre}
	\displaystyle\sum_{(\underline{k},j)\in \mathcal{F}''}\left( a_{\underline{k},j} +\sum_{(\underline{k}_0,j_0)\in \mathcal{F}'\setminus\mathcal{F}''}a_{\underline{k}_0,j_0}  \lambda_{\underline{k},j}^{\underline{k}_0,j_0}  \right)V_{\underline{k},j}^{red}= V^{red}.
\end{equation}

We want to examine when the vectors $(V_{\underline{k},j}^{red})_{(\underline{k},j)\in \mathcal{F}''}$ and $V^{red}$ are linearly dependent. Let  $N^{red}$ be the infinite matrix with columns  $(V_{\underline{k},j}^{red})_{(\underline{k},j)\in \mathcal{F}''}$ and $V^{red}$.
\begin{lemma}\label{lemma:depth-non-homog}
	The vectors  $(V_{\underline{k},j}^{red})_{(\underline{k},j)\in \mathcal{F}''}$ and $V^{red}$ are linearly dependent if and only if all the minors of maximal order of $N^{red}$ up to the row $\underline{p}$ with: 	$$|\underline{p}| \leq 2.3^{(d')^{i-1}+\cdots+(d')^2+d'+1}    
	d_{\underline{s}} (d')^{(d')^{i}+\cdots+(d')^2+d'+1} $$  vanish. 
\end{lemma}
\begin{proof}
	The vectors  $(V_{\underline{k},j}^{red})_{(\underline{k},j)\in \mathcal{F}''}$ and $V^{red}$ are linearly dependent if and only if all the minors of $N^{red}$ of maximal order vanish: see \cite[ Lemma 1]{hickel-matu:puiseux-alg}.
	
	Conversely, we suppose that the vectors are linearly independent. So, there is a  minor of $N$ of maximal order which is nonzero. Let $\underline{p}$ be the smallest multi-index for $\leq_{\mathrm{grlex}}$ such that there is such a nonzero minor of $N^{red}$ of maximal order with lowest row of index $\underline{p}$. Hence, there is a subminor of it based on the columns indexed in $\mathcal{F}''$ which is nonzero, say $\det(B)$. The lowest row of $B$ is at most  $\underline{p}$. So, by minimality of $\hat{\underline{p}}$ (see before  (\ref{equ:wilcz}) in the previous section), $\underline{p}\geq_{\mathrm{grlex}}\hat{\underline{p}}$. If $\underline{p}=\hat{\underline{p}}$, then $|\underline{ p}|\leq  2d_{\underline{s}}'d'$ and we are done. If $\underline{p}>_{\mathrm{grlex}}\hat{\underline{p}}$, let us denote by  $\underline{\tilde{p}}$ the predecessor of $\underline{p}$ for $\leq_{\textrm{grlex}}$. Then $\tilde{\underline{p}}\geq_{\mathrm{grlex}}\hat{\underline{p}}$.
	For any multi-index $\underline{ m}\in\N^r$, denote by $N_{\underline{m}}^{red}, V_{\underline{k},j,\underline{m}}^{red},V_{\underline{m}}^{red}$  the truncations up to the row $\underline{m}$ of $N^{red},V_{\underline{k},j}^{red},V^{red}$ respectively. 	By definition of $\underline{p}$, the rank of the matrix $N^{red}_{\underline{p}}$  is $|\mathcal{F}''|+1$, whereas the rank of $N^{red}_{\underline{\tilde{p}}}$  is  $|\mathcal{F}''|$. There exists a nonzero vector $((a_{\underline{i},j})_{(\underline{i},j)\in\mathcal{F}''},-a)$ of elements of $K$ such that 
	\begin{equation}\label{equ:profond-non-homog}
		N_{\underline{\tilde{p}}}^{red}\,\cdot\, \left( \begin{array}{c}
			(a_{\underline{i},j})_{(\underline{i},j)\in\mathcal{F}''}\\
			-a
		\end{array}   \right)= 0,
	\end{equation}
	where $a$ can be chosen to be 1 since the vectors $\left(V_{\underline{k},j,\underline{\tilde{p}}}^{red}\right)_{(\underline{k},j)\in \mathcal{F}''}$ are independent. The components of the resulting vector $N_{\underline{\tilde{p}}}^{red}\,\cdot\, \left( \begin{array}{c}
		(a_{\underline{i},j})_{(\underline{i},j)\in\mathcal{F}''}\\
		-1
	\end{array}   \right)$  are exactly the coefficients $e_{\underline{k}}$, $(\underline{k},0)\notin \mathcal{G}'$ and $\underline{k}\leq_{\mathrm{grlex}} \underline{\tilde{p}}$, of the expansion of $\displaystyle\sum_{(\underline{i},j)\in\mathcal{F}''}a_{\underline{i},j}\,\underline{s}^{\underline{i}}\,(y_0')^j-z'$.
	By computing the coefficients $a_{\underline{k},0}$ for $(\underline{k},0)\in\mathcal{G}'$ as:
	\begin{equation}\label{equ:terme-cst-bis}
		a_{\underline{k},0}=-\displaystyle\sum_{(\underline{i},j)\in\mathcal{F}'', \underline{k}>\underline{i}} a_{\underline{i},j}c_{\underline{k}-\underline{i}}^{(j)}+f_{\underline{k}},
	\end{equation}
	where $f_{\underline{k}}$ denotes the coefficient of ${\underline{s}}^{\underline{k}}$ in $z'$,
	we obtain the vanishing of the first terms of $Q(\underline{s},y_0',\ldots,y'_i):=\displaystyle\sum_{(\underline{i},j)\in\mathcal{F}''\cup\mathcal{G}'} a_{\underline{i},j} \underline{s}^{\underline{i}}(y_0')^j-z'$ up to  $\tilde{\underline{p}}$. So, $w_{\underline{s}}(Q(\underline{s},y_0',\ldots,y'_i))\geq_{\mathrm{grlex}}{\underline{p}}$ and, therefore,  $\ord(Q(\underline{s},y_0',\ldots,y'_i))\geq |{\underline{p}}|$. 
	
	On the contrary, 
	%	Moreover, for any nonzero vector $((b_{\underline{i},j})_{(\underline{i},j)\in\mathcal{F}''},b)$	 of elements of $K$, 
	we have:
	\begin{equation}\label{equ:annul}
		N_{\underline{p}}^{red}\,\cdot\, \left( \begin{array}{c}
			(a_{\underline{i},j})_{(\underline{i},j)\in\mathcal{F}''}\\
			-1
		\end{array}   \right)\neq 0. 
	\end{equation}
	From (\ref{equ:profond-non-homog}) and (\ref{equ:annul}), we deduce that the coefficient $e_{\underline{p}}$  of $\underline{s}^{\underline{p}}$ in the expansion of\\ $\displaystyle\sum_{(\underline{i},j)\in\mathcal{F}''}a_{\underline{i},j}\,\underline{x}^{\underline{i}}\,(y_0')^j-z'$ is nonzero. 
	Observe that this term of the latter series does not overlap with the terms of $\displaystyle\sum_{(\underline{i},0)\in\mathcal{G}'}a_{\underline{i},0}\,\underline{s}^{\underline{i}}$ since $(\underline{p},0)\notin\mathcal{G}'$. Therefore,  $w_{\underline{s}}(Q(\underline{s},y_0',\ldots,y'_i))={\underline{p}}$. In particular, $Q(\underline{s},y_0',\ldots,y'_i)\neq 0$, so the bound (\ref {equ:estim1}) in Theorem \ref{propo:nested} applies:
	$$|\underline{p}| \leq 2.3^{(d')^{i-1}+\cdots+(d')^2+d'+1}    
	d_{\underline{s}} (d')^{(d')^{i}+\cdots+(d')^2+d'+1} .$$ 
\end{proof}

Let us return to (\ref{equ:wilcz-2nd-membre}). Let  $A$ be the square matrix defined after (\ref{equ:wilcz}).  For any $(\underline{k},j)\in\mathcal{F}''$, we denote by $A_{\underline{k},j}$ the matrix deduced from $A$ by substituting the corresponding part of $V^{red}$ instead of the column indexed by $(\underline{k},j)$. Equality  (\ref{equ:wilcz-2nd-membre}) holds if and only if the vectors    $(V_{\underline{k},j}^{red})_{(\underline{k},j)\in \mathcal{F}''}$ and $V^{red}$ are linearly dependent, and by  Cramer's rule, one has:
\begin{equation}\label{equ:reconstr_non-homog}
	\forall (\underline{k},j)\in \mathcal{F}'',\ \  a_{\underline{k},j} +\sum_{(\underline{k}_0,j_0)\in \mathcal{F}'\setminus\mathcal{F}''}a_{\underline{k}_0,j_0}  \lambda_{\underline{k},j}^{\underline{k}_0,j_0}=\frac{\det (A_{\underline{k},j}) }{\det( A)}.
\end{equation}
Recall that one determines that  $(V_{\underline{k},j}^{red})_{(\underline{k},j)\in \mathcal{F}''}$ and $V^{red}$ are  linearly dependent by examining the dependence of the finite truncation of these vectors according to Lemma \ref{lemma:depth-non-homog}. Finally, the remaining  coefficients $a_{\underline{k},0}$ for  $(\underline{k},0)\in\mathcal{G}'$ are each uniquely determined  as follows: 
\begin{equation}\label{equ:terme-cst}
	a_{\underline{k},0}=-\displaystyle\sum_{(\underline{i},j)\in\mathcal{F}',\, \underline{i}\leq \underline{k}} a_{\underline{i},j}c_{\underline{k}-\underline{i}}^{(j)}+f_{\underline{ k}}\, ,
\end{equation} 
where $f_{\underline{k}}$ denotes the coefficient of ${\underline{s}}^{\underline{k}}$ in $z'$. 

As a conclusion, we obtain the affine space of $P(\underline{s},y')\in K[\underline{s},y']\setminus\{0\}$ such that $P(\underline{s},y_0')=z'$ as a parametric family of its coefficients with free parameters the  $a_{\underline{k}_0,j_0} $'s for $(\underline{k}_0,j_0)\in \mathcal{F}'\setminus\mathcal{F}''$.

\subsubsection{Second case.}\label{sect:secondcase}
Let  $\d_{\underline{ s}}'\in \N$ and $y_0'=\displaystyle\sum_{\underline{m}\in\mathbb{N}^\tau} c_{\underline{m}}\underline{s}^{\underline{m}}\in K\left[\left[\underline{s}\right]\right]$ be algebraic of degrees  $d'_{y'}$ and $\d_{\underline{ s}}'$, \emph{but not algebraic relatively to $(\mathcal{F}',\mathcal{G}')$}. 
Let  $i,\,d_{\underline{s}},\,d' \in \mathbb{N}$, $d'\geq 3$, $d_{\underline{ s}}'\leq d_{\underline{ s}}$ and $d_{y'}'\leq d'$.  For any $j=0,\ldots,i$, we consider power series   $y_j'=\displaystyle\sum_{\underline{m}\in\N^\tau} c_{j,\underline{m}}\underline{s}^{\underline{m}}\in K\left[\left[\underline{s}\right]\right]$  
which satisfy some equations ${P}_j(\underline{s},y'_0,\ldots,y'_j)=0$ % and  ${P}_{1}(\underline{x},c_0,c_1)=0$
where 	$P_{j}\in K\left[\underline{s},z_0,z_1,\ldots,z_j \right]\setminus\{0 \}$, $P_{j}(\underline{s},y_0',\ldots,y_{j-1}',z_j)\not\equiv 0$, $\deg_{\underline{s}}P_{j}\leq d_{\underline{s} },$ 	
$\deg_{z_k}{P}_{j}\leq d'$ for $k=0,\ldots,j$.    In particular, $c_{\underline{m}}=c_{0,\underline{m}} $ for any $\underline{m}$. 
Let $z'=R(\underline{s},y'_0,\ldots,y'_i)\in K[[\underline{s}]]\setminus\{0\}$, where $R\in K\left[\underline{s},z_0,z_1,\ldots,z_j \right]\setminus\{0 \}$ with $\deg_{\underline{s}}R\leq d_{\underline{s} },$ 	
$\deg_{z_k}R\leq d'$ for $k=0,\ldots,j$.

As in the previous section, our purpose is to determine when  there is a polynomial  $P(\underline{s},y')=\displaystyle\sum_{(\underline{k},j)\in\mathcal{F}'\cup\mathcal{G}'} a_{\underline{k},j}\underline{s}^{\underline{k}}{y'}^j\in K[\underline{s},y']\setminus\{0\}$ such that $P(\underline{s},y_0')=z'$.  Note that such a polynomial is necessarily unique, since $y_0'$ is not algebraic relatively to $(\mathcal{F}',\mathcal{G}')$. 

We consider the corresponding reduced Wilczynski matrix $M_{\mathcal{F}',\mathcal{G}'}^{red}$. Proceeding as in  Lemma 3.7 in  \cite{hickel-matu:puiseux-alg-multivar} and using Lemma \ref{lemme:ordreQ-alg}, there is a nonzero minor $\det(B)$ of maximal order where the lowest row of $B$ is indexed by $\underline{m}$ such that $ |\underline{m}|\leq \left(\d_{\underline{s}}'+ d'_{\underline{s}}\right)d'_{y'}$. 

We resume the notations of the previous section. There is a polynomial $P$ such that $P(\underline{s},y_0')=z'$ if and only if the vectors $(V_{\underline{k},j}^{red})_{(\underline{k},j)\in \mathcal{F}'}$ and $V^{red}$ are $K$-linearly dependent, since the vectors $(V_{\underline{k},j}^{red})_{(\underline{k},j)\in \mathcal{F}'}$ are independent. One determines that  $(V_{\underline{k},j}^{red})_{(\underline{k},j)\in \mathcal{F}'}$ and $V^{red}$ are  linearly dependent by examining the dependence of the finite truncation of these vectors according to the following lemma.

\begin{lemma}\label{lemma:depth-non-homog_bis}
	The vectors  $(V_{\underline{k},j}^{red})_{(\underline{k},j)\in \mathcal{F}'}$ and $V^{red}$ are linearly dependent if and only if, in the corresponding matrix denoted by $N^{red}$, all the minors of maximal order up to the row $\underline{p}$ with 	$|\underline{p}| \leq 2.3^{(d')^{i-1}+\cdots+(d')^2+d'+1}    
	d_{\underline{s}} (d')^{(d')^{i}+\cdots+(d')^2+d'+1} $  vanish. 	 
\end{lemma} 
\begin{proof}
	The proof is analogous to that of Lemma \ref{lemma:depth-non-homog}, also using  Theorem \ref{propo:nested}.
\end{proof}

We proceed as in the previous section. For any $(\underline{k},j)\in\mathcal{F}'$, we denote by $B_{\underline{k},j}$ the matrix deduced from $B$ by substituting the corresponding part of $V^{red}$ instead of the column indexed by $(\underline{k},j)$. If the condition of the previous lemma holds, by  Cramer's rule, one has:
\begin{equation}\label{equ:reconstr_non-homog_bis}
	\forall (\underline{k},j)\in \mathcal{F}',\ \  a_{\underline{k},j} =\frac{\det (B_{\underline{k},j}) }{\det( B)}.
\end{equation}

Then it suffices to note that the remaining  coefficients $a_{\underline{k},0}$ for  $(\underline{k},0)\in\mathcal{G}'$ are each uniquely determined  as follows: 
\begin{equation}\label{equ:terme-cst_bis}
	a_{\underline{k},0}=-\displaystyle\sum_{(\underline{i},j)\in\mathcal{F}',\, \underline{i}\leq \underline{k}} a_{\underline{i},j}c_{\underline{k}-\underline{i}}^{(j)}+f_{\underline{ k}}\, ,
\end{equation} 
where $f_{\underline{k}}$ denotes the coefficient of ${\underline{s}}^{\underline{k}}$ in $z'$.

\subsection{Total algebraic reconstruction with several algebraic series.}\label{sect:alg-reconstr-multi}
Let  $i,\,d_{\underline{s}},\,d' \in \mathbb{N}$, $d'\geq 3$. %, $d_{\underline{ s}}'\leq d_{\underline{ s}}$ and $d_{y'}'\leq d'$.
For any $j=0,\ldots,i$, we consider power series   $y_j'=\displaystyle\sum_{\underline{m}\in\N^\tau} c_{j,\underline{m}}\underline{s}^{\underline{m}}\in K\left[\left[\underline{s}\right]\right]$ 
which satisfy some equations ${P}_j(\underline{s},y'_0,\ldots,y'_j)=0$ % and  ${P}_{1}(\underline{x},c_0,c_1)=0$
where 	$P_{j}\in K\left[\underline{s},z_0,z_1,\ldots,z_j \right]\setminus\{0 \}$, $P_{j}(\underline{s},y_0',\ldots,y_{j-1}',z_j)\not\equiv 0$, $\deg_{\underline{s}}P_{j}\leq d_{\underline{s} },$ 	
$\deg_{z_k}{P}_{j}\leq d'$ for $k=0,\ldots,j$. 

Let $\mathcal{K}'$ and $\mathcal{L}'$,  $\mathcal{K}'\neq \emptyset$, be two strictly increasing finite sequences of pairs $(\underline{k},\underline{l})\in\left(\mathbb{N}^\tau\times\mathbb{N}^{i+1}\right)$ ordered anti-lexicographically: 
$$(\underline{k}_1,\underline{l}_1) \leq_{\textrm{alex}*} (\underline{k}_2,\underline{l}_2)\Leftrightarrow  \underline{l}_1 <_{\textrm{grlex}} \underline{l}_2\textrm{ or } (\underline{l}_1 = \underline{l}_2\ \textrm{and}\ \underline{k}_1 \leq_{\textrm{grlex}} \underline{k}_2).$$
We suppose additionally that  $\mathcal{K}'\geq_{\textrm{alex}*} \left(\underline{0},(0,\ldots,0,1)\right)>_{\textrm{alex}*}\mathcal{L}'$ (thus the elements of $\mathcal{L}'$ are  ordered tuples of the form $(\underline{k},\underline{0})$, and those of  $\mathcal{K}'$ are of the form  $(\underline{k},\underline{l}),\ |\underline{l}|\geq 1$). \\
We set $d_{y'_j}':=\max\{l_j,\ (\underline{k},\underline{l})\in\mathcal{K}'\}$ for $j=0,\ldots,i$, and  $d_{\underline{ s}}':=\max\{|\underline{k}|,\ (\underline{k},\underline{l})\in\mathcal{K}'\cup\mathcal{L}'\}$. We assume that 	$d_{y'_j}'\leq d'$ for  $j=0,\ldots,i$, and $d_{\underline{ s}}'\leq d_{\underline{s}}$.

Let us set $\underline{z}=(z_0,\ldots,z_i)$ and $\underline{y}'=(y_0',\ldots,y_i')$. We assume that $\underline{y}'\neq \underline{0}$.   We want to determine when  there is a polynomial  $P(\underline{s},\underline{z})=\displaystyle\sum_{(\underline{k},\underline{l})\in\mathcal{K}'\cup\mathcal{L}'} a_{\underline{k},\underline{l}}\underline{s}^{\underline{k}}{\underline{z}}^{\underline{l}}\in K[\underline{s},\underline{z}]\setminus\{0\}$ such that $P(\underline{s},\underline{y}')=0$ and, subsequently, to reconstruct all such possible $P$'s. It is a generalization of Section \ref{sect:reconstr-alg}.\\

For any $j=0,\ldots,i$, for any $l_j\in\mathbb{N}$, we denote the multinomial expansion of ${y_j'}^{l_j}$ by:
$${y_j'}^{l_j}=\displaystyle\sum_{\underline{n}_j\in\mathbb{N}^\tau} c_{j,\underline{n}_j}^{(l_j)}\underline{s}^{\underline{n}_j}.$$
So the coefficient of $\underline{s}^{\underline{m}}$ in $\underline{y}'^{\underline{l}}={y_0'}^{l_0}\cdots {y_i'}^{l_i}$ is equal to: $$c_{\underline{m}}^{(\underline{l})}:=\displaystyle\sum_{\underline{n}_0\in \N^\tau,\ldots,\underline{n}_i\in\N^\tau,\ \underline{n}_0+\cdots+\underline{n}_i=\underline{m}}   c_{0,\underline{n}_0}^{(l_0)}\cdots c_{i,\underline{n}_i}^{(l_i)}.$$

\begin{defi}\label{defi:mat_Wilc-multi}
	%Supposons donnés deux entiers $d_x\in\mathbb{N}_{>0}$ et $d_y\in\mathbb{N}_{>0}$. 
	\begin{enumerate}
		\item 
		Given an ordered pair $(\underline{k},\underline{l})\in\mathbb{N}^\tau\times\mathbb{N}^{i+1}$, we call \textbf{Wilczynski vector} $V_{\underline{k},\underline{l}}$ the infinite vector with components $\gamma_{\underline{m}}^{\underline{k},\underline{l}}$ with $\underline{m}\in\mathbb{N}^\tau$ ordered with $\leq_{\textrm{grlex}}$:\\
		- if $\underline{l}\geq_{\mathrm{grlex}} (0,\ldots,0,1)$:
		$$V_{\underline{k},\underline{l}}:= \left(\gamma_{\underline{m}}^{\underline{k},\underline{l}}\right)_{\underline{m}\in\mathbb{N}^\tau} \textrm{ with } \gamma_{\underline{m}}^{\underline{k},\underline{l}}=\left\{\begin{array}{ll}
			=c_{\underline{m}-\underline{k}}^{(\underline{l})}& \textrm{ if }\underline{m}\geq\underline{k}\\
			=0 & \textrm{ otherwise }
		\end{array}\right.$$
		- otherwise: 1  in the $\underline{k}$th position and 0 for the other coefficients,
		$$V_{\underline{k},\underline{0}}:=(0,\ldots,1,0,0,\ldots,0,\ldots).$$
		So $\gamma_{\underline{m}}^{\underline{k},\underline{l}}$ is the coefficient of $s^{\underline{m}}$ in the expansion of $s^{\underline{k}}{\underline{y}'}^{\underline{l}}$.
		\item Let $\mathcal{K}'$ and $\mathcal{L}'$ be two sequences as above. We associate to $\mathcal{K}'$ and $\mathcal{L}'$ the \textbf{ (infinite)  Wilczynski matrix } whose columns are the corresponding vectors $V_{\underline{k},\underline{l}}$:
		$$M_{\mathcal{K}',\mathcal{L}'}:=(V_{\underline{k},\underline{l}})_{(\underline{k},\underline{l})\in\mathcal{K}'\cup \mathcal{L}'}\, ,$$
		$\mathcal{K}'\cup\mathcal{L}'$ being ordered by $ \leq_{\textrm{alex}*}$ as above.\\ 
		We also define the \textbf{reduced Wilczynski matrix}, $M_{\mathcal{K}',\mathcal{L}'}^{red}$: it is the matrix obtained from $M_{\mathcal{K}',\mathcal{L}'}$ by removing the columns indexed in $\mathcal{L}'$,  and also removing the corresponding rows (suppress the $\underline{k}$th row for any $(\underline{k},\underline{0})\in\mathcal{L}'$). This amounts exactly to remove the rows containing the coefficient 1 for some Wilczynski vector indexed in $\mathcal{L}'$. For $(\underline{i},\underline{l})\in\mathcal{K}'$, we also denote by $V_{\underline{i},\underline{l}}^{red}$ the corresponding vectors obtained from $V_{\underline{i},\underline{l}}$ by suppressing the $\underline{k}$th row for any $(\underline{k},\underline{0})\in\mathcal{L}'$ and we call them \textbf{reduced Wilczynski vectors}.
	\end{enumerate}
\end{defi}

\begin{lemma}[generalized Wilczynski]\label{lemme:wilcz-multi}
	There exists a nonzero polynomial with support included in $\mathcal{K}'\cup \mathcal{L}'$ which vanishes at $\underline{y}'$ if and only if all the minors of order $|\mathcal{K}'\cup\mathcal{L}'|$ of the  Wilczynski matrix $M_{\mathcal{K}',\mathcal{L}'}$ vanish, or also if and only if all the minors of order  $|\mathcal{K}'|$ of the reduced Wilczynski matrix $M_{\mathcal{K}',\mathcal{L}'}^{red}$ vanish.
\end{lemma}
\begin{demo}
	By construction of the Wilczynski matrix $M_{\mathcal{K}',\mathcal{L}'}$, the existence of such a polynomial is equivalent to the fact that the corresponding Wilczynski vectors are $K$-linearly dependent. This is in turn equivalent to the vanishing of all the minors of maximal order of $M_{\mathcal{K}',\mathcal{L}'}$. 
	
	Suppose that we are given a nonzero vector $(a_{\underline{k},\underline{l}})_{(\underline{k},\underline{l})\in\mathcal{K}'\cup\mathcal{L}'}$ such that 
	$$ M_{\mathcal{K}',\mathcal{L}'}\cdot(a_{\underline{k},\underline{l}})_{(\underline{k},\underline{l})\in\mathcal{K}'\cup\mathcal{L}'}=0.$$
	Observe that, necessarily, the vector $(a_{\underline{k},\underline{l}})_{(\underline{k},\underline{l})\in\mathcal{K}'}$ is also nonzero (since the vectors $V_{\underline{k},\underline{0}}$ for $(\underline{k},\underline{0})\in \mathcal{L}'$ are independent). Let us remark that:
	$$ M_{\mathcal{K}',\mathcal{L}'}^{red}\cdot(a_{\underline{k},\underline{l}})_{(\underline{k},\underline{l})\in\mathcal{K}'}=0$$
	since the latter vector is deduced from the former one by deleting the rows corresponding to  $(\underline{k},\underline{0})\in \mathcal{L}'$. So, the columns of $ M_{\mathcal{K}',\mathcal{L}'}^{red}$ are linked, which is equivalent to the vanishing of its minors of maximal order. Conversely, suppose that there exists a nonzero $(a_{\underline{k},\underline{l}})_{(\underline{k},\underline{l})\in\mathcal{K}'}$ such that 	$$ M_{\mathcal{K}',\mathcal{L}'}^{red}\cdot(a_{\underline{k},\underline{l}})_{(\underline{k},\underline{l})\in\mathcal{K}'}=0.$$
	Then, we can complete the list of coefficients $(a_{\underline{k},\underline{l}})_{(\underline{k},\underline{l})\in\mathcal{K}'\cup\mathcal{L}'}$ by setting:
	\begin{equation}\label{equ:2nd-membre-multi}
		a_{\underline{k},\underline{0}}=- \sum_{(\underline{i},\underline{l})\in\mathcal{K}',\, \underline{i}\leq \underline{k}} a_{\underline{i},\underline{l}}\,c_{\underline{k}-\underline{i}}^{(\underline{l})}.
	\end{equation}
	
\end{demo}

\begin{lemma}\label{lemme:wilcz-multi-profondeur}
	There exists a nonzero polynomial with support included in $\mathcal{K}'\cup \mathcal{L}'$ which vanishes at $\underline{y}'$ if and only if  all the minors of the reduced Wilczynski matrix $M_{\mathcal{K}',\mathcal{L}'}^{red}$ of order  $|\mathcal{K}'|$ and with lowest row indexed by $\underline{m}$ with:
	$$		|\underline{m}|\leq 2.3^{(d')^{i-1}+\cdots+(d')^2+d'+1}    
	d_{\underline{s}}  (d')^{(d')^{i}+\cdots+(d')^2+d'+1},$$
	vanish.
\end{lemma}
\begin{demo}
	The direct part follows from the previous lemma. Suppose that there is no nonzero polynomial with support included in $\mathcal{K}'\cup \mathcal{L}'$ which vanishes at $\underline{y}'$. So there is a nonzero minor of the reduced Wilczynski matrix $M_{\mathcal{K}',\mathcal{L}'}^{red}$ of order  $|\mathcal{K}'|$ and with lowest row indexed by $\underline{m}$ that we assume to be minimal for $\leq_{\textrm{grlex}}$. Reasoning as in the proof of Lemma \ref{lemma:depth-non-homog}, we obtain a nonzero polynomial $Q(\underline{s},z_0,\ldots,z_i)$ with $\textrm{Supp}(Q)\subseteq \mathcal{K}'\cup \mathcal{L}'$, such that $Q(\underline{s},\underline{y}')\neq 0$, and with $\ord_{\underline{s}}\left(Q(\underline{s},\underline{y}')\right)\geq |\underline{m}|$. Since 	$d_{y'_j}'\leq d'$ for  $j=0,\ldots,i$, and $d_{\underline{ s}}'\leq d_{\underline{s}}$, by Theorem \ref{propo:nested}, we obtain that:
	$$\ord_{\underline{s}}\left(Q(\underline{s},\underline{y}')\right)\leq  2.3^{(d')^{i-1}+\cdots+(d')^2+d'+1}    
	d_{\underline{s}}  (d')^{(d')^{i}+\cdots+(d')^2+d'+1},$$
	which gives the expected result.
\end{demo}

Let us suppose that there is a nonzero  polynomial $P$  with support included in $\mathcal{K}'\cup \mathcal{L}'$ which vanishes at $\underline{y}'$. Our purpose is to determine the space of all such polynomials.
For this, we consider a maximal  family $\mathcal{K}''\subsetneq\mathcal{K}'$ such that the corresponding reduced Wilczynski vectors are $K$-linearly independent. This is equivalent to the fact that, for the matrix consisting of the  $(V_{\underline{k},\underline{l}}^{red})$ with $(\underline{k},\underline{l})\in \mathcal{K}''$,  there is a nonzero minor $\det(A)$  of maximal order and with lowest row indexed by $ \underline{m} $ with 
$$		|\underline{m}|\leq 2.3^{(d')^{i-1}+\cdots+(d')^2+d'+1}    
d_{\underline{s}}  (d')^{(d')^{i}+\cdots+(d')^2+d'+1}.$$
For any $(\underline{k}_0,\underline{l}_0)\in\mathcal{K}'\setminus\mathcal{K}''$, the corresponding  family of reduced Wilczynski vectors $(V_{\underline{k},\underline{l}}^{red})$ with $(\underline{k},\underline{l})\in \mathcal{F}''\cup\{(\underline{k}_0,\underline{l}_0)\}$ is $K$-linearly dependent. There is a unique relation:
\begin{equation}\label{equ:wilcz-multi}
	V_{\underline{k}_0,\underline{l}_0}^{red} =\sum_{(\underline{k},\underline{l})\in \mathcal{K}''}\lambda_{\underline{k},\underline{l}}^{\underline{k}_0,\underline{l}_0} V_{\underline{k},\underline{l}}^{red}\ \  \textrm{ with }\ \  \lambda_{\underline{k},\underline{l}}^{\underline{k}_0,\underline{l}_0}\in K.
\end{equation}
which can be computed by Cramer's rule based on $\det(A)$.
%between any column of index  $(\underline{k},j)\in\mathcal{F}'\setminus\mathcal{F}''$ and the columns in $\mathcal{F}''$.
The coefficients $\lambda_{\underline{k},\underline{l}}^{\underline{k}_0,\underline{l}_0}$ of such a linear combination are quotients of homogeneous polynomials with integer coefficients  in terms of the entries of these restricted matrices, hence quotients of polynomials in the corresponding $c_{\underline{m}}$'s, $|\underline{m}|\leq 2.3^{(d')^{i-1}+\cdots+(d')^2+d'+1}    
d_{\underline{s}}  (d')^{(d')^{i}+\cdots+(d')^2+d'+1}$. \\

Let $\underline{z}=(z_0,\ldots,z_i)$, and  $P(\underline{s},\underline{z})=\displaystyle\sum_{(\underline{k},\underline{l})\in\mathcal{K}'\cup\mathcal{L}'} a_{\underline{k},\underline{l}}\underline{s}^{\underline{k}}{\underline{z}}^{\underline{l}}\in K[\underline{s},\underline{z}]\setminus\{0\}$. One has $P(\underline{s},\underline{y}')=0$ if and only if (\ref {equ:2nd-membre-multi}) holds as well as:
\begin{center}
	$\displaystyle	\sum_{(\underline{k},\underline{l})\in \mathcal{K}''}a_{\underline{k},\underline{l}} V_{\underline{k},\underline{l}}^{red}+\sum_{(\underline{k}_0,\underline{l}_0)\in \mathcal{K}'\setminus\mathcal{K}''}a_{\underline{k}_0,\underline{l}_0} V_{\underline{k}_0,\underline{l}_0}^{red}=\underline{0} 
	$\\
	$\Leftrightarrow  	\displaystyle\sum_{(\underline{k},\underline{l})\in \mathcal{K}''}a_{\underline{k},\underline{l}} V_{\underline{k},\underline{l}}^{red}+\sum_{(\underline{k}_0,\underline{l}_0)\in \mathcal{K}'\setminus\mathcal{K}''}a_{\underline{k}_0,\underline{l}_0} \left(\sum_{(\underline{k},\underline{l})\in \mathcal{K}''}\lambda_{\underline{k},\underline{l}}^{\underline{k}_0,\underline{l}_0} V_{\underline{k},\underline{l}}^{red}\right)=\underline{0} $\\
	$\Leftrightarrow  	\displaystyle\sum_{(\underline{k},\underline{l})\in \mathcal{K}''}\left( a_{\underline{k},\underline{l}} +\sum_{(\underline{k}_0,\underline{l}_0)\in \mathcal{K}'\setminus\mathcal{K}''}a_{\underline{k}_0,\underline{l}_0}  \lambda_{\underline{k},\underline{l}}^{\underline{k}_0,\underline{l}_0}  \right)V_{\underline{k},\underline{l}}^{red}=\underline{0} $\\
	$\Leftrightarrow  \forall 	(\underline{k},\underline{l})\in \mathcal{K}'',\ \  a_{\underline{k},\underline{l}} =-\displaystyle\sum_{(\underline{k}_0,\underline{l}_0)\in \mathcal{K}'\setminus\mathcal{K}''}a_{\underline{k}_0,\underline{l}_0}  \lambda_{\underline{k},\underline{l}}^{\underline{k}_0,\underline{l}_0}.$\\
\end{center}

\begin{lemma}\label{lemma:total-reconstr-multi}
	Let $\mathcal{K}',\mathcal{L}',d_{\underline{s}},d', \underline{y}',\mathcal{K}''$ be as above. Then, the set of polynomials   $P(\underline{s},\underline{z})=\displaystyle\sum_{(\underline{k},\underline{l})\in\mathcal{K}'\cup\mathcal{L}'} a_{\underline{k},\underline{l}}\underline{s}^{\underline{k}}{\underline{z}}^{\underline{l}}\in K[\underline{s},\underline{z}]$ such that $P(\underline{s},\underline{y}')=0$ is the set of polynomials such that 
	\begin{equation}\label{equ:coeff-param_1}
		\forall 	(\underline{k},\underline{l})\in \mathcal{K}'',\ \  a_{\underline{k},\underline{l}} =-\displaystyle\sum_{(\underline{k}_0,\underline{l}_0)\in \mathcal{K}'\setminus\mathcal{K}''}a_{\underline{k}_0,\underline{l}_0}  \lambda_{\underline{k},\underline{l}}^{\underline{k}_0,\underline{l}_0},
	\end{equation} 	
	and 
	\begin{equation}\label{equ:terme-cst1_1}
		\forall (\underline{k},\underline{0})\in\mathcal{L}',\ \ a_{\underline{k},\underline{0}}=-\displaystyle\sum_{(\underline{i},\underline{l})\in\mathcal{K}',\, \underline{i}\leq \underline{k}} a_{\underline{i},\underline{l}}c_{\underline{k}-\underline{i}}^{(j)}\, ,
	\end{equation} 
	where the $\lambda_{\underline{k},\underline{l}}^{\underline{k}_0,\underline{l}_0}$'s are computed as in (\ref{equ:wilcz-multi}) as quotients of polynomials with integer coefficients in the $c_{\underline{m}}$'s for $|\underline{m}|\leq  2.3^{(d')^{i-1}+\cdots+(d')^2+d'+1}    
	d_{\underline{s}}  (d')^{(d')^{i}+\cdots+(d')^2+d'+1}$.
\end{lemma}

\begin{remark}
	Note that the set of polynomials   $P(\underline{s},\underline{z})\in K[\underline{s},\underline{z}]$ with support in $\mathcal{K}'\cup\mathcal{L}'$ such that $P(\underline{s},\underline{y}')=0$ is a $K$-vector space of dimension $|\mathcal{K}'|-|\mathcal{K}''|\geq 1$.
\end{remark}

\section{Reconstruction of an equation for an algebroid series.}\label{sect:reconstr-algebroid}

\subsection{The reconstruction algorithm}\label{sect:reconstr-algebroid-algo}

%\textcolor{red}{ATTENTION les $q_i$ ici sont pour $y_0$ deja avec $c_0\neq 0$ : ce ne sont pas les memes que dans l'enonce du theoreme principal ??}

We resume the notations of Section \ref{section:preliminaries}, in particular Lemma \ref{lemma:support-equ} and after. In particular, recall that $\tau$ is the number of variables in $\underline{s}$, and so $r-\tau$ is the number of variables in $\underline{t}$.

\begin{defi}\label{defi:algebroid-relative}
	Let $\mathcal{F}$ and $\mathcal{G}$ be two strictly increasing  sequences of triples $(\underline{k},\underline{l},j)\in\mathbb{N}^\tau\times \mathbb{N}^{r-\tau}\times\mathbb{N}$ ordered as follows: 
	
	$$(\underline{k}_1,\underline{l}_1,j_1) \leq_{*\textrm{alex}*} (\underline{k}_2,\underline{l}_2,j_2):\Leftrightarrow  j_1 < j_2\textrm{ or } (j_1 = j_2\ \textrm{and}\ (\underline{k}_1,\underline{l}_1) \leq_{\textrm{alex}{*}} (\underline{k}_2,\underline{l}_2))$$
	with
	$$(\underline{k}_1,\underline{l}_1) \leq_{\textrm{alex}{*}} (\underline{k}_2,\underline{l}_2):\Leftrightarrow  \underline{l}_1 <_{\mathrm{grlex}} \underline{l}_2\textrm{ or } (\underline{l}_1 = \underline{l}_2\ \textrm{and}\ \underline{k}_1 \leq_{\textrm{grlex}} \underline{k}_2).$$
	We suppose additionally that $(\underline{k}_1,\underline{l}_1,j_1)\geq_{*\textrm{alex}*}\left(\underline{0},\underline{0},1\right)>_{*\textrm{alex}*}(\underline{k}_2,\underline{l}_2,j_2)$ for any  $(\underline{k}_1,\underline{l}_1,j_1)\in\mathcal{F} $ and $(\underline{k}_2,\underline{l}_2,j_2)\in\mathcal{G}$ (thus the elements of $\mathcal{G}$ are  ordered triples of the form $(\underline{k}_2,\underline{l}_2,0)$, and those of  $\mathcal{F}$ are of the form  $(\underline{k}_1,\underline{l}_1,j_1),\ j_1\geq 1$). Moreover, we assume that there is $d\in\N$, $d\geq 1$, such that $j\leq d$ for any $(\underline{k},\underline{l},j)\in \mathcal{F}\cup\mathcal{G}$, and we set $d:= \max\{j\ |\, \exists (\underline{k},\underline{l},j)\in \mathcal{F}\cup\mathcal{G}\}$.  We say that a series $y_0=\displaystyle\sum_{(\underline{m},\underline{n})\in\mathbb{N}^\tau\times\mathbb{N}^{r-\tau}} c_{\underline{m},\underline{n}}\underline{s}^{\underline{m}}\underline{t}^{\underline{n}}\in K[[\underline{s},{\underline{t}}]]$, $c_{\underline{0},\underline{0}}\neq 0$, is \textbf{algebroid relatively to $(\mathcal{F},\mathcal{G})$} if there exists a polynomial $P(\underline{s},\underline{t},y)=\displaystyle\sum_{(\underline{k},\underline{l},j)\in\mathcal{F}\cup\mathcal{G}} a_{\underline{k},\underline{l},j}\underline{s}^{\underline{k}}\underline{t}^{\underline{l}}y^j\in K[[\underline{s},\,\underline{t}]][y]\setminus\{0\}$ such that $P(\underline{s},\underline{t},y_0)=0$.
\end{defi}

For any $\mathcal{F},\mathcal{G}$ satisfying Conditions (i), (ii), (iii) of Lemma \ref{lemma:support-equ}, let us denote by \\ $\left(K[\underline{s}][[\underline{t}]][y]\right)_{\mathcal{F},\mathcal{G}}$ the subset of polynomials in $ K[\underline{s}][[\underline{t}]][y]\setminus\{0\}$ with support in  $\mathcal{F}\cup\mathcal{G}$.

The purpose of the following discussion is to make more explicit the conditions in Lemma \ref{lemma:algebraicity} for the vanishing of a polynomial $P\in \left(K[\underline{s}][[\underline{t}]][y]\right)_{\mathcal{F},\mathcal{G}}$ for some $\mathcal{F},\mathcal{G}$ corresponding to (i), (ii), (iii) in Lemma \ref{lemma:support-equ}, at a formal power series $y_0\in K[[\underline{s}]][[\underline{t}]]$.  As we have seen in Section \ref{section:preliminaries}, one can always assume that $y_0=\displaystyle\sum_{\underline{m}\in\mathbb{N}^\tau,\,\underline{n}\in \mathbb{N}^{r-\tau}} c_{\underline{m} ,\underline{n}}\underline{s}^{\underline{m}}\underline{t}^{\underline{n}} =\displaystyle\sum_{\underline{n}\in \mathbb{N}^{r-\tau}} c_{\underline{n}}(\underline{s})\, \underline{t}^{\underline{n}}$ is such that $c_{\underline{0},\underline{0}}\neq 0$.\\ %(In fact we could even assume that  $\underline{n}\geq \underline{0}$ but we will not use this restriction).\\

Let us consider a series $$Y_0=\displaystyle\sum_{\underline{n}\in \mathbb{N}^{r-\tau}}\, \left(\displaystyle\sum_{\underline{m}\in \mathbb{N}^\tau} C_{\underline{m},\underline{n}}\underline{s}^{\underline{m}}\right)\underline{t}^{\underline{n}} = \displaystyle\sum_{\underline{n}\in \mathbb{N}^{r-\tau}}C_{\underline{n}}(\underline{s})\,\underline{t}^{\underline{n}} \in K[(C_{\underline{m} , \underline{n}})_{\underline{m}\in\mathbb{N}^\tau,\,\underline{n}\in \mathbb{N}^{r-\tau}}][[\underline{s}]][[\underline{t}]]$$ where $\underline{s}$, $\underline{t}$ and the $C_{\underline{m},\underline{n}}$'s are variables. We denote the multinomial expansion of the $j$th power ${Y_0}^j$  of $Y_0$ by:
$${Y_0}^j=\displaystyle\sum_{\underline{n}\in \mathbb{N}^{r-\tau}}\, \left(\displaystyle\sum_{\underline{m}\in \mathbb{N}^\tau} C_{\underline{m},\underline{n}}^{(j)}\underline{s}^{\underline{m}}\right)\underline{t}^{\underline{n}} = \displaystyle\sum_{\underline{n}\in \mathbb{N}^{r-\tau}}C_{\underline{n}}^{(j)}(\underline{s})\,\underline{t}^{\underline{n}}$$
where $C_{\underline{m},\underline{n}}^{(j)}\in  K\left[(C_{\underline{k},\underline{l}})_{\underline{k}\leq \underline{m},\,\underline{l}\leq \underline{n}}\right]$ and $$C_{\underline{n}}^{(j)}(\underline{s})\in  K\left[\left(C_{\underline{l}}(\underline{s})\right)_{\underline{l}\leq \underline{n}}\right]\subseteq K\left[(C_{\underline{k},\underline{l}})_{\underline{k}\leq \underline{m},\,\underline{l}\leq_{\mathrm{grlex}}\underline{n}}\right][[\underline{s}]]. $$  
We also set ${Y_0}^0:=1$.

Now, suppose we are given a series $y_0=\displaystyle\sum_{\underline{m}\in\mathbb{N}^\tau,\,\underline{n}\in \mathbb{N}^{r-\tau}} c_{\underline{m}, \underline{n}}\underline{s}^{\underline{m}}\underline{t}^{\underline{n}}\in K[[\underline{s},\underline{t}]]$ with  $c_{\underline{0},\underline{0}}\neq 0$. For any $j\in\mathbb{N}$, we denote the multinomial expansion of ${y_0}^j$ by:
\begin{equation}\label{equ:multinomial-expansion}
	{y_0}^j=\displaystyle\sum_{\underline{m}\in\mathbb{N}^\tau,\,\underline{n}\in \mathbb{N}^{r-\tau}} c_{\underline{m},\underline{n}}^{(j)}\underline{s}^{\underline{m}}\underline{t}^{\underline{n}}= \displaystyle\sum_{\underline{n}\in \mathbb{N}^{r-\tau}}c_{\underline{n}}^{(j)}(\underline{s})\,\underline{t}^{\underline{n}}. 
\end{equation}
So, $c_{\underline{m},\underline{n}}^{(j)}=C_{\underline{m},\underline{n}}^{(j)}\left(c_{\underline{0},\underline{0}},\ldots,c_{\underline{m},\underline{n}}\right)$ and $c_{\underline{n}}^{(j)}(\underline{s})=C_{\underline{n}}^{(j)}\left(c_{\underline{0}}(\underline{s}),\ldots,c_{\underline{n}}(\underline{s})\right)$. We also set ${y_0}^0:=1$.

\begin{lemma}\label{lemma:recurrence}
	For a polynomial $P\in\left(K[\underline{s}][[\underline{t}]][y]\right)_{\mathcal{F},\mathcal{G}}\setminus\{0\}$, we denote $$P(\underline{s},\underline{t},y)=\displaystyle\sum_{(\underline{k},\underline{l},j)\in\mathcal{F}\cup\mathcal{G}} a_{\underline{k},\underline{l},j}\underline{s}^{\underline{k}}\underline{t}^{\underline{l}}y^j =\displaystyle\sum_{\underline{l}\in \mathbb{N}^{r-\tau},\,  j=0,..,d} a_{\underline{l},j}(\underline{s})\underline{t}^{\underline{l}}y^j.$$ 
	A series $y_0\in K[[\underline{s}]][[\underline{t}]]$, $y_0=\displaystyle\sum_{\underline{m}\in\mathbb{N}^\tau,\,\underline{n}\in \mathbb{N}^{r-\tau}} c_{\underline{m} ,\underline{n}}\underline{s}^{\underline{m}}\underline{t}^{\underline{n}} =\displaystyle\sum_{\underline{n}\in \mathbb{N}^{r-\tau}} c_{\underline{n}}(\underline{s})\, \underline{t}^{\underline{n}}$, is a root of $P$ if and only if the following polynomial relations hold when evaluated at the series $c_{\underline{0}}(\underline{s}),\ldots, c_{\underline{n}}(\underline{s})$:
	\begin{equation}\label{equ:recurrence}
		\forall \underline{l}\in \mathbb{N}^{r-\tau},\ \displaystyle\sum_{j=0,..,d}  a_{\underline{l},j}(\underline{s}) {C_{\underline{0}}}^{j}(\underline{s})=- \displaystyle\sum_{\underline{i}<\underline{l},\, j=0,..,d} a_{\underline{i},j}(\underline{s}) C_{\underline{l}-\underline{i}}^{(j)}(\underline{s})\,.
	\end{equation} 
	%\textcolor{red}{in fact, $j=1,..,d$ in the RHS: to be modified or indicate that $y_0^0=1$ translates as vector $(1,0,\ldots)$ ??}
\end{lemma}
\begin{demo}
	Let us compute:
	\begin{center}
		$P(\underline{s},\underline{t},y_0)=\displaystyle\sum_{{\underline{i}}\in \mathbb{N}^{r-\tau},\,  j=0,..,d} a_{{\underline{i}},j}(\underline{s})\underline{t}^{{\underline{i}}}{y_0}^j$\\
		$=\displaystyle\sum_{{\underline{i}}\in \mathbb{N}^{r-\tau},\,  j=0,..,d} a_{{\underline{i}},j}(\underline{s})\underline{t}^{{\underline{i}}}\left(\displaystyle\sum_{\underline{n}\in \mathbb{N}^{r-\tau}}c_{\underline{n}}^{(j)}(\underline{s})\,\underline{t}^{\underline{n}}\right)$\\
		$=\displaystyle\sum_{\underline{l}\in \mathbb{N}^{r-\tau}}\left(\displaystyle\sum_{\underline{i}\leq\underline{l},\,  j=0,..,d } a_{\underline{i},j}(\underline{s})c_{\underline{l}-\underline{i}}^{(j)}(\underline{s})\right)\underline{t}^{\underline{l}}.$
	\end{center}
	So, $y_0$ is a root of $P$ if and only if, in the latter formula, the coefficient of $\underline{t}^{\underline{l}}$ for each $\underline{l}$ vanishes, which is equivalent to the vanishing of (\ref{equ:recurrence}) (noticing that $ C_{\underline{0}}^{(j)}= {C_{\underline{0}}}^{j}$ for all $j$).
\end{demo}

Let $\mathcal{F},\mathcal{G}$ be as in Definition \ref{defi:algebroid-relative} and  satisfying Conditions (i), (ii), (iii) of Lemma \ref{lemma:support-equ}. Let $y_0=\displaystyle\sum_{(\underline{m},\underline{n})\in\mathbb{N}^\tau\times\mathbb{N}^{r-\tau}} c_{\underline{m},\underline{n}}\underline{s}^{\underline{m}}\underline{t}^{\underline{n}}=\displaystyle\sum_{\underline{n}\in \mathbb{N}^{r-\tau}}c_{\underline{n}}(\underline{s})\,\underline{t}^{\underline{n}}\in K[[\underline{s},{\underline{t}}]]$, $c_{\underline{0},\underline{0}}\neq 0$, be a series  {algebroid relatively to $(\mathcal{F},\mathcal{G})$}. Let $P\in\left(K[\underline{s}][[\underline{t}]][y]\right)_{\mathcal{F},\mathcal{G}}\setminus\{0\}$ be a polynomial such that $P(\underline{s},\underline{t},y_0)=0$. 
%Let us denote $\underline{l}_0:=$. 
We notice that  $w_{\underline{t}}(P)$ is the index of the first non-trivial relation (\ref{equ:recurrence}), for $\mathbb{N}^{r-\tau}$ ordered with $\leq_{\mathrm{grlex}}$. Let $\hat{\underline{l}}_0\in \N^{r-\tau}$ be such that $w_{\underline{t}}(P)\leq_{\textrm{grlex}} \hat{\underline{l}}_0$. If  $w_{\underline{t}}(P)$ is known, then one can take  $\hat{\underline{l}}_0=w_{\underline{t}}(P)$.

\subsubsection{First step}\label{sect:first-step}\hspace{0.3cm} \\  

For any $\underline{l}\in\N^{r-\tau}$, we denote by ${\mathcal{F}}_{\underline{l}}'$ and ${\mathcal{G}}_{\underline{l}}'$ the corresponding sets of tuples $(\underline{k},j)\in \N^{\tau}\times \N$ where $(\underline{k},\underline{l},j)\in \mathcal{F}$ and $(\underline{k},\underline{l},0)\in \mathcal{G}$ respectively. We denote  $d'_{\underline{s},\underline{l}}:=\max\{|\underline{k}|\ |\, (\underline{k},j)\in {\mathcal{F}}_{\underline{l}}'\cup{\mathcal{G}}_{\underline{l}}' \}$ (which is well-defined thanks to Condition (iii) of  Lemma \ref{lemma:support-equ}). By (\ref{equ:control-deg}) in Remark \ref{rem:control-deg}, we have that:
$$ d'_{\underline{s},\underline{l}}\leq a|\underline{l}|+b,  $$
where $a$ and $b$ are as in Lemma \ref{lemme:control-deg-supp}.

Let $\underline{l}\leq_{\textrm{grlex}} \hat{\underline{l}}_0$ (or directly $\underline{l}=w_{\underline{t}}(P)$ if known). As we are interested in the first non trivial relation in (\ref{equ:recurrence}), we consider its following instance:
\begin{equation}\label{equ:recurrence-init}
	\displaystyle\sum_{j=0,..,d}  a_{\underline{l},j}(\underline{s}) {C_{\underline{0}}}^{j}=\displaystyle\sum_{(\underline{k},j)\in\mathcal{F}'_{\underline{l}}\cup\mathcal{G}'_{\underline{l}}} a_{\underline{k},\underline{l},j}\underline{s}^{\underline{k}}{C_0}^j=0\,.
\end{equation} 
By Lemma \ref{lemma:recurrence}, there is $\underline{l}\leq_{\textrm{grlex}}  \hat{\underline{l}}_0$ such that $c_{\underline{0}}$ satisfies the latter relation, i.e. $c_{\underline{0}}$ is algebraic relatively to $(\mathcal{F}'_{\underline{l}},\mathcal{G}'_{\underline{l}})$. In particular, $c_{\underline{0}}$ is algebraic relatively to $\left(\displaystyle\bigcup_{\underline{l}\leq_{\textrm{grlex}} \hat{\underline{l}}_0}\mathcal{F}'_{\underline{l}},\displaystyle\bigcup_{\underline{l}\leq_{\textrm{grlex}} \hat{\underline{l}}_0}\mathcal{G}'_{\underline{l}}\right)$. We denote  $d'_{\underline{s}}:=\displaystyle\max_{\underline{l}\leq_{\textrm{grlex}} \hat{\underline{l}}_0}\left(d'_{\underline{s},\underline{l}}\right)$. Let us now describe the reconstruction method for this first step:
\begin{enumerate}
	\item 	We determine the multi-indices  $\underline{l}\leq_{\textrm{grlex}}  \hat{\underline{l}}_0$ such that  $\mathcal{F}'_{\underline{l}}\cup \mathcal{G}'_{\underline{l}}\neq \emptyset$.
	\item For each $\underline{l}\leq_{\textrm{grlex}} \hat{\underline{l}}_0$ as above,
	we determine whether $c_{\underline{0}}$ is algebraic relatively to $(\mathcal{F}'_{\underline{l}},\mathcal{G}'_{\underline{l}})$ by computing the first minors of maximal order of the corresponding Wilczynski matrix $M_{\mathcal{F}'_{\underline{l}},\mathcal{G}'_{\underline{l}}}^{red}$. Proceeding as in  \cite[Lemma 3.7]{hickel-matu:puiseux-alg-multivar} or Lemma \ref{lemma:depth-non-homog}, it suffices to compute them up to the row indexed by the biggest $\underline{ m}\in\N^{\tau}$ such that $|\underline{ m}|\leq 2 d\,d'_{\underline{s}}$. 
	\item Let $\underline{l}\leq_{\textrm{grlex}} \hat{\underline{l}}_0$ such that $c_{\underline{0}}$ is algebraic relatively to $(\mathcal{F}'_{\underline{l}},\mathcal{G}'_{\underline{l}})$. We  reconstruct the $K$-vector space of polynomials  corresponding to Equation (\ref{equ:recurrence-init}) according to the method in Section \ref{sect:reconstr-alg}, in particular Lemma \ref{lemma:total-reconstr}, applied to $(\mathcal{F}'_{\underline{l}},\mathcal{G}'_{\underline{l}})$ and $c_{\underline{0}}$. We denote by ${E}_{\underline{l}}$ this space.
	\item For each ${\underline{l}}'<_{\mathrm{grlex}} {\underline{l}}$, we set $a_{\underline{k}, {\underline{l}}',j}:=0$ for $(\underline{k}, {\underline{l}}',j)\in \mathcal{F}\cup \mathcal{G}$.
\end{enumerate} 	%If  $c_{\underline{0}}$ is not algebraic relatively to $(\mathcal{F}'_{\underline{l}},\mathcal{G}'_{\underline{l}})$, then we remove $(\mathcal{F}'_{\underline{l}},\mathcal{G}'_{\underline{l}})$ from $(\mathcal{F},\mathcal{G})$ and we resume the previous step.

\subsubsection{Second step} \label{sect:second-step}\hspace{1cm} \\  

With the notations of the previous section, let $\underline{l}$ be such that $E_{\underline{l}}\neq \{0\}$. Let us consider  the  instances of (\ref{equ:recurrence}) corresponding to the ${\underline{l}}'$ such that: \begin{equation}\label{equ:ineq}
	{\underline{l}}<_{\mathrm{grlex}} {\underline{l}}'<_{\mathrm{grlex}} {\underline{l}}+(0,\ldots,0,1),
\end{equation}
For such  ${\underline{l}}'$, we claim that the set of indices $\underline{ i}$ such that   $\underline{ i}<{\underline{l}}'$ and ${\underline{i}}\geq_{\mathrm{grlex}} {\underline{l}}$  is empty. Indeed, by (\ref{equ:ineq}), note that $|\underline{ l}'|=|\underline{ l}|$. For such $\underline{ i}$, one necessarily has $|\underline{ i}|<|\underline{ l}'|=|\underline{ l}|$, but also  $|\underline{ i}|\geq |\underline{ l}|$: a contradiction.

According to (4) at the end of First Step above and to the previous claim, the  right hand sides of such  instances are equal to 0. Hence, they also are of the same form as (\ref{equ:recurrence-init}):
\begin{equation}\label{equ:recurrence-init-bis}
	\displaystyle\sum_{j=0,..,d}  a_{\underline{l}',j}(\underline{s}) {C_{\underline{0}}}^{j}=\displaystyle\sum_{(\underline{k},j)\in\mathcal{F}'_{\underline{l}'}\cup\mathcal{G}'_{\underline{l}'}} a_{\underline{k},\underline{l}',j}\underline{s}^{\underline{k}}{C_{\underline{0}}}^j=0\,.
\end{equation} 
We perform the same method of reconstruction as in the First Step \ref{sect:first-step} to determine $E_{\underline{l}' }$ the $K$-vector space of polynomials  corresponding to this equation. Note that $E_{\underline{l}' }$ might be equal to $\{0\}$.

At this step, for each $\underline{l}\leq_{\textrm{grlex}} \hat{\underline{l}}_0$ such that $E_{\underline{l}}\neq\{0\}$ from the First Step, we have built the vector spaces $E_{\underline{l}' }$ (possibly $\{0\}$) of  all the coefficients $ a_{\underline{k},\underline{l}',j}$ for  $(\underline{k}, {\underline{l}}',j)\in \mathcal{F}\cup \mathcal{G}$ satisfying the instances of (\ref{equ:recurrence}) for  
${\underline{l}}'<_{\mathrm{grlex}} {\underline{l}}+(0,\ldots,0,1)$. 

\subsubsection{Third step} \label{sect:second-step1}\hspace{1cm} \\  

Let $\underline{l}\leq_{\textrm{grlex}} \hat{\underline{l}}_0$ such that $E_{\underline{l}}\neq\{0\}$ as in the First Step \ref{sect:first-step}. We consider the  instance of (\ref{equ:recurrence}) corresponding to ${\underline{l}}+(0,\ldots,0,1)$. Note that for  $\underline{i}< \underline{l}+(0,\ldots,0,1) $, we have that $\underline{i}\leq_{\textrm{grlex}} \underline{l}$. Applying (4) from the end of the First Step, we obtain:
\begin{equation}\label{equ:recurrence_step2}
	\displaystyle\sum_{j=0,..,d}  a_{{\underline{l}}+(0,\ldots,0,1),j}(\underline{s}) {C_{\underline{0}}}^{j}=- \displaystyle\sum_{j=0,..,d} a_{\underline{l},j}(\underline{s}) C_{(0,\ldots,0,1)}^{(j)}\,.
\end{equation} 
Noticing that
$C_{(0,\ldots,0,1)}^{(j)}=j\,{C_{\underline{0}}}^{j-1}C_{(0,\ldots,0,1)}$, we get:
\begin{equation}\label{equ:recurrence_step2bis}
	\displaystyle\sum_{(\underline{k},j)\in\mathcal{F}'_{\underline{l}+(0,\ldots,0,1)}\cup\mathcal{G}'_{\underline{l}+(0,\ldots,0,1)}} a_{\underline{k},\underline{l}+(0,\ldots,0,1),j}\underline{s}^{\underline{k}}{C_{\underline{0}}}^j=- 
	\left(\displaystyle\sum_{(\underline{k},j)\in\mathcal{F}'_{\underline{l}}\cup\mathcal{G}'_{\underline{l}}} a_{\underline{k},\underline{l},j}\underline{s}^{\underline{k}}j\,{C_{\underline{0}}}^{j-1}\right) C_{(0,\ldots,0,1)}\,.
\end{equation} 

There is $\underline{l}\leq_{\textrm{grlex}}  \hat{\underline{l}}_0$ such that $c_{\underline{0}}$ and $c_{(0,\ldots,0,1)}$ satisfy the latter relation, and $c_{\underline{0}}$ satisfies the relations (\ref{equ:recurrence-init}) and (\ref{equ:recurrence-init-bis}). 

If $c_{(0,\ldots,0,1)}=0$, then there are two cases. Either $\mathcal{F}'_{\underline{l}+(0,\ldots,0,1)}\cup\mathcal{G}'_{\underline{l}+(0,\ldots,0,1)}=\emptyset$ i.e. there is no coefficient $a_{\underline{k},\underline{l}+(0,\ldots,0,1),j}$ to reconstruct. Or else, we obtain an equation like (\ref{equ:recurrence-init}) and we derive $E_{\underline{l}+(0,\ldots,0,1)}$ as in the first and second step.

If $c_{(0,\ldots,0,1)}\neq 0$, let us denote $\theta_{\underline{s},(0,\ldots,0,1)}:= \left(| \hat{\underline{l}}_0|+d\right)a +b $ where $a$ and $b$ are as in Lemma \ref{lemme:control-deg-supp}.	 
By this lemma,  %Equations (\ref{equ:recurrence-init}), (\ref{equ:recurrence-init-bis}) and (\ref{equ:recurrence_step2bis}) satisfied by  $c_{\underline{0}}$ and $c_{(0,\ldots,0,1)}$ for a certain $\underline{l}$,
there are non-trivial polynomial relations $P_0(\underline{s},z_0)=0$ and $P_1(\underline{s},z_0,z_1)=0$ satisfied by $c_{\underline{0}}$ and $c_{(0,\ldots,0,1)}$ with $\deg_{\underline{s}}P_j\leq \theta_{\underline{s},(0,\ldots,0,1)}$,  $\deg_{z_0}P_j\leq d$ and $\deg_{z_1}P_1\leq d$. There are several cases. 

\noindent$\bullet$ Suppose that $\mathcal{F}'_{\underline{l}+(0,\ldots,0,1)}\cup\mathcal{G}'_{\underline{l}+(0,\ldots,0,1)}=\emptyset$. Equation (\ref{equ:recurrence_step2bis}) reduces to: 
\begin{equation}\label{equ:double-root}
	\displaystyle\sum_{(\underline{k},j)\in\mathcal{F}'_{\underline{l}}\cup\mathcal{G}'_{\underline{l}}} a_{\underline{k},\underline{l},j}\underline{s}^{\underline{k}}j\,{c_{\underline{0}}}^{j-1}=
	\displaystyle\sum_{(\underline{k},j)\in\mathcal{F}'_{\underline{l}}} a_{\underline{k},\underline{l},j}\underline{s}^{\underline{k}}j\,{c_{\underline{0}}}^{j-1}=0,
\end{equation}
which means that $c_{\underline{0}}$ is at least a double root of (\ref{equ:recurrence-init}). We resume the notations of Section \ref{sect:reconstr-alg}. Let us denote by $\mathcal{F}''_{\underline{l}}$ the family corresponding to $\mathcal{F}''$ for (\ref{equ:recurrence-init}), and $\lambda_{\underline{l},\,\underline{k},j}^{\underline{k}_0,j_0}$ the coefficients corresponding to  $\lambda_{\,\underline{k},j}^{\underline{k}_0,j_0}$. Formula (\ref{equ:coeff-param}) of Lemma \ref{lemma:total-reconstr} becomes:
\begin{equation*}
	\forall 	(\underline{k},j)\in \mathcal{F}''_{\underline{l}},\ \  a_{\underline{k},{\underline{l}},j} =-\displaystyle\sum_{(\underline{k}_0,j_0)\in \mathcal{F}'_{\underline{l}}\setminus\mathcal{F}''_{\underline{l}}}a_{\underline{k}_0,{\underline{l}},j_0}\,  \lambda_{\underline{l},\,\underline{k},j}^{\underline{k}_0,j_0}\ .
\end{equation*} 	
Substituting this formula in (\ref{equ:double-root}) gives:
\begin{equation*}
	\displaystyle\sum_{(\underline{k}_0,j_0)\in\mathcal{F}'_{\underline{l}}\setminus\mathcal{F}''_{\underline{l}}}  a_{\underline{k}_0,{\underline{l}},j_0} \underline{s}^{\underline{k}_0}j_0\,{c_{\underline{0}}}^{j_0-1}\ + \  \displaystyle\sum_{(\underline{k},j)\in\mathcal{F}''_{\underline{l}}} 
	\left(  -\displaystyle\sum_{(\underline{k}_0,j_0)\in \mathcal{F}'_{\underline{l}}\setminus\mathcal{F}''_{\underline{l}}}a_{\underline{k}_0,{\underline{l}},j_0}\,  \lambda_{\underline{l},\,\underline{k},j}^{\underline{k}_0,j_0} \right)
	\underline{s}^{\underline{k}}j\,{c_{\underline{0}}}^{j-1} =0 \ ,
\end{equation*}
which is:
\begin{equation}\label{equ:substit}
	\displaystyle\sum_{(\underline{k}_0,j_0)\in\mathcal{F}'_{\underline{l}}\setminus\mathcal{F}''_{\underline{l}}}  a_{\underline{k}_0,{\underline{l}},j_0}\left(  \underline{s}^{\underline{k}_0}j_0\,{c_{\underline{0}}}^{j_0-1}\ - \  \displaystyle\sum_{(\underline{k},j)\in\mathcal{F}''_{\underline{l}}} 
	\,  \lambda_{\underline{l},\,\underline{k},j}^{\underline{k}_0,j_0}
	\underline{s}^{\underline{k}}j\,{c_{\underline{0}}}^{j-1} \right) =0 \ .
\end{equation}

Either, the latter relation is trivial, i.e. for all $(\underline{k}_0,j_0)\in\mathcal{F}'_{\underline{l}}\setminus\mathcal{F}''_{\underline{l}}$, the contents of the parenthesis are all 0. In this case, the space $E_{\underline{l}}$ of possible equations for $c_{\underline{0}}$ remains unchanged. Or, the dimension of $E_{\underline{l}}$ drops. Since the contents of these parenthesis are polynomials in $\underline{ s}$ and $c_{\underline{0}}$, by Lemma \ref{lemme:ordreQ-alg}, the  $\underline{ s}$-adic order of the non-vanishing ones is  at most $2d'_{\underline{s}}d$. The vanishing of (\ref{equ:substit}) follows from the vanishing of the terms of  $\underline{ s}$-adic order up to $2d'_{\underline{s}}d$. This gives  linear relations (with at least one that is nontrivial) between the $a_{\underline{k}_0,{\underline{l}},j_0}$'s for $(\underline{k}_0,j_0)\in\mathcal{F}'_{\underline{l}}\setminus\mathcal{F}''_{\underline{l}}$. Accordingly, we derive a new space of possible equations for $c_{\underline{0}}$, that we still denote by $E_{\underline{l}}$ for simplicity. In the particular case where $E_{\underline{l}}=\{0\}$, we exclude $\underline{l}$ from the list of admissible multi-indices.\\

\noindent$\star$ Suppose now that $\mathcal{F}'_{\underline{l}+(0,\ldots,0,1)}\cup\mathcal{G}'_{\underline{l}+(0,\ldots,0,1)}\neq\emptyset$. We determine whether $c_{\underline{0}}$ is algebraic relatively to $(\mathcal{F}'_{\underline{l}+(0,\ldots,0,1)},\mathcal{G}'_{\underline{l}+(0,\ldots,0,1)})$. For this, we examine the vanishing of the minors of maximal order of $M_{\mathcal{F}'_{\underline{l}+(0,\ldots,0,1)},\mathcal{G}'_{\underline{l}+(0,\ldots,0,1)}}^{red}$ up to the lowest row of order $2d'_{\underline{s},\underline{l}+(0,\ldots,0,1)}d$. There are two subcases.

\noindent $\star \bullet$ If $c_{\underline{0}}$ is algebraic relatively to $(\mathcal{F}'_{\underline{l}+(0,\ldots,0,1)},\mathcal{G}'_{\underline{l}+(0,\ldots,0,1)})$, according to Equation (\ref{equ:recurrence_step2bis}), we set $z'=- 
\left(\displaystyle\sum_{(\underline{k},j)\in\mathcal{F}'_{\underline{l}} } a_{\underline{k},\underline{l},j}\underline{s}^{\underline{k}}j\,{c_{\underline{0}}}^{j-1}\right) c_{(0,\ldots,0,1)}$. We have to  determine whether there exists a relation $P(\underline{ s}, c_{\underline{0}})=z'$ with  $P$ having support in $\mathcal{F}'_{\underline{l}+(0,\ldots,0,1)}\cup\mathcal{G}'_{\underline{l}+(0,\ldots,0,1)}$. We consider as in Section \ref{sect:non-homog-1stcase}, a subfamily  $\mathcal{F}''_{\underline{l}+(0,\ldots,0,1)}$ of $\mathcal{F}'_{\underline{l}+(0,\ldots,0,1)}$, the vectors $(V_{\underline{l}+(0,\ldots,0,1),\,\underline{k},j}^{red})_{(\underline{k},j)\in \mathcal{F}''_{\underline{l}+(0,\ldots,0,1)}}$ and $V^{red}_{\underline{l}+(0,\ldots,0,1)}$ for $z'$, and the corresponding matrix  $N^{red}_{\underline{l}+(0,\ldots,0,1)}$.  According to Lemma \ref{lemma:depth-non-homog}, the existence of such a polynomial $P$ is equivalent to the vanishing of the minors of $N^{red}_{\underline{l}+(0,\ldots,0,1)}$ of maximal order up to the row $\underline{p}$ with 	$|\underline{p}| \leq 2.3.   
{\theta_{\underline{s},(0,\ldots,0,1)} }{d}^{d+1} $.  Let us consider one of these minors, say $\det(D)$. For $(\underline{k},j)\in\mathcal{F}'_{\underline{l}}$, we denote by $W_{\underline{k},j}^{red}$ the infinite vector corresponding to $\underline{s}^{\underline{k}}j\,{c_{\underline{0}}}^{j-1}  c_{(0,\ldots,0,1)}$. Hence, we have:
\[V^{red}_{\underline{l}+(0,\ldots,0,1)}= -\displaystyle\sum_{(\underline{k},j)\in\mathcal{F}'_{\underline{l}} } a_{\underline{k},\underline{l},j} W_{\underline{k},j}^{red}.\] 
For each $(\underline{k},j)\in\mathcal{F}'_{\underline{l}}$, we set $D_{\underline{k},j}$ the matrix obtained from $D$ by substituting to its last column, i.e. the part of $V^{red}_{\underline{l}+(0,\ldots,0,1)} $, the corresponding part of the $W_{\underline{k},j}^{red}$. By multilinearity of the determinant, one obtains:
$$ \det(D)=-\sum_{(\underline{k},j)\in\mathcal{F}'_{\underline{l}}} \det(D_{\underline{k},j})a_{\underline{k},\underline{l},j}.$$
So, the vanishing of $ \det(D)$ is equivalent to the vanishing of a linear form in the $a_{\underline{k},\underline{l},j}$'s for $(\underline{k},j)\in\mathcal{F}'_{\underline{l}}$. Considering the linear relations for all these $D$'s, we derive from  $E_{\underline{l}}$ a new space of possible equations for $c_{\underline{0}}$, that we still denote by $E_{\underline{l}}$ for simplicity. In the particular case where $E_{\underline{l}}=\{0\}$, we exclude $\underline{l}$ from the list of admissible multi-indices.

If $E_{\underline{l}}\neq\{0\}$, for each $\underline{a}_{\underline{l}}:=(a_{{\underline{k},\underline{l},j}})_{(\underline{k},j)\in\mathcal{F}'_{\underline{l}}\cup \mathcal{G}'_{\underline{l} }}$ list of coefficients of a polynomial in $ E_{\underline{l}}$, we perform the method in Section \ref{sect:non-homog-1stcase} and we reconstruct the space $\Phi_{\underline{l}+(0,\ldots,0,1)}(\underline{a}_{\underline{l}})$ of coefficients $(a_{\underline{k},\underline{l}+(0,\ldots,0,1),j})_{ (\underline{k},j)\in\mathcal{F}'_{\underline{l}+(0,\ldots,0,1)}\cup\mathcal{G}'_{\underline{l}+(0,\ldots,0,1)}}$ for a relation (\ref{equ:recurrence_step2bis}). By (\ref{equ:reconstr_non-homog}) and (\ref{equ:terme-cst}), it is an affine space $\phi_{\underline{l}+(0,\ldots,0,1)}(\underline{a}_{\underline{l}}) + F_{\underline{l}+(0,\ldots,0,1)} $ where $\phi_{\underline{l}+(0,\ldots,0,1)}(\underline{a}_{\underline{l}})$ is a point and $F_{\underline{l}+(0,\ldots,0,1)}$ a vector space. Note that $\phi_{\underline{l}+(0,\ldots,0,1)}(\underline{a}_{\underline{l}})$ depends linearly on $\underline{a}_{\underline{l}}$ and 
that its computation is done  by computing a finite number of minors of matrices given by the $W_{\underline{k}',j'}^{red}$'s, $(\underline{k}',j')\in\mathcal{F}'_{\underline{l}}$ , and the $V_{\underline{k}'',j''}^{red}$'s, $(\underline{k}'',j'')\in \mathcal{F}''_{\underline{l}+(0,\ldots,0,1)}$.
Also, we have that $ F_{\underline{l}+(0,\ldots,0,1)} $ is independent of $\underline{a}_{\underline{l}}$. Finally, we observe that, for a given $\underline{l}$, the set of admissible \\ $\left((a_{{\underline{k},\underline{l},j}})_{(\underline{k},j)\in\mathcal{F}'_{\underline{l}}\cup \mathcal{G}'_{\underline{l} }}\ ,\ (a_{\underline{k},\underline{l}+(0,\ldots,0,1),j})_{ (\underline{k},j)\in\mathcal{F}'_{\underline{l}+(0,\ldots,0,1)}\cup\mathcal{G}'_{\underline{l}+(0,\ldots,0,1)}}\right)$'s is a nonzero $K$-vector space. \\

\noindent $\star\star$ If $c_{\underline{0}}$ is not algebraic relatively to $(\mathcal{F}'_{\underline{l}+(0,\ldots,0,1)},\mathcal{G}'_{\underline{l}+(0,\ldots,0,1)})$, we have to  determine whether there exists a relation $P(\underline{ s}, c_{\underline{0}})=z'$ with  $P$ having support in $\mathcal{F}'_{\underline{l}+(0,\ldots,0,1)}\cup\mathcal{G}'_{\underline{l}+(0,\ldots,0,1)}$. Note that in this case, such a polynomial $P$ is necessarily unique for a given $z'$. We proceed as above with $\mathcal{F}'_{\underline{l}+(0,\ldots,0,1)}$ instead of $\mathcal{F}''_{\underline{l}+(0,\ldots,0,1)}$ and as in 
Section \ref{sect:secondcase}, in particular Lemma \ref{lemma:depth-non-homog_bis} with  $2.3.   
{\theta_{\underline{s},(0,\ldots,0,1)} }{d}^{d+1} $ as bound for the depth of the minors involved. 
This determines from  $E_{\underline{l}}$ a new space of possible equations for $c_{\underline{0}}$, that we still denote by $E_{\underline{l}}$ for simplicity. In the particular case where $E_{\underline{l}}=\{0\}$, we exclude $\underline{l}$ from the list of admissible multi-indices. Also, if $E_{\underline{l}}\neq\{0\}$, for each $\underline{a}_{\underline{l}}\in E_{\underline{l}}\neq\{0\}$, we reconstruct the list of coefficients $\phi_{\underline{l}+(0,\ldots,0,1)}(\underline{a}_{\underline{l}}):= (a_{\underline{k},\underline{l}+(0,\ldots,0,1),j})_{ (\underline{k},j)\in\mathcal{F}'_{\underline{l}+(0,\ldots,0,1)}\cup\mathcal{G}'_{\underline{l}+(0,\ldots,0,1)}}$ for a relation (\ref{equ:recurrence_step2bis}). By (\ref{equ:reconstr_non-homog_bis}) and (\ref{equ:terme-cst_bis}), $\phi_{\underline{l}+(0,\ldots,0,1)}(\underline{a}_{\underline{l}})$ depends linearly on $\underline{a}_{\underline{l}}$ and 
its computation is done  by computing a finite number of minors of matrices given by the $W_{\underline{k}',j'}^{red}$'s, $(\underline{k}',j')\in\mathcal{F}'_{\underline{l}}$ , and the $V_{\underline{k}'',j''}^{red}$'s, $(\underline{k}'',j'')\in \mathcal{F}'_{\underline{l}+(0,\ldots,0,1)}$.  Again, we observe that, for a given $\underline{l}$, the set of admissible  $\left((a_{{\underline{k},\underline{l},j}})_{(\underline{k},j)\in\mathcal{F}'_{\underline{l}}\cup \mathcal{G}'_{\underline{l} }}\ ,\ (a_{\underline{k},\underline{l}+(0,\ldots,0,1),j})_{ (\underline{k},j)\in\mathcal{F}'_{\underline{l}+(0,\ldots,0,1)}\cup\mathcal{G}'_{\underline{l}+(0,\ldots,0,1)}}\right)$'s is a nonzero $K$-vector space.\\

To sum up Sections \ref{sect:first-step} to \ref{sect:second-step1}, we have reconstructed a finite number of multi-indices $\underline{l}$ (i.e. possible initial steps $\underline{l}_0:=w_{\underline{t}}(P)$) and, for each of these $\underline{l}$'s, the nonzero $K$-vector space $E_{\underline{l},{\underline{l}}+(0,\ldots,0,1)}$ of coefficients $(a_{{\underline{k},\underline{l}',j}})_{(\underline{k},\underline{l}',j)\in\mathcal{F} \cup \mathcal{G}\,,\, \underline{l}\leq_{\mathrm{grlex}}\underline{l}'\leq_{\mathrm{grlex}} \underline{l}+(0,\ldots,0,1) }$  for the initial part of a possible vanishing polynomial for $y_0$.
%depending linearly on 

%$d''_{\underline{s}}$

\subsubsection{Induction step.}\label{sect:induction-step}\hspace{1cm} \\  

For each $\underline{l}\leq_{\textrm{grlex}}\hat{\underline{l}}_0$ possible initial step as above, we assume that up to some $\tilde{\underline{l}}\geq_{\mathrm{grlex}} \underline{l}+(0,\ldots,0,1)$ we have reconstructed the nonzero $K$-vector space, say $E_{\underline{l},\tilde{\underline{l}}}$, of coefficients $(a_{{\underline{k},\underline{l}',j}})_{(\underline{k},\underline{l}',j)\in\mathcal{F} \cup \mathcal{G}\,,\, \underline{l}'\leq_{\mathrm{grlex}} \tilde{\underline{l}}}$  for the initial part of a possible vanishing polynomial for $y_0$.  Recall that, for $\underline{\lambda}\in\N^r$, $S(\underline{\lambda})$ (respectively $A(\underline{\lambda})$ for $\underline{\lambda}\neq 0$) denotes the successor (respectively the predecessor) for $\leq_{\mathrm{grlex}}$ of $\underline{\lambda}$ in $\N^r$. Equation (\ref{equ:recurrence}) gives:
$$  \displaystyle\sum_{j=0,..,d}  a_{S(\tilde{\underline{l}}),j}(\underline{s}) {C_{\underline{0}}}^{j}=- \displaystyle\sum_{\underline{i}<S(\tilde{\underline{l}}),\, j=0,..,d} a_{\underline{i},j}(\underline{s}) C_{S(\tilde{\underline{l}})-\underline{i}}^{(j)}\,,$$
which we write as:
\begin{equation}\label{equ:recurrence_indstep}
	\displaystyle\sum_{(\underline{k},j)\in\mathcal{F}'_{S(\tilde{\underline{l}})}\cup\mathcal{G}'_{S(\tilde{\underline{l}})}} a_{\underline{k},S(\tilde{\underline{l}}),j}\underline{s}^{\underline{k}}{C_{\underline{0}}}^j=- \displaystyle\sum_{\underline{i}<S(\tilde{\underline{l}}) } 
	\left(\displaystyle\sum_{(\underline{k},j)\in\mathcal{F}'_{\underline{i}}\cup\mathcal{G}'_{\underline{i}}} a_{\underline{k},\underline{i},j}\underline{s}^{\underline{k}}\,C_{S(\tilde{\underline{l}})-\underline{i}}^{(j)}\right)\,.
\end{equation} 
Let us denote $\theta_{\underline{s},S(\underline{\tilde{l}})}:= \left(|\hat{\underline{l}}_0|+d\, |S(\underline{\tilde{l}})|\right)a +b $ where $a$ and $b$ are as in Lemma \ref{lemme:control-deg-supp}.	By this lemma, there exist polynomials $\left(P_{\underline{\lambda}}(\underline{s},z_{\underline{0}},\ldots,z_{\underline{ \lambda}})\right)_{\underline{\lambda}= \underline{0},\ldots,S(\underline{\tilde{l}})}$ such that $P_{\underline{\lambda}}(\underline{s},c_{\underline{0}},\ldots,c_{\underline{ \lambda}})=0$, $P_{\underline{\lambda}}(\underline{s},c_{\underline{0}},\ldots,c_{A(\underline{ \lambda})},z_{\underline{ \lambda}})\not\equiv0$, $\deg_{\underline{s}}P_{\underline{\lambda}}\leq \theta_{\underline{s},S(\underline{\tilde{l}})}$, $\deg_{z_{\underline{\mu}}}P_{\underline{\lambda}}\leq d$ for $\underline{\mu}\leq_{\mathrm{grlex}}\underline{\lambda}$. Let us denote 
\begin{equation}\label{eq:indice-multiind}
	i_{S(\underline{\tilde{l}})}:=\left(\begin{array}{c}
		|S(\underline{\tilde{l}})|+r-\tau\\ |S(\underline{\tilde{l}})|
	\end{array}\right)-1.
\end{equation}

Note that $i_{S(\underline{\tilde{l}})}+1$ is at most the number of multi-indices $\underline{\lambda}$ such that  $\underline{\lambda}\leq_{\mathrm{grlex}}S(\underline{\tilde{l}})$.

\noindent$\bullet$ Suppose that $\mathcal{F}'_{S(\tilde{\underline{l}})}\cup\mathcal{G}'_{S(\tilde{\underline{l}})}=\emptyset$. Equation (\ref{equ:recurrence_indstep}) evaluated at $c_{\underline{0}},\ldots,c_{S(\underline{\tilde{l}})}$ reduces to:
\begin{equation}\label{equ:recurrence_indstep-case1}
	\displaystyle\sum_{\underline{i}<S(\tilde{\underline{l}}) } 
	\left(\displaystyle\sum_{(\underline{k},j)\in\mathcal{F}'_{\underline{i}}\cup\mathcal{G}'_{\underline{i}}} a_{\underline{k},\underline{i},j}\underline{s}^{\underline{k}}\,c_{S(\tilde{\underline{l}})-\underline{i}}^{(j)}\right)=0\,.
\end{equation} 
Let us expand $c_{\underline{n}}^{(j)} $ in (\ref{equ:multinomial-expansion}):
$$ 
{y_0}^j= \displaystyle\sum_{\underline{n}\in \mathbb{N}^{r-\tau}}c_{\underline{n}}^{(j)}\,\underline{t}^{\underline{n}}= \left(\displaystyle\sum_{\underline{\gamma}\in \mathbb{N}^{r-\tau}}c_{\underline{\gamma}}\,\underline{t}^{\underline{\gamma}}\right)^j,$$
so,
$$c_{\underline{n}}^{(j)}=\displaystyle\sum_{\underline{j}\ /\ |\underline{j}|=j\atop
	\ \ \ \ g(\underline{j})=\underline{n}}
\displaystyle\frac{j!}{\underline{j}!}\underline{c}^{\underline{j}}$$
where $\underline{j}:=(j_{\underline{0}},\ldots,j_{\underline{n}})$ and $\underline{c}^{\underline{j}}:= c_{\underline{0}}^{j_{\underline{0}}}\cdots c_{\underline{n}}^{j_{\underline{n}}}$ (and where $g$ is as in Notation \ref{nota:FS}).

Let us expand the left hand side of (\ref{equ:recurrence_indstep-case1}):
$$\displaystyle\sum_{\underline{i}<S(\tilde{\underline{l}}) } 
\left(\displaystyle\sum_{(\underline{k},j)\in\mathcal{F}'_{\underline{i}}\cup\mathcal{G}'_{\underline{i}}} a_{\underline{k},\underline{i},j}\underline{s}^{\underline{k}}\,c_{S(\tilde{\underline{l}})-\underline{i}}^{(j)}\right)= \displaystyle\sum_{\underline{i}<S(\tilde{\underline{l}}) } 
\left(\displaystyle\sum_{(\underline{k},j)\in\mathcal{F}'_{\underline{i}}\cup\mathcal{G}'_{\underline{i}}} a_{\underline{k},\underline{i},j}\underline{s}^{\underline{k}}\, \displaystyle\sum_{\underline{j}\ /\ |\underline{j}|=j\atop
	\ \ \ \ g(\underline{j})=S(\tilde{\underline{l}})-\underline{i}}
\displaystyle\frac{j!}{\underline{j}!}\underline{c}^{\underline{j}}\right)$$
(where  $\underline{j}:=(j_{\underline{0}},\ldots,j_{S(\tilde{\underline{l}})})$ and $\underline{c}^{\underline{j}}:= c_{\underline{0}}^{j_{\underline{0}}}\cdots c_{S(\tilde{\underline{l}})}^{j_{S(\tilde{\underline{l}})}}$). 

We set $\mathcal{K}'_{S(\tilde{\underline{l}})}$ the set of $(\underline{k},\underline{j})$ where $\underline{k}\in \N^\tau$ and $\underline{j}:=(j_{\underline{0}},\ldots,j_{S(\tilde{\underline{l}})})$, $\underline{j}\neq \underline{0}$,  such that $j:=|\underline{j}|\in\{0,\ldots,d\}$ and there exists $\underline{i}\in\N^{r-\tau}$ with $\underline{i}<S(\tilde{\underline{l}})$,  $(\underline{k},j)\in\mathcal{F}'_{\underline{i}}\cup\mathcal{G}'_{\underline{i}}$, $g(\underline{j})=S(\tilde{\underline{l}})-\underline{i}$. 
Equation (\ref{equ:recurrence_indstep-case1}) becomes:
$$  \displaystyle\sum_{(\underline{k},\underline{j})\in \mathcal{K}'_{S(\tilde{\underline{l}})}\cup \mathcal{L}'_{S(\tilde{\underline{l}})} } 
\displaystyle\frac{j!}{\underline{j}!}a_{\underline{k},S(\tilde{\underline{l}})-g(\underline{j}),j} 
\,\underline{s}^{\underline{k}}\underline{c}^{\underline{j}}=0
.$$
%Note that $L_{\underline{k},\underline{j}}$ is a linear form with positive integer coefficients.

Thanks to Remark \ref{rem:control-deg}, for any $(\underline{k},\underline{j})\in \mathcal{K}'_{S(\tilde{\underline{l}})}\cup \mathcal{L}'_{S(\tilde{\underline{l}})} $, we have that $|\underline{k}|\leq a\,|S(\underline{\tilde{l}})|+b\leq \theta_{\underline{s},S(\underline{\tilde{l}})}$. We are in position to apply the method of reconstruction of Section \ref{sect:alg-reconstr-multi} of all the polynomials such that 
$$\displaystyle\sum_{(\underline{k},\underline{j})\in \mathcal{K}'_{S(\tilde{\underline{l}})}\cup \mathcal{L}'_{S(\tilde{\underline{l}})} } 
b_{\underline{k},\underline{j}}\,\underline{s}^{\underline{k}}\underline{c}^{\underline{j}}=0.$$
This requires computations of minors of the corresponding Wilczynski matrix up to a finite depth bounded by 
$$2.3^{d^{i_{S(\tilde{\underline{l}})}-1}+\cdots+d^2+d+1}    
\theta_{\underline{s},S(\underline{\tilde{l}})}\,  d^{d^{i_{S(\tilde{\underline{l}})}}+\cdots+d^2+d+1}$$
(see Lemma \ref{lemme:wilcz-multi-profondeur}). By Lemma \ref{lemma:total-reconstr-multi}, the formulas (\ref{equ:coeff-param_1}) and (\ref{equ:terme-cst1_1}) give us with a vector space $B_{S(\underline{\tilde{l}})}$ (possibly zero) of coefficients $b_{\underline{k},\underline{j}}$, hence a corresponding vector space $A_{S(\underline{\tilde{l}})}$ of coefficients $a_{\underline{k},S(\tilde{\underline{l}})-g(\underline{j}),j}=\displaystyle\frac{\underline{j}!}{j!}b_{\underline{k},\underline{j}}$. 
%Let us denote by $L_{S(\underline{\tilde{l}})}$ the linear map whose components are the linear forms $L_{\underline{k},\underline{j}}$ for $(\underline{k},\underline{j})\in \mathcal{K}'_{S(\tilde{\underline{l}})}\cup \mathcal{L}'_{S(\tilde{\underline{l}})} $. We obtain a vector space $A_{S(\underline{\tilde{l}})}:=L_{S(\underline{\tilde{l}})}^{-1}\left( B_{S(\underline{\tilde{l}})}\right) $ of admissible coefficients $ a_{\underline{k},\underline{i},j}$ for Equation (\ref{equ:recurrence_indstep-case1}). 
We take the intersection of $A_{S(\underline{\tilde{l}})}$ with $E_{\underline{l},\tilde{\underline{l}}}$ and we obtain another vector space of admissible coefficients that we still denote by  $E_{\underline{l},\tilde{\underline{l}}}$ for simplicity.  In the particular case where the projection of  $E_{\underline{l},\tilde{\underline{l}}}$ on $E_{\underline{l}}$ is $\{0\}$, we exclude $\underline{l}$ from the list of admissible multi-indices.
\\

\noindent$\star$ Suppose that $\mathcal{F}'_{S(\tilde{\underline{l}})}\cup\mathcal{G}'_{S(\tilde{\underline{l}})}\neq\emptyset$.  We determine whether $c_{\underline{0}}$ is algebraic relatively to $(\mathcal{F}'_{S(\tilde{\underline{l}})},\mathcal{G}'_{S(\tilde{\underline{l}})})$. For this, we examine the vanishing of the minors of maximal order of $M_{\mathcal{F}'_{S(\tilde{\underline{l}})},\mathcal{G}'_{S(\tilde{\underline{l}})}}^{red}$ up to the lowest row of order $2d'_{\underline{s},S(\tilde{\underline{l}})}d$ (see Section \ref{sect:first-step} for the notation). There are two subcases. \vspace{0.3cm}

\noindent $\star \bullet$ If $c_{\underline{0}}$ is algebraic relatively to $(\mathcal{F}'_{S(\tilde{\underline{l}})},\mathcal{G}'_{S(\tilde{\underline{l}})})$, according to Equation (\ref{equ:recurrence_indstep}), we set $z':=- 
\displaystyle\sum_{\underline{i}<S(\tilde{\underline{l}}) } 
\left(\displaystyle\sum_{(\underline{k},j)\in\mathcal{F}'_{\underline{i}}\cup\mathcal{G}'_{\underline{i}}} a_{\underline{k},\underline{i},j}\underline{s}^{\underline{k}}\,c_{S(\tilde{\underline{l}})-\underline{i}}^{(j)}\right)$. We have to  determine whether there exists a relation $P(\underline{ s}, c_{\underline{0}})$ $=z'$ with  $P$ having support in $\mathcal{F}'_{S(\tilde{\underline{l}})}\cup\mathcal{G}'_{S(\tilde{\underline{l}})}$. We consider as in Section \ref{sect:non-homog-1stcase}, a subfamily  $\mathcal{F}''_{S(\tilde{\underline{l}})}$ of $\mathcal{F}'_{S(\tilde{\underline{l}})}$, the vectors $(V_{S(\tilde{\underline{l}}),\,\underline{k},j}^{red})_{(\underline{k},j)\in \mathcal{F}''_{S(\tilde{\underline{l}})}}$ and $V^{red}_{S(\tilde{\underline{l}})}$ for $z'$, and the corresponding matrix  $N^{red}_{S(\tilde{\underline{l}})}$. 

According to Lemma \ref{lemma:depth-non-homog}, the existence of such a polynomial $P$ is equivalent to the vanishing of the minors of $N^{red}_{S(\tilde{\underline{l}})}$ of maximal order up to the row $\underline{p}$ with \begin{Large}$$ 
	|\underline{p}| \leq 2.3^{d^{i_{S(\tilde{\underline{l}})}-1}+\cdots+d^2+d+1}    
	\theta_{\underline{s},S(\tilde{\underline{l}})}\,. d^{d^{i_{S(\tilde{\underline{l}})}}+\cdots+d^2+d+1}  
	$$\end{Large}	
Let us consider one of these minors, say $\det(D)$. For $\underline{i}<S(\tilde{\underline{l}})$, for $(\underline{k},j)\in\mathcal{F}'_{\underline{i}}\cup\mathcal{G}'_{\underline{i}}$, we denote by $W_{\underline{k},\underline{i},j}^{red}$ the infinite vector corresponding to $\underline{s}^{\underline{k}}\,c_{S(\tilde{\underline{l}})-\underline{i}}^{(j)}$. 
We set $D_{\underline{k},\underline{i},j}$ the matrix obtained from $D$ by substituting to its last column, i.e. the part of $V^{red}_{S(\tilde{\underline{l}})} $, the corresponding parts of the $W_{\underline{k},\underline{i},j}^{red}$'s. Since $V^{red}_{S(\tilde{\underline{l}})}= \displaystyle\sum_{\underline{i}<S(\tilde{\underline{l}}) }\left( \sum_{(\underline{k},j)\in \mathcal{F}'_{\underline{i}}\cup\mathcal{G}'_{\underline{i}} } a_{\underline{k},\underline{i},j}.W_{\underline{k},\underline{i},j}^{red}\right) $, one has:
$$ \det(D)=- \displaystyle\sum_{\underline{i}<S(\tilde{\underline{l}}) }\left( \sum_{(\underline{k},j)\in \mathcal{F}'_{\underline{i}}\cup\mathcal{G}'_{\underline{i}} } \det(D_{\underline{k},\underline{i},j})\,a_{\underline{k},\underline{i},j}\right).$$
So, the vanishing of $ \det(D)$ is equivalent to the vanishing of a linear form in the $a_{\underline{k},\underline{i},j}$'s for $\underline{i}<S(\tilde{\underline{l}}) $ and  $(\underline{k},j)\in\mathcal{F}'_{\underline{i}}\cup\mathcal{G}'_{\underline{i}}$. Considering these linear relations, we derive from  $E_{\underline{l},\tilde{\underline{l}}}$ a new space of possible coefficients $(a_{{\underline{k},\underline{l}',j}})_{(\underline{k},\underline{l}',j)\in\mathcal{F} \cup \mathcal{G}\,,\, \underline{l}'\leq_{\mathrm{grlex}} \tilde{\underline{l}}}$, that we still denote by   $E_{\underline{l},\tilde{\underline{l}}}$ for simplicity. In the particular case where the projection of  $E_{\underline{l},\tilde{\underline{l}}}$ on $E_{\underline{l}}$ is $\{0\}$, we exclude $\underline{l}$ from the list of admissible multi-indices.

If this projection is not $\{0\}$, so in particular $E_{\underline{l}}\neq\{0\}$, for each $\underline{a}_{\underline{\tilde{l}}}:=(a_{{\underline{k},\underline{l}',j}})_{(\underline{k},\underline{l}',j)\in\mathcal{F} \cup \mathcal{G},\, \underline{l}'\leq_{\mathrm{grlex}} \underline{\tilde{l}} }$ list of coefficients of a polynomial in $E_{\underline{l},\tilde{\underline{l}}}$, we perform the method in Section \ref{sect:non-homog-1stcase} and we reconstruct the space $\Phi_{S(\tilde{\underline{l}})}(\underline{a}_{\underline{\tilde{l}}})$ of coefficients $(a_{\underline{k},S(\tilde{\underline{l}}),j})_{ (\underline{k},j)\in\mathcal{F}'_{S(\tilde{\underline{l}})}\cup\mathcal{G}'_{S(\tilde{\underline{l}})}}$ for a relation (\ref{equ:recurrence_indstep}). By (\ref{equ:reconstr_non-homog}) and (\ref{equ:terme-cst}), it is an affine space $\phi_{S(\tilde{\underline{l}})}(\underline{a}_{\underline{\tilde{l}}}) + F_{S(\tilde{\underline{l}})} $ where $\phi_{S(\tilde{\underline{l}})}(\underline{a}_{\underline{\tilde{l}}})$ is a point and $F_{S(\tilde{\underline{l}})}$ a vector space. Note that $\phi_{S(\underline{\tilde{l}})}(\underline{a}_{\underline{\tilde{l}}})$ depends linearly on $\underline{a}_{\underline{\tilde{l}}}$ and 
that its computation is done  by computing a finite number of minors of matrices given by the $W_{\underline{k}',\underline{i},j'}^{red}$'s, $\underline{i}<S(\tilde{\underline{l}})$,  $(\underline{k}',j')\in\mathcal{F}'_{\underline{i}}\cup\mathcal{G}'_{\underline{i}}$, and the $V_{\underline{k}'',j''}^{red}$'s, $(\underline{k}'',j'')\in \mathcal{F}''_{S(\tilde{\underline{l}})}$.
Also, we have that $ F_{S(\tilde{\underline{l}})} $ is independent of $\underline{a}_{\underline{\tilde{l}}}$. Finally, we observe that the set of admissible  $\left((a_{{\underline{k},\underline{l}',j}})_{(\underline{k},\underline{l}',j)\in\mathcal{F} \cup \mathcal{G},\, \underline{l}'\leq_{ \mathrm{grlex}} \tilde{\underline{l}} }\ ,\ (a_{\underline{k},S(\tilde{\underline{l}}),j})_{ (\underline{k},j)\in\mathcal{F}'_{S(\tilde{\underline{l}})} \cup\mathcal{G}'_{S(\tilde{\underline{l}})}}\right)$'s, for a given $\underline{l}$, is a nonzero $K$-vector space which we denote by $E_{\underline{l},S(\tilde{\underline{l}})}$. \vspace{0.3cm} \\

\noindent $\star \star$ If $c_{\underline{0}}$ is not algebraic relatively to $(\mathcal{F}'_{S(\tilde{\underline{l}})},\mathcal{G}'_{S(\tilde{\underline{l}})})$, according to Equation (\ref{equ:recurrence_indstep}), we set $z'=- 
\displaystyle\sum_{\underline{i}<S(\tilde{\underline{l}}) } 
\left(\displaystyle\sum_{(\underline{k},j)\in\mathcal{F}'_{\underline{i}}\cup\mathcal{G}'_{\underline{i}}} a_{\underline{k},\underline{i},j}\underline{s}^{\underline{k}}\,c_{S(\tilde{\underline{l}})-\underline{i}}^{(j)}\right)$. We want to  determine if there exists a relation $P(\underline{ s}, c_{\underline{0}})=z'$ with  $P$ having support in $\mathcal{F}'_{S(\tilde{\underline{l}})}\cup\mathcal{G}'_{S(\tilde{\underline{l}})}$. As in Section \ref{sect:secondcase}, we consider the vectors $(V_{S(\tilde{\underline{l}}),\,\underline{k},j}^{red})_{(\underline{k},j)\in \mathcal{F}'_{S(\tilde{\underline{l}})}}$, $V^{red}_{S(\tilde{\underline{l}})}$ for $z'$, and the corresponding matrix  $N^{red}_{S(\tilde{\underline{l}})}$. 

According to Lemma \ref{lemma:depth-non-homog_bis}, the existence of such a polynomial $P$ is equivalent to the vanishing of the minors of $N^{red}_{S(\tilde{\underline{l}})}$ of maximal order up to the row $\underline{p}$ with \begin{Large}$$ 
	|\underline{p}| \leq 2.3^{d^{i_{S(\tilde{\underline{l}})}-1}+\cdots+d^2+d+1}    
	\theta_{\underline{s},S(\tilde{\underline{l}})}\,. d^{d^{i_{S(\tilde{\underline{l}})}}+\cdots+d^2+d+1}  
	$$\end{Large}
where $i_{S(\tilde{\underline{l}})}$ is defined by (\ref{eq:indice-multiind}).	

As previously, for any of such minors, say  $ \det(D)$, the vanishing of $ \det(D)$ is equivalent to the vanishing of a linear form in the $a_{\underline{k},\underline{i},j}$'s for $\underline{i}<S(\tilde{\underline{l}}) $ and  $(\underline{k},j)\in\mathcal{F}'_{\underline{i}}\cup\mathcal{G}'_{\underline{i}}$. Considering these linear relations, we derive from  $E_{\underline{l},\tilde{\underline{l}}}$ a new space of possible coefficients $(a_{{\underline{k},\underline{l}',j}})_{(\underline{k},\underline{l}',j)\in\mathcal{F} \cup \mathcal{G}\,,\, \underline{l}'\leq_{\mathrm{grlex}} \tilde{\underline{l}}}$, that we still denote by   $E_{\underline{l},\tilde{\underline{l}}}$ for simplicity. In the particular case where the projection of  $E_{\underline{l},\tilde{\underline{l}}}$ on $E_{\underline{l}}$ is $\{0\}$, we exclude $\underline{l}$ from the list of admissible multi-indices.

If this projection is not $\{0\}$, so in particular $E_{\underline{l}}\neq\{0\}$, for each $\underline{a}_{\underline{\tilde{l}}}:=(a_{{\underline{k},\underline{l}',j}})_{(\underline{k},\underline{l}',j)\in\mathcal{F} \cup \mathcal{G},\, \underline{l}'\leq_{\mathrm{grlex}} \underline{\tilde{l}} }$ list of coefficients of a polynomial in $E_{\underline{l},\tilde{\underline{l}}}$, we perform the method in Section \ref{sect:secondcase} and we reconstruct the \emph{unique} list of coefficients $(a_{\underline{k},S(\tilde{\underline{l}}),j})_{ (\underline{k},j)\in\mathcal{F}'_{S(\tilde{\underline{l}})}\cup\mathcal{G}'_{S(\tilde{\underline{l}})}}$ for a relation (\ref{equ:recurrence_indstep}). Note that this list depends linearly on  $(a_{{\underline{k},\underline{l}',j}})_{(\underline{k},\underline{l}',j)\in\mathcal{F} \cup \mathcal{G},\, \underline{l}'\leq_{ \mathrm{grlex}} \tilde{\underline{l}} }$ by relations (\ref{equ:reconstr_non-homog_bis}) and (\ref{equ:terme-cst_bis}).  Finally,  we denote by $E_{\underline{l},S(\tilde{\underline{l}})}$ the $K$-vector space of   $\left((a_{{\underline{k},\underline{l}',j}})_{(\underline{k},\underline{l}',j)\in\mathcal{F} \cup \mathcal{G},\, \underline{l}'\leq_{ \mathrm{grlex}} \tilde{\underline{l}} }\ ,\ (a_{\underline{k},S(\tilde{\underline{l}}),j})_{ (\underline{k},j)\in\mathcal{F}'_{S(\tilde{\underline{l}})} \cup\mathcal{G}'_{S(\tilde{\underline{l}})}}\right)$ admissible. 
\vspace{1cm}

As a conclusion, we obtain:
\begin{theo}\label{theo:reconstr-var-u}
	Let $\tilde{\underline{n}}^0\in\N^r$, $p\in\N^*$, $\underline{q}\in\N^{r-1}\setminus\{\underline{0}\}$, $d\in\N^*$ be given. Let $\mathcal{F},\mathcal{G}$ be as in Definition \ref{defi:algebroid-relative} and  satisfying Conditions (i), (ii), (iii) of Lemma \ref{lemma:support-equ}. Let $y_0=\displaystyle\sum_{(\underline{m},\underline{n})\in\mathbb{N}^\tau\times\mathbb{N}^{r-\tau}} c_{\underline{m},\underline{n}}\underline{s}^{\underline{m}}\underline{t}^{\underline{n}}=\displaystyle\sum_{\underline{n}\in \mathbb{N}^{r-\tau}}c_{\underline{n}}(\underline{s})\,\underline{t}^{\underline{n}}\in K[[\underline{s},{\underline{t}}]]$, $c_{\underline{0},\underline{0}}\neq 0$, be a series  {algebroid relatively to $(\mathcal{F},\mathcal{G})$}. Let $\hat{\underline{l}}_0\in \N^{r-\tau}$ be given. Assume that there exists a polynomial  $P\in\left(K[\underline{s}][[\underline{t}]][y]\right)_{\mathcal{F},\mathcal{G}}\setminus\{0\}$ such that $P(\underline{s},\underline{t},y_0)=0$ and $w_{\underline{t}}(P)\leq_{\textrm{grlex}} \hat{\underline{l}}_0$.
	
	For any $\underline{l}\leq_{\mathrm{grlex}} \hat{\underline{l}}_0$, for any  $\tilde{\underline{l}}\geq_{\mathrm{grlex}}\underline{l}$,  Sections \ref{sect:first-step} to \ref{sect:induction-step} provide the vector space $E_{\underline{l},\tilde{\underline{l}}}$ of all the polynomials $Q_{\underline{l},\tilde{\underline{l}}}\in\left(K[\underline{s}][[\underline{t}]][y]\right)_{\mathcal{F},\mathcal{G}}$ such that:
	\[w_{\underline{t}}(Q_{\underline{l},\tilde{\underline{l}}})={\underline{l}}\hspace{1cm} \textrm{ and }\hspace{1cm} w_{\underline{t}}(Q_{\underline{l},\tilde{\underline{l}}}(\underline{s},\underline{t},y_0) )>_{\textrm{grlex}}\tilde{\underline{l}}. \]
	
\end{theo}

\subsection{Proof of Theorem \ref{theo:reconstr}}\label{sect:proof-theo-reconstr}

Theorem \ref{theo:reconstr} will be a corollary of the following result:

\begin{theo}\label{theo:reconstr-optimise}
	Let $d\in\N^*$ and $\tilde{\nu}_0\in\N$. Let $\tilde{y}_0\in \mathcal{K}_r$, more precisely $\tilde{y}_0=\displaystyle\frac{\tilde{f}}{\tilde{g}}$ for some formal power series $\tilde{f},\tilde{g}\in K\left[\left[\left(\frac{x_1}{x_2^{q_1}}\right)^{1/p},\ldots, \left(\frac{x_{r-1}}{x_r^{q_{r-1}}}\right)^{1/p} ,x_r^{1/p}\right]\right]$. We assume that $\tilde{y}_0$ is algebroid of degree bounded by $d$, and that there is a vanishing polynomial $\tilde{P}$  of degree bounded by $d$ and of $(\underline{x})$-adic order bounded by $\tilde{\nu}_0$. Let $q_i'\geq q_i$, $i=1,\ldots,r-1$, be such that the transform $fg$ of $\tilde{f}\tilde{g}$ under the change of variables $u_i:=\left(\frac{x_i}{x_{i+1}^{q_i'}}\right)^{1/p}$, $i=1,\ldots,r-1$, $u_r={x_r}^{1/p}$, is monomialized with respect to the $u_i$'s:
	\[(fg)(\underline{u}):=(\tilde{f}\tilde{g})\left(  u_1^pu_{2}^{pq'_{1}}\cdots u_{r}^{pq'_{1}q'_{2}\cdots q'_{r-1}}\    ,\ \ldots\ ,\ u_{r-1}^pu_{r}^{p q'_{r-1}} ,\ u_r^{p} ,\ y\right)\]
	We resume the notations of  (\ref{equ:var-paquets}), (\ref{equ:formula}), (\ref{equ:translation}), in particular, $x_i\in \underline{\xi}_k$ if and only if $q_i'>0$, and otherwise $x_i \in \underline{x}_k$ for some $k$:
	\begin{equation}
		\underline{x}^{\underline{ n}}y^j = \underline{x}_0^{\underline{ n}_0}\,\underline{\xi}_1^{\underline{ m}_1}\,\underline{x}_1^{\underline{ n}_1}\cdots\underline{\xi}_\sigma^{\underline{ m}_\sigma}\,\underline{x}_\sigma^{\underline{ n}_\sigma}y^j. 
	\end{equation}
	where $\underline{n}=(\underline{ n}_0, \underline{m }_1,\underline{ n}_1,\ldots,\underline{ m}_\sigma,\underline{ n}_\sigma)$. For $k=1,\ldots,\sigma$, we denote $\underline{ \xi}_k=(x_{i_k},\ldots,x_{j_k-1})$ and $\underline{ x}_k=(x_{j_k},\ldots,x_{i_{k+1}-1})$, and accordingly $\underline{ m}_k=(n_{i_k},\ldots,n_{j_k-1})$ and $\underline{ n}_k=(n_{j_k},\ldots,n_{i_{k+1}-1})$ with $i_{\sigma+1}:=r+1$. For $k=0$ when $\underline{ x}_0$ is not empty, we denote $\underline{ x}_0=(x_{j_0},\ldots,x_{i_{1}-1})$ and $\underline{ n}_0=(n_{j_0},\ldots,n_{i_{1}-1})$ with $j_0:=1$. When $\underline{ x}_0$ is  empty, we set $\underline{ n}_0=0$.

	We set:
	\[\begin{array}{lccl}
		\tilde{L}_k:&\Z^{i_{k+1}-i_k}&\rightarrow &\Z\\
		&(\underline{m}_k,\underline{n}_k)=(n_{i_k},\ldots,n_{i_{k+1}-1})&\mapsto & \tilde{L}_k(\underline{m}_k,\underline{0})+ |\underline{n}_k|
	\end{array}\]
	where: \[
	\tilde{L}_k(\underline{m}_k,\underline{0}):=q'_{j_k-1}q'_{j_k-2}\cdots q'_{i_k}n_{i_k}+\cdots+q'_{j_k-1}q'_{j_k-2}n_{j_k-2} + q'_{j_k-1}n_{j_k-1}.\]
	Moreover, let \[\tilde{L}(\underline{n}):=|\underline{n}_0|+\sum_{k=1,\ldots,\sigma}\tilde{L}_k(\underline{m}_k,\underline{n}_k).\]

	%\textcolor{red}{NON pas de $p$, deja pris en compte ci-dessous quand on considere $(1/p)\underline{n}$ puis $L(\underline{n})$}
	The algorithm described in Section \ref{sect:reconstr-algebroid-algo} provides for any $\nu\in\N$ all the polynomials $\tilde{Q}_\nu(\underline{x},y)\in K[[\underline{x}]][y]$ with $\deg_y\tilde{Q}_\nu\leq d$ and of $(\underline{x})$-adic order bounded by $\tilde{\nu}_0$ such that, for any $\frac{1}{p}\underline{n}=\frac{1}{p}(n_1,\ldots,n_r)\in \mathrm{Supp}\,\tilde{Q}_\nu(\underline{x},\tilde{y}_0) $, one has:
	$$ \tilde{L}(\underline{n})\geq \nu.$$
\end{theo}

Recall that, by the Monomialization Lemma \ref{lemme:monomialisation} and by Remark \ref{rem:monom}, if $\underline{\beta}=(\beta_1,\ldots,\beta_r)$ is the lexicographic valuation of $\tilde{f}\tilde{g}$ with respect to the variables $\zeta_i:=\left(\frac{x_i}{x_{i+1}^{q_i}}\right)^{1/p}$ for $i=1,\ldots,r-1$, $\zeta_r:=x_r^{1/p}$, then the assumptions of Theorem \ref{theo:reconstr-optimise} are satisfied with  $q_i':=q_i+\beta_{i+1}+1$. Therefore, Theorem \ref{theo:reconstr} follows.\\

Let us now deduce Theorem \ref{theo:reconstr-optimise} from Theorem \ref{theo:reconstr-var-u}. Suppose that $\ord_{\underline{x}}\tilde{P} \leq \tilde{\nu}_0$. Let  $\mathcal{F},\mathcal{G}$ be as in Definition \ref{defi:algebroid-relative} and such that $\mathcal{F}\cup\mathcal{G}$ is the total family of multi-indices $(\underline{\alpha},j)$ satisfying Conditions (i), (ii), (iii) of Lemma \ref{lemma:support-equ} with $q_i'$ instead of $q_i$. By the transformations described in (\ref{equ:eclt1}), (\ref{equ:y0-y0tilde}) and  (\ref{equ:m}) associated to the  change of variables  $u_i:=\left(\frac{x_i}{x_{i+1}^{q_i'}}\right)^{1/p}$, $i=1,\ldots,r-1$, $u_r={x_r}^{1/p}$, we obtain a polynomial 
\[P(\underline{u},y):=\underline{u}^{\tilde{\underline{m}}^0}\tilde{P}\left(  u_1^pu_{2}^{pq'_{1}}\cdots u_{r}^{pq'_{1}q'_{2}\cdots q'_{r-1}}\    ,\ \ldots\ ,\ u_r^{p} ,\ \underline{u}^{\tilde{\underline{n}}^0}y\right)\in \left(K[[\underline{u}]][y]\right)_{\mathcal{F},\mathcal{G}}.\]
Recall that  we denote by $\underline{x}_k$, $\underline{\xi}_k$    the sub-tuple of variables $x_i$ corresponding to $\underline{t}_k$, $\underline{s}_k$ respectively.	For $k=0$ when $\underline{ t}_0$ is not empty, we denote $\underline{ x}_0=(x_{j_0},\ldots,x_{i_{1}-1})$, $\underline{ t}_0=(u_{j_0},\ldots,u_{i_{1}-1})=({x_{j_0}}^{1/p},\ldots,{x_{i_{1}-1}}^{1/p})$ and $\underline{ n}_0=(n_{j_0},\ldots,n_{i_{1}-1})$ with $j_0:=1$.

According to  (\ref{equ:var-paquets}), (\ref{equ:formula}), (\ref{equ:translation}), a monomial $\underline{x}^{\underline{ n}}$  is transformed into a monomial  $\underline{u}^{\underline{\alpha}}=\underline{s}^{\underline{\beta}}\underline{t}^{\underline{\gamma}}$ such that, for $k=1,\ldots,\sigma$, we have:
\[
\begin{array}{l}
	{\underline{\xi}_k}^{\underline{ m}_k}\,{\underline{x}_k}^{\underline{ n}_k}= {s_{i_k}}^{pn_{i_k}} {s_{i_k+1}}^{p(n_{i_k+1}+q'_{i_k}n_{i_k})}\cdots {s_{j_k-1}}^{p(n_{j_k-1}+q'_{j_k-2}n_{j_k-2}+q'_{j_k-2}q'_{j_k-3}n_{j_k-3}+\cdots+ q'_{j_k-2}q'_{j_k-3}\cdots q'_{i_k}n_{i_k})} \\
	\hspace{1cm}{ t_{j_k}}^{p(n_{j_k}+q'_{j_k-1}n_{j_k-1}+q'_{j_k-1}q'_{j_k-2}n_{j_k-2}+\cdots+ q'_{j_k-1}q'_{j_k-2}\cdots q'_{i_k}n_{i_k})}{ t_{j_k+1}}^{pn_{j_k+1}}\cdots {t_{i_{k+1}-1}}^{pn_{i_{k+1}-1}}.
\end{array}\]
Hence, a monomial $\underline{x}^{\underline{ n}}y^j$ of $\tilde{P}(\underline{x},y)$ gives a monomial  $\underline{u}^{\underline{\alpha}}\underline{u}^{\tilde{\underline{ m}}^0+j\tilde{\underline{n}}^0}y^j=\underline{s}^{\underline{\beta}}\underline{t}^{\underline{\gamma}}\underline{u}^{\tilde{\underline{ m}}^0+j\tilde{\underline{n}}^0}y^j$ of $P(\underline{u},y)$.
Since $\supp (\tilde{P})$ contains a monomial $\underline{x}^{\underline{ n}}y^j$ such that 
\[|\underline{n}|= |\underline{n}_0|+\sum_{k=1}^{\sigma}\left(|\underline{m}_{k}|+|\underline{n}_{k}|\right)\leq \tilde{\nu}_0,\]
we have that:
\begin{equation}\label{equ:ord}
	\ord_{\underline{t}}P\leq p|\underline{n}_0|+ \sum_{k=1}^{\sigma} \left(pq'_{j_k-1}q'_{j_k-2}\cdots q'_{i_k}|\underline{m}_{k}|+p|\underline{n}_{k}|\right)\ +\ \left|\left(\tilde{\underline{ m}}^0+j\tilde{\underline{n}}^0\right)_{|\underline{t}}\right|\\
	\leq p.\kappa.\tilde{\nu}_0 + d.\rho 
\end{equation}
where $\underline{n}_{|\underline{t}}$ denotes the components of $\underline{n}$ corresponding to the exponents of the variables $\underline{t}$ in $\underline{u}^{\underline{n}}$,  $ \kappa:=\displaystyle\max_{k=1,..,\sigma}(q'_{j_k-1}q'_{j_k-2}\cdots q'_{i_k})$ and  $\rho:=\displaystyle\sum_{k=0}^\sigma \left( |\tilde{n}^0_{j_k}|+\cdots+|\tilde{n}^0_{i_{k+1}-1}|\right)$. We set \begin{equation}\label{equ:l0chap}
	\hat{\underline{l}}_0:= (p.\kappa.\tilde{\nu}_0 + d.\rho,0,\ldots,0)\in \N^{r-\tau},
\end{equation}
so that $w_{\underline{t}}(P)\leq_{\mathrm{grlex}}\hat{\underline{l}}_0$.

Given $\tilde{Q}_\nu(\underline{x},y)$ as in Theorem \ref{theo:reconstr-optimise}, let us denote by $Q_\nu(\underline{u},y)$ its transform via  (\ref{equ:eclt1}), (\ref{equ:y0-y0tilde}), (\ref{equ:m}) as recalled between $\tilde{P}$ and $P$ above. One gets $\tilde{Q}_\nu(\underline{x},\tilde{y}_0)=\underline{u}^{\underline{\tilde{m}}^0}Q_{\nu}(\underline{u},y_0)$. According to  (\ref{equ:var-paquets}), (\ref{equ:formula}), (\ref{equ:translation}), a monomial $\underline{x}^{\underline{ n}/p}$ of $\tilde{Q}_\nu(\underline{x},\tilde{y}_0)$ is transformed into a monomial $\underline{u}^{\underline{\alpha}}=\underline{s}^{\underline{\beta}}\underline{t}^{\underline{\gamma}}$ such that, for $k=1,\ldots,\sigma$, we have:
\[
\begin{array}{l}
	{\underline{\xi}_k}^{\underline{ m}_k/p}\,{\underline{x}_k}^{\underline{ n}_k/p}= {s_{i_k}}^{n_{i_k}} {s_{i_k+1}}^{n_{i_k+1}+q'_{i_k}n_{i_k}}\cdots {s_{j_k-1}}^{n_{j_k-1}+q'_{j_k-2}n_{j_k-2}+q'_{j_k-2}q'_{j_k-3}n_{j_k-3}+\cdots+ q'_{j_k-2}q'_{j_k-3}\cdots q'_{i_k}n_{i_k}} \\
	\hspace{1cm}{ t_{j_k}}^{n_{j_k}+q'_{j_k-1}n_{j_k-1}+q'_{j_k-1}q'_{j_k-2}n_{j_k-2}+\cdots+ q'_{j_k-1}q'_{j_k-2}\cdots q'_{i_k}n_{i_k}}{ t_{j_k+1}}^{n_{j_k+1}}\cdots {t_{i_{k+1}-1}}^{n_{i_{k+1}-1}}.
\end{array}\]
So the monomials of $Q_\nu(\underline{u},y_0)$ are of the form $\underline{u}^{\underline{\alpha}-\underline{\tilde{m}}^0}$. As in the computation of (\ref{equ:ord}), $\ord_{\underline{x}}\tilde{Q}_\nu(\underline{x},y)\leq \tilde{\nu}_0$ implies that $\ord_{\underline{t}}Q_\nu(\underline{u},y)\leq p.\kappa.\tilde{\nu}_0 + d.\rho$, so $w_{\underline{t}}(Q_\nu(\underline{u},y))\leq_{\mathrm{grlex}} \hat{l}_0$. 

Moreover, since $\tilde{Q}_\nu(\underline{x},\tilde{y}_0)=\underline{u}^{\underline{\tilde{m}}^0}Q_{\nu}(\underline{u},y_0)$,  the condition such that for any\\ $\frac{1}{p}\underline{n}=\frac{1}{p}(n_1,\ldots,n_r)\in \mathrm{Supp}\,\tilde{Q}_\nu(\underline{x},\tilde{y}_0) $, $\tilde{L}(\underline{n})\geq \nu$, is equivalent to $\ord_{\underline{t}}(Q_\nu(\underline{u},y_0))+\left|{\tilde{\underline{m}}^0}_{\,|\underline{t}}\right|\geq \nu$. This is in turn equivalent to  $w_{\underline{t}}(Q_\nu(\underline{u},y_0))\geq \left(0,\ldots,0,\nu-\left|{\tilde{\underline{m}}^0}_{\,|\underline{t}}\right|\right)$. We set\\ $\tilde{\underline{l}}_\nu:= \left(0,\ldots,0,\nu-\left|{\tilde{\underline{m}}^0}_{\,|\underline{t}}\right|\right)$, and $\underline{l}:=w_{\underline{t}}(Q_\nu(\underline{u},y))$.

A polynomial $\tilde{Q}_\nu(\underline{x},y)$ satisfying the conditions of Theorem \ref{theo:reconstr-optimise} comes from a polynomial $Q_\nu(\underline{u},y)$ as above satisfying 
\[w_{\underline{t}}(Q_\nu(\underline{u},y))\leq_{\mathrm{grlex}}\hat{l}_0\hspace{0.3cm} \textrm{ and }\hspace{0.3cm}  w_{\underline{t}}(Q_\nu(\underline{u},y_0))\geq \tilde{\underline{l}}_\nu. 
\]
The construction of such polynomials $Q_\nu(\underline{u},y)=Q_{\underline{l},\tilde{\underline{l}}_\nu}(\underline{u},y)$ is given by Theorem \ref{theo:reconstr-var-u}.

This achieves the proofs of Theorems \ref{theo:reconstr-optimise} and \ref{theo:reconstr}.

\subsection{Plan of the algorithm and example}

For the convenience of the reader, we now give several flowcharts in order to describe the algorithm. The first one provides the plan of the algorithm. The others consist of the details of the corresponding steps. \newpage

\includegraphics[scale=1]{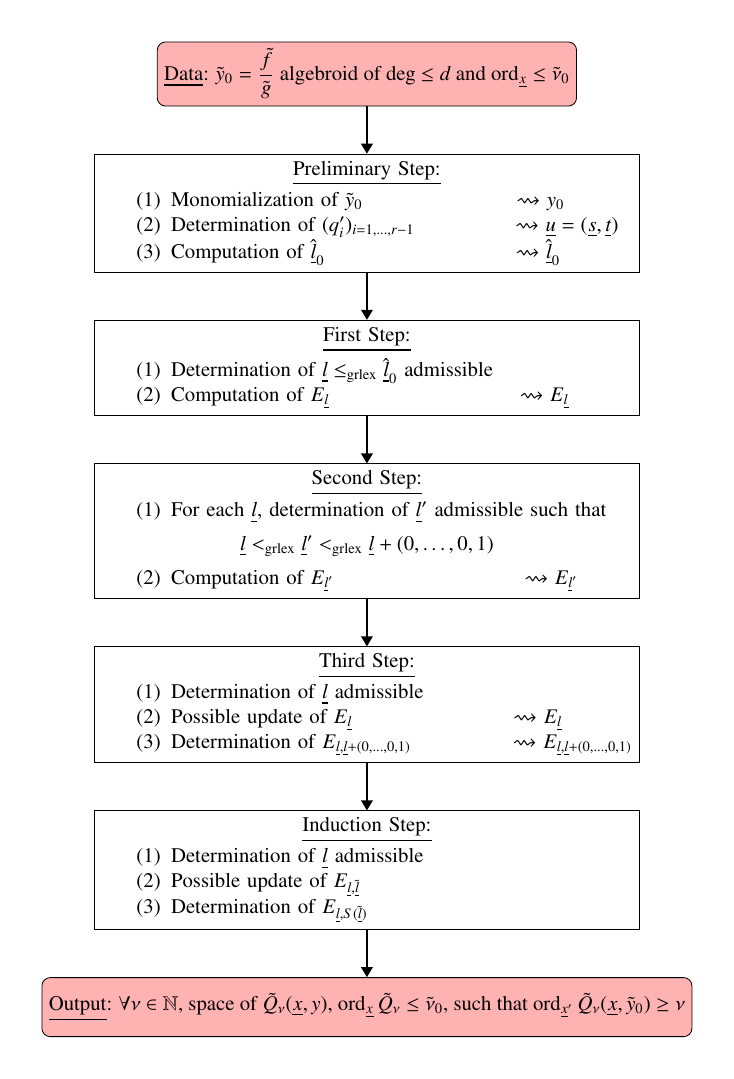}

\includegraphics[scale=1]{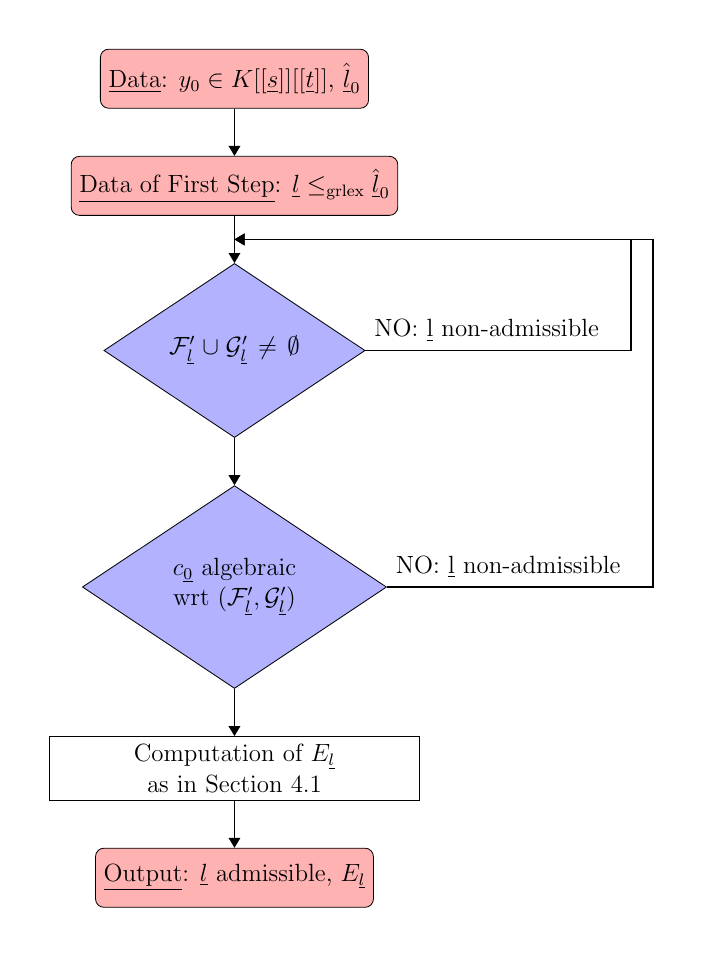}

\includegraphics[scale=0.9]{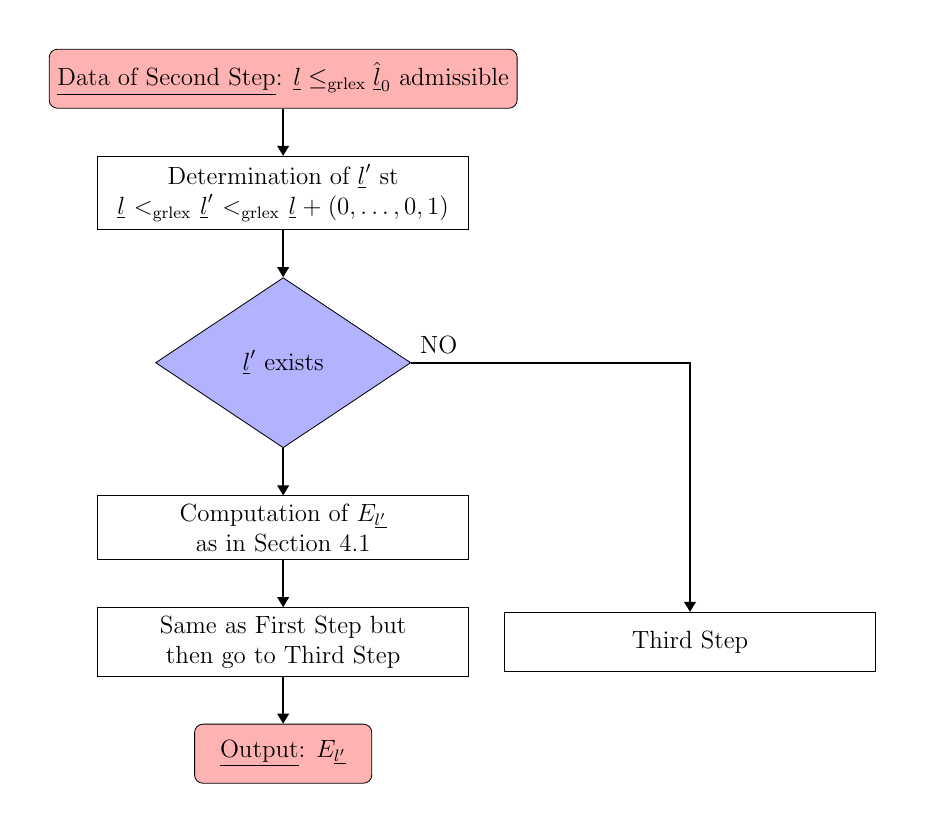}

\includegraphics[scale=0.6]{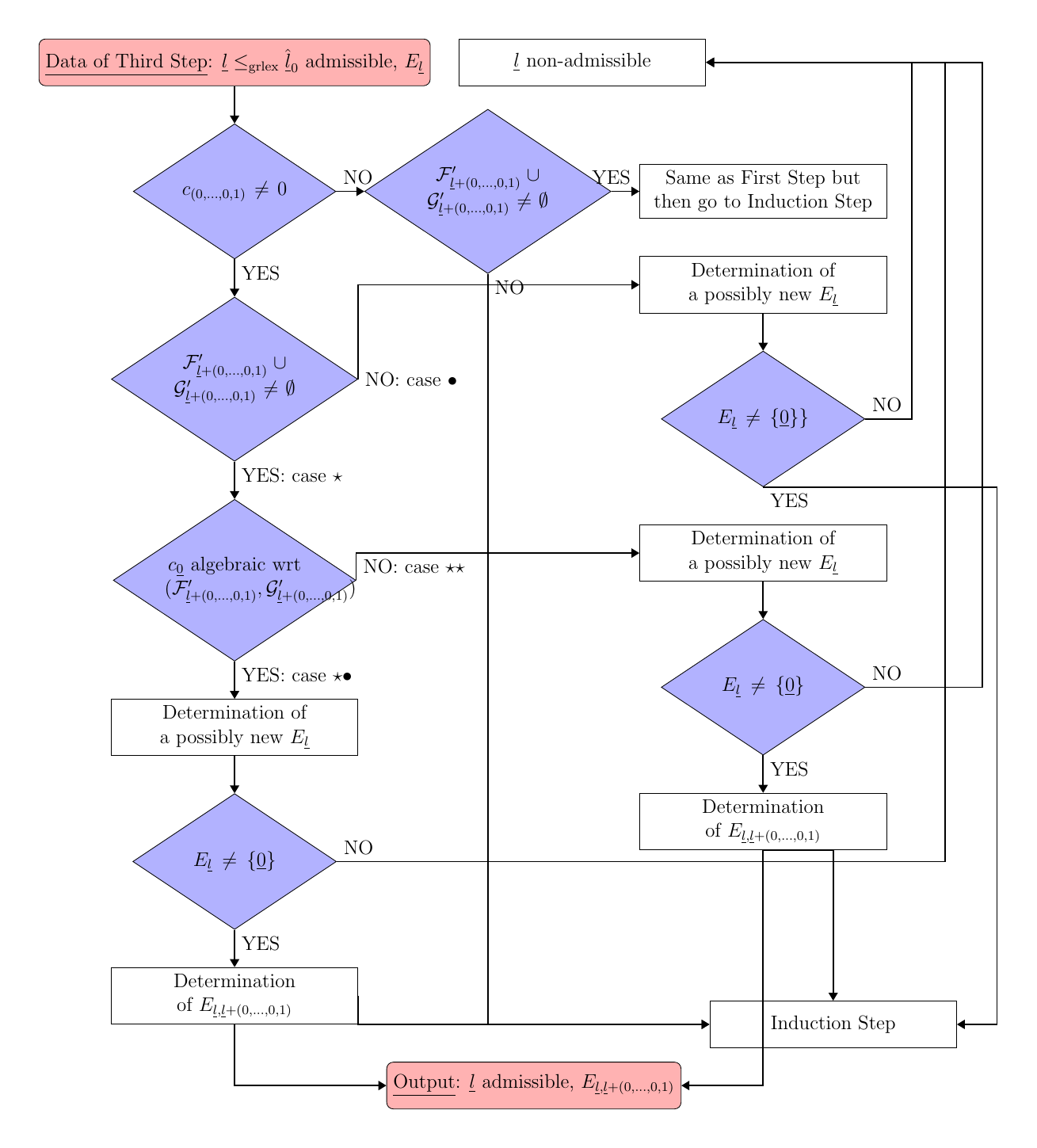}

\includegraphics[scale=0.6]{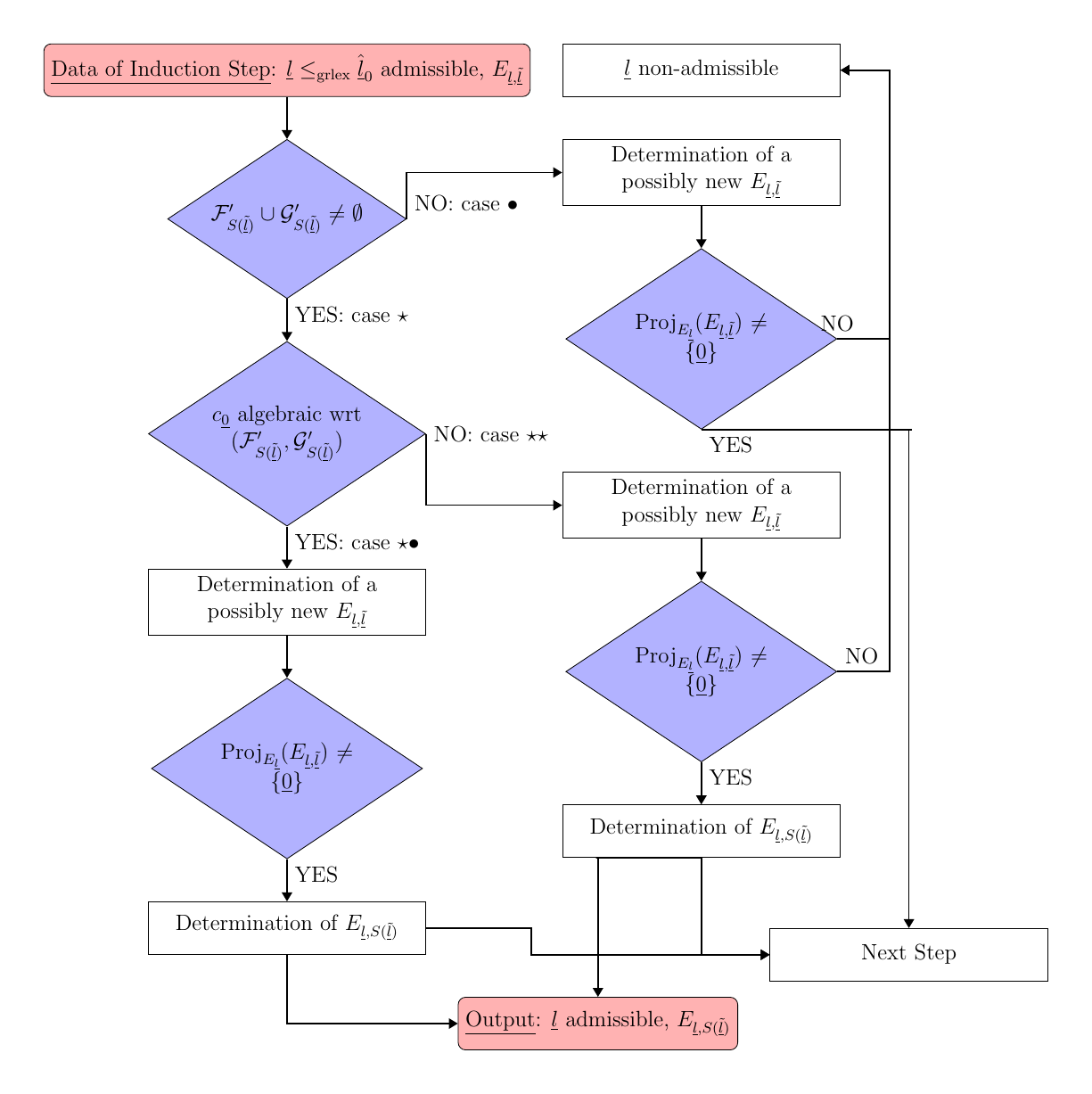}

\begin{ex}\label{ex:eclt} 
	The purpose of the present example is to illustrate the various points of our Theorem \ref{theo:reconstr-optimise}. For $r=d=p=2$ and $q_1=\tilde{\nu}_0=1$, let us consider 
	$\tilde{y}_0=\displaystyle\frac{\tilde{f}}{\tilde{g}}\in \mathcal{K}_2$ with $ \tilde{f},\tilde{g}\in K\left[\left[\left(\frac{x_1}{x_2}\right)^{1/2},x_2^{1/2}\right]\right]$ a root of the following equation:
	\begin{equation}\label{eq:exemple}
		\tilde{P}(x_1,x_2,y)\,:=\,\sin(x_1+x_2)y^2+e^{x_1}x_1x_2y-{x_2}^2\cos(x_1x_2)\,=\, 0. 
	\end{equation}
	For instance, 
	\begin{align*}
		{\tilde{y}_0} & := & {\frac {  - {{\rm e}^{x_1}}x_1{x_2}+ \sqrt{  {{\rm e}^{2x_1}}{x_1}^2{x_2}^2+4\,{x_2}^2\cos \left( x_1x_2 \right) \sin \left( x_1+x_2 \right) }  }{2\,\sin \left( x_1+x_2 \right) }}\\
		&=&  {\frac {  - {{\rm e}^{{{\frac{x_1}{x_2}}}{x_2}}}{\frac{x_1}{x_2}}{x_2}+ {x_2}^{1/2}\sqrt{ {{\rm e}^{2{\frac{x_1}{x_2}}{x_2}}}  {\left(\frac{x_1}{x_2}\right)}^{2}  {x_2}+4\,\cos \left( {\frac{x_1}{x_2}}{x_2}^2 \right) \sin \left( {\frac{x_1}{x_2}}{x_2}+{x_2} \right)/{x_2} }  }{2\,\sin \left( {\frac{x_1}{x_2}}{x_2}+{x_2} \right) / {x_2}}} 	 
	\end{align*}
	
	and therefore:
	\begin{align*}
		\tilde{f}&:=&  \left[ 2+\frac{x_1}{x_2}-{\frac {1}{4}}\left(\frac{x_1}{x_2}\right)^{2}+{\frac {1}{8}}\left(\frac{x_1}{x_2}\right)^{3}-{\frac {5}{64}\left(\frac{x_1}{x_2}\right)^{4}}+{\frac {7}{128}\left(\frac{x_1}{x_2}\right)^{5}} \right] {x_2}^{1/2} -\frac{x_1}{x_2}x_2 \\
		&&+\left[ {\frac {1}{4}\left(\frac{x_1}{x_2}\right)^{2}}-{\frac {1}{8}\left(\frac{x_1}{x_2}\right)^{3}}+{\frac {3}{32}\left(\frac{x_1}{x_2}\right)^{4}}-{\frac {5}{64}\left(\frac{x_1}{x_2}\right)^{5}} \right]{x_2}^{3/2} -{\left(\frac{x_1}{x_2}\right)}^2{x_2}^{2} \\
		&&+\left[ -{\frac{1}{6}}-{\frac {5}{12}\frac{x_1}{x_2}}-{\frac {5}{16}\left(\frac{x_1}{x_2}\right)^{2}}+{\frac {43}{96}\left(\frac{x_1}{x_2}\right)^{3}}-{\frac {199}{768}\left(\frac{x_1}{x_2}\right)^{4}}+{\frac {107}{512}\left(\frac{x_1}{x_2}\right)^{5}} \right] {x_2}^{5/2}\\	 
		&&-{\frac {1}{2}\left({\frac{x_1}{x_2}}\right)}^{2}{x_2}^{3}+\cdots %\hspace{1.5cm}
		\\
		\tilde{g}&:=&  \left[2+2\,{{\frac{x_1}{x_2}}}\right]-\left[ {\frac {1}{3}}+{{\frac{x_1}{x_2}}}+{\left({\frac{x_1}{x_2}}\right)}^{2}+{\frac {1}{3}}{\left({\frac{x_1}{x_2}}\right)}^{3}\right]{x_2}^{2}\\
		&&+  \left[{\frac {1}{60}}+{\frac {1}{12}}{\frac{x_1}{x_2}}+{\frac {1}{6}}{\left({\frac{x_1}{x_2}}\right)}^{2}+{\frac {1}{6}}{\left({\frac{x_1}{x_2}}\right)}^{3}+{\frac {1}{12}}{\left({\frac{x_1}{x_2}}\right)}^{4}   +{\frac {1}{60}}{\left({\frac{x_1}{x_2}}\right)}^{5}\right]{x_2}^{4} \\
		&&-\frac {1}{2520}\left[\sum_{k=0}^7 \frac{7!}{k!(7-k)!} \,{\left({\frac{x_1}{x_2}}\right)}^{k}
		% {\frac {1}{2520}}+{\frac {1}{360}}{\frac{x_1}{x_2}}+{\frac {1}{120}}{\left({\frac{x_1}{x_2}}\right)}^{2}+{\frac {1}{72}}{\left({\frac{x_1}{x_2}}\right)}^{3}+{\frac {1}{72}}{\left({\frac{x_1}{x_2}}\right)}^{4}+{\frac {1}{120}}{\left({\frac{x_1}{x_2}}\right)}^{5}+{\frac {1}{360}}{\left({\frac{x_1}{x_2}}\right)}^{6}+{\frac {1}{2520}}{\left({\frac{x_1}{x_2}}\right)}^{7}
		\right]{x_2}^{6}+\cdots 
	\end{align*}
	
	% and assume that $\tilde{y}_0$ is algebroid of degree bounded by $2$, and that there is a vanishing polynomial $\tilde{P}$  of degree bounded by $2$ and of $(\underline{x})$-adic order bounded by $1$.
	
	In this case,  note that the transform $fg$ of $\tilde{f}\tilde{g}$ under the change of variables $u_1:=\left(\frac{x_1}{x_{2}}\right)^{1/2}$, $u_2={x_2}^{1/2}$, is monomialized with respect to $(u_1,u_2)$, so that $q_1'=q_1=1$ and $(u_1,u_2)=(s,t)$. Hence, $r-\tau=\tau=1$.  Therefore, one can expand $\tilde{y}_0$ as a monomialized power series in $(s,t)$: 	$\tilde{y}_0=ty_0$ with %IL MANQUE $c2$ ???
	\begin{align*}
		{y}_0	&=& 1-{\frac {1}{2}{s}^{2}}+{\frac {3}{8}{s}^{4}}-{\frac {5}{16}{s}^{6}}+{\frac {35}{128}{s}^{8}}-{\frac {63}{256}{s}^{10}}+\cdots\\
		&&+ \left( -{\frac {1}{2}{s}^{2}}+{\frac {1}{2}{s}^{4}}-{\frac {1}{2}{s}^{6}}+{\frac {1}{2}{s}^{8}}-{\frac {1}{2}{s}^{10}}+\cdots\right)t\\
		&&+ \left({\frac {1}{8}}{s}^{4}-{\frac {3}{16}{s}^{6}}+{\frac {15}{64}{s}^{8}}-{\frac {35}{128}{s}^{10}}+\cdots \right)t^2\\
		&&+\left(-{\frac {1}{2}{s}^{4}}+{\frac {1}{2}{s}^{6}}-{\frac {1}{2}{s}^{8}}+{\frac {1}{2}{s}^{10}}+\cdots\right)t^3\\
		&&+\left( {\frac{1}{12}}+{\frac {1}{8}{s}^{2}}+{\frac {1}{32}{s}^{4}}+{\frac {47}{192} {s}^{6}}-{\frac {195}{512}{s}^{8}}+{\frac {499}{1024}{s}^{10}}+\cdots \right)t^4\\
		&& \left( -{\frac {1}{12}{s}^{2}}-{\frac {1}{12}{s}^{4}}-{\frac {1}{4}{s}^{6}}+{\frac {1}{4}{s}^{8}}-{\frac {1}{4}{s}^{10}}+\cdots \right)t^5  +\cdots\\
		&=& \displaystyle\sum_{n\in \mathbb{N} }c_{n}(s)\,\underline{t}^{{n}}\hspace{4cm} \textrm{ with } c_{0,0}=1\neq 0
	\end{align*}
	As described after (\ref{equ:ord}),  now we are in position to apply the algorithm as stated in Theorem \ref{theo:reconstr-var-u} with $\tilde{\underline{n}}^0=(0,1)$ and  $\tilde{\underline{n}}^0=(0,0)$ and \[\hat{\underline{l}}_0:= p.\kappa.\tilde{\nu}_0 + d.\rho=2\times 1\times 1+2\times1=4.\] The corresponding support of the vanishing polynomial $P$ belongs to some $\mathcal{F}\cup \mathcal{G}$ as in Definition \ref{defi:algebroid-relative} and  satisfying Conditions (i), (ii), (iii) of Lemma \ref{lemma:support-equ}, namely for any $(k,l,j)\in \mathcal{F}\cup \mathcal{G}$: 
	\begin{itemize}
		\item[(i)] ($k,l)\geq (0,j)$;
		\item[(ii)] $k$ and $l-j$ are even;
		\item[(iii)] $k\leq l-j$.
	\end{itemize}
	For the first step of the algorithm (Section \ref{sect:first-step}), the list of plausible indices to begin with are all the non-negative  integers $l\leq \hat{\underline{l}}_0=4$. We resume the notations of Section \ref{sect:first-step}  (see also the method in Section \ref{sect:reconstr-alg}). For simplicity, let us write $c_0$ for $c_0(s)$. \vspace{0.1cm}\\
	\emph{Step 1.}
	
	If $l=0$ then $j=0$ and thefore $l=k=0$, so $\mathcal{F}'_0=\emptyset$ and $\mathcal{G}'_0=\{(0,0,0)\}$. Equation (\ref{equ:recurrence-init}) translates as $a_{0,0,0}=0$, which contradicts the assumption that such an equation should be non-trivial. Hence, we exclude $l=0$ from the list of admissible indices.
	
	If $l=1$ then $j=0$ or 1. But $l-j$ has to be even, so $j=1$ and $l-j=0=k$. Thus, $\mathcal{F}'_1=\{(0,1,1)\}$ and $\mathcal{G}'_1=\emptyset$.  Equation (\ref{equ:recurrence-init}) translates as 
	\[a_{0,1,1}.s.C_0=0\Leftrightarrow a_{0,1,1}=0,\] which contradicts the assumption that such an equation should be non-trivial. Hence, we exclude $l=1$ from the list of admissible indices.
	
	If $l=2$ then $j\in\{0,1,2\}$. But  $l-j$ has to be even, so $j=0$ or 2. Since $k$ is even, in the former case, $k=0$ or 2, and in the latter case $k=0$. Thus, $\mathcal{F}'_2=\{(0,2,2)\}$ and $\mathcal{G}'_2=\{(0,2,0), (2,2,0)\}$.  Equation (\ref{equ:recurrence-init}) translates as 
	\[a_{0,2,2}.{C_0}^2+a_{0,2,0}+a_{2,2,0}.s^2=0.\]
	However, since ${c_0}^2=1-{s}^{2}+{s}^{4}-{s}^{6}+{s}^{8}-{s}^{10}+\cdots$ is not a polynomial of degree at most 2, the only possibility is $a_{0,2,2}=a_{0,2,0}=a_{2,2,0}=0$ which contradicts the assumption that such an equation should be non-trivial. Hence, we exclude $l=2$ from the list of admissible indices.
	
	If $l=3$ then $j\in\{0,1,2\}$ (recall that $\deg_y P=2\leq d=2$). But  $l-j$ has to be even, so $j=1$. Since $k$ is even, $k=0$ or 2. Thus, $\mathcal{F}'_3=\{(0,3,1), (2,3,1)\}$ and $\mathcal{G}'_3=\emptyset$.  Equation (\ref{equ:recurrence-init}) translates as 
	\[(a_{0,3,1}+a_{2,3,1}.s^2).C_0=0\Leftrightarrow a_{0,3,1}=a_{2,3,1}=0,\]
	which contradicts the assumption that such an equation should be non-trivial. Hence, we exclude $l=3$ from the list of admissible indices.
	
	If $l=4$, again since $l-j$ has to be even, we have that $j=0$ or 2. Since $k$ is even, in the former case, $k\in\{0,2,4\}$, and in the latter case $k\in\{0,2\}$. Thus, $\mathcal{F}'_4=\{(0,4,2), (2,4,2)\}$ and $\mathcal{G}'_2=\{(0,4,0), (2,4,0),(4,4,0)\}$.  Equation (\ref{equ:recurrence-init}) translates as 
	\begin{equation}\label{equ:exstep1}
		(a_{0,4,2}+a_{2,4,2}.s^2).{C_0}^2+a_{0,4,0}+a_{2,4,0}.s^2+a_{4,4,0}.s^4=0.
	\end{equation}
	Let us consider the corresponding Wilczynski matrices, where for simplicity the lines consists only of the coefficients of 1, $s^2$, $s^4$, etc.
	\[ M_{\mathcal{F}'_4,\mathcal{G}'_4} :=\left[\begin{array}{ccccc}1 & 0 & 0 & 1& 0 \\
		0 & 1 & 0 & -1& 1 \\
		0 & 0 & 1 & 1& -1 \\
		0 & 0 & 0 & -1& 1 \\
		0 & 0 & 0 & 1& -1 \\
		0 & 0 & 0 & -1& 1 \\
		\vdots & \vdots & \vdots & \vdots &\vdots \\
	\end{array}\right]\hspace{0.4cm}\textrm{ and }\hspace{0.4cm}
	M_{\mathcal{F}'_4,\mathcal{G}'_4}^{\textrm{red}} :=\left[\begin{array}{cc}
		-1& 1 \\
		1& -1 \\
		-1& 1 \\
		1& -1 \\
		-1& 1 \\
		\vdots &\vdots \\
	\end{array}\right]\]
	(Recall that here the reduced matrix is obtained by removing the 3 first rows and columns.) One can easily check that all the minors of maximal order vanish up to order $2d_sd=2\times 4\times 2=16$: as expected, $c_0$ is algebraic relatively to $(\mathcal{F}'_4,\mathcal{G}'_4)$. Moreover, a first non-zero minor of order 1 in $M_{\mathcal{F}'_4,\mathcal{G}'_4}^{\textrm{red}}$ is obtained e.g. with the coefficient 1 of the second column (this is the coefficient of $s^6$ in the expansion of $s^2.{c_0}^2$). Using the Cramer's rule, we identify it, up to a multiplicative constant $\lambda\in K$, with $a_{2,4,2}$, and we also get $a_{0,4,2}=\lambda$. According to (\ref{equ:terme-cst1}), we derive $a_{0,4,0}=-\lambda$ and $a_{2,4,0}=a_{4,4,0}=0$. 
	
	As a conclusion, the $K$-vector space $E_4$ of polynomials corresponding to Equation (\ref{equ:recurrence-init}) is 
	\[E_4:=\left\{ \lambda\left[(1+s^2)y^2-1\right]t^4+R(s,t,y)\ |\  \lambda\in K,\, R\in \left(K[s][[t]][y]\right)_{\mathcal{F},\mathcal{G}},\, w_t(R)\geq 5\right\}.\]
	Here, the linear form $\tilde{L}$ of Theorem \ref{theo:reconstr-optimise} is given by:
	\[\tilde{L}(n_1,n_2)=1n_1+n_2=n_1+n_2.\]
	We go back to the variables $(x_1,x_2)$ by the following transformation:
	\[Q(s, t,y)=\tilde{Q}(s^2t^2,t^2,ty).\]
	The space $E_4$ corresponds to the  space of polynomials in $K[[x_1,x_2]][y]$ of the form:
	\[\lambda \left[(x_1+x_2)y^2-{x_2}^2\right]+ \tilde{R}(x_1,x_2,y)\]
	\textrm{ with } $ \lambda\in K,\, \tilde{R}\in K[[x_1,x_2]][y]$ such that:
	\[\tilde{R}=\tilde{a}_0+\tilde{a}_1y+\tilde{a}_2y^2\]
	with $\ord_{\underline{x}}(\tilde{a}_0)\geq 3,\ \ord_{\underline{x}}(\tilde{a}_1)\geq 2$ and $\ord_{\underline{x}}(\tilde{a}_2)\geq 2$.
	\vspace{0.1cm}\\
	\emph{Step 2.} 
	
	Here, there isn't any $l'>4$ as in (\ref{equ:ineq}).
	\vspace{0.1cm}\\
	\emph{Step 3.} 
	
	We consider the case where $l+1=5$ corresponding to  Third Step \ref{sect:second-step1}. By applying Conditions (i), (ii), (iii) of Lemma \ref{lemma:support-equ} as before, we obtain:
	\[\mathcal{F}'_5=\left\{(0,5,1),(2,5,1),(4,5,1) \right\} \textrm{ and } \mathcal{G}'_5=\emptyset. \]
	The instance of (\ref{equ:recurrence_step2bis}) is:
	\begin{equation}\label{equ:ex3rd}
		\begin{array}{lcl}
			(a_{0,5,1}+a_{2,5,1}.s^2++a_{4,5,1}.s^4).{C_0}&=& -\left( a_{0,4,2}+a_{2,4,2}.s^2\right)2C_0C_1\\
			&=& -\lambda(1+s^2) 2C_0C_1.
		\end{array}
	\end{equation}
	Here, $c_1\neq 0$, and  $c_0$ is not algebraic relatively to $(\mathcal{F}'_5,\mathcal{G}'_5)$ since $\mathcal{G}'_5=\emptyset$, so we are in the case $\star\star$ of Third Step \ref{sect:second-step1}. Note that $\theta_{s,1}=(4+2)a+b$ with $a=1$, $b=0$ (see Lemma \ref{lemme:control-deg-supp}), so $\theta_{s,1}=6$.  According to Lemma \ref{lemma:depth-non-homog_bis},  we are assured to find a non zero reconstruction minor at depth at most $2.3.{\theta_{\underline{s},(0,\ldots,0,1)} }{d}^{d+1}=2\times 3\times 6\times 2^3=288$. However, here, the Wilczynski matrices (where again for simplicity we only consider the lines consisting of the coefficients of 1, $s^2$, $s^4$, etc.) are triangular with non zero diagonal coefficients:
	\[ M_{\mathcal{F}'_5,\mathcal{G}'_5}=M_{\mathcal{F}'_5,\mathcal{G}'_5}^{\textrm{red}} =\left[\begin{array}{ccc}
		1 & 0 & 0 \\
		-1/2 & 1& 0 \\
		3/8 & -1/2& 1 \\
		-5/16 & 3/8& -1/2 \\
		35/128 & -5/16& 3/8 \\
		\vdots & \vdots &\vdots \\
	\end{array}\right]. \]
	A first nonzero minor is obtained with the three first lines, and is equal to 1. But we notice that, here, Equation (\ref{equ:ex3rd}) can be simplified by $C_0$ (since $c_0\neq 0$) and we get:
	\[ a_{0,5,1}+a_{2,5,1}.s^2++a_{4,5,1}.s^4=  -\lambda(1+s^2) 2 C_1.\]
	By evaluating at $c_1=-{\frac {1}{2}{s}^{2}}+{\frac {1}{2}{s}^{4}}-{\frac {1}{2}{s}^{6}}+{\frac {1}{2}{s}^{8}}-{\frac {1}{2}{s}^{10}}+\cdots$, we see that:
	\[-\lambda(1+s^2) 2 c_1=\lambda s^2\]
	and therefore $a_{0,5,1}= a_{4,5,1}=0$ and $a_{2,5,1}=\lambda$. As a conclusion, the $K$-vector space $E_{4,5}$ of polynomials corresponding to  Third Step \ref{sect:second-step1} is 
	\begin{center}
		 	$E_{4,5}:=$\\
		 	$\left\{ \lambda\left[(1+s^2)y^2-1\right]t^4+(\lambda s^2 y)\, t^5+R(s,t,y)\ |\  \lambda\in K,\, R\in \left(K[s][[t]][y]\right)_{\mathcal{F},\mathcal{G}},\, w_t(R)\geq 6\right\}.$
	\end{center}

	As before, we go back to the variables $(x_1,x_2)$ by the following transformation:
	\[Q(s, t,y)=\tilde{Q}(s^2t^2,t^2,ty).\]
	The space $E_{4,5}$ corresponds to the  space of polynomials in $K[[x_1,x_2]][y]$ of the form:
	\[\lambda \left[(x_1+x_2)y^2+ x_1x_2 y-{x_2}^2\right]+ \tilde{R}(x_1,x_2,y)\]
	\textrm{ with } $ \lambda\in K,\, \tilde{R}\in K[[x_1,x_2]][y]$ such that:
	\[\tilde{R}=\tilde{a}_0+\tilde{a}_1y+\tilde{a}_2y^2\]
	with $\ord_{\underline{x}}(\tilde{a}_0)\geq 3,\ \ord_{\underline{x}}(\tilde{a}_1)\geq 3$ and $\ord_{\underline{x}}(\tilde{a}_2)\geq 2$.
	\vspace{0.1cm}\\
	\emph{Step 4.}

	We consider the case where $S(\tilde{l})=6$ corresponding to  Induction Step \ref{sect:induction-step}.  By applying Conditions (i), (ii), (iii) of Lemma \ref{lemma:support-equ} as before, we obtain:
	\[\mathcal{F}'_6=\left\{(0,6,2),(2,6,2),(4,6,2) \right\} \textrm{ and } \mathcal{G}'_6=\left\{(0,6,0),(2,6,0),(4,6,0),(6,6,0) \right\}. \]
	The instance of (\ref{equ:recurrence_step2bis}) is:
	\begin{equation}\label{equ:ex4th}
		\begin{array}{l}
			(a_{0,6,2}+a_{2,6,2}.s^2++a_{4,6,2}.s^4).{C_0}^2+	a_{0,6,0}+a_{2,6,0}.s^2+a_{4,6,0}.s^4+a_{6,6,0}.s^6\\ \hspace{0.4cm}=-\left(( a_{0,4,2}+a_{2,4,2}.s^2)(2C_0C_2+ {C_1}^2) +(a_{0,5,1}+a_{2,5,1}.s^2++a_{4,5,1}.s^4).{C_1}\right)\\
			\hspace{0.4cm}	= -\lambda\left[(1+s^2) (2C_0C_2+ {C_1}^2) + s^2 C_1\right].
		\end{array}
	\end{equation}
	Note that we are in the case $\star \bullet$ of  Induction Step \ref{sect:induction-step} since $c_0$ is algebraic relatively to $(\mathcal{F}'_6,\mathcal{G}'_6)$. Moreover, when evaluating at $c_0$, $c_1$ and $c_2={\frac {1}{8}}{s}^{4}-{\frac {3}{16}{s}^{6}}+{\frac {15}{64}{s}^{8}}-{\frac {35}{128}{s}^{10}}+\cdots$, we obtain that the right-hand side of (\ref{equ:ex4th}) vanishes. So we get:
	\[	(a_{0,6,2}+a_{2,6,2}.s^2++a_{4,6,2}.s^4).{C_0}^2+	a_{0,6,0}+a_{2,6,0}.s^2+a_{4,6,0}.s^4+a_{6,6,0}.s^6=0 \]  
	which is of the same type as (\ref{equ:exstep1}). The corresponding Wilczynski matrices (where again for simplicity the lines consists only of the coefficients of 1, $s^2$, $s^4$, etc.) are
	\[ M_{\mathcal{F}'_6,\mathcal{G}'_6} :=\left[\begin{array}{ccccccc}1&0 & 0 & 0 & 1& 0&0 \\
		0 & 1 & 0 &0  & -1& 1&0 \\
		0 & 0 & 1&0 & 1& -1&1 \\
		0 & 0 & 0&1 & -1& 1&-1 \\
		0 & 0 & 0&0 & 1& -1&1 \\
		0 & 0 & 0&0 & -1& 1&-1 \\
		\vdots & \vdots & \vdots & \vdots &\vdots&\vdots&\vdots \\
	\end{array}\right]\hspace{0.4cm}\textrm{ and }\hspace{0.4cm}
	M_{\mathcal{F}'_6,\mathcal{G}'_6}^{\textrm{red}} :=\left[\begin{array}{ccc}
		-1& 1 &-1\\
		1& -1 &1\\
		-1& 1 &-1\\
		1& -1 &1\\
		-1& 1 &-1 \\
		\vdots &\vdots&\vdots \\
	\end{array}\right]\]
	We apply the reconstruction method of Section \ref{sect:reconstr-alg} with maximal subfamily $\mathcal{F}''_6=\{(2,6,2)\}$. According to Lemma \ref{lemma:total-reconstr}, we obtain:
	\[a_{2,6,2}= a_{0,6,2}  \lambda_{2,6,2}^{0,6,2}+ a_{4,6,2}  \lambda_{2,6,2}^{4,6,2}\]
	where here $\lambda_{2,6,2}^{0,6,2}=-1$ is the coefficient relating the column $(0,6,2)$ to the column $(2,6,2)$. Likewise,  $  \lambda_{2,6,2}^{4,6,2}=-1$. Let us consider $a_{0,6,2} $ and $a_{4,6,2} $ as parameters $\alpha,\beta\in K$, so $a_{2,6,2}=-\alpha-\beta$. Moreover, we compute the coefficients of $\mathcal{G}'_6$ according to (\ref{equ:terme-cst1}) in Lemma \ref{lemma:total-reconstr}:
	\[\begin{array}{lclcl}
		a_{0,6,0}&=&-a_{0,6,2}. 1&=&-\alpha  \\
		a_{2,6,0}&=&a_{0,6,2}. 1 -a_{2,6,2}.1&=&2\alpha+\beta \\
		a_{4,6,0}&=&- a_{0,6,2}. 1 +a_{2,6,2}.1 -a_{4,6,2}.1&=&-2\alpha-2\beta \\
		a_{6,6,0}&=& a_{0,6,2}. 1 -a_{2,6,2}.1 +a_{4,6,2}.1&=&2\alpha+2\beta 
	\end{array}
	\]
	As a conclusion, the $K$-vector space $E_{4,6}$ of polynomials corresponding to Induction Step \ref{sect:induction-step} is 
	\[\begin{array}{l}
		E_{4,6}:=\left\{ \lambda\left[(1+s^2)y^2-1\right]t^4+(\lambda s^2 y)\, t^5 +\right.\\ 
		\left[ \left(\alpha - (\alpha +\beta )s^2 +\beta s^4\right) y^2 -\alpha +(2\alpha+\beta)s^2- 2(\alpha +\beta)s^4+  2(\alpha +\beta)s^6   \right]t^6 
		+R(s,t,y)\ |\  \\ 
		\hspace{5cm}\left.\lambda,\alpha,\beta\in K,\, R\in \left(K[s][[t]][y]\right)_{\mathcal{F},\mathcal{G}},\, w_t(R)\geq 7\right\}.
	\end{array}
	\]
	As before, we go back to the variables $(x_1,x_2)$ by the following transformation:
	\[Q(s, t,y)=\tilde{Q}(s^2t^2,t^2,ty).\]
	The space $E_{4,6}$ corresponds to the  space of polynomials in $K[[x_1,x_2]][y]$ of the form:
	\[ \begin{array}{l}
		(\lambda x_1+\lambda x_2+ \alpha {x_2}^2- (\alpha +\beta )x_1x_2 +\beta x_1^2)y^2+ \lambda x_1x_2 y \\
		\hspace{2cm} -\lambda{x_2}^2 -\alpha {x_2}^3 +(2\alpha+\beta)x_1{x_2}^2- 2(\alpha +\beta){x_1}^2x_2+  2(\alpha +\beta){x_1}^3  
		+ \tilde{R}(x_1,x_2,y)
	\end{array} \]
	\textrm{ with } $ \lambda,\alpha,\beta \in K,\, \tilde{R}\in K[[x_1,x_2]][y]$ such that:
	\[\tilde{R}=\tilde{a}_0+\tilde{a}_1y+\tilde{a}_2y^2\]
	with $\ord_{\underline{x}}(\tilde{a}_0)\geq 4,\ \ord_{\underline{x}}(\tilde{a}_1)\geq 3$ and $\ord_{\underline{x}}(\tilde{a}_2)\geq 3$.
	
	Note that we recover the beginning of the analytic expansion of $\tilde{P}$ at $\underline{0}$ in (\ref{eq:exemple}) for $\lambda=1$ and $\alpha=\beta=0$.
	
\end{ex}

\section{A generalization of the Flajolet-Soria Formula.}\label{section:hensel}
In the monovariate context, let $Q(x,y)=\displaystyle\sum_{i,j}a_{i,j}x^iy^j\,\in K[x,y]$ with $Q(0,0)=\displaystyle\frac{\partial Q}{\partial y}(0,0)=0$ and $Q(x,0)\neq 0$. 
In \cite{flajolet-soria:coeff-alg-series}, P. Flajolet and M. Soria give the following formula for the coefficients of the unique formal solution $y_0=\displaystyle\sum_{n\geq 1}c_nx^n$ of the implicit equation $y=Q(x,y)$: 

\begin{theo}[Flajolet-Soria's Formula \cite{flajolet-soria:coeff-alg-series}]\label{theo:formule-FS0}
	$$ c_n=\displaystyle\sum_{m=1}^{2n-1}\frac{1}{m}\displaystyle\sum_{|\underline{k}|=m,\ ||\underline{k}||=m-1,\ g(\underline{k})=n}\frac{m!}{\prod_{i,j}k_{i,j}!}\prod_{i,j}a_{i,j}^{k_{i,j}},$$
	where $\underline{k}=\displaystyle(k_{i,j})_{i,j}$, $\ |\underline{k}|=\displaystyle\sum_{i,j}k_{i,j}$,  $\ ||\underline{k}|| = \displaystyle\sum_{i,j}j\, k_{i,j}$ and $\ g(\underline{k}) = \displaystyle\sum_{i,j}i\, k_{i,j}$.
\end{theo}

Note that in the particular case where the coefficients of $Q$ verify $a_{0,j}=0$ for all $j$, one has $m\leq n$ in the summation. 

One can derive immediately from  Theorems 3.5 and 3.6 in \cite{sokal:implicit-function}   a multivariate version of the Flajolet-Soria Formula in the case where $Q(\underline{x},y)\in K\left[\underline{x},y\right]$. The purpose of the present section is to generalize the latter result to the case where $Q\left(\underline{x},y\right)\in K\left(\left(u_1^{\mathbb{Z}},\ldots,u_r^\mathbb{Z}\right)\right)^{\textrm{grlex}}_{\textrm{Mod}}\left[y\right]$.\\

%\begin{remark}\label{rem:FS} A MODIFIER, INTRODUIRE LA FORMULE DE F.S. VERSION GENERALE ET SIMPLIFIEE
%Let us consider the particular case where the coefficients of $Q$ verify $a_{\underline{0},j}=0$ for all $j$.  So, for any $\underline{k}$ such that $|\underline{k}|=m$ and $\displaystyle\prod_{\underline{i},j}a_{\underline{i},j}^{k_{\underline{i},j}}\neq 0$, we have that $\|\underline{k}\|_1\geq_{\textrm{grlex}} (0,\ldots,0,m)$. Thus, to have $\|\underline{k}\|_1=\underline{n}$, one needs to have $(0,\ldots,0,m)\leq_{\textrm{grlex}} \underline{n}$, hence $m\leq |\underline{n}|$. Flajolet-Soria's Formula can be written:
%$$ c_{\underline{n}}=\displaystyle\sum_{m=1}^{|\underline{n}|}\frac{1}{m}\displaystyle\sum_{|\underline{k}|=m,\ ||\underline{k}||_1=\underline{n},\ ||\underline{k}||=m-1}\frac{m!}{\prod_{\underline{i},j}k_{\underline{i},j}!}\prod_{\underline{i},j} a_{\underline{i},j}^{k_{\underline{i},j}}.$$
%\end{remark}

We will need a special version of Hensel's Lemma for multivariate power series elements of  $K((x_1^{\mathbb{Z}},\ldots,x_r^\mathbb{Z}))^{\textrm{grlex}}$. Recall that the latter denotes the field of generalized series $\left(K\left(\left(X^{\mathbb{Z}^r}\right)\right)^{\mathrm{grlex}},\, w\right)$  where $w$ is the graded lexicographic valuation as described in Section \ref{section:preliminaries}. Generalized series fields are known to be Henselian \cite[Theorem 4.1.3 and Remark 4.1.8]{engler-prestel:valued-fields}. For the convenience of the reader, we give a short proof in our particular context.

\begin{defi}\label{defi:equ-hensel-red}
	We call \textbf{strongly reduced Henselian equation} any equation of the following type:
	$$y=F(\underline{u},y)\ \textrm{ with }\ F(\underline{u},y)\in K\left(\left(u_1^{\mathbb{Z}},\ldots,u_r^\mathbb{Z}\right)\right)^{\textrm{grlex}}_{\textrm{Mod}},$$
	such that $w\left(F(\underline{u},y)\right)>_{\textrm{grlex}}\underline{0}$ and $F(\underline{u},0)\nequiv 0$.
\end{defi}

\begin{theo}[Hensel's lemma]\label{theo:hensel}
	Any strongly reduced Henselian equation admits a unique solution $y_0= \displaystyle\sum_{\underline{n}>_{\textrm{grlex}}\underline{0}}c_{\underline{n}}\underline{u}^{\underline{n}}\in  K((u_1^{\mathbb{Z}},\ldots,u_r^\mathbb{Z}))^{\mathrm{grlex}}$. 
\end{theo}
\begin{demo} Let 
	\begin{equation}\label{equ:hensel}y=F\left(\underline{u},y\right)
	\end{equation}
	be a strongly reduced Henselian equation and let  $y_0=\displaystyle\sum_{\underline{n}>_{\textrm{grlex}}\underline{0}}c_{\underline{n}}\underline{u}^{\underline{n}} \in K((u_1^{\mathbb{Z}},\ldots,u_r^{\mathbb{Z}}))^{\textrm{grlex}}$. For $\underline{n}\in\mathbb{Z}^r$, $\underline{n}>_{\textrm{grlex}}0$, let us denote $\tilde{z}_{\underline{n}}:= \displaystyle\sum_{\underline{m}<_{\textrm{grlex}}\underline{n}} c_{\underline{m}}\underline{u}^{\underline{m}}$.
	We get started with the following key lemma:
	\begin{lemma}\label{lemma:hensel} The following are equivalent:
		\begin{enumerate}
			\item a series $y_0$ is a solution of (\ref{equ:hensel});
			\item for any $\underline{n}\in\mathbb{Z}^r$, $\underline{n}>_{\mathrm{grlex}}\underline{0}$, 
			$$ w\left(\tilde{z}_{\underline{n}}-F\left(\underline{u},\tilde{z}_{\underline{n}}\right)\right)=w\left(y_0-\tilde{z}_{\underline{n}}\right);$$
			\item for any $\underline{n}\in\mathbb{Z}^r$, $\underline{n}>_{\textrm{grlex}}\underline{0}$, 
			$$ w\left(\tilde{z}_{\underline{n}}-F\left(\underline{u},\tilde{z}_{\underline{n}}\right)\right)\geq_{\mathrm{grlex}}\underline{n}.$$
		\end{enumerate}
	\end{lemma}
	\begin{demo}
		For $\underline{n}>_{\textrm{grlex}}\underline{0}$, let us denote $\tilde{y}_{\underline{n}}:=y_0-\tilde{z}_{\underline{n}}=\displaystyle\sum_{\underline{m}\geq_{\textrm{grlex}}\underline{n}} c_{\underline{m}}\underline{u}^{\underline{m}}$. We apply Taylor's Formula to $G(\underline{u},y):=y-F(\underline{u},y)$ at $\tilde{z}_{\underline{n}}$:
		$$G\left(\underline{u},\tilde{z}_{\underline{n}}+y\right) =\tilde{z}_{\underline{n}}-F\left(\underline{u},\tilde{z}_{\underline{n}}\right)+\left(1-\displaystyle\frac{\partial F}{\partial y}\left(\underline{u},\tilde{z}_{\underline{n}}\right)\right)y +y^2H\left(\underline{u},y\right),$$
		where $H\left(\underline{u},y\right)\in K((u_1^{\mathbb{Z}},\ldots,u_r^{\mathbb{Z}}))^{\textrm{grlex}}[y]$ with $w\left(R(\underline{u},y)\right)>_{\textrm{grlex}}\underline{0}$. The series
		$y_0$ is a solution of (\ref{equ:hensel}) iff for any $\underline{n}$, $\tilde{y}_{\underline{n}}$ is a root of $G\left(\underline{u},\tilde{z}_{\underline{n}}+y\right)=0$, i.e.:
		\begin{equation}\label{equ:henselian-trunc}
			\tilde{z}_{\underline{n}}-F\left(\underline{u},\tilde{z}_{\underline{n}}\right)+\left(1-\displaystyle\frac{\partial F}{\partial y}\left(\underline{u},\tilde{z}_{\underline{n}}\right)\right)\tilde{y}_{\underline{n}}+ \tilde{y}_{\underline{n}}^2H\left(\underline{u},\tilde{y}_{\underline{n}}\right)=0.
		\end{equation}  Now consider $y_0$ a solution of (\ref{equ:hensel}) and $\underline{n}\in\mathbb{Z}^r$, $\underline{n}>_{\mathrm{grlex}}\underline{0}$. Either $\tilde{y}_{\underline{n}}=0$, i.e. $y_0=\tilde{z}_{\underline{n}}$: (2) holds trivially. Or  $\tilde{y}_{\underline{n}}\neq 0$, so we have:
		$$\underline{n}\leq_{\textrm{grlex}} w\left(\left(1-\displaystyle\frac{\partial G}{\partial y}\left(\underline{u},\tilde{z}_{\underline{n}}\right)\right)\tilde{y}_{\underline{n}}\right) =w\left(\tilde{y}_{\underline{n}}\right)<_{\textrm{grlex}} 2w\left(\tilde{y}_{\underline{n}}\right)<_{\textrm{grlex}} w\left(\tilde{y}_{\underline{n}}^2H\left(\underline{u},\tilde{y}_{\underline{n}}\right)\right).$$
		So we must have $w\left(\tilde{z}_{\underline{n}}-G\left(\underline{u},\tilde{z}_{\underline{n}}\right)\right)=w\left(\tilde{y}_{\underline{n}}\right)$. \\
		Now, $(2)\,\Rightarrow\, (3)$ since $w\left(\tilde{y}_{\underline{n}}\right)\geq_{\textrm{grlex}}\underline{n}$.\\
		Finally, suppose that for any $\underline{n}$,  $w\left(\tilde{z}_{\underline{n}}-F\left(\underline{u},\tilde{z}_{\underline{n}}\right)\right)\geq_{\textrm{grlex}}\underline{n}$. If $y_0-F\left(\underline{x},y_0\right)\neq 0$, denote  $\underline{n}_0:= w\left(y_0-F\left(\underline{u},y_0\right)\right)$. For $\underline{n}>_{\textrm{grlex}}\underline{n}_0$, one has $$\underline{n}_0=w\left(\tilde{z}_{\underline{n}}-F\left(\underline{u},\tilde{z}_{\underline{n}}\right)\right)\geq_{\textrm{grlex}}\underline{n}.$$ A contradiction.
	\end{demo}
	Let us return to the proof of Theorem \ref{theo:hensel}. Note that, if $y_0$ is a solution of (\ref{equ:hensel}), then its support needs to be included in the monoid $\mathcal{S}$ generated by the $\underline{i}$'s from the nonzero coefficients $a_{\underline{i},j}$ of $F(\underline{x},y)$. If not, consider the smallest index $\underline{n}$ for $\leq_{\mathrm{grlex}}$ which is not in  $\mathcal{S}$. Property (2) of Lemma \ref{lemma:hensel} gives a contradiction for this index.
	$\mathcal{S}$ is a well-ordered subset of $(\mathbb{Z}^r)_{\geq_\mathrm{grlex}\underline{0}}$ by \cite[Theorem 3.4]{neumann:ord-div-rings}.
	%the set $\mathcal{S}$ is well-ordered in $\mathbb{Z}^r$. 
	Let us prove  by transfinite induction on $\underline{n}\in \mathcal{S}$  the existence and uniqueness of a sequence of series  $\tilde{z}_{\underline{n}}$ as in the statement of the previous lemma.  Suppose that for some $\underline{n}\in\mathcal{S}$,  we are given a series $\tilde{z}_{\underline{n}}$ with support included in $\mathcal{S}$ and $<_{\textrm{grlex}}\underline{n}$, such that $w\left(\tilde{z}_{\underline{n}}-F\left(\underline{u},\tilde{z}_{\underline{n}}\right)\right)\geq_{\textrm{grlex}}\underline{n}$. Then by Taylor's formula as in the proof of the previous lemma, denoting by $\underline{m}$ the successor of $\underline{n}$ in $\mathcal{S}$ for $\leq_{\textrm{grlex}}$:
	$$G\left(\underline{u},\tilde{z}_{\underline{m}}\right)=G\left(\underline{u},\tilde{z}_{\underline{n}}+c_{\underline{n}}\underline{u}^{\underline{n}}\right) =\tilde{z}_{\underline{n}}-F\left(\underline{u},\tilde{z}_{\underline{n}}\right)+\left(1-\displaystyle\frac{\partial F}{\partial y}\left(\underline{u},\tilde{z}_{\underline{n}}\right)\right)c_{\underline{n}}\underline{u}^{\underline{n}} +c_{\underline{n}}^2\underline{u}^{2\underline{n}}H\left(\underline{u},\tilde{z}_{\underline{n}}\right).$$
	Note that $w\left(H\left(\underline{u},\tilde{z}_{\underline{n}}\right)\right)\geq_{\textrm{grlex}}\underline{0}$ since $w(\tilde{z}_{\underline{n}})>_{\textrm{grlex}}\underline{0}$ and $w\left(F\left(\underline{u},y\right)\right)>_{\textrm{grlex}}\underline{0}$.
	Therefore, one has:
	$$ w\left(G\left(\underline{u},\tilde{z}_{\underline{m}}\right)\right)=w\left(\tilde{z}_{\underline{m}}-F\left(\underline{u},\tilde{z}_{\underline{m}}\right)\right)\geq_{\textrm{grlex}}\underline{m}>_{\textrm{grlex}}\underline{n}$$
	if and only if $c_{\underline{n}}$ is equal to the coefficient of $\underline{u}^{\underline{n}}$ in $F\left(\underline{u},\tilde{z}_{\underline{n}}\right)$. This determines $\tilde{z}_{\underline{m}}$ in a unique way as desired.
\end{demo}

We prove now our generalized version of the Flajolet-Soria Formula \cite{flajolet-soria:coeff-alg-series}. Our proof, as the one in \cite{sokal:implicit-function}, uses the classical Lagrange Inversion Formula in one variable. We will use Notation \ref{nota:FS}.

\begin{theo}[Generalized multivariate Flajolet-Soria Formula]\label{theo:formule-FS}\indent \\
	Let $y=F\left(\underline{u},y\right)=\displaystyle\sum_{\underline{i},j}a_{\underline{i},j} \underline{u}^{\underline{i}}y^j$ be a strongly reduced Henselian equation.  Define $\underline{\iota}_0=(\iota_{0,1},\ldots,\iota_{0,r})$ by: $$-\iota_{0,k}:=\min\left\{0,\, i_k\, /\, a_{\underline{i},j}\neq 0,\, \underline{i} = (i_1,\ldots,i_k,\ldots,i_r)\right\},\ \ \ k=1,\ldots,r.$$
	Then the  coefficients $c_{\underline{n}}$ of the unique solution $y_0=\displaystyle\sum_{\underline{n}>_{\mathrm{grlex}}\underline{0} } c_{\underline{n}}\underline{u}^{\underline{n}}\in K((u_1^{\mathbb{Z}},\ldots,u_r^\mathbb{Z}))^{\mathrm{grlex}}$ are given by:
	\begin{equation}\label{eq:flajo-so}
		c_{\underline{n}}=\displaystyle\sum_{m=1}^{\mu_{\underline{n}}}\frac{1}{m}\displaystyle\sum_{|\underline{M}|=m,\ ||\underline{M}||=m-1,\ g(\underline{M})=\underline{n}}\frac{m!}{\underline{M}!}\underline{A}^{\underline{M}}
	\end{equation}
	where $\mu_{\underline{n}}$ is the greatest integer $m$ such that there exists an $\underline{M}$ with $|\underline{M}|=m,\ ||\underline{M}||=m-1$ and $g(\underline{M})=\underline{n}$. Moreover, for  $\underline{n}=(n_1,\ldots,n_r)$,  $\mu_{\underline{n}}\leq \displaystyle\sum_{k=1}^r\lambda_k\, n_k$ with:
	$$\lambda_k=\left\{ \begin{array}{ll}
		\displaystyle\prod_{j=k+1}^{r-1}(1+\iota_{0,j})+\displaystyle\prod_{j=1}^{r-1}(1+\iota_{0,j}) & \textrm{ if } k<r-1; \\
		1+\displaystyle\prod_{j=1}^{r-1}(1+\iota_{0,j})& \textrm{ if } k=r-1;\\
		\displaystyle\prod_{j=1}^{r-1}(1+\iota_{0,j})& \textrm{ if } k=r.\end{array}  \right.$$
\end{theo}
\begin{remark}\label{rem:entier}
	\begin{enumerate}
		\item In (\ref{eq:flajo-so}), note that the second sum is finite. Indeed, let $\underline{M}=(m_{\underline{i},j})$ be such that $|\underline{M}|=m,\ ||\underline{M}||=m-1,\ g(\underline{M})=\underline{n}$. Since $F\in	K\left(\left(u_1^{\mathbb{Z}},\ldots,u_r^\mathbb{Z}\right)\right)^{\textrm{grlex}}_{\textrm{Mod}}[y]$, if $\underline{i}$ has a component negative enough, then $a_{\underline{i},j}=0$. On the other hand, since $|\underline{M}|=m$ and $g(\underline{M})=\underline{n}$, the positive components of $\underline{i}$ are bounded.
		%		K\left(\left(u_1^{\mathbb{Z}},\ldots,u_r^\mathbb{Z}\right)\right)^{\textrm{grlex}}_{\textrm{Mod}}
		
		\item By \cite[Lemma 2.6]{hickel-matu:puiseux-alg-multivar}, $\displaystyle\frac{1}{m}\cdot \displaystyle\frac{m!}{\underline{M}!}\in\mathbb{N}$. If we set $m_j:=\displaystyle\sum_{\underline{i}}m_{\underline{i },j}$ and $\underline{N}=(m_j)_j$, then $|\underline{N}|=m$, $\|\underline{N}\|=m-1$ and:
		$$\displaystyle\frac{1}{m}\cdot \displaystyle\frac{m!}{\underline{M}!}= \displaystyle\frac{1}{m}\cdot \displaystyle\frac{m!}{\underline{N}!} \cdot \displaystyle\frac{N!}{\underline{M}!},$$
		where $\displaystyle\frac{N!}{\underline{M}!}$ is a product of multinomial coefficients and $\displaystyle\frac{1}{m}\cdot \displaystyle\frac{m!}{\underline{N}!}$ is an integer again by  \cite[Lemma 2.6]{hickel-matu:puiseux-alg-multivar}. 
		Thus, each $c_n$ is the evaluation at the $a_{i,j}$'s of a polynomial with coefficients in $\mathbb{Z}$.
	\end{enumerate}
	
\end{remark}
\begin{demo}
	For a given strongly reduced Henselian equation $y=F(\underline{u},y)$, one can expand:
	$$f\left(\underline{u},y\right):=\displaystyle\frac{y}{F\left(\underline{u},y\right)}=\displaystyle\sum_{n\geq 1}b_n(\underline{u})y^n\,\in K((u_1^{\mathbb{Z}},\ldots,u_r^\mathbb{Z}))^{\mathrm{grlex}}[[y]]\ \mathrm{with}\ b_1\neq 0,$$
	which admits a unique formal inverse in $K((u_1^{\mathbb{Z}},\ldots,u_r^\mathbb{Z}))^{\mathrm{grlex}}[[y]]$:
	$$\tilde{f}\left(\underline{u},y\right)= \displaystyle\sum_{m\geq  1}d_m(\underline{u}) y^m.$$
	%let us consider the following associated equation with unknown $y$ and variable $t$ over $K[\underline{u}]$:
	%\begin{equation}\label{equ:FS}
	%y=tQ\left(\underline{u},y\right)\ \Leftrightarrow \ t=
	%\end{equation}
	The Lagrange Inversion Theorem (see e.g. \cite[Theorem 2]{henrici:lagr-burmann} with $\mathcal{F}=K((u_1^{\mathbb{Z}},\ldots,u_r^\mathbb{Z}))^{\mathrm{grlex}}$ and $P=f(\underline{u},y)$) applies: for any $m$, $d_m(\underline{u})$ is equal to the coefficient of $y^{m-1}$ in  $\left[F\left(\underline{u},y\right)\right]^m$, divided by $m$. Hence, according to the multinomial expansion of $\left[F\left(\underline{u},y\right)\right]^m=\left[\displaystyle\sum_{\underline{i},j}a_{\underline{i},j} \underline{u}^{\underline{i}}y^j\right]^m$:
	$$d_m(\underline{u})=\displaystyle\frac{1}{m}\displaystyle\sum_{|\underline{M}|=m,\  ||\underline{M}||=m-1}\frac{m!}{\underline{M}!}\underline{A}^{\underline{M}}\underline{u}^{g(\underline{M})}.$$
	Note that the powers $\underline{n}$ of $\underline{u}$ that  appear in $d_m$ are nonzero elements of the monoid generated by the exponents $\underline{i} $ of the monomials $\underline{u}^{\underline{i}}y^j$ appearing  in $F\left(\underline{u},y\right)$, so they are $>_{\mathrm{grlex}}0$.
	%In particular, $\underline{n}>_{\textrm{grlex}}\underline{0}$ and 
	Now, it will suffice to show that, for any fixed $\underline{n}$,  the number $\displaystyle\sum_{k=1}^r\lambda_k\, n_k$ is indeed a bound for the number $\mu_{\underline{n}}$ of $m$'s for which $d_m$ can contribute to the coefficient of $\underline{u}^{\underline{n}}$. Indeed, this will show that $\tilde{f}\left(\underline{u},y\right)\in K[y]((u_1^{\mathbb{Z}},\ldots,u_r^\mathbb{Z}))^{\textrm{grlex}}$. But, by definition of $\tilde{f}$, one has that:
	$$\tilde{f}\left(\underline{u},y\right)=y\,F\left(\underline{u},\tilde{f}\left(\underline{u},y\right)\right)\,\in K((u_1^{\mathbb{Z}},\ldots,u_r^\mathbb{Z}))^{\textrm{grlex}}[[y]].$$
	Hence, both members of this equality are in fact in $ K[y]((u_1^{\mathbb{Z}},\ldots,u_r^\mathbb{Z}))^{\textrm{grlex}}$.
	So, for $y=1$, we get that $\tilde{f}\left(\underline{u},1\right)\in K((u_1^{\mathbb{Z}},\ldots,u_r^\mathbb{Z}))^{\textrm{grlex}}$ is a solution with $w\left(\tilde{f}\left(\underline{u},1\right)\right)>_{\mathrm{grlex}}\underline{0}$ of the equation: $$f(\underline{u},y)=\displaystyle\frac{y}{F\left(\underline{u},y\right)}=1\, \Leftrightarrow\, y=F(\underline{u},y).$$
	It is equal to the unique solution $y_0$ of Theorem \ref{theo:hensel}: 
	$$y_0=\tilde{f}\left(\underline{u},1\right)= \displaystyle\sum_{m\geq  1}d_m(\underline{u}).$$
	%Note that this will also give a confirmation that the family $(d_m)_{m\geq 1}$ is summable, i.e. that this number  $\mu_{\underline{n}}$ is always finite. 
	We consider the relation:
	$$g(\underline{M})=\underline{n}\ \Leftrightarrow\ \left\{\begin{array}{lcl}
		\displaystyle\sum_{\underline{i},j}m_{\underline{i},j}\,i_1&=&n_1;\\
		&\vdots&\\
		\displaystyle\sum_{\underline{i},j}m_{\underline{i},j}\,i_r&=&n_r.
	\end{array}\right.$$
	Let us decompose $m=|M|=\displaystyle\sum_{\underline{i},j}m_{\underline{i},j}$ as follows:
	$$|M|=\displaystyle\sum_{|\underline{i}|>0}m_{\underline{i},j}+\displaystyle\sum_{|\underline{i}|=0,\, i_1>0}m_{\underline{i},j}+\cdots+ \displaystyle\sum_{|\underline{i}|=0=i_1=\cdots=i_{r-2},\, i_{r-1}>0}m_{\underline{i},j}.$$
	So, the relation $g(\underline{M})=\underline{n}$ can be written as:
	\begin{equation}\label{equ:borneFS}
		%g(\underline{M})=\underline{n}\ \Leftrightarrow\ 
		\left\{\begin{array}{ccl}
			\displaystyle\sum_{|\underline{i}|>0}m_{\underline{i},j}\,i_1+\displaystyle\sum_{|\underline{i}|=0,\, i_1>0}m_{\underline{i},j}\,i_1&=&n_1;\\
			&\vdots&\\
			\displaystyle\sum_{|\underline{i}|>0}m_{\underline{i},j}\,i_k+\displaystyle\sum_{|\underline{i}|=0,\, i_1>0}m_{\underline{i},j}\,i_k+\cdots+ \displaystyle\sum_{|\underline{i}|=0=i_1=\cdots=i_{k-1},\, i_{k}>0}m_{\underline{i},j}\,i_k&=&n_k;\\
			&\vdots&\\
			\displaystyle\sum_{\underline{i},j}m_{\underline{i},j}\,i_r&=&n_r.
		\end{array}\right.\end{equation}
	Firstly, let us show by induction on $k\in\{0,\ldots,r-1\}$ that:
	
	$$\begin{array}{l}
		\displaystyle\sum_{|\underline{i}|=0=i_1=\cdots=i_{k-1},\, i_{k}>0}m_{\underline{i},j}\ \ \ \ \leq\ \ \ \  
		\displaystyle\sum_{q=1}^{k-1} 
		\left[\iota_{0,k}\left(\displaystyle\prod_{p=q+1}^{k-1}(1+\iota_{0,p}) + \displaystyle\prod_{p=1}^{k-1}(1+\iota_{0,p}) \right)\right]n_q \\
		\ \ \ \ \ \ \ \ \ \ \ \ \ \ \ \ \ \ \ \ \	\ \ \ \ \ \ \ \ \ \ \ \ \ \ \ \ \ \ \ \ \ \ \ \ \ \ \ \ \ \ \ \ \ \ \ \ +\left[1+\iota_{0,k}\displaystyle\prod_{p=1}^{k-1}(1+\iota_{0,p}) \right]n_k\\
		\ \ \ \ \ \ \ \ \ \ \ \ \ \ \ \ \ \ \ \ \ \ \ \ \ \ \ \ \ \ \ \ \ \ \ \ +\left[\iota_{0,k}\displaystyle\prod_{p=1}^{k-1}(1+\iota_{0,p})\right]n_{k+1}
		+\cdots+\left[\iota_{0,k}\displaystyle\prod_{p=1}^{k-1}(1+\iota_{0,p})\right]n_r ,
	\end{array}$$
	the initial step $k=0$ being:
	$$ \displaystyle\sum_{|\underline{i}|>0}m_{\underline{i},j}\leq n_1+\ldots+n_r.        $$
	This case $k=0$ follows directly from (\ref{equ:borneFS}), by summing its $r$ relations:
	$$ \displaystyle\sum_{|\underline{i}|>0}m_{\underline{i},j}\leq\displaystyle\sum_{|\underline{i}|>0}m_{\underline{i},j}|\underline{i}|\leq n_1+\ldots+n_r.         $$
	Suppose that we have the desired property until some rank $k-1$. Recall that for any $\underline{i}$, $i_k\geq -\iota_{0,k}$. By the $k$'th equation in (\ref{equ:borneFS}), we have:\\
	
	$\begin{array}{l}
		\displaystyle\sum_{|\underline{i}|=0=i_1=\cdots=i_{k-1},\, i_{k}>0}m_{\underline{i},j}\, \leq \, \displaystyle\sum_{|\underline{i}|=0=i_1=\cdots=i_{k-1},\, i_{k}>0}m_{\underline{i},j}\,i_k
	\end{array}$
	
	$\ \ \ \ \ \ \ \ \ \ \ \ \ \ \ \ \ \ \ \ \ \  \begin{array}{lcl}
		&\leq & n_k-\left(
		\displaystyle\sum_{|\underline{i}|>0}m_{\underline{i},j}\,i_k+\displaystyle\sum_{|\underline{i}|=0,\, i_1>0}m_{\underline{i},j}\,i_k+\cdots+ \displaystyle\sum_{|\underline{i}|=0=i_1=\cdots=i_{k-2},\, i_{k-1}>0}m_{\underline{i},j}\,i_k\right)\\
		&\leq & n_k+\iota_{0,k}\left(
		\displaystyle\sum_{|\underline{i}|>0}m_{\underline{i},j} +\displaystyle\sum_{|\underline{i}|=0,\, i_1>0}m_{\underline{i},j} +\cdots+ \displaystyle\sum_{|\underline{i}|=0=i_1=\cdots=i_{k-2},\, i_{k-1}>0}m_{\underline{i},j} \right).
	\end{array}$\\
	
	We apply the induction hypothesis to these $k$ sums and obtain an inequality of type:
	$$\displaystyle\sum_{|\underline{i}|=0=i_1=\cdots=i_{k-1},\, i_{k}>0}m_{\underline{i},j}\leq \alpha_{k,1}\,n_1+\cdots+\alpha_{k,r}\,n_r.$$
	For $q>k$, let us compute:
	$$\begin{array}{lcl}
		\alpha_{k,q}&=&\iota_{0,k}\left( 1+  \iota_{0,1}+ \iota_{0,2}(1+\iota_{0,1})+\iota_{0,3}(1+\iota_{0,1})(1+\iota_{0,2})+\cdots + \iota_{0,k-1} \displaystyle\prod_{p=1}^{k-2}(1+\iota_{0,p})     \right)\\
		&=& \iota_{0,k}\displaystyle\prod_{p=1}^{k-1}(1+\iota_{0,p}).
	\end{array}$$
	For $q=k$, we have the same computation, plus the contribution of the isolated term $n_k$. Hence:
	$$\alpha_{k,k}=1+\iota_{0,k}\displaystyle\prod_{p=1}^{k-1}(1+\iota_{0,p}).$$
	For $q<k$, we have a part of the terms leading again by the same computation to the formula $\iota_{0,k} \displaystyle\prod_{p=1}^{k-1}(1+\iota_{0,p})$. The other part consists of terms starting to appear at the rank $q$ and whose sum can be computed as:
	$$\iota_{0,k}\left( 1+  \iota_{0,q+1}+ \iota_{0,q+2}(1+\iota_{0,q+1})+\cdots + \iota_{0,k-1} \displaystyle\prod_{p=q+1}^{k-2}(1+\iota_{0,p})     \right)\\
	= \iota_{0,k} \displaystyle\prod_{p=q+1}^{k-1}(1+\iota_{0,p}).$$
	So we obtain as desired:
	$$\alpha_{k,q}= \iota_{0,k}\left[ \displaystyle\prod_{p=q+1}^{k-1}(1+\iota_{0,p})+ \displaystyle\prod_{p=1}^{k-1}(1+\iota_{0,p})\right].  $$
	Subsequently, we obtain an inequality for  $m=|M|=\displaystyle\sum_{\underline{i},j}m_{\underline{i},j}$ of type:
	$$\begin{array}{lcl}
		m&=&\displaystyle\sum_{|\underline{i}|>0}m_{\underline{i},j}+\displaystyle\sum_{|\underline{i}|=0,\, i_1>0}m_{\underline{i},j}+\cdots+ \displaystyle\sum_{|\underline{i}|=0=i_1=\cdots=i_{r-2},\, i_{r-1}>0}m_{\underline{i},j}\\
		&\leq & \alpha_1\, n_1+\cdots +\alpha_r\,n_r,
	\end{array}$$
	with  $\alpha_k= 1+\displaystyle\sum_{l=1}^{r-1}\alpha_{l,k}$ for any $k$.   For $k=r$, let us compute in a similar way as before for $\alpha_{k,q}$:
	$$\begin{array}{lcl}
		\alpha_r&=&1+\iota_{0,1}+\iota_{0,2}(1+\iota_{0,1})+\cdots +\iota_{0,k}\displaystyle\prod_{p=1}^{k-1}(1+\iota_{0,p})+\cdots +\iota_{0,r-1}\displaystyle\prod_{p=1}^{r-2}(1+\iota_{0,p})\\
		&=& \displaystyle\prod_{p=1}^{r-1}(1+\iota_{0,p})=\lambda_r.
	\end{array}$$
	For $k=r-1$, we have the same computation plus 1 coming from the term $\alpha_{r-1,r-1}$. Hence:
	$$ \alpha_{r-1}=1+ \displaystyle\prod_{p=1}^{r-1}(1+\iota_{0,p})=\lambda_{r-1}.$$
	For $k\in \{1,\ldots,r-2\}$, we have a part of the terms leading again by the same computation to the formula $\displaystyle\prod_{p=1}^{r-1}(1+\iota_{0,p})$. The other part consists of terms starting to appear at the rank $k$ and whose sum can be computed as:
	$$1+\iota_{0,k+1}+\iota_{0,k+2}(1+\iota_{0,k+1})+\cdots+\iota_{0,r-1}\displaystyle\prod_{p=k+1}^{r-2}(1+\iota_{0,p})=\displaystyle\prod_{p=k+1}^{r-1}(1+\iota_{0,p})$$
	Altogether, we obtain as desired:
	$$\alpha_k=\displaystyle\prod_{p=k+1}^{r-1}(1+\iota_{0,p})+\displaystyle\prod_{p=1}^{r-1}(1+\iota_{0,p})=\lambda_k.$$
\end{demo}
%i_{0,k}\left( 1+ \displaystyle\sum_{ l=1}^{k-1}  \left[i_{0,l}\left(\displaystyle\prod_{p=q+1}^{l-1}(1+i_{0,p})  + \displaystyle\prod_{p=1}^{l-1}(1+i_{0,p}) \right)\right] \right)

\begin{remark}\label{rem:FS-contraint}\indent 
	\begin{enumerate}
		\item Note that  for any $k\in\{1,\ldots,r-1\}$,  $\lambda_k=\lambda_r\left(\displaystyle\frac{1}{(1+\iota_{0,1})\cdots(1+\iota_{0,k})}+1\right)$, so $\lambda_1\geq \lambda_k>\lambda_r$. Thus, we obtain that:
		$$ \mu_{\underline{n}}\leq \lambda_1|\underline{n}|.$$
		Moreover, in the particular case where $\underline{\iota}_0=\underline{0}$ -- i.e. when $Q(\underline{x},y)\in K[[\underline{x}]][y]$ and $y_0\in K[[\underline{x}]]$ as in \cite{sokal:implicit-function} -- we have $\lambda_k=2$ for $k\in\{1,\ldots,r-1\}$ and $\lambda_r=1$. Thus  we obtain:
		$$ \mu_{\underline{n}}\leq 2|\underline{n}|-n_r\leq 2|\underline{n}|.$$
		Note that :
		$$|\underline{n}| \leq 2|\underline{n}|-n_r\leq 2|\underline{n}|$$
		which can be related in this context with the effective bounds  $2|\underline{n}|-1$ (case\\ $w_{\underline{x}}(Q(\underline{x},y))\geq_{\mathrm{grlex}}\underline{0}$) and $|\underline{n}|$ (case $w_{\underline{x}}(Q(\underline{x},y))>_{\mathrm{grlex}}\underline{0}$) given in \cite[Remark 2.4]{sokal:implicit-function}.
		%for any $\underline{n}$ such that $n_k\geq n_{0,k}$ for some fixed $n_{0,k}\geq \mathbb{Z}$, 
		%let us denote $\underline{m}:=\underline{n}-\underline{n}_0$ Indeed,.  So, $\mu_{\underline{n}}\leq \displaystyle\sum_{k=1}^r\lambda_k\, n_k\leq \lambda_1\,|n|$.
		\item With the notation from Theorem \ref{theo:formule-FS}, any strongly reduced Henselian equation $y=Q(\underline{x},y)$ can be written:
		$$\underline{x}^{\underline{\iota}_0}y=\tilde{Q}(\underline{x},y)$$
		$\textrm{ with }\tilde{Q}(\underline{x},y)\in K[[\underline{x}]][y]$ and $w_{\underline{x}}(\tilde{Q}(\underline{x},y))>_{\mathrm{grlex}}\underline{\iota}_0$. 
		Any element $\underline{n}$ of $\mathrm{Supp}\, y_0$, being in the monoid $\mathcal{S}$ of the proof of Theorem \ref{theo:hensel}, is of the form:
		$$\underline{n}=\underline{m}-k\,\underline{\iota}_0\ \ \mathrm{with}\ \underline{m}\in\mathbb{N}^r,\ k\in\mathbb{N}\ \mathrm{and}\ k\,|\underline{\iota}_0|\leq |\underline{m}|.$$
	\end{enumerate}
\end{remark}

\begin{ex}\label{ex:FS-gene}
	Let us consider the following example of strongly reduced Henselian equation:
	$$\begin{array}{lcl}
		y &=& a_{1,-1,2}x_1{x_2}^{-1} y^2 + a_{-1,2,0}{x_1}^{-1}{x_2}^2 +a_{0,1,1}x_2y+ a_{-1,3,0}{x_1}^{-1}{x_2}^3 +a_{0,2,1}{x_2}^2y\\
		&&+\left(a_{1, 1, 0}+ a_{1,1,2}y^2\right)x_1 x_2  +a_{1,2,0} x_1{x_2}^2+a_{2,1,1}y{x_1}^2x_2\\
		&&+ a_{1,3,0} x_1{x_2}^3 +a_{2,2,1} y{x_1}^2{x_2}^2+a_{3,1,2}y^2{x_1}^3x_2.
	\end{array}$$
	The support of the solution is included in the monoid $\mathcal{S}$ generated by the exponents of $(x_1,x_2)$, which is equal to the pairs $\underline{n}=(n_1,n_2)\in\mathbb{Z}^2$ with $n_2=-n_1+ l$ and  $n_1\geq -l$ for  $l\in\mathbb{N}$. We have $\underline{\iota}_0=(1,1)$, so $(\lambda_1,\lambda_2)=(3,2)$ and $\mu_{\underline{n}}\leq 3n_1+2n_2=n_1+2l$. We are in position to compute the first coefficients of the unique solution $y_0$. Let us give the details for the computation of the first terms, for $l=0$. In this case, to compute $c_{n_1,-n_1}$, $n_1>0$, we consider $m$ such that  $1\leq m\leq \mu_{n_1,-n_1}\leq n_1$, and $\underline{M}=(m_{\underline{i},j})_{\underline{i},j}$ such that:
	$$\left\{\begin{array}{lcl}
		|\underline{M}|=m & \Leftrightarrow &  \displaystyle\sum_{\underline{i},j}m_{\underline{i},j}=m\leq n_1;\\
		\|\underline{M}\|=m-1 & \Leftrightarrow &  \displaystyle\sum_{\underline{i},j}m_{\underline{i},j}j=m-1\leq n_1-1;\\
		g(\underline{M})=\underline{n} &\Leftrightarrow &  \left\{\begin{array}{lcl}
			\displaystyle\sum_{\underline{i},j}m_{\underline{i},j}\,i_1&=&n_1>0;\\
			\displaystyle\sum_{\underline{i},j}m_{\underline{i},j}\,i_2&=&-n_1<0.
		\end{array}\right.
	\end{array}\right.$$
	The last condition implies that $m_{1,-1,2}\geq n_1$. 
	But, according to the second condition, this gives $n_1-1\geq \|\underline{M}\|\geq 2\, m_{1,-1,2} \geq 2\,n_1$, a contradiction. Hence, $c_{n_1,-n_1}=0$ for any $n_1>0$.\\
	In the case $l=1$, we consider the corresponding conditions to compute $c_{n_1,-n_1+1}$ for $n_1\geq -1$. We obtain that $1\leq m\leq \mu_{n_1,-n_1+1}\leq n_1+2$. Suming the two conditions in $g(\underline{M})=(n_1,-n_1+1)$, we get $m_{-1,2,0}+m_{0,1,1}=1$ and $m_{\underline{i},j}=0$ for any $\underline{i}$ such that $i_1+i_2\geq 2$. So we are left with the following linear system:
	$$\left\{\begin{array}{cccccccccc}
		(L_1)&m_{1,-1,2}&+&m_{-1,2,0}&+&m_{0,1,1}&=& m&\leq& n_1+2\\
		(L_2)&2\,m_{1,-1,2}&+&&&m_{0,1,1}&=& m-1&\leq& n_1+1\\
		(L_3)&m_{1,-1,2}&-&m_{-1,2,0}& &&=& n_1&&\\
		(L_4)&-m_{1,-1,2}&+&2\,m_{-1,2,0}&+&m_{0,1,1}&=& -n_1+1&&\\
	\end{array}\right.$$
	By comparing $(L_2)-(L_3)$ and $(L_1)$, we get that $m=m-1-n_1$, so $n_1=-1$. Consequently, by $(L_1)$, $m=1$, and by $(L_2)$, $m_{1,-1,2}=m_{0,1,1}=0$. Since $m_{-1,2,0}+m_{0,1,1}=1$, we obtain $m_{-1,2,0}=1$ which indeed gives the only solution. Finally,  $c_{n_1,-n_1+1}=0$ for any $n_1\geq 0$ and: 
	$$ c_{-1,2}=\displaystyle\frac{1}{1}\displaystyle\frac{1!}{1!0!}{a_{-1,2,0}}^1=a_{-1,2,0}.$$
	Similarly, we claim that one can determine that:
	$$\begin{array}{lcll}
		c_{-2,4}&=&0,&\mu_{\underline{n}}\leq 2;\\
		c_{-1,3}&=&a_{{-1,3,0}}+a_{{0,1,1}}a_{{-1,2,0}}+a_{{1,-1,2}}{a_{{-1,2,0}}}^{2},& \mu_{\underline{n}}\leq 3;\\
		c_{0,2}&=&0,&  \mu_{\underline{n}}\leq 4;\\
		c_{1,1}&=&a_{{1,1,0}},&  \mu_{\underline{n}}\leq 5;\\
		c_{n_1,-n_1+2}&=&0\ \ \ \ \mathrm{for}\ \ \ \ n_1\geq 0,\ n_1\neq 1&\mu_{\underline{n}}\leq n_1+4;\\
		c_{n_1,-n_1+3}&=&0\ \ \ \ \mathrm{for}\ \ \ \ -3\leq n_1\leq -2,&\mu_{\underline{n}}\leq n_1+6;\\
		c_{-1,4}&=&a_{{0,2,1}}a_{{-1,2,0}}+a_{{0,1,1}}a_{{-1,3,0}}+2\,a_{{1,-1,2}}a_{{-1,2,0}}a_{{-1,3,0}} &\\
		&&+{a_{{0,1,1}}}^{2}a_{{-1,2,0}}+3\,a_{{0,1,1}}a_{{1,-1,2}}{a_{{-1,2,0}}}^{2}+2\,{a_{{1,-1,2}}}^{2}{a_{{-1,2,0}}}^{3},& \mu_{\underline{n}}\leq 5;\\
		&\vdots&\end{array}$$
	%_{\mathrm{grlex}}
\end{ex}

\section{Closed-form expression of an algebroid multivariate series.}\label{section:flajo-soria}

The field $K$ of coefficients has still characteristic zero. Our purpose is to determine the coefficients of an algebroid series in terms of the coefficients of a vanishing polynomial. We consider the following polynomial of degree  in $y$ bounded by $d_y$ and satisfying the conditions (i) to (iii) of Lemma \ref{lemma:support-equ}: 
$$ \begin{array}{lcl}
	P(\underline{u},y)&=&\displaystyle\sum_{\underline{i}\in \N^r}\displaystyle\sum_{j=0}^{d_y}a_{\underline{i},j}\underline{u}^{\underline{i}}y^j ,\ \textrm{  with } P(\underline{u},y)\in K[[\underline{u}]][y]\setminus\{0\}\\
	&=& \displaystyle\sum_{\underline{i}\in\N^r}\pi_{\underline{i}}^P(y)\underline{u}^{\underline{i}}\\
	&=& \displaystyle\sum_{j=0}^{d_y}a_{j}^P(\underline{u})y^j,
\end{array}$$
and a formal power series:
$$y_0=\displaystyle\sum_{\underline{n}\geq_{\mathrm{grlex}} \underline{0}}c_ {\underline{n}}\underline{u}^{\underline{n}},  \textrm{  with } y_0\in K[[\underline{u}]],\   c_{\underline{0}}\neq 0.$$
The field $K((\underline{u}))$ is endowed with the  graded lexicographic valuation $w$.\\
%$$\begin{array}{lccl}w:&K((\underline{x}))\setminus\{0\}&\rightarrow&(\mathbb{Z}^r,\leq_{\textrm{grlex}})\\
	%&y&\mapsto& \min(\textrm{support}(y)).
	%\end{array}$$

	\begin{notation}\label{nota:P_k}
		For any $\underline{k}\in\mathbb{N}^r$ and for any $Q(\underline{u},y)=\displaystyle\sum_{j=0}^da_j^Q(\underline{u})y^j\in K((u_1^{\mathbb{Z}},\ldots,u_r^{\mathbb{Z}}))^{\textrm{grlex}}[y]$, we denote: 
		\begin{itemize}
			\item $S(\underline{k})$ the successor element of $\underline{k}$ in $(\mathbb{N}^r,\leq_{\textrm{grlex}})$;
			\item $w(Q):=\min\left\{w \left(a_j^Q(\underline{u})\right),\ j=0,..,d\right\}$;
			\item For any $\underline{k}\in\N^r $, $z_{\underline{k}}:=\displaystyle\sum_{\underline{n}=\underline{0}}^{\underline{k}}c_{\underline{n}}\underline{u }^{\underline{n}}$;
			\item $y_{\underline{k}}:=y_0-z_{\underline{k}}=\displaystyle\sum_{\underline{n}\geq_{\textrm{grlex}}S( \underline{k})}c_{\underline{n}}\underline{u}^{\underline{n}}$;
			\item $Q_{\underline{k}}(\underline{u},y):=Q(\underline{u},z_{\underline{k}}+\underline{u}^{S(\underline{k})}y) =\displaystyle\sum_{\underline{i}\geq_{\textrm{grlex}}\underline{i}_{\underline{k}}}\pi^Q_{\underline{k},\underline{i}}(y)\underline{u}^{\underline{i}}$ where $\underline{i}_{\underline{k}}:=w( Q_{\underline{k}})$.
			% and $d_k:=\deg_uQ_k$.
			Note that the sequence $(\underline{i}_{\underline{k}})_{\underline{k}\in\mathbb{N}^r}$ is nondecreasing since $Q_{S(\underline{k})}(\underline{u},y)=Q_{\underline{k}}(\underline{u},c_{S(\underline{k})}+\underline{u}^{\underline{n}}y)$ for  $\underline{n}=S^2({\underline{k}})-S({\underline{k}})>_{\textrm{grlex}}\underline{0}$, $\underline{n}\in \mathbb{Z}^r$.
		\end{itemize}
	\end{notation}
	
	As for the algebraic case \cite{hickel-matu:puiseux-alg-multivar}, we consider $y_0$ solution of the equation $P=0$ via an adaptation in several variables of the  algorithmic method of Newton-Puiseux, also with two stages:
	\begin{enumerate}
		\item a first stage of separation of the  solutions, which illustrates the following fact: $y_0$ may share an initial part with other roots of $P$.  But, if $y_0$ is a simple root of $P$, this step concerns \emph{only finitely many} of the first terms of $y_0$ since $w\left(\partial P/\partial y\,(\underline{u},y_0)\right)$ is finite.
		\item a second stage of unique "automatic" resolution: for $y_0$ a simple root of $P$, once it has been separated from the other solutions, we will show that the remaining part of $y_0$ is a root of a strongly reduced  Henselian equation, in the sense of Definition \ref{defi:equ-hensel-red}, naturally derived from $P$ and an initial part of $y_0$.
	\end{enumerate}

	\begin{lemma}\label{lemme:double-simple}
		\begin{itemize}
			\item[(i)] The series $y_0$ is a root of  $P(\underline{u},y)$ if and only if the sequence $(\underline{i}_{\underline{k}})_{\underline{k}\in\mathbb{N}^r}$ where $\underline{i}_{\underline{k}}:=w( P_{\underline{k}})$ is strictly increasing.
			\item[(ii)] The series $y_0$ is a simple root of $P(\underline{u},y)$ if and only if the sequence $(\underline{i}_{\underline{k}})_{\underline{k}\in\mathbb{N}^r}$ is strictly increasing and there exists a lowest multi-index $\underline{k}_0$ such that $\underline{i}_{S(\underline{k}_0)}=\underline{i}_{\underline{k}_0}-S(\underline{k}_0)+S^2(\underline{k}_0)$. In that case, one has that  $\underline{i}_{S(\underline{k})}=\underline{i}_{\underline{k}}-S(\underline{k})+S^2(\underline{k})=\underline{i}_{\underline{k}_0}-S(\underline{k}_0)+S^2(\underline{k})$  for any $\underline{k}\geq_{\textrm{grlex}} \underline{k}_0$. 
		\end{itemize}
	\end{lemma}
	\begin{proof}
		(i) Note that for any $\underline{k}\in\mathbb{N}^r,\ $ $\underline{i}_{\underline{k}}\leq_{\textrm{grlex}} w(P_{\underline{k}}(\underline{u},0)=w(P(\underline{u},z_{\underline{k}}))$. Hence, if the sequence $(\underline{i}_{\underline{k}})_{\underline{k}\in\mathbb{N}^r}$ is strictly increasing in $(\mathbb{N}^r,\leq_{\textrm{grlex}})$, it tends to  $+\infty$ (i.e.  $\forall \underline{n}\in\mathbb{N}^r$, $\exists \underline{k}_0\in\mathbb{N}^r$, $\forall \underline{k}\geq_{\mathrm{grlex}} \underline{k}_0$, $\underline{i}_{\underline{k}}\geq_{\mathrm{grlex}}\underline{n}$),  and so does $w(P(\underline{u},z_{\underline{k}}))$. The series $y_0$ is indeed a root of $P(\underline{u},y)$. Conversely, suppose that there exist $\underline{k}<_{\textrm{grlex}}\underline{l}$ such that $\underline{i}_{\underline{k}}\geq_{\textrm{grlex}} \underline{i}_{\underline{l}}$.
		Since the sequence $(\underline{i}_{\underline{n}})_{\underline{n}\in\mathbb{N}^r}$ is nondecreasing,  one has that $\underline{i}_{\underline{l}}\geq \underline{i}_{\underline{k}}$, so  $\underline{i}_{\underline{l}}=\underline{i}_{\underline{k}}$.
		We apply the multivariate Taylor's formula to $P_{\underline{j}}(\underline{u},y)$ for $\underline{j}>_{\textrm{grlex}}\underline{k}$:
		\begin{equation}\label{equ:taylor}
			\begin{array}{lcl}
				P_{\underline{j}}(\underline{u},y)&=&P_{\underline{k}}\left(\underline{u},c_{S(\underline{k})}+ c_{S^2(\underline{k})}\underline{u}^{S^2(\underline{k})-S(\underline{k})} +\cdots+c_{\underline{j}}\underline{u}^{\underline{j}-S(\underline{k})}+\underline{u}^{S(\underline{j})-S(\underline{k})}y\right)\\
				&=& \displaystyle\sum_{\underline{i}\geq_{\textrm{grlex}} \underline{i}_{\underline{k}}} \pi^P_{\underline{k},\underline{i}} \left(c_{S(\underline{k})}+ c_{S^2(\underline{k})}\underline{u}^{S^2(\underline{k})-S(\underline{k})} +\cdots+\underline{u}^{S(\underline{j})-S(\underline{k})}y\right) \underline{u}^{\underline{i}}\\
				&=& \pi^P_{\underline{k},\underline{i}_{\underline{k}}}(c_{S(\underline{k})})\underline{u}^{\underline{i}_{\underline{k}}}+b_{S(\underline{i}_{\underline{k}})} \underline{u}^{S(\underline{i}_{\underline{k}})}+ \cdots.
			\end{array}
		\end{equation}
		Note that $b_{S(\underline{i}_{\underline{k}})}= \pi^P_{\underline{k},S(\underline{i}_{\underline{k}})}(c_{S(\underline{k})})$ or  $b_{S(\underline{i}_{\underline{k}})}=  (\pi^P_{\underline{k},\underline{i}_{\underline{k}}} )'(c_{S(\underline{k})})\, c_{S^2(\underline{k})}+\pi^P_{\underline{k},S(\underline{i}_{\underline{k}})}(c_{S(\underline{k})})$ depending on whether $S(\underline{i}_{\underline{k}})<_{\textrm{grlex}} \underline{i}_{\underline{k}}+S^2(\underline{k})-S(\underline{k})$ or $S(\underline{i}_{\underline{k}})=\underline{i}_{\underline{k}}+S^2(\underline{k})-S(\underline{k})$.
		For $\underline{j}=\underline{l}$, we deduce that $\pi^P_{\underline{k},\underline{i}_{\underline{k}}}(c_{S(\underline{k})})\neq 0$. This implies that for any $\underline{j}>_{\mathrm{grlex}}\underline{k}$, $\underline{i}_{\underline{j}}=\underline{i}_{\underline{k}}$
		and    $w\left(P_{\underline{j}}(\underline{u},0)\right)=w\left(P(\underline{u},z_{\underline{j}})\right)=\underline{i}_{\underline{k}}$. Hence $w\left(P(\underline{u},y_0)\right)=\underline{i}_{\underline{k}}\neq +\infty$.
		\vspace{2 pt} 
		
		\noindent (ii) The series $y_0$ is a double root of $P$ if and only if it is a root of $P$ and $\partial P/\partial y$. Let $y_0$ be a root of $P$. Let us expand the multivariate Taylor's formula (\ref{equ:taylor}) for  $\underline{j}=S(\underline{k})$: 
		\begin{equation}\label{equ:p_{k+1}}
			\begin{array}{l}
				\begin{array}{lcl}
					P_{S(\underline{k})}(\underline{u},y)
					&=&  \pi^P_{\underline{k},\underline{i}_{\underline{k}}}(c_{S(\underline{k})})\underline{u}^{\underline{i}_{\underline{k}}}+ \pi^P_{\underline{k},S(\underline{i}_{\underline{k}})}(c_{S(\underline{k})})\underline{u}^{S(\underline{i}_{\underline{k}})}+\cdots\\
					&& +\left[(\pi^P_{\underline{k},\underline{i}_{\underline{k}}})'(c_{S(\underline{k})})\, y+\pi^P_{\underline{k},\underline{i}_{\underline{k}}+S^2(\underline{k})-S(\underline{k})}(c_{S(\underline{k})})\right]\underline{u}^{\underline{i}_{\underline{k}}+S^2(\underline{k})-S(\underline{k})}+\cdots +
				\end{array}\\
				\left[\displaystyle\frac{(\pi^P_{\underline{k},\underline{i}_{\underline{k}}})''(c_{S(\underline{k})})}{2}\, y^2+(\pi^P_{\underline{k},\underline{i}_{\underline{k}}+S^2(\underline{k})-S(\underline{k})})'(c_{S(\underline{k})})\,y+\pi^P_{\underline{k},\underline{i}_{\underline{k}}+2(S^2(\underline{k})-S(\underline{k}))}(c_{S(\underline{k})})\right]\underline{u}^{\underline{i}_{\underline{k}}+2(S^2(\underline{k})-S(\underline{k}))}+\cdots
			\end{array}\\
		\end{equation}
		
		\noindent Note that if $S(i_{\underline{k}})=\underline{i}_{\underline{k}}+S^2(\underline{k})-S(\underline{k})$, then there are no intermediary terms between the first one and the one with valuation $\underline{i}_{\underline{k}}+S^2(\underline{k})-S(\underline{k})$.
		We have by definition of $P_{\underline{k}}$:
		\begin{center}
			$\displaystyle\frac{\partial P_{\underline{k}}}{\partial y}(\underline{u},y)=\underline{u}^{S(\underline{k})}\left(\displaystyle\frac{\partial P}{\partial y}\right)_{\underline{k}}(\underline{u},y)=\displaystyle\sum_{\underline{i}\geq_{\textrm{grlex}} \underline{i}_{\underline{k}}}(\pi^P_{\underline{k},\underline{i}})'(y)\underline{u}^{\underline{i}}$
		\end{center}
		One has that $\pi^P_{\underline{k},\underline{i}_{\underline{k}}}(y)\nequiv 0$ and $\pi^P_{\underline{k},\underline{i}_{\underline{k}}}(c_{S(\underline{k})})=0$ (see the point (i) above), so $(\pi^P_{\underline{k},\underline{i}_{\underline{k}}})'(y)\nequiv 0$. Thus: \begin{equation}\label{equ:val-DP} w\left(\left(\displaystyle\frac{\partial P}{\partial y}\right)_{\underline{k}}\right)=\underline{i}_{\underline{k}}-S(\underline{k}).
		\end{equation}
		We perform the Taylor's expansion of   $\left(\displaystyle\frac{\partial P}{\partial y}\right)_{S(\underline{k})}$:
		$$\begin{array}{lcl}
			\left(\displaystyle\frac{\partial P}{\partial y}\right)_{S(\underline{k})}(\underline{u},y)&=& \left(\displaystyle\frac{\partial P}{\partial y}\right)_{\underline{k}}\left(\underline{u},c_{S(\underline{k})}+\underline{u}^{S^2(\underline{k})-S(\underline{k})}y\right)\\
			&=&( \pi^P_{\underline{k},\underline{i}_{\underline{k}}})'(c_{S(\underline{k})})u^{\underline{i}_{\underline{k}}-S(\underline{k})}+\cdots\\
			&&+ \left[(\pi^P_{\underline{k},\underline{i}_{\underline{k}}})''(c_{S(\underline{k})})\, y+(\pi^P_{\underline{k},\underline{i}_{\underline{k}}+S^2(\underline{k})-S(\underline{k})})'(c_{S(\underline{k})})\right]\underline{u}^{\underline{i}_{\underline{k}}+S^2(\underline{k})-2S(\underline{k})}+\cdots.
		\end{array}$$
		By the point (i) applied to  $\displaystyle\frac{\partial P}{\partial y}$, if $y_0$ is a double root $P$, we must have $ (\pi^P_{\underline{k},\underline{i}_{\underline{k}}})'(c_{S(\underline{k})})=0$. Moreover, if  $\pi^P_{\underline{k},\underline{i}}(c_{S(\underline{k})})\neq 0$ for some $\underline{i}\in\left\{S(\underline{i}_{\underline{k}}),\,\ldots\,,\,\underline{i}_{\underline{k}}+S^2(\underline{k})-S(\underline{k})\right\}$, by Formula (\ref{equ:p_{k+1}}) we would have $i_{S(\underline{k})}\leq_{\textrm{grlex}}\underline{i}_{\underline{k}}+S^2(\underline{k})-S(\underline{k})$ and even $\underline{i}_{\underline{j}}\leq_{\textrm{grlex}}\underline{i}_{\underline{k}}+S^2(\underline{k})-S(\underline{k})$ for every  $\underline{j}>_{\textrm{grlex}}\underline{k}$ according to Formula (\ref{equ:taylor}): $y_0$ could not be a root of $P$. So, $\pi^P_{\underline{k},\underline{i}}(c_{S(\underline{k})})= 0$ for $\underline{i}=S(\underline{i}_{\underline{k}}),..,\underline{i}_{\underline{k}}+S^2(\underline{k})-S(\underline{k})$, and, accordingly, $i_{S(\underline{k})}>_{\textrm{grlex}} \underline{i}_{\underline{k}}+S^2(\underline{k})-S(\underline{k})$.\\
		If $y_0$ is a simple root of $P$, from the point (i) and its proof there exists a lowest  $\underline{k}_0$ such that the sequence $(\underline{i}_{\underline{k}}-S(\underline{k}))_{\underline{k}\in\mathbb{N}^r}$ is no longer strictly increasing, that is to say, such that  $(\pi^P_{\underline{k}_0,\underline{i}_{\underline{k}_0}})'(c_{S(\underline{k}_0)})\neq 0$. For any $\underline{k}\geq_{\mathrm{grlex}} \underline{k}_0$, we consider the  Taylor's expansion of $\left(\displaystyle\frac{\partial P}{\partial y}\right)_{S(\underline{k})}=\left(\displaystyle\frac{\partial P}{\partial y}\right)_{\underline{k}_0}(c_{S(\underline{k}_0)}+\cdots+\underline{u}^{S^2(\underline{k})-S(\underline{k_0 })}y)$:
		\begin{equation}\label{equ:partialP}
			\begin{array}{l}
				\left(\displaystyle\frac{\partial P}{\partial y}\right)_{S(\underline{k})}(\underline{u},y)\,=\,(\pi^P_{\underline{k}_0,\underline{i}_{\underline{k}_0}})'(c_{S(\underline{k}_0)})\underline{u}^{\underline{i}_{\underline{k}_0}-S(\underline{k}_0)}+\cdots\\
				\ \ \ \ \ \  \ \ \  \ \ \ +\left[(\pi^P_{\underline{k}_0,\underline{i}_{\underline{k}_0}})''(c_{S(\underline{k}_0)})c_{S^2(k_0)}+(\pi^P_{\underline{k}_0, \underline{i}_{\underline{k}_0}+S^2(\underline{k}_0)-S(\underline{k}_0)})' (c_{S(\underline{k}_0)})\right]\underline{u}^{\underline{i}_{\underline{k}_0}+ S^2(\underline{k}_0)-S(\underline{k}_0)} +\cdots
			\end{array}
		\end{equation}
		and we get that:
		\begin{equation}\label{equ:taylor-deriv}
			w\left(\displaystyle\frac{\partial P}{\partial y}\left(z_{S(\underline{k})},0\right) \right)=w\left(\left(\displaystyle\frac{\partial P}{\partial y}\right)_{S(\underline{k})}(\underline{u},0)\right)=w\left(\left(\displaystyle\frac{\partial P}{\partial y}\right)_{S(\underline{k})}\right)=\underline{i}_{\underline{k}_0}-S(\underline{k}_0).
		\end{equation} 
		%As $(\pi^P_{S(\underline{k}),\underline{i}_{S(\underline{k})}})'(y)\nequiv 0$, we obtain that $\underline{i}_{S(k)}=w(P_{S(\underline{k})})=w\left(\displaystyle\frac{\partial P_{S(\underline{k})}}{\partial y}\right)=w\left(\displaystyle\frac{\partial P}{\partial y}\right)_{S(\underline{k})}+S^2(\underline{k})=\underline{i}_{\underline{k}_0}-S(\underline{k}_0)+S^2(\underline{k})$ by Equation (\ref{equ:taylor-deriv}). 
		By Equation (\ref{equ:val-DP}), we obtain that $w\left(\left(\displaystyle\frac{\partial P}{\partial y}\right)_{S(\underline{k})}\right)=\underline{i}_{S(\underline{k})}-S^2(\underline{k})$. So, $\underline{i}_{S(k)}=\underline{i}_{\underline{k}_0}-S(\underline{k}_0)+S^2(\underline{k})$. As every $\underline{k}>_{\mathrm{grlex}} \underline{k}_0$ is the successor of some $\underline{k}'\geq_{\mathrm{grlex}} \underline{k}_0$, we get that for every $\underline{k}\geq_{\mathrm{grlex}} \underline{k}_0$, $\underline{i}_{\underline{k}}-S(\underline{k})=\underline{i}_{\underline{k}_0}-S(\underline{k}_0)$. So, finally, $\underline{i}_{S(k)}=\underline{i}_{\underline{k}}-S(\underline{k})+S^2(\underline{k})$ as desired.
		% Hence, beginning on rank $\underline{k}_0$, the sequence $(\underline{i}_{\underline{k}})$ increases  one by one. 
	\end{proof}
	
	Resuming the notations  of Lemma \ref{lemme:double-simple}, the  multi-index $\underline{k}_0$ represents the length of the initial part in the stage of separation of the solutions. In the following lemma, we bound it using  the discriminant $\Delta_P$ of $P$ (see just before Notation \ref{nota:FS}). 
	
	\begin{lemma}\label{lemme:partie-princ}
		Let $P(\underline{u},y)$ be a  nonzero polynomial with  $\deg_y(P)\leq d_y$ and with only simple roots. Let  $y_0=\displaystyle\sum_{\underline{n}\in\N^r}c_ {\underline{n}}\underline{u}^{\underline{n}},\   c_{\underline{0}}\neq 0$ be one of these roots.
		The  multi-index $\underline{k}_0$ of Lemma \ref{lemme:double-simple} verifies that:
		$$|\underline{k}_0|\leq \ord_{\underline{u}}(\Delta_P(\underline{u})).$$
	\end{lemma}
	\begin{proof}
		By definition of $\underline{k}_0$ and by Formula (\ref{equ:taylor-deriv}), for any $\underline{k}\geq_{\mathrm{grlex}} \underline{k}_0$, $$w\left( \displaystyle\frac{\partial P}{\partial y}\left(\underline{u},z_{S(\underline{k})}\right)\right)=w\left(\displaystyle\frac{\partial P}{\partial y}\left(\underline{u},z_{S(\underline{k}_0)}\right)\right)=\underline{i}_{\underline{k}_0}-S(\underline{k}_0).$$ So, $w\left(\displaystyle\frac{\partial P}{\partial y}(\underline{u},y_0)\right)=w\left(\displaystyle\frac{\partial P}{\partial y}(\underline{u},z_{S(\underline{k}_0)})\right)$. 
		Moreover,  by minimality of $\underline{k}_0$, the sequence $(\underline{i}_{\underline{k}}-S(\underline{k}))_{\underline{k}}$ is strictly increasing up to $\underline{k}_0$, so by Formula (\ref{equ:val-DP}): $$w\left( \displaystyle\frac{\partial P}{\partial y}(\underline{u},y_0)\right)=w\left(\displaystyle\frac{\partial P}{\partial y}(\underline{u},z_{S(\underline{k}_0)})\right)=w\left(\left(\displaystyle\frac{\partial P}{\partial y}\right)_{S(\underline{k}_0)}(\underline{u},0)\right)\geq_{\mathrm{grlex}} w\left(\left(\displaystyle\frac{\partial P}{\partial y}\right)_{S(\underline{k}_0)}\right)\geq_{\mathrm{grlex}} \underline{k}_0.$$
		So:
		$$\left|\underline{k}_0\right|\leq \left|w\left( \displaystyle\frac{\partial P}{\partial y}(\underline{u},y_0)\right)\right|=\mathrm{ord}_{\underline{u}} \displaystyle\frac{\partial P}{\partial y}(\underline{u},y_0).$$
		Since  $P$  has only  simple roots, its discriminant  $\Delta_P$  is nonzero and one has a Bezout identity:
		$$A(\underline{u},y)P(\underline{u},y)+B(\underline{u},y)\frac{\partial P}{\partial y}(\underline{u},y)=\Delta_P(\underline{u})$$
		with $A,B\in K[[\underline{u}]][y]$. 
		By evaluating this identity at $y=y_0$, we obtain that $\ord_{\underline{u}} \left(\displaystyle\frac{\partial P}{\partial y}(\underline{u},y_0) \right)\leq \ord_{\underline{u}}(\Delta_P(\underline{u}))$, so $\left|\underline{k}_0\right|\leq \ord_{\underline{u}}(\Delta_P(\underline{u}))$ as desired.
	\end{proof}

	\begin{notation}\label{nota:omega_0}
		Resuming Notation \ref{nota:P_k} and the content of Lemma \ref{lemme:double-simple}, we set: $$\omega_0:=(\pi^P_{\underline{k}_0,\underline{i}_{\underline{k}_0}})'(c_{S(\underline{k}_0)}).$$
		By Formula (\ref{equ:partialP}), we note that $$
		\left(\displaystyle\frac{\partial P}{\partial y}\right)(\underline{u},y_0)=\omega_0\,\underline{u}^{\underline{i}_{\underline{k}_0}-S(\underline{k}_0)}+\cdots.$$
		Thus, $\omega_0$ is the  initial coefficient of $\left(\displaystyle\frac{\partial P}{\partial y}\right)(\underline{u},y_0)$ with respect to $\leq_{\mathrm{grlex}}$, hence $\omega_0\neq 0$.
	\end{notation}
	
	\begin{theo}\label{theo:FS}
		Consider the following nonzero polynomial in $K[[\underline{u}]][y]$ of degree in $y$ bounded by $d_y$: 
		$$ P(\underline{u},y)=\displaystyle\sum_{\underline{i}\in\N^r}\displaystyle\sum_{j=0}^{d_y}a_{\underline{i},j}\underline{u}^{\underline{i}}y^j = \displaystyle\sum_{\underline{i}\geq_{\mathrm{grlex}} \underline{0}} \pi^P_{\underline{i}}(y)\underline{u}^{\underline{i}},$$
		and a formal power series which is a simple root:
		$$y_0=\displaystyle\sum_{\underline{n}\geq_{\mathrm{grlex}} \underline{0}}c_{\underline{n}}\underline{u}^{\underline{n}}\ \in K[[\underline{u}]],\   c_{\underline{0}}\neq 0.$$
		Resuming Notations \ref{nota:P_k} and \ref{nota:omega_0} and the content of Lemma \ref{lemme:double-simple}, recall that \\ $\omega_0:=(\pi^P_{\underline{k}_0,\underline{i}_{\underline{k}_0}})'(c_{S(\underline{k}_0)})\neq 0$.
		Then, for any $\underline{k}>_{\mathrm{grlex}}\underline{k}_0$:
		\begin{itemize}
			\item either the polynomial $z_{S(\underline{k})}=\displaystyle\sum_{\underline{n}=\underline{0}}^{S(\underline{k})}c_{\underline{n}}\underline{u}^{\underline{n}}$ is a solution of $P(\underline{u},y)=0$;
			\item or $ _{\underline{k}}R(\underline{u},y):=\displaystyle\frac{P_{\underline{k}}(\underline{u},y+c_{S(\underline{k})})}{-\omega_0\underline{u}^{\underline{i}_{\underline{k}}}}=-y+\, _{\underline{k}}Q(\underline{u},y)\in K\left(\left(u_1^{\mathbb{Z}},\ldots,u_r^\mathbb{Z}\right)\right)^{\textrm{grlex}}_{\textrm{Mod}}[y]$  defines a strongly reduced Henselian equation:
			$$ y=\, _{\underline{k}}Q(\underline{u},y)$$
			%with $w\left(Q_{\underline{k}(\underline{u},y)\right)>_{\mathrm{grlex}} \underline{0}$
				as in Definition \ref{defi:equ-hensel-red} and satisfied by:
				$$t_{S(\underline{k})}:=\frac{y_0-z_{S(\underline{k})}}{\underline{u}^{S(\underline{k})}}=c_{S^2(\underline{k})}\underline{u}^{S^2(\underline{k})-S(\underline{k})}+c_{S^3(\underline{k})}\underline{u}^{S^3(\underline{k})-S(\underline{k})}+\cdots.$$
			\end{itemize} 
		\end{theo}
		\begin{proof}
			%A MODIFIER : PREUVE SANS RECURRENCE EN S'APPUYANT SUR LES FORMULES DANS 5.2
			We show by induction on $\underline{k}\in(\mathbb{N}^r,\leq_{\mathrm{grlex}})$, $\underline{k}>_{\mathrm{grlex}}\underline{k}_0$,  that $_{\underline{k}}R(u,y)=-y+ \, _{\underline{k}}Q(u,y)$ with $_{\underline{k}}Q(\underline{u},y)\in K\left(\left(u_1^{\mathbb{Z}},\ldots,u_r^\mathbb{Z}\right)\right)^{\textrm{grlex}}_{\textrm{Mod}}\left[y\right]$ is
			such that $w\left( _{\underline{k}}Q\left(\underline{u},y\right)\right) >_{\textrm{grlex}}\underline{0}$.  Let us apply  Formula  (\ref{equ:p_{k+1}}) with parameter $\underline{k}=\underline{k}_0$. Since $\underline{i}_{S(\underline{k}_0)}=\underline{i}_{\underline{k}_0}+S^2(\underline{k}_0)-S(\underline{k}_0)$, we have that  $\pi^P_{\underline{k}_0,\underline{i}}(c_{S(\underline{k}_0)})=0$ for $\underline{i}_{\underline{k}_0}\leq_{\mathrm{grlex}} \underline{i}<_{\mathrm{grlex}}i_{\underline{k}_0}+S^2(\underline{k}_0)-S(\underline{k}_0)$,  and accordingly:
			$$P_{S(\underline{k}_0)}(\underline{u},y)=\left[\omega_0\,y+\pi^P_{\underline{k}_0,\underline{i}_{\underline{k}_0}+S^2(\underline{k}_0)-S(\underline{k}_0)}(c_{S(\underline{k}_0)})\right]\underline{u}^{\underline{i}_{\underline{k}_0}+S^2(\underline{k}_0)-S(\underline{k}_0)}+\, _{S(\underline{k}_0)}T(\underline{u},y)$$
			where $_{S(\underline{k}_0)}T(\underline{u},y)\in K[[\underline{u}]][y]$ with $w\left( _{S(\underline{k}_0)}T(\underline{u},y)\right)>_{\mathrm{grlex}}\underline{i}_{\underline{k}_0}+S^2(\underline{k}_0)-S(\underline{k}_0)$.
			Since $\underline{i}_{S^2(\underline{k}_0)}=\underline{i}_{\underline{k}_0}+S^3(\underline{k}_0)-S(\underline{k}_0)>_{\mathrm{grlex}}i_{\underline{k}_0}+S^2(\underline{k}_0)-S(\underline{k}_0)$, we obtain that: $$\pi^P_{S(\underline{k}_0),i_{\underline{k}_0}+S^2(\underline{k}_0)-S(\underline{k}_0)}(y)=\omega_0\,y+\pi^P_{\underline{k}_0,i_{\underline{k}_0}+S^2(\underline{k}_0)-S(\underline{k}_0)}(c_{S(\underline{k}_0)})$$ vanishes at $c_{S^2(\underline{k}_0)}$, which implies that $$c_{S^2(\underline{k}_0)}= \displaystyle\frac{-\pi^P_{\underline{k}_0,\underline{i}_{\underline{k}_0}+S^2(\underline{k}_0)-S(\underline{k}_0)}(c_{S(\underline{k}_0)})}{\omega_0}.$$ Computing $_{S(\underline{k}_0)}R(\underline{u},y)$, it follows that:
			\begin{center}
				$_{S(\underline{k}_0)}R(\underline{u},y)=-y+\, _{S(\underline{k}_0)}Q(\underline{u},y) $,\end{center} \textrm{ with }$_{S(\underline{k}_0)}Q(\underline{u},y)=\displaystyle\frac{_{S(\underline{k}_0)}T(\underline{u},y +c_{S^2(\underline{k}_0)})}{-\omega_0\underline{u}^{i_{\underline{k}_0}+S^2(\underline{k}_0)-S(\underline{k}_0)}}$.
			So $_{S(\underline{k}_0)}Q(\underline{u},y)\in K\left(\left(u_1^{\mathbb{Z}},\ldots,u_r^\mathbb{Z}\right)\right)^{\textrm{grlex}}_{\textrm{Mod}}\left[y\right]$ with $w\left( _{S(\underline{k}_0)}Q(\underline{u},y)\right)>_{\mathrm{grlex}} \underline{0}$.  
			
			Now suppose that the property holds true at a rank $\underline{k}\geq_{\mathrm{grlex}}S(\underline{k}_0)$, which means that $_{\underline{k}}R(\underline{u},y):=\displaystyle\frac{P_{\underline{k}}(\underline{u},y+c_{S(\underline{k})})}{-\omega_0\underline{u}^{\underline{i}_{\underline{k}}}}=-y+\, _{\underline{k}}Q(\underline{u},y)$. Therefore, for $_{\underline{k}}\check{Q}(\underline{u},y)=-\omega_0\, _{\underline{k}}Q(\underline{u},y-c_{S(\underline{k})})\in K\left(\left(u_1^{\mathbb{Z}},\ldots,u_r^\mathbb{Z}\right)\right)^{\textrm{grlex}}_{\textrm{Mod}}\left[y\right]$ which is such that $w\left( _{\underline{k}}\check{Q}(\underline{u}, y)\right)  >_{\textrm{grlex}}\underline{0}$, we can write:
			$$\begin{array}{lcl}
				P_{\underline{k}}(\underline{u},y)&=&\omega_0(y-c_{S(\underline{k})})\underline{u}^{\underline{i}_{\underline{k}}}+ \underline{u}^{\underline{i}_{\underline{k}}} \cdot\, _{\underline{k}}\check{Q}(\underline{u},y)\\
				&=&\pi^P_{\underline{k},\underline{i}_{\underline{k}}}(y)\underline{u}^{\underline{i}_{\underline{k}}}+\pi^P_{\underline{k},S(\underline{i}_{\underline{k}})}(y)\underline{u}^{S(\underline{i}_{\underline{k}})}+ \cdots.
			\end{array}$$
			Since $P_{S(\underline{k})}(\underline{u},y)= P_{\underline{k}}\left(\underline{u},c_{S(\underline{k})}+\underline{u}^{S^2(\underline{k})-S(\underline{k})}y\right)$ and $\underline{i}_{S(\underline{k})}=\underline{i}_{\underline{k}}+S^2(\underline{k})-S(\underline{k})$ by Lemma \ref{lemme:double-simple}, we have that:
			$$ P_{S(\underline{k})}(\underline{u},y)=\left[\omega_0\,y+\pi^P_{\underline{k},\underline{i}_{\underline{k}}+S^2(\underline{k})-S(\underline{k})}(c_{S(\underline{k})})\right]\underline{u}^{\underline{i}_{\underline{k}}+S^2(\underline{k})-S(\underline{k})}+\pi^P_{S(\underline{k}),S(\underline{i}_{S(\underline{k})})}(y)\underline{u}^{S(\underline{i}_{S(\underline{k})})}+\cdots.$$
			But, again by Lemma \ref{lemme:double-simple}, $i_{S^2(\underline{k})}=\underline{i}_{S(\underline{k})}+S^3(\underline{k})-S^2(\underline{k}) >_{\mathrm{grlex}} \underline{i}_{S(\underline{k})}=\underline{i}_{\underline{k}}+S^2(\underline{k})-S(\underline{k})$. So we must have $\pi^P_{S(\underline{k}),\underline{i}_{S(\underline{k})}}(c_{S^2(\underline{k})})=0$, i.e. $c_{S^2(\underline{k})}=\displaystyle\frac{-\pi^P_{\underline{k},\underline{i}_{\underline{k}}+S^2(\underline{k})-S(\underline{k})}(c_{S(\underline{k})})}{\omega_0}$. It follows that:
			$$P_{S(\underline{k})}(\underline{u},y)=\omega_0\left(y-c_{S^2(\underline{k})}\right)\underline{u}^{\underline{i}_{\underline{k}}+S^2(\underline{k})-S(\underline{k})}+\pi^P_{S(\underline{k}),S(\underline{i}_{S(\underline{k})})}(y)\underline{u}^{S(\underline{i}_{S(\underline{k})})}+\cdots,$$
			Since, by definition, $_{S(\underline{k})}R(\underline{u},y):=\displaystyle\frac{P_{S(\underline{k})}(\underline{u},y+c_{S^2(\underline{k})})}{-\omega_0\underline{u}^{\underline{i}_{S(\underline{k})}}}=-y+\, _{S(\underline{k})}Q(\underline{u},y)$, we get that:
			$$\begin{array}{lcl}
				_{S(\underline{k})}R(\underline{u},y)&=&-y-
				\displaystyle\frac{\pi^P_{S(\underline{k}),S(\underline{i}_{S(\underline{k})})}(y+c_{S^2(\underline{k})})}{\omega_0}\underline{u}^{S(\underline{i}_{S(\underline{k})})-\underline{i}_{S(\underline{k})}}+ \cdots\\
				&=&-y+\, _{S(\underline{k})}Q(\underline{u},y),\ \ \ \ _{S(\underline{k})}Q\in K\left(\left(u_1^{\mathbb{Z}},\ldots,u_r^\mathbb{Z}\right)\right)^{\textrm{grlex}}_{\textrm{Mod}}\left[y\right],
			\end{array} $$
			with $w\left( _{\underline{k}}Q\left(\underline{u},y\right)\right) >_{\textrm{grlex}}\underline{0}$ as desired.
			
			To conclude the proof, it suffices to note that the equation $ _{\underline{k}}R(\underline{u},y)=0$ is strongly reduced Henselian if and only if  $ _{\underline{k}}Q\left(\underline{u},0\right)\nequiv 0$, which is equivalent to $z_{S(\underline{k})}$ not being a root of $P$.
			%In particular, $Q_{k}(0,0)=\displaystyle\frac{\partial Q_{k}}{\partial y}(0,0)=0$. So the equation $y=Q_k(u,y)$ is reduced Henselian if and only if $Q_k(u,0)\nequiv 0$, which is equivalent to $z_{k+1}$ not being a root of $P$.
		\end{proof}

		We will need the following lemma:
		\begin{lemma}\label{lemme:coincidence}
			Let $P(\underline{u},y)\in K[[\underline{u}]][y]\setminus\{0\}$ be a  polynomial of degree $\deg_y(P)\leq d_y$ with only simple roots. Assume that $y_0, y_1\in K[[\underline{u}]]$  are two distinct roots. One has that:
			$$\mathrm{ord}_{\underline{u}}\, (y_0-y_1)\leq \ord_{\underline{u}}(\Delta_P(\underline{u})).$$
		\end{lemma}
		\begin{proof}
			Note that the hypothesis imply that $d_y\geq 2$. Let us write $y_1-y_0=\delta_{1,0}$ and $\underline{k}:=w(y_1-y_0)=w(\delta_{1,0})\in \mathbb{N}^r$. By Taylor's Formula, we have:
			$$\begin{array}{lcl}
				P(\underline{u},y_0+\delta_{1,0})&=&0\\
				&=& P(\underline{u},y_0)+\displaystyle\frac{\partial P}{\partial y}(\underline{u},y_0) \delta_{1,0}+\cdots+\displaystyle\frac{1}{d_y!}\displaystyle\frac{\partial^{d_y} P}{\partial y^{d_y}}(\underline{u},y_0){\delta_{1,0}}^{d_y}\\
				&=&\delta_{1,0}\left(\displaystyle\frac{\partial P}{\partial y}(\underline{u},y_0)+\cdots+\displaystyle\frac{1}{d_y!}\displaystyle\frac{\partial^{d_y} P}{\partial y^{d_y}}(\underline{u},y_0){\delta_{1,0}}^{d_y-1}\right).
			\end{array}$$
			Since $\delta_{1,0}\neq 0$ and $\displaystyle\frac{\partial P}{\partial y}(\underline{u},y_0)\neq 0$, one has that:
			$$\displaystyle\frac{\partial P}{\partial y}(\underline{u},y_0)=-\delta_{1,0}\left(\displaystyle\frac{1}{2}\displaystyle\frac{\partial^2 P}{\partial y^2}(\underline{u},y_0)+\cdots+\displaystyle\frac{1}{d_y!}\displaystyle\frac{\partial^{d_y} P}{\partial y^{d_y}}(\underline{u},y_0){\delta_{1,0}}^{d_y-2}\right)$$
			The valuation of the right hand side being at least $\underline{k}$, we obtain that:
			$$ w\left(\displaystyle\frac{\partial P}{\partial y}(\underline{u},y_0)\right)\geq_{\mathrm{grlex}} \underline{k}.$$
			But, by Lemma \ref{lemme:partie-princ}, we must have $\mathrm{ord}_u\left(\displaystyle\frac{\partial P}{\partial y}(\underline{u},y_0)\right)\leq  \ord_{\underline{u}}(\Delta_P(\underline{u}))$. So $|\underline{k}|\leq  \ord_{\underline{u}}(\Delta_P(\underline{u}))$.
		\end{proof}

		For the courageous reader, in the case where $y_0$ is a series which is not a polynomial, we deduce from Theorem \ref{theo:FS} and from the generalized  Flajolet-Soria's Formula \ref{theo:formule-FS} a closed-form expression for the coefficients of $y_0$ in terms of the coefficients $a_{\underline{i},j}$ of $P$ and of the coefficients of an initial part  $z_{\underline{k}}$ of $y_0$ sufficiently large, in particular for any $\underline{k}\in\mathbb{N}^r$ such that $|\underline{k}|\geq \ord_{\underline{u}}(\Delta_P(\underline{u}))+1$. Recall that $i_{\underline{k}}=w\left( P_{\underline{k}}(\underline{u},y)\right)$. Note that for such a $\underline{k}$, since $y_0$ is not a polynomial, by Lemma \ref{lemme:coincidence}, $z_{S(\underline{k})}$ cannot be a root of $P$.

		\begin{coro}\label{coro:FS} Let $P(\underline{u},y)\in K[[\underline{u}]][y]\setminus\{0\}$ be a  polynomial of degree $\deg_y(P)\leq d_y$ with only simple roots. 
			Let $\underline{k}\in\mathbb{N}^r$ be such that $|\underline{k}|\geq \ord_{\underline{u}}(\Delta_P(\underline{u}))+1$. For any $\underline{p}>_{\mathrm{grlex}} S(\underline{k})$, consider $\underline{n}:=\underline{p}-S(\underline{k})$. Then:
			$$c_{\underline{p}}=c_{S(\underline{k})+\underline{n}}=\displaystyle\sum_{q=1}^{\mu_{\underline{n}}} \displaystyle\frac{1}{q}\left(\displaystyle\frac{-1}{\omega_0}\right)^q\displaystyle\sum_{|\underline{S}|=q,\ \|\underline{S}\|\geq q-1}\underline{A}^{\underline{S}}\left(\displaystyle\sum_{|\underline{T}_{\underline{S}}|=\|\underline{S}\|-q+1 \atop  g(\underline{T}_{\underline{S}})=\underline{n}+q\underline{i}_{\underline{k}}-(q-1)S(\underline{k})-g(\underline{S})}e_{\underline{T}_{\underline{S}}}\underline{C}^{\underline{T}_{\underline{S}}}\right),$$
			where $\mu_{\underline{n}}$ is as in Theorem \ref{theo:formule-FS} for the equation $ y=\, _{\underline{k}}Q(\underline{u},y)$ of Theorem \ref{theo:FS}, $S=(s_{\underline{i},j})_{ \underline{i}\in\N^r,  \atop j=0,\ldots,d_y}$ with finite support, and as in Notation \ref{nota:FS}, $\underline{A}^{\underline{S}}=\displaystyle\prod_{\underline{i},\, j}a_{\underline{i},j}^{s_{\underline{i},j}}$, $\underline{T}_{\underline{S}}=(t_{\underline{S},\underline{i}})$,  $\underline{C}^{\underline{T}_{\underline{S}}}=\displaystyle\prod_{\underline{i}=\underline{0}}^{S(\underline{k})}c_{\underline{i}}^{t_{\underline{S},\underline{i}}}$, %$\|T_S\|=\displaystyle\sum_{i=1}^{k+1}i\,t_{S,i}$,
			and $e_{\underline{T}_{\underline{S}}}\in\mathbb{N}$ is of the form:\\
			\begin{center}
				 $\ \ e_{\underline{T}_{\underline{S}}}=$\\
				$\displaystyle\sum_{\left(n^{\underline{l},m}_{\underline{i},j,\underline{L}}\right)} \displaystyle\frac{q!}{\displaystyle\prod_{{\underline{l} =S(\underline{i}_{\underline{k}})-\underline{i}_{\underline{k}},\ldots,\atop d_yS(\underline{k})+(d_u,0,\ldots,0)-\underline{i}_{\underline{k}}}  \atop m=0,\ldots,m_{\underline{l}}}\displaystyle\prod_{|\underline{i}|=0,\ldots,d_u \atop j=m,\ldots,d_y}\displaystyle\prod_{|\underline{L}|=j-m \atop g(\underline{L})=\underline{l}+\underline{i}_{\underline{k}}-mS(\underline{k})-\underline{i}}n^{\underline{l},m}_{\underline{i},j,\underline{L}}!} \displaystyle\prod_{{\underline{l}=S(\underline{i}_{\underline{k}})-\underline{i}_{\underline{k}},\ldots,\atop d_y S(\underline{k})+(d_u,0,\ldots,0)-\underline{i}_{\underline{k}}} \atop m=0,\ldots,m_{\underline{l}}}\displaystyle\prod_{|\underline{i}|=0,\ldots,d_u \atop j=m,\ldots,d_y}\displaystyle\prod_{|\underline{L}|=j-m \atop g(\underline{L})=\underline{l}+\underline{i}_{\underline{k}}-mS(\underline{k})-\underline{i}}\left(\displaystyle\frac{j!}{m!\,L!}\right)^{n^{\underline{l},m}_{\underline{i},j,\underline{L}}},$ 
			\end{center}			
			where we denote $m_{\underline{l}}:=\min\left\{d_y,\ \max\left\{m\in\mathbb{N}\, / \, mS(\underline{k})\leq_{\mathrm{grlex}}\underline{l }+\underline{i}_{\underline{k}}\right\}\right\}$,\\ $\underline{L}=\underline{L}_{\underline{i},j}^{\underline{l},m}=\left(l_{\underline{i},j,\underline{0}}^{\underline{l},m},\ldots,l_{\underline{i},j,S(\underline{k})}^{\underline{l},m}\right)$, %$L!=\displaystyle\prod_{t}l_t!$
			and where the sum is taken over the set of tuples\\  $\left(n^{\underline{l},m}_{\underline{i},j,\underline{L}}\right)_{\underline{l}= S(\underline{i}_{\underline{k}})-\underline{i}_{\underline{k}},\ldots,d_yS(\underline{k})+(d_u,0,\ldots,0)-\underline{i}_{\underline{k}},\ m=0,\ldots,m_{\underline{l}} \atop |\underline{i}|=0,\ldots,d_u,\ j=m,\ldots,d_y,\ |\underline{L}|=j-m,\  g(\underline{L})=\underline{l}+\underline{i}_{\underline{k}}-mS(\underline{k})-\underline{i}}$ such that:
			\begin{center}
				$\displaystyle\sum_{\underline{l},m}\displaystyle\sum_{\underline{L}} n^{\underline{l},m}_{\underline{i},j,\underline{L}}=s_{\underline{i},j}$,  $\ \ \displaystyle\sum_{\underline{l},m}\displaystyle\sum_{\underline{i},j} \displaystyle\sum_{\underline{L}}n^{\underline{l},m}_{\underline{i},j,\underline{L}}=q\ \ \ $ and $\ \ \ \displaystyle\sum_{\underline{l},m}\displaystyle\sum_{\underline{i},j} \displaystyle\sum_{\underline{L}}n^{\underline{l},m}_{\underline{i},j,\underline{L}}\underline{L}= \underline{T}_{\underline{S}}$.
			\end{center}
			%où le dernier produit est pris sur tous les $n_L$ tels que $\displaystyle\sum n_L=q$ et $T_S=\displaystyle\sum_{l,m}\displaystyle\sum_{i,j} \displaystyle\sum_{L}n_L L$.
		\end{coro}
		\begin{remark}\label{rem:coeff-e_T}
			Note that the coefficients $e_{\underline{T}_{\underline{S}}}$ are indeed natural numbers, since they are sums of products of multinomial coefficients because  $\displaystyle\sum_{\underline{l},m}\displaystyle\sum_{\underline{i},j}\displaystyle\sum_{\underline{L}} n^{\underline{l},m}_{\underline{i},j,\underline{L}}=q$ and $m+|\underline{L}|=j$. In fact,  $\displaystyle\frac{1}{q}e_{\underline{T}_{\underline{S}}}\in\mathbb{N}$ by Remark \ref{rem:entier} as we will see along the proof.
		\end{remark}
		\begin{proof}
			We get started by computing the coefficients of $\omega_0\underline{u}^{\underline{i}_{\underline{k}}}\, _{\underline{k}}R$,  in order to get those of $\, _{\underline{k}}Q$:
			$$\begin{array}{lcl}
				-\omega_0\underline{u}^{\underline{i}_{\underline{k}}}\, _{\underline{k}}R&=&P_{\underline{k}}(\underline{u},\,y+c_{S(\underline{k})})\\
				&=&P(\underline{u},z_{S(\underline{k})}+\underline{u}^{S(\underline{k})}y)\\
				&=& \displaystyle\sum_{\underline{i}\in \N^r\, ,\ j=0,\ldots,d_y}a_{\underline{i},j}\underline{u}^{\underline{i}}\left(z_{S(\underline{k})}+\underline{u }^{S(\underline{k})}y\right)^{j}\\
				&=& \displaystyle\sum_{\underline{i}\in \N^r\, ,\ j=0,\ldots,d_y}a_{\underline{i},j}\underline{u}^{\underline{i}}\displaystyle\sum_{m=0}^{j}\displaystyle\frac{j!}{m!\,(j-m)!}z_{S(\underline{k})}^{j-m}\underline{u}^{mS(\underline{k})}y^m.
			\end{array}$$
			For $L=(l_{\underline{0} },\cdots,l_{S(\underline{k})})$, we denote %$L!:=l_1!\cdots l_{k+1}!$, $|L|:=l_1+ \cdots+l_{k+1}$, $\|L\|:=1l_1+\cdots+(k+1)l_{k+1}$, 
			$\underline{C}^{\underline{L}}:=c_{\underline{0}}^{l_{\underline{0}}}\cdots c_{S(\underline{k})}^{l_{S(\underline{k})}}$. One  has that:
			$$z_{S(\underline{k})}^{j-m}=\displaystyle\sum_{|\underline{L}|=j-m}\displaystyle\frac{(j-m)!}{\underline{L}!}\underline{C}^{\underline{L}}\underline{u}^{g(\underline{L})}.$$
			So:
			$$ -\omega_0\underline{u}^{\underline{i}_{\underline{k}}}\, _{\underline{k}}R=\displaystyle\sum_{m=0}^{d_y} \displaystyle\sum_{\underline{i}\in \N^r\atop j=m,\ldots,d_y}a_{\underline{i},j}\displaystyle\sum_{|\underline{L}|=j-m}\displaystyle\frac{j!}{m!\,\underline{L}!}\underline{C}^{\underline{L}}\underline{u}^{g(\underline{L}) +mS(\underline{k})+\underline{i}}\,y^m.$$
			We set $\underline{\hat{l}}=g(\underline{L})+mS(\underline{k})+\underline{i}$. It verifies: $\underline{\hat{l}}\geq mS(\underline{k})$. Thus: 
			$$-\omega_0\underline{u}^{\underline{i}_{\underline{k}}}\, _{\underline{k}}R=\displaystyle\sum_{m=0,\ldots,d_y } \displaystyle\sum_{\hat{l}\,\geq\, mS(\underline{k})} \displaystyle\sum_{\underline{i}\,\leq\, \underline{\hat{l}}- mS(\underline{k}) \atop j=m,\ldots,d_y}a_{\underline{i},j}\displaystyle\sum_{|\underline{L}|=j-m \atop g(\underline{L})=\underline{\hat{l}}-mS(\underline{k})-\underline{i}}\displaystyle\frac{j!}{m!\,\underline{L}!}\underline{C}^{\underline{L}}\underline{u}^{\underline{\hat{l}}}y^m.$$
			Since $_{\underline{k}}R(\underline{u},y)=-y+\, _{\underline{k}}Q(\underline{u},y)$ with $w(\, _{\underline{k}}Q(\underline{u},y))>_{\mathrm{grlex}} \underline{0}$, the coefficients of $_{\underline{k}}Q$ are obtained for $\underline{\hat{l}}\geq_{\mathrm{grlex}}S(\underline{i}_{\underline{k}})$. 
			%??\footnote{Note that our assumptions ensure that $ _{\underline{k}}Q(\underline{u},y)\neq 0$, so $S(\underline{i}_{\underline{k}}) \leq_{\mathrm{grlex}} d_yS(\underline{k})+(d_u,0,\ldots,0)$.}. 
			We set $\underline{l}:=\underline{\hat{l}}-\underline{i}_{\underline{k}}$
			%, $m_l:=\min\left\{\left\lfloor\displaystyle\frac{l+i_k}{k+1}\right\rfloor,d_y\right\}$
			and  $$m_{\underline{l}}:=\min\left\{d_y,\ \max\left\{m\in\mathbb{N}\, / \, mS(\underline{k})\leq \underline{l }+\underline{i}_{\underline{k}}\right\}\right\}.$$
			% Since $\underline{i}_{S(\underline{k})}-\underline{i}_{\underline{k}}=S^2(\underline{k})-S(\underline{k})$, 
			We obtain: $$_{\underline{k}}Q(\underline{u},y)=\displaystyle\sum_{\underline{l}\,\geq_{\mathrm{grlex}} S(\underline{i}_{\underline{k}})-\underline{i}_{\underline{k}} \atop m=0,\ldots,m_{\underline{l}}}b_{\underline{l},m}\underline{u}^{\underline{l}}y^m,$$
			with:
			$$b_{\underline{l},m}=\displaystyle\frac{-1}{\omega_0}\displaystyle\sum_{\underline{i}\,\leq\, \underline{l}+\underline{i}_{\underline{k}}- mS(\underline{k})  \atop j=m,\ldots,d_y}a_{\underline{i},j}\displaystyle\sum_{|\underline{L}|=j-m \atop g(\underline{L})=\underline{l}+\underline{i}_{\underline{k}}-mS(\underline{k})-\underline{i}}\displaystyle\frac{j!}{m!\,\underline{L}!}\underline{C}^{\underline{L}}.$$
			According to Lemma \ref{lemme:partie-princ}, Theorem \ref{theo:FS} and Lemma \ref{lemme:coincidence}, we are in position to apply  the generalized Flajolet-Soria's Formula of Theorem \ref{theo:formule-FS}  in order to compute the coefficients of the solution $t_{S(\underline{k})}=c_{S^2(\underline{k})}\underline{u}^{S^2(\underline{k})-S(\underline{k})}+c_{S^3(\underline{k})}\underline{u}^{S^3(\underline{k})-S(\underline{k})}+\cdots$.  Thus, denoting $\underline{B}:=(b_{\underline{l},m})$, $\underline{Q}:=(q_{\underline{l},m})$ with finite support and $\underline{B}^{\underline{Q}}:=\displaystyle\prod_{\underline{l},m} b_{\underline{l},m}^{q_{\underline{l},m}}$  for $\underline{l}\geq_{\mathrm{grlex}}S(\underline{i}_{\underline{k}})-\underline{i}_{\underline{k}}$ and $m=0,\ldots,m_{\underline{l}}$, we obtain for $\underline{n}>_{\mathrm{grlex}}\underline{0}$: 
			$$c_{S(\underline{k})+\underline{n}}=\displaystyle\sum_{q=1}^{\mu_{\underline{n}}}\displaystyle\frac{1}{q}\displaystyle\sum_{|\underline{Q}|=q,\,\|\underline{Q}\|=q-1 ,\, g(\underline{Q})=\underline{n}}\displaystyle\frac{q!}{\underline{Q}!}\underline{B}^{\underline{Q}}.$$
			As in Remark \ref{rem:entier} (1), the previous sum is finite, and as in Remark \ref{rem:entier} (2), we have $\displaystyle\frac{1}{q}\cdot\displaystyle\frac{q!}{\underline{Q}!}\in\mathbb{N}$.
			Let us compute:
			\begin{equation}\label{equ:b_lm}
				\begin{array}{l}
					\begin{array}{lcl}
						b_{\underline{l},m}^{q_{\underline{l},m}}&=&\left(\displaystyle\frac{-1}{\omega_0}\right)^{q_{\underline{l},m}}\left(\displaystyle\sum_{\underline{i}\,\leq\, \underline{l}+\underline{i}_{\underline{k}}- mS(\underline{k})   \atop j=m,\ldots,d_y}a_{\underline{i},j}\displaystyle\sum_{|\underline{L}|=j-m \atop g(\underline{L})=\underline{l} +\underline{i}_{\underline{k}}-mS(\underline{k})-\underline{i}}\displaystyle\frac{j!}{m!\,\underline{L}!}\underline{C}^{\underline{L}}\right)^{q_{\underline{l},m}}\\
						&=&\left(\displaystyle\frac{-1}{\omega_0}\right)^{q_{\underline{l},m}}\displaystyle\sum_{|\underline{M}_{\underline{l},m}|=q_{\underline{l},m}}\displaystyle\frac{q_{\underline{l},m}!}{\underline{M}_{\underline{l},m}!} \underline{A}^{\underline{M}_{\underline{l},m}} \displaystyle\prod_{\underline{i}\,\leq\, \underline{l}+\underline{i}_{\underline{k}}- mS(\underline{k})   \atop j=m,\ldots,d_y}\left(\displaystyle\sum_{|\underline{L}|=j-m \atop g(\underline{L})=\underline{l}+\underline{i}_{\underline{k}}-mS(\underline{k})- \underline{i}}\displaystyle\frac{j!}{m!\,\underline{L}!}\underline{C}^{\underline{L}}\right)^{m^{\underline{l},m}_{\underline{i},j}}
					\end{array}\\
					\textrm{where } \underline{M}_{\underline{l},m}=(m^{\underline{l},m}_{\underline{i},j})\textrm{ for } \underline{i}\leq\underline{l}+\underline{i}_{\underline{k}}- mS(\underline{k})  ,\  j=0,\ldots,d_y\textrm{ and }m^{\underline{l},m}_{\underline{i},j}=0\textrm{ for }j<m.
				\end{array}
			\end{equation}
			Note that, in the previous formula, $\left(-\omega_0\right)^{q_{\underline{l},m}}b_{\underline{l},m}^{q_{\underline{l},m}}$ is the evaluation at $\underline{A}$ and $\underline{C}$ of a polynomial with coefficients in $\mathbb{N}$. Since  $\displaystyle\frac{1}{q}\cdot\displaystyle\frac{q!}{\underline{Q}!}\in\mathbb{N}$, the expansion of $\left(-\omega_0\right)^{q}\displaystyle\frac{1}{q}\cdot\displaystyle\frac{q!}{\underline{Q}!}\underline{B}^{\underline{Q}}$ as a polynomial in $\underline{A}$ and $\underline{C}$ will only have natural numbers as coefficients.

			Let us expand the expression $\displaystyle\prod_{\underline{i}\,\leq\, \underline{l}+\underline{i}_{\underline{k}}- mS(\underline{k})   \atop j=m,\ldots,d_y}\left(\displaystyle\sum_{|\underline{L}|=j-m \atop g(\underline{L})=\underline{l}+\underline{i}_{\underline{k}}-mS(\underline{k})-\underline{i}}\displaystyle\frac{j!}{m!\,\underline{L}!}\underline{C}^{\underline{L}}\right)^{m^{\underline{l},m}_{\underline{i},j}}$.
			For each $(\underline{l},m,\underline{i},j)$, we enumerate the terms $\displaystyle\frac{j!}{m!\,\underline{L}!}\underline{C}^{\underline{L}}$ with $h=1,\ldots,\alpha_{\underline{i},j}^{\underline{l},m}$. Subsequently:
			$$\begin{array}{lcl}
				\left(\displaystyle\sum_{|\underline{L}|=j-m \atop g(\underline{L})=\underline{l}+\underline{i}_{\underline{k}}-mS(\underline{k})-\underline{i}}\displaystyle\frac{j!}{m!\,\underline{L}!}\underline{C}^{\underline{L}}\right)^{m^{\underline{l},m}_{\underline{i},j}}&=& \left(\displaystyle\sum_{h=1}^{\alpha_{\underline{i},j}^{\underline{l},m}}\displaystyle\frac{j!}{m!\,\underline{L}_{\underline{i},j,h}^{\underline{l},m}!} \underline{C}^{\underline{L}_{\underline{i},j,h}^{\underline{l},m}}\right)^{ m^{\underline{l},m}_{\underline{i},j}}\\
				&=& \displaystyle\sum_{|\underline{N}^{\underline{l},m}_{\underline{i},j}|=m^{\underline{l},m}_{\underline{i},j}}\displaystyle\frac{m^{\underline{l},m}_{\underline{i},j}!}{\underline{N}^{\underline{l},m}_{\underline{i},j}!}\left( \displaystyle\prod_{h=1}^{\alpha_{\underline{i},j}^{\underline{l},m}} \left(\displaystyle\frac{j!}{m!\,\underline{L}_{\underline{i},j,h}^{\underline{l},m}!}\right)^{ n^{\underline{l},m}_{\underline{i},j,h}}\right) \underline{C}^{\sum_{h=1}^{\alpha^{\underline{l},m}_{\underline{i},j}} n^{\underline{l},m}_{\underline{i},j,h} \underline{L}_{\underline{i},j,h}^{\underline{l},m}},
			\end{array} $$
			where $\underline{N}^{\underline{l},m}_{\underline{i},j}= \left(n^{\underline{l},m}_{\underline{i},j,h}\right)_{h=1,\ldots,\alpha_{\underline{i},j}^{\underline{l},m}}$, %
			$\, \underline{N}^{\underline{l},m}_{\underline{i},j}!= \displaystyle\prod_{h=1}^{\alpha_{\underline{i},j}^{\underline{l},m}} n^{\underline{l},m}_{\underline{i},j,h}!$. 
			Denoting  \[\underline{H}_{\underline{l},m}=\left(h^{\underline{l},m}_{\underline{0}},\ldots,h^{\underline{l},m}_{S(\underline{k})}\right):= \displaystyle\sum_{\underline{i}\,\leq\, \underline{l}+\underline{i}_{\underline{k}}- mS(\underline{k})\atop j=m,\ldots,d_y}\displaystyle\sum_{h=1}^{ \alpha_{\underline{i},j}^{\underline{l},m}}n^{\underline{l},m}_{\underline{i},j,h} \underline{L}_{\underline{i},j,h}^{\underline{l},m},\] one computes:
			\begin{equation}\label{equ:Ulm}
				\begin{array}{lcl}
					|\underline{H}_{\underline{l},m}|&=&\displaystyle\sum_{\underline{i}\,\leq\, \underline{l}+\underline{i}_{\underline{k}}- mS(\underline{k}) \atop j=m,\ldots,d_y}\displaystyle\sum_{h=1}^{\alpha_{\underline{i},j}^{\underline{l},m}} n^{\underline{l},m}_{\underline{i},j,h}|\underline{L}_{\underline{i},j,h}^{\underline{l},m}|\\
					&=&\displaystyle\sum_{\underline{i}\,\leq\, \underline{l}+\underline{i}_{\underline{k}}- mS(\underline{k}) \atop j=m,\ldots,d_y}\left(\displaystyle\sum_{h=1}^{\alpha_{\underline{i},j}^{\underline{l},m}} n^{\underline{l},m}_{\underline{i},j,h}\right)(j-m)\\
					&=&\displaystyle\sum_{\underline{i}\,\leq\, \underline{l}+\underline{i}_{\underline{k}}- mS(\underline{k}) \atop j=m,\ldots,d_y}m^{\underline{l},m}_{\underline{i},j}(j-m)\\
					&=& \|\underline{M}_{\underline{l},m}\|-m\,q_{\underline{l},m}.
			\end{array} \end{equation} 
			Likewise, one computes:
			\begin{equation}\label{equ:Ulm2}\begin{array}{lcl}
					g(\underline{H}_{\underline{l},m})&=&\displaystyle\sum_{\underline{i}\,\leq\, \underline{l}+\underline{i}_{\underline{k}}- mS(\underline{k}) \atop j=m,\ldots,d_y}\displaystyle\sum_{h=1}^{\alpha_{\underline{i},j}^{\underline{l},m}} n^{\underline{l},m}_{\underline{i},j,h}g(\underline{L}_{\underline{i},j,h}^{\underline{l},m})\\
					&=&\displaystyle\sum_{\underline{i}\,\leq\, \underline{l}+\underline{i}_{\underline{k}}- mS(\underline{k}) \atop j=m,\ldots,d_y}\left(\displaystyle\sum_{h=1}^{\alpha_{\underline{i},j}^{\underline{l},m}} n^{\underline{l},m}_{\underline{i},j,h}\right)(\underline{l}+\underline{i}_{\underline{k}}-mS(\underline{k})-\underline{i})\\
					&=&\displaystyle\sum_{\underline{i}\,\leq\, \underline{l}+\underline{i}_{\underline{k}}- mS(\underline{k}) \atop j=m,\ldots,d_y}m^{\underline{l},m}_{\underline{i},j}(\underline{l}+\underline{i}_{\underline{k}}-mS(\underline{k})-\underline{i})\\
					&=& q_{\underline{l},m}[\underline{l}+\underline{i}_{\underline{k}}-mS(\underline{k})]-g(\underline{M}_{\underline{l},m}).
			\end{array} \end{equation} 
			So, according to Formula (\ref{equ:b_lm}) and the new way of writing the expression\\ $\displaystyle\prod_{\underline{i}\,\leq\, \underline{l}+\underline{i}_{\underline{k}}- mS(\underline{k}) \atop j=m,\ldots,d_y}\left(\displaystyle\sum_{|\underline{L}|=j-m \atop g(\underline{L})=\underline{l}+\underline{i}_{\underline{k}}-mS(\underline{k})-\underline{i}}\displaystyle\frac{j!}{m!\,\underline{L}!}\underline{C}^{\underline{L}}\right)^{m^{\underline{l},m}_{\underline{i},j}}$, we obtain:
			\begin{center}
				$\begin{array}{lcl}
					b_{\underline{l},m}^{q_{\underline{l},m}}&=& \left(\displaystyle\frac{-1}{\omega_0}\right)^{q_{\underline{l},m}}\displaystyle\sum_{|\underline{M}_{\underline{l},m}|=q_{\underline{l},m}}\underline{A}^{\underline{M}_{\underline{l},m}}\displaystyle\sum_{|\underline{H}_{\underline{l},m}|=\|\underline{M}_{\underline{l},m}\|-m\,q_{\underline{l},m} \atop g(\underline{H}_{\underline{l},m})=q_{\underline{l},m}[\underline{l}+\underline{i}_{\underline{k}}-mS(\underline{k})]-g(\underline{M}_{\underline{l},m})} d_{\underline{H}_{\underline{l},m}}\underline{C}^{\underline{H}_{\underline{l},m}}\\
					&& \textrm{ with }d_{\underline{H}_{\underline{l},m}}:=\displaystyle\sum_{ \left(\underline{N}^{\underline{l},m}_{\underline{i},j}\right)}\displaystyle\frac{q_{\underline{l},m}!}{\displaystyle\prod_{\underline{i}\,\leq\, \underline{l}+\underline{i}_{\underline{k}}- mS(\underline{k}) \atop j=m,\ldots,d_y}\underline{N}^{\underline{l},m}_{\underline{i},j}!}\displaystyle\prod_{\underline{i}\,\leq\, \underline{l}+\underline{i}_{\underline{k}}- mS(\underline{k}) \atop j=m,\ldots,d_y}\displaystyle\prod_{h=1}^{\alpha_{\underline{i},j}^{\underline{l},m}} \left(\displaystyle\frac{j!}{m!\,\underline{L}_{\underline{i},j,h}^{\underline{l},m}!}\right)^{n^{\underline{l},m}_{\underline{i},j,h}},
				\end{array}$
			\end{center} 
			\textrm{ where the sum is taken over }$$\left\{\left(\underline{N}^{\underline{l},m}_{\underline{i},j}\right)_{\underline{i}\,\leq\, \underline{l}+\underline{i}_{\underline{k}}- mS(\underline{k}) \atop j=m,\ldots,d_y}\textrm{ such that }|\underline{N}^{\underline{l},m}_{\underline{i},j}|=m^{\underline{l},m}_{\underline{i},j}\textrm{ and }\displaystyle\sum_{\underline{i}\,\leq\, \underline{l}+\underline{i}_{\underline{k}}- mS(\underline{k}) \atop j=m,\ldots,d_y}\sum_{h=1}^{\alpha_{\underline{i},j}^{\underline{l},m}} n^{\underline{l},m}_{\underline{i},j,h} \underline{L}_{\underline{i},j,h}^{\underline{l},m}=\underline{H}_{\underline{l},m}\right\}.$$
			Note that, if the latter set is empty, then $d_{\underline{H}_{\underline{l},m}}=0$.\\
			
			Recall that we consider $\underline{Q}:=(q_{\underline{l},m})$ with finite support and such that $|\underline{Q}|=q$, $\|\underline{Q}\|=q-1$ and $g(\underline{Q})=\underline{n}$. We deduce that:
			$$\begin{array}{lcl}
				\underline{B}^{\underline{Q}}&=&\displaystyle\prod_{\underline{l}\,\geq_{\mathrm{grlex}}S(\underline{i}_{\underline{k}})-\underline{i}_{\underline{k}}\atop m=0,\ldots,m_{\underline{l}}}b_{\underline{l},m}^{q_{\underline{l},m}}\\
				&=& \left(\displaystyle\frac{-1}{\omega_0}\right)^{q}\displaystyle\prod_{\underline{l},m} \left[\displaystyle\sum_{|\underline{M}_{\underline{l},m}|=q_{\underline{l},m}}\underline{A}^{\underline{M}_{\underline{l},m}}\displaystyle\sum_{|\underline{H}_{\underline{l},m}|=\|\underline{M}_{\underline{l},m}\|-m\,q_{\underline{l},m} \atop \|\underline{H}_{\underline{l},m}\|=q_{\underline{l},m}\left(\underline{l}+\underline{i}_{\underline{k}}-mS(\underline{k})\right)-g(\underline{M}_{\underline{l},m})}d_{\underline{H}_{\underline{l},m}} \underline{C}^{\underline{H}_{\underline{l},m}}\right].
			\end{array}$$ 
			Now, in order to expand the latter product of sums, we consider the corresponding sets:	\begin{center} $\mathcal{S}_{\underline{Q}}:=$
		
				$\left\{\displaystyle\sum_{\underline{l},m}\underline{M}_{\underline{l},m}\ \ / \ \   \exists (\underline{M}_{\underline{l},m})\ \mathrm{s.t.}\ |\underline{M}_{\underline{l},m}|=q_{\underline{l},m}\textrm{ and }\forall \underline{l},m,\  m^{\underline{l},m}_{\underline{i},j}=0\textrm{ for }j<m \textrm{ or } \underline{i}\,\not\leq\, \underline{l}+\underline{i}_{\underline{k}}- mS(\underline{k})\right\}$	\end{center}
				and, for any   $\underline{S}\in\mathcal{S}_{\underline{Q}}$,

			\begin{center} $ \mathcal{H}_{\underline{Q},\underline{S}}:=\left\{\left(\underline{H}_{\underline{l},m}\right)\ \  / \ \  \exists (\underline{M}_{\underline{l},m})\ \mathrm{s.t.}\ |\underline{M}_{\underline{l},m}|=q_{\underline{l},m}\textrm{ and }\forall \underline{l},m,\  m^{\underline{l},m}_{\underline{i},j}=0\textrm{ for }j<m  \textrm{ or } \underline{i}\,\not\leq\, \underline{l}+\underline{i}_{\underline{k}}- mS(\underline{k}), \right.$\\
				$\ \ \ \ \    \ \ \  \ \  \    \left. \displaystyle\sum_{\underline{l},m}\underline{M}_{\underline{l},m}=\underline{S},\,  |\underline{H}_{\underline{l},m}|=\|\underline{M}_{\underline{l},m}\|-m\,q_{\underline{l},m} \textrm{ and }g(\underline{H}_{\underline{l},m})=q_{\underline{l},m}\left(\underline{l}+\underline{i}_{\underline{k}}-mS(\underline{k})\right)-g(\underline{M}_{\underline{l},m})  \displaystyle\frac{}{}\right\}$\end{center}
			and
			\begin{center} $\mathcal{T}_{\underline{Q},\underline{S}}:=\left\{\displaystyle\sum_{\underline{l},m} \underline{H}_{\underline{l},m} \ \  / \ \  \left(\underline{H}_{\underline{l},m}\right)\in  \mathcal{H}_{\underline{Q},\underline{S}} \right\}.$\end{center}
			We have:
			\begin{equation}\label{equ:BQ}
				\begin{array}{lcl}
					\underline{B}^{\underline{Q}}&=& \left(\displaystyle\frac{-1}{\omega_0}\right)^{q}\displaystyle\sum_{\underline{S}\in\mathcal{S}_{\underline{Q}}} \underline{A}^{\underline{S}}\displaystyle\sum_{\underline{T}_{\underline{S}} \in\mathcal{T}_{\underline{Q},\underline{S}}} \left(\displaystyle\sum_{ \left(\underline{H}_{\underline{l},m}\right)\in  \mathcal{H}_{\underline{Q},\underline{S}} \atop \sum_{\underline{l},m} \underline{H}_{\underline{l},m}=\underline{T}_{\underline{S}}}\displaystyle\prod_{\underline{l},m} d_{\underline{H}_{\underline{l},m}}\right)       \underline{C}^{\underline{T}_{\underline{S}}}\\
					&=&\left(\displaystyle\frac{-1}{\omega_0}\right)^{q}\displaystyle\sum_{\underline{S}\in\mathcal{S}_{\underline{Q}}} \underline{A}^{\underline{S}}\displaystyle\sum_{\underline{T}_{\underline{S}} \in\mathcal{T}_{\underline{Q},\underline{S}}}e_{\underline{Q},\underline{T}_{\underline{S}}} \underline{C}^{\underline{T}_{\underline{S}}}.\end{array}\end{equation}
			\textrm{ where }: $$e_{\underline{Q},\underline{T}_{\underline{S}}}:= \displaystyle\sum_{\left(\underline{N}^{\underline{l},m}_{\underline{i},j}\right)} \displaystyle\frac{\displaystyle\prod_{\underline{l},m}q_{\underline{l},m}!}{\displaystyle\prod_{\underline{l},m}\displaystyle\prod_{\underline{i},j} \underline{N}^{\underline{l},m}_{\underline{i},j}!} \displaystyle\prod_{\underline{l},m} \displaystyle\prod_{\underline{i},j}\displaystyle\prod_{h}\left(\displaystyle\frac{j!}{m!\,L_{\underline{i},j,h}^{\underline{l},m}!}\right)^{n^{\underline{l},m}_{\underline{i},j,h}}$$
			and  where the previous sum is taken over:  
			\begin{center}$ \mathcal{E}_{\underline{Q},\underline{T}_{\underline{S}}}:=\left\{\left( \underline{N}^{\underline{l},m}_{\underline{i},j}\right)_{\underline{l}\,\geq_{\mathrm{grlex}}S(\underline{i}_{\underline{k}})-\underline{i}_{\underline{k}},\, m=0,\ldots,m_{\underline{l}} \atop \underline{i}\,\leq\, \underline{l}+\underline{i}_{\underline{k}}- mS(\underline{k}),\ j=m,\ldots,d_y}\ \ /\ \  \forall \underline{i},j,\ \displaystyle\sum_{\underline{l},m} \sum_{h=1}^{\alpha_{\underline{i},j}^{\underline{l},m}}n^{\underline{l},m}_{\underline{i},j,h}=s_{\underline{i},j}, \right.$\\
				$\left. \ \ \ \ \ \  \ \ \ \ \  \ \ \ \ \ \ \ \  \  \ \  \ \ \ \ \ \ \ \  \ \ \ \ \ \  \ \ \ \ \ \ \ \ \ \  \forall \underline{l},m,\ \displaystyle\sum_{\underline{i},j}|\underline{N}^{\underline{l},m}_{\underline{i},j}|=q_{\underline{l},m}, \textrm{ and }  \displaystyle\sum_{\underline{l},m} \displaystyle\sum_{\underline{i}, j}\sum_{h=1}^{\alpha_{\underline{i},j}^{\underline{l},m}} n^{\underline{l},m}_{\underline{i},j,h} \underline{L}_{\underline{i},j,h}^{\underline{l},m} =\underline{T}_{\underline{S}} \right\}.$\end{center}
			Note that, if the latter set is empty, then $e_{\underline{Q},\underline{T}_{\underline{S}}}=0$.\\
			
			Observe that $\displaystyle\frac{1}{q}\displaystyle\frac{q!}{Q!}e_{\underline{Q},\underline{T}_{\underline{S}}}$ lies in $\mathbb{N}$ as a coefficient of $(-{\omega_0}) ^{q}\displaystyle\frac{1}{q}\displaystyle\frac{q!}{Q!}B^{Q}$ as seen before.
			Note also that, for any $\underline{Q}$ and for any $\underline{S}\in\mathcal{S}_{\underline{Q}}$,  $|\underline{S}|=\displaystyle\sum_{\underline{l},m}q_{\underline{l},m}=q$ and $\|\underline{S}\|\geq \displaystyle\sum_{\underline{l},m}mq_{\underline{l},m}=\|\underline{Q}\|=q-1$. Moreover, for any $\underline{T}_{\underline{S}}\in\mathcal{T}_{\underline{Q},\underline{S}}$: 
			$$\begin{array}{lcl}
				|\underline{T}_{\underline{S}}|&=&\displaystyle\sum_{\underline{l},m} \|\underline{M}_{\underline{l},m}\|-m\,q_{\underline{l},m}\\
				&=&\|\underline{S}\|-\|\underline{Q}\|\\
				&=&\|\underline{S}\|-q+1
			\end{array}$$ and:
			$$\begin{array}{lcl}
				g(\underline{T}_{\underline{S}})&=& \displaystyle\sum_{l,m}q_{l,m}\left(\underline{l}+\underline{i}_{\underline{k}}-mS(\underline{k})\right)-g(\underline{M}_{\underline{l},m})\\
				&=&g(\underline{Q})+|\underline{Q}|\,\underline{i}_{\underline{k}}-\|\underline{Q}\|\,S(\underline{k})-g(\underline{S})\\
				&=&\underline{n}+q\,\underline{i}_{\underline{k}}-(q-1)\,S(\underline{k})-g(\underline{S}).
			\end{array}$$ Let us show that:
			\begin{equation}\label{equ:BQAS}
				\begin{array}{lcl}
					\displaystyle\sum_{|\underline{Q}|=q,\,\|\underline{Q}\|=q-1,\, g(\underline{Q})=\underline{n}}\displaystyle\frac{q!}{\underline{Q}!}\underline{B}^{\underline{Q}}&=&  \left(\displaystyle\frac{-1}{\omega_0}\right)^{q}\displaystyle\sum_{|\underline{S}|=q,\, \|\underline{S}\|\geq q-1}\underline{A}^{\underline{S}}\displaystyle\sum_{|\underline{T}_{\underline{S}}|=\|\underline{S}\|-q+1 \atop g(\underline{T}_{\underline{S}})=\underline{n}+q\underline{i}_{\underline{k}}-(q-1)S(\underline{k})-g(\underline{S})}e_{\underline{T}_{\underline{S}}}\underline{C}^{\underline{T}_{\underline{S}}},\end{array} 
			\end{equation} 
			\textrm{ where } $e_{\underline{T}_{\underline{S}}}:=\displaystyle\sum_{ \left(\underline{N}^{\underline{l},m}_{\underline{i},j}\right)}\displaystyle\frac{q!}{\displaystyle\prod_{\underline{l},m}\displaystyle\prod_{\underline{i},j} \underline{N}^{\underline{l},m}_{\underline{i},j}!}\displaystyle\prod_{\underline{l},m} \displaystyle\prod_{\underline{i},j}\displaystyle\prod_{h}\left(\displaystyle\frac{j!}{m!\,\underline{L}_{\underline{i},j,h}^{\underline{l},m}!}\right)^{n^{\underline{l},m}_{\underline{i},j,h}}$  and  \textrm{ where the sum is taken over }
			
			\begin{center}
				$\mathcal{E}_{\underline{T}_{\underline{S}}}:=\left\{\left(\underline{N}^{\underline{l},m}_{\underline{i},j}\right)_{\underline{l}\,\geq_{\mathrm{grlex}}S(\underline{i}_{\underline{k}})-\underline{i}_{\underline{k}},\, m=0,\ldots,m_{\underline{l}} \atop \underline{i}\,\leq\, \underline{l}+\underline{i}_{\underline{k}}- mS(\underline{k}),\ j=m,\ldots,d_y}\textrm{ s.t. }\displaystyle\sum_{\underline{l},m}\displaystyle\sum_{h} n^{\underline{l},m}_{\underline{i},j,h}=s_{\underline{i},j},\ 
				\displaystyle\sum_{\underline{l},m}\displaystyle\sum_{\underline{i},j}|\underline{N}^{\underline{l},m}_{\underline{i},j}|=q\right.$\\
				$ \  \ \ \ \ \ \ \ \  \  \ \  \ \ \ \ \ \ \ \  \ \ \ \ \ \  \ \ \ \ \ \ \ \ \ \  \  \ \ \ \ \ \ \ \  \  \ \  \ \ \ \ \ \ \ \  \ \ \ \ \ \  \ \ \ \ \ \ \  \ \ \ \ \ \ \  \ \ \  \left.\textrm{ and } \displaystyle\sum_{\underline{l},m}\displaystyle\sum_{\underline{i},j} \sum_{h=1}^{\alpha_{\underline{i},j}^{\underline{l},m}}n^{\underline{l},m}_{\underline{i},j,h} \underline{L}_{\underline{i},j,h}^{\underline{l},m}=\underline{T}_{\underline{S}}\right\}.$	 
			\end{center} 
			Note that, if the latter set is empty, then $e_{ \underline{T}_{\underline{S}}}=0$.\\
			Recall that  $\underline{N}^{\underline{l},m}_{\underline{i},j}!= \displaystyle\prod_{h=1}^{\alpha_{\underline{i},j}^{\underline{l},m}} n^{\underline{l},m}_{\underline{i},j,h}!$ and that the $\underline{L}^{\underline{l},m}_{\underline{i},j,h}$'s enumerate the $\underline{L}$'s such that $|\underline{L}|=j-m$ and $g(\underline{L})=\underline{l}+\underline{i}_{\underline{k}}-m\,S(\underline{k})-\underline{i}$ for given $\underline{l},m,\underline{i},j$.
			% Note that $\displaystyle\frac{1}{q}e_{\underline{T}_{\underline{S}}}$ lies in $\mathbb{N}$ as a coefficient of $(-{\omega_0}) ^{q}\displaystyle\frac{1}{q}\displaystyle\frac{q!}{Q!}B^{Q}$ as seen before.

			Let us consider  $\underline{S}$ and  $\underline{T}_{\underline{S}}$ such that $|\underline{S}|=q,\, \|\underline{S}\|\geq q-1$, $|\underline{T}_{\underline{S}}|=\|\underline{S}\|-q+1,\  g(\underline{T}_{\underline{S}})=\underline{n}+q\underline{i}_{\underline{k}}-(q-1)S(\underline{k})-g(\underline{S})$ and such that $\mathcal{E}_{T_S}\neq \emptyset$. Take an element $( n^{l,m}_{i,j,h})\in \mathcal{E}_{\underline{T}_{\underline{S}}}$. Define $ m^{\underline{l},m}_{\underline{i},j}:=\displaystyle\sum_{h=1}^{\alpha_{\underline{i},j}^{\underline{l},m}} n^{\underline{l},m}_{\underline{i},j,h}$ for each $\underline{i},\,j,\,\underline{l},\,m$ with $j\geq m$, and $m^{\underline{l},m}_{\underline{i},j}:=0$ if $j<m$  \textrm{ or } $\underline{i}\,\not\leq\, \underline{l}+\underline{i}_{\underline{k}}- mS(\underline{k})$. Set $\underline{M}_{\underline{l},m}:=(m^{\underline{l},m}_{\underline{i},j})_{\underline{i},j}$ for each $\underline{l},\,m$. So, $\displaystyle\sum_{\underline{l},m}m^{\underline{l},m}_{\underline{i},j}=\displaystyle\sum_{\underline{l},m}\sum_{h=1}^{\alpha_{\underline{i},j}^{\underline{l},m}} n^{\underline{l},m}_{\underline{i},j,h}=s_{\underline{i},j}$, and $\underline{S}=\displaystyle\sum_{\underline{l},m}\underline{M}_{\underline{l},m}$. Define $q_{\underline{l},m}:=\displaystyle\sum_{\underline{i},j}m^{\underline{l},m}_{\underline{i},j}=|\underline{M}_{\underline{l},m}|$ for each $\underline{l},\,m$, and $\underline{Q}:=(q_{\underline{l},m})$. Let us show that $|\underline{Q}|=q$, $ g(\underline{Q})=\underline{n}$ and $\|\underline{Q}\|=q-1$. By definition of $\mathcal{E}_{\underline{T}_{\underline{S}}}$, $$|\underline{Q}|:=\displaystyle\sum_{\underline{l},m}q_{\underline{l},m}= \displaystyle\sum_{\underline{l},m}\displaystyle\sum_{\underline{i},j} \sum_{h=1}^{\alpha_{\underline{i},j}^{\underline{l},m}}n^{\underline{l},m}_{\underline{i},j,h}=q.$$
			Recall that $\|\underline{Q}\|:=\displaystyle\sum_{\underline{l},m}mq_{\underline{l},m}$. We have:
			$$\begin{array}{ll}
				&|\underline{T}_{\underline{S}}|= \left|\displaystyle\sum_{\underline{l},m}\displaystyle\sum_{\underline{i},j} \sum_{h=1}^{\alpha_{\underline{i},j}^{\underline{l},m}}n^{\underline{l},m}_{\underline{i},j,h}\underline{L}_{\underline{i},j,h}^{\underline{l},m}\right|=\|\underline{S}\| -q+1\\
				\Leftrightarrow & \displaystyle\sum_{\underline{l},m}\displaystyle\sum_{\underline{i},j} \sum_{h=1}^{\alpha_{\underline{i},j}^{\underline{l},m}}n^{\underline{l},m}_{\underline{i},j,h}|\underline{L}_{\underline{i},j,h}^{\underline{l},m}|=\displaystyle\sum_{\underline{i},j}js_{\underline{i},j}-q+1\\
				\Leftrightarrow & \displaystyle\sum_{\underline{l},m}\displaystyle\sum_{\underline{i},j} \sum_{h=1}^{\alpha_{\underline{i},j}^{\underline{l},m}}n^{\underline{l},m}_{\underline{i},j,h}(j-m)= \displaystyle\sum_{\underline{i},j}js_{\underline{i},j}-q+1\\
				\Leftrightarrow &\displaystyle\sum_{\underline{i},j} j\displaystyle\sum_{\underline{l},m} \sum_{h=1}^{\alpha_{\underline{i},j}^{\underline{l},m}}n^{\underline{l},m}_{\underline{i},j,h}- \displaystyle\sum_{\underline{l},m}m\displaystyle\sum_{\underline{i},j} \sum_{h=1}^{\alpha_{\underline{i},j}^{\underline{l},m}}n^{\underline{l},m}_{\underline{i},j,h}= \displaystyle\sum_{\underline{i},j}js_{\underline{i},j}-q+1\\
				\Leftrightarrow &\displaystyle\sum_{\underline{i},j} js_{\underline{i},j}-\displaystyle\sum_{\underline{l},m}mq_{\underline{l},m} =\displaystyle\sum_{\underline{i},j}js_{\underline{i},j}-q+1\\
				\Leftrightarrow & \|\underline{Q}\|=q-1.
			\end{array} $$
			Recall that $ g(\underline{Q}):=\displaystyle\sum_{\underline{l},m}q_{\underline{l},m}\underline{l}$. We have:
			$$\begin{array}{ll}
				&g(\underline{T}_{\underline{S}})= g\left(\displaystyle\sum_{\underline{l},m}\displaystyle\sum_{\underline{i},j} \sum_{h=1}^{\alpha_{\underline{i},j}^{\underline{l},m}}n^{\underline{l},m}_{\underline{i},j,h}\underline{L}_{\underline{i},j,u }^{\underline{l},m}\right)=\underline{n}+q\,\underline{i}_{\underline{k}} -(q-1)S(\underline{k}) -g(\underline{S})\\
				\Leftrightarrow & \displaystyle\sum_{\underline{l},m}\displaystyle\sum_{\underline{i},j} \sum_{h=1}^{\alpha_{\underline{i},j}^{\underline{l},m}}n^{\underline{l},m}_{\underline{i},j,h}g(\underline{L}_{\underline{i},j,h}^{\underline{l},m})=\underline{n}+q\,\underline{i}_{\underline{k}} -(q-1)S(\underline{k}) -g(\underline{S})\\
				\Leftrightarrow & \displaystyle\sum_{\underline{l},m}\displaystyle\sum_{\underline{i},j} \sum_{h=1}^{\alpha_{\underline{i},j}^{\underline{l},m}}n^{\underline{l},m}_{\underline{i},j,h}(\underline{l}+\underline{i}_{\underline{k}} -mS(\underline{k}) -\underline{i})= \underline{n}+q\, \underline{i}_{\underline{k}} -(q-1)S(\underline{k}) - g(\underline{S})\\
				\Leftrightarrow & \begin{array}{l}
					\displaystyle\sum_{\underline{l},m}\underline{l}\displaystyle\sum_{\underline{i},j} \sum_{h=1}^{\alpha_{\underline{i},j}^{\underline{l},m}}n^{\underline{l},m}_{\underline{i},j,h}+
					\underline{i}_{\underline{k}} \displaystyle\sum_{\underline{l},m}\displaystyle\sum_{\underline{i},j} \sum_{h=1}^{\alpha_{\underline{i},j}^{\underline{l},m}}n^{\underline{l},m}_{\underline{i},j,h}
					-S(\underline{k}) \displaystyle\sum_{\underline{l},m}m\displaystyle\sum_{\underline{i},j} \sum_{h=1}^{\alpha_{\underline{i},j}^{\underline{l},m}}n^{\underline{l},m}_{\underline{i},j,h}
					\\     
					\ \ \ \ \ \ \ \ \ \ \ \ \ \ \  \ \ \ \ \ \ \ \ \ \ \ \ \ \ \ \ \  \ \ \ \ \ \ \ \ \ \ \ \ \ \ \  -\displaystyle\sum_{\underline{i},j}\underline{i} \displaystyle\sum_{\underline{l},m} \sum_{h=1}^{\alpha_{\underline{i},j}^{\underline{l},m}}n^{\underline{l},m}_{\underline{i},j,h}=\underline{n}+q\, \underline{i}_{\underline{k}} -(q-1)S(\underline{k}) -g(\underline{S}) \\
					
				\end{array}\\
				\Leftrightarrow & \displaystyle\sum_{\underline{l},m}q_{\underline{l},m}\underline{l}+q\,\underline{i}_{\underline{k}}-S(\underline{k}) \displaystyle\sum_{\underline{l},m}m\, q_{\underline{l},m}-\displaystyle\sum_{\underline{i},j} s_{i,j}\underline{i}= \underline{n}+q\, \underline{i}_{\underline{k}} -(q-1)S(\underline{k}) -g(\underline{S}) \\
				\Leftrightarrow & g(\underline{Q})+q\, \underline{i}_{\underline{k}}-\|Q\|S(\underline{k})-g(\underline{S})=\underline{n}+q\, \underline{i}_{\underline{k}} -(q-1)S(\underline{k}) -g(\underline{S}).
			\end{array} $$
			Since $\|\underline{Q}\|=q-1$, we deduce that $g(\underline{Q})=\underline{n}$ as desired. So, $\underline{S}\in \mathcal{S}_{\underline{Q}}$ for $\underline{Q}$ as in the left-hand side of (\ref{equ:BQAS}).\\
			Now, set $\underline{H}_{\underline{l},m}:=\displaystyle\sum_{\underline{i},j} \sum_{h=1}^{\alpha_{\underline{i},j}^{\underline{l},m}}n^{\underline{l},m}_{\underline{i},j,h} \underline{L}_{\underline{i},j,h}^{\underline{l},m}$, so $\displaystyle\sum_{\underline{l},m}\underline{H}_{\underline{l},m}=\underline{T}_{\underline{S}}$. Let us show that $(\underline{H}_{\underline{l},m})\in \mathcal{H}_{\underline{Q},\underline{S}}$, which implies that $\underline{T}_{\underline{S}}\in \mathcal{T}_{\underline{Q},\underline{S}}$ as desired. The existence of $(\underline{M}_{\underline{l},m})$ such that $|\underline{M}_{\underline{l},m}|=q_{\underline{l},m}\textrm{ and } \  m^{\underline{l},m}_{\underline{i},j}=0\textrm{ for }j<m$ and $\displaystyle\sum_{\underline{l},m}\underline{M}_{\underline{l},m}=\underline{S}$ follows by construction. Conditions $|\underline{H}_{\underline{l},m}|=\|\underline{M}_{\underline{l},m}\|-m\,q_{\underline{l},m}  \textrm{ and }g(\underline{H}_{\underline{l},m})=q_{\underline{l},m}[\underline{l}+\underline{i}_{\underline{k}}-mS(\underline{k})]-g(\underline{M}_{\underline{l},m})$ are obtained exactly as in (\ref{equ:Ulm}) and (\ref{equ:Ulm2}). This shows that $(n^{\underline{l },m}_{\underline{i},j,h}) \in \mathcal{E}_{\underline{Q},\underline{T}_{\underline{S}}}$, so:
			$$ \mathcal{E}_{\underline{T}_{\underline{S}}}\subseteq \bigcupdot_{|\underline{Q}|=q,\, g(\underline{Q})=\underline{n},\,\|Q\|=q-1}  \mathcal{E}_{\underline{Q},\underline{T}_{\underline{S}}}. $$
			The reverse inclusion holds trivially since $|\underline{Q}|=q$, so:
			$$ \mathcal{E}_{\underline{T}_{\underline{S}}} = \bigcupdot_{|\underline{Q}|=q,\, g(\underline{Q})=\underline{n},\,\|\underline{Q}\|=q-1}  \mathcal{E}_{\underline{Q},\underline{T}_{\underline{S}}}. $$
			We deduce that:
			$$ e_{\underline{T}_{\underline{S}}}=\displaystyle\sum_{|\underline{Q}|=q,\, g(\underline{Q})=\underline{n},\,\|\underline{Q}\|=q-1}\displaystyle\frac{q!}{\underline{Q}!} e_{\underline{Q},\underline{T}_{\underline{S}}}.$$
			We conclude that any term occuring in the right-hand side of (\ref{equ:BQAS}) comes from a term from the left-hand side. 
			
			Conversely, for any $\underline{Q}$ as in the  left-hand side of  Formula (\ref{equ:BQAS}), $\underline{S}\in\mathcal{S}_{\underline{Q}}$ and $\underline{T}_{\underline{S}} \in \mathcal{T}_{\underline{Q},\underline{S}}$ verify the following conditions:
			$$|\underline{S}|=q,\ \ \|\underline{S}\|\geq q-1,\ \ |\underline{T}_{\underline{S}}|=\|\underline{S}\|-q+1 ,\ \ \|\underline{T}_{\underline{S}}\|=\underline{n}+q\,\underline{i}_{\underline{k}}-(q-1)S(\underline{k})-g(\underline{S})$$
			and
			$$ \mathcal{E}_{\underline{T}_{\underline{S}}} = \bigcupdot_{|\underline{Q}|=q,\, g(\underline{Q})=\underline{n},\,\|\underline{Q}\|=q-1}  \mathcal{E}_{\underline{Q},\underline{T}_{\underline{S}}},\ \ \ \  e_{\underline{T}_{\underline{S}}}=\displaystyle\sum_{|\underline{Q}|=q,\, g(\underline{Q})=\underline{n},\,\|\underline{Q}\|=q-1}\displaystyle\frac{q!}{\underline{Q}!} e_{\underline{Q},\underline{T}_{\underline{S}}}. $$
			Hence, any term occuring in the expansion of $\underline{B}^{\underline{Q}}$ contributes to the right hand side of Formula (\ref{equ:BQAS}).
			
			Thus we obtain  Formula (\ref{equ:BQAS}) from which the statement of Corollary \ref{coro:FS} follows. Note also that:
			$$ \displaystyle\frac{1}{q}e_{\underline{T}_{\underline{S}}}=\displaystyle\sum_{|\underline{Q}|=q,\, g(\underline{Q})=\underline{n},\,\|\underline{Q}\|=q-1}\displaystyle\frac{1}{q}\displaystyle\frac{q!}{\underline{Q}!} e_{\underline{Q},\underline{T}_{\underline{S}}},$$
			so $ \displaystyle\frac{1}{q}e_{\underline{T}_{\underline{S}}}\in\N$.
		\end{proof}
		
		\begin{remark}\label{rem:omega_0}
			We have seen in Theorem \ref{theo:FS} and its proof (see Formula (\ref{equ:p_{k+1}}) with $\underline{k}=\underline{k}_0$) that $\omega_0=(\pi^P_{\underline{k}_0,\underline{i}_{\underline{k}_0}})'(c_{S(\underline{k}_0)})$ is the coefficient of the monomial $\underline{u}^{\underline{i}_{S(\underline{k}_0)}}y$ in the expansion of $P_{S(\underline{k}_0)}(\underline{u},y)=P(\underline{u},c_{\underline{0}}u_r+\cdots+c_{S(\underline{k}_0)}\underline{u}^{S(\underline{k}_0)}+ \underline{u}^{S^2(\underline{k}_0)}y)$, and that  $c_{S^2(\underline{k}_0)}=\displaystyle\frac{-\pi^P_{\underline{k}_0,\underline{i}_{S(\underline{k}_0)}}(c_{S(\underline{k}_0)})}{\omega_0}$ where $\pi^P_{\underline{k}_0,\underline{i}_{S(\underline{k}_0)}}(c_{S(\underline{k}_0)})$ is the coefficient of $\underline{u}^{\underline{i}_{S(\underline{k}_0)}}$ in the expansion of $P_{S(\underline{k}_0)}(\underline{u},y)$. Expanding $P_{S(\underline{k}_0)}(\underline{u},y)$, having done the whole computations, we deduce that:
			$$\left\{\begin{array}{lcl}
				\omega_0&=&\displaystyle\sum_{\underline{i}\,\leq\, \underline{l}+\underline{i}_{\underline{k}}- mS(\underline{k}),\ j=1,..,d_y}\ \ \displaystyle\sum_{|\underline{L}|=j-1,\  g(\underline{L})=\underline{i}_{\underline{k}_0}-S(\underline{k}_0)-\underline{i}}\displaystyle\frac{j!}{\underline{L}!}a_{\underline{i},j}\underline{C}^{\underline{L}}\ ;\\
				c_{S^2(k_0)}&=& \displaystyle\frac{-1}{\omega_0}\displaystyle\sum_{\underline{i}\,\leq\, \underline{l}+\underline{i}_{\underline{k}}- mS(\underline{k}),\ j=0,..,d_y}\ \ \displaystyle\sum_{|\underline{L}|=j,\  g(\underline{L})=\underline{i}_{S(\underline{k}_0)}-\underline{i}}\ \ \displaystyle\frac{j!}{L!}a_{\underline{i},j}\underline{C }^{\underline{L}},
			\end{array}\right. $$
			where $\underline{C}:=\left(c_{\underline{0}},\ldots,c_{S(\underline{k}_0)}\right)$ and $\underline{L}:=\left(l_{\underline{0}},\ldots,l_{S(\underline{k}_0)}\right)$.
		\end{remark}

		%Cela justifie a posteriori notre choix de représentation des solutions avec des termes en $\frac{x_1^{1/p}}{x_2^{q/p}}$ dans notre théorème ci-dessus. 

		% Bibliographic references with the natbib package:
		% Parenthetical: \citep{Bai92} produces (Bailyn 1992).
		% Textual: \citet{Bai95} produces Bailyn et al. (1995).
		% An affix and part of a reference:
		%   \citep[e.g.][Ch. 2]{Bar76}
		%   produces (e.g. Barnes et al. 1976, Ch. 2).
		
		%\begin{thebibliography}{}
		
		%\end{thebibliography}
		\bibliographystyle{amsalpha}

		\newcommand{\etalchar}[1]{$^{#1}$}
		\def\cprime{$'$} \def\cprime{$'$}
		\providecommand{\bysame}{\leavevmode\hbox to3em{\hrulefill}\thinspace}
		\providecommand{\MR}{\relax\ifhmode\unskip\space\fi MR }
		\providecommand{\MRhref}[2]{%
			\href{http://www.ams.org/mathscinet-getitem?mr=#1}{#2}
		}
		\providecommand{\href}[2]{#2}

		% Important: Do not put any empty line here.
		% Use \affiliationthree{} for any address positioned under \affiliationone
		% Use \affiliationfour{}  for any address positioned under \affiliationtwo

		%
	\end{document}